\documentclass{article}
\usepackage{amssymb,amsfonts,amsmath,amsthm,amsopn,amstext,amscd,latexsym,xy,color,mathrsfs}
\usepackage{hyperref}
\usepackage{stmaryrd}

\input xy
\xyoption{all}

\newtheorem{theorem}{Theorem}
\newtheorem{lemma}[theorem]{Lemma}

\newtheorem{proposition}[theorem]{Proposition}
\newtheorem{corollary}[theorem]{Corollary}
\newtheorem{definition}[theorem]{Definition}
\newtheorem{conjecture}[theorem]{Conjecture}
\theoremstyle{remark}

\newtheorem{remark}[theorem]{Remark}

\numberwithin{theorem}{section}
\numberwithin{equation}{section}

\renewcommand{\)}{\textup{)}}

\DeclareMathOperator{\Ad}{Ad}
\DeclareMathOperator{\Bun}{Bun}
\DeclareMathOperator{\Aut}{Aut}
\DeclareMathOperator{\Out}{Out}
\DeclareMathOperator{\GL}{GL}
\DeclareMathOperator{\SL}{SL}
\DeclareMathOperator{\PGL}{PGL}

\DeclareMathOperator{\depth}{depth}

\DeclareMathOperator{\Def}{Def}
\DeclareMathOperator{\PDef}{PDef}

\DeclareMathOperator{\Frac}{Frac}
\DeclareMathOperator{\Frob}{Frob}

\DeclareMathOperator{\Hom}{Hom}

\DeclareMathOperator{\rank}{rank}

\DeclareMathOperator{\Sets}{Sets}

\DeclareMathOperator{\Gal}{Gal}

\DeclareMathOperator{\Stab}{Stab}

\DeclareMathOperator{\End}{End}
\DeclareMathOperator{\tr}{tr}

\DeclareMathOperator{\Lie}{Lie}

\DeclareMathOperator{\Res}{Res}
\DeclareMathOperator{\Pic}{Pic}
\DeclareMathOperator{\Isom}{Isom}
\DeclareMathOperator{\Ind}{Ind}

\DeclareMathOperator{\res}{res}

\DeclareMathOperator{\Map}{Map}
\DeclareMathOperator{\Art}{Art}

\DeclareMathOperator{\Sh}{Sh}
\DeclareMathOperator{\Spec}{Spec}
\DeclareMathOperator{\Eis}{Eis}
\DeclareMathOperator{\Vect}{Vect}
\DeclareMathOperator{\Whit}{Whit}

\DeclareMathOperator{\IC}{IC}

\newcommand{\cA}{{\mathcal A}}
\newcommand{\cB}{{\mathcal B}}
\newcommand{\cC}{{\mathcal C}}
\newcommand{\cD}{{\mathcal D}}
\newcommand{\cE}{{\mathcal E}}
\newcommand{\cF}{{\mathcal F}}
\newcommand{\cG}{{\mathcal G}}
\newcommand{\cH}{{\mathcal H}}

\newcommand{\cL}{{\mathcal L}}

\newcommand{\cO}{{\mathcal O}}
\newcommand{\cP}{{\mathcal P}}
\newcommand{\cQ}{{\mathcal Q}}

\newcommand{\cS}{{\mathcal S}}

\newcommand{\cX}{{\mathcal X}}
\newcommand{\cY}{{\mathcal Y}}
\newcommand{\cZ}{{\mathcal Z}}
\newcommand{\fra}{{\mathfrak a}}
\newcommand{\frb}{{\mathfrak b}}

\newcommand{\frg}{{\mathfrak g}}
\newcommand{\frh}{{\mathfrak h}}

\newcommand{\frl}{{\mathfrak l}}
\newcommand{\ffrm}{{\mathfrak m}}
\newcommand{\frn}{{\mathfrak n}}

\newcommand{\frp}{{\mathfrak p}}
\newcommand{\frq}{{\mathfrak q}}

\newcommand{\frt}{{\mathfrak t}}

\newcommand{\frz}{{\mathfrak z}}
\newcommand{\bbA}{{\mathbb A}}

\newcommand{\bbC}{{\mathbb C}}

\newcommand{\bbF}{{\mathbb F}}
\newcommand{\bbG}{{\mathbb G}}

\newcommand{\bbP}{{\mathbb P}}
\newcommand{\bbQ}{{\mathbb Q}}
\newcommand{\bbR}{{\mathbb R}}

\newcommand{\bbT}{{\mathbb T}}

\newcommand{\bbZ}{{\mathbb Z}}
\newcommand{\ilim}{\mathop{\varinjlim}\limits}
\newcommand{\plim}{\mathop{\varprojlim}\limits}
\newcommand{\hG}{{\widehat{G}}}
\newcommand{\hH}{{\widehat{H}}}
\newcommand{\hB}{{\widehat{B}}}
\newcommand{\hM}{{\widehat{M}}}
\newcommand{\hT}{{\widehat{T}}}

\renewcommand{\Re}{\text{Re }}

\DeclareMathOperator{\im}{im}
\DeclareMathOperator{\inv}{inv}

\newcommand{\dquot}{{\,\!\sslash\!\,}}

\newcommand{\nc}{\newcommand}
\nc{\renc}{\renewcommand}
\nc{\ssec}{\subsection}
\nc{\sssec}{\subsubsection}
\nc{\on}{\operatorname}

\nc\ol{\overline}
\nc\wt{\widetilde}
\nc\tboxtimes{\wt{\boxtimes}}
\nc{\alp}{\alpha}

\nc{\BunBb}{\overline{\Bun}_B}

\title{$\hG$-local systems on smooth projective curves are potentially automorphic}
\author{Gebhard B\"ockle, Michael Harris, Chandrashekhar Khare, and Jack A. Thorne}

\begin{document}
\maketitle
\begin{abstract}
Let $X$ be a smooth, projective, geometrically connected curve over a finite field $\bbF_q$, and let $G$ be a split semisimple algebraic group over $\bbF_q$. Its dual group $\hG$ is a split reductive group over $\bbZ$. Conjecturally, any $l$-adic $\hG$-local system on $X$ (equivalently, any conjugacy class of continuous homomorphisms $\pi_1(X) \to \hG(\overline{\bbQ}_l)$) should be associated to an everywhere unramified automorphic representation of the group $G$. 

We show that for any homomorphism $\pi_1(X) \to \hG(\overline{\bbQ}_l)$ of Zariski dense image, there exists a finite Galois cover $Y \to X$ over which the associated local system becomes automorphic. 
\end{abstract}
\setcounter{tocdepth}{2}
\tableofcontents
\section{Introduction}

Let $X$ be a smooth, projective, geometrically connected curve over the finite field $\bbF_q$, and let $G$ be a split semisimple algebraic group over $\bbF_q$. Let $\hG$ denote the dual group of $G$, considered as a split semisimple group scheme over $\bbZ$. Fix a prime $l \nmid q$ and an algebraic closure $\overline{\bbQ}_l$ of $\bbQ_l$. In the Langlands program, one considers a conjectural duality between the following two kinds of objects:
\begin{itemize}
\item Everywhere unramified automorphic representations of the adele group $G(\bbA_K)$: these appear as irreducible subrepresentations $\Pi$ of the space of functions $f : G(K) \backslash G(\bbA_K) / G(\prod_v \cO_{K_v}) \to \overline{\bbQ}_l$.
\item $\hG$-local systems on $X$: equivalently, $\hG(\overline{\bbQ}_l)$-conjugacy classes of continuous homomorphisms $\sigma : \pi_1(X) \to \hG(\overline{\bbQ}_l)$.
\end{itemize}
(We omit here the extra conditions required in order to get a conjecturally correct statement.) In recent work \cite{Laf12}, V. Lafforgue has established one direction of this conjectural correspondence: he associates to each everywhere unramified, cuspidal automorphic representation $\Pi$ of $G(\bbA_K)$ a corresponding homomorphism $\sigma_\Pi : \pi_1(X) \to \hG(\overline{\bbQ}_l)$. The goal of this paper is to establish the following `potential' converse to this result.
\begin{theorem}\label{thm_intro_theorem}
Let $\sigma : \pi_1(X) \to \hG(\overline{\bbQ}_l)$ be a continuous homomorphism which has Zariski dense image. Then we can find a finite Galois extension $K' / K$ and a cuspidal automorphic representation $\Pi = \otimes'_v \Pi_v$ of $G(\bbA_{K'})$ satisfying the following condition: for every place $v$ of $K'$, $\Pi_v^{G(\cO_{K'_v})} \neq 0$, and $\Pi_v$ and $\sigma|_{W_{K'_v}}$ are matched under the unramified local Langlands correspondence.
\end{theorem}
One expects that this theorem should be true with $K' = K$, but we are not able to prove this. The theorem is already known in the case $G = \PGL_n$, for then the entire global Langlands correspondence is known in its strongest possible form, by work of L. Lafforgue \cite{Laf02} (and in this case one can indeed take $K' = K$). Moreover, it is likely that analytic arguments used to establish functoriality for classical groups over number fields can be extended to function fields.  Combined with the results of \cite{Laf02}, our main theorem  would then follow easily for split classical groups, again with $K' = K$. \footnote{Either the descent method of Ginzburg, Rallis, and Soudry \cite{GRS} or Arthur's more general approach using the twisted trace formula \cite{Art13} would yield automorphy under the hypotheses of our main theorem.}    However, for semisimple groups in the exceptional series, this is the first theorem of this kind. We note that using the deformation-theoretic techniques that we develop in this paper, one can construct plentiful examples of Zariski dense representations $\sigma$ to which the above theorem applies, for any group $G$.

We comment on the hypotheses of the theorem. Zariski density of the representation is convenient at several points. From the automorphic perspective, taking Arthur's conjectures into account, it removes the possibility of having to deal with automorphic representations which are cuspidal but non-tempered. Zariski density also has the important consequence that $\sigma$ can be placed in a compatible system of $\hG$-local systems that is determined uniquely up to equivalence by the conjugacy classes of Frobenius elements. Without this density condition, it is not even clear what should be meant by the phrase `compatible system', and the definition we use here (see Definition \ref{def_compatible_sytems}) should therefore be regarded as provisional. (For a possible solution to these issues, see \cite{Dri16}.) A related point is that it is subtle to define what it means for a Galois representation to be automorphic; this is discussed further (along with the simplifying role played by Zariski density of representations) in \S \ref{sec_summary_of_lafforgue}. 

The condition that $\sigma$ be everywhere unramified (or in other words, that $X$ be projective) is less serious. Indeed, V. Lafforgue's work attaches Galois representations to cuspidal automorphic representations with an arbitrary level of ramification, and Theorem \ref{thm_intro_theorem} immediately implies an analogue where $\sigma$ is allowed to be finitely ramified at a finite number of points. It seems likely that one could prove the same result with no restrictions on the ramification of $\sigma$ (other than ramification at a finite number of points), using the techniques of this paper, but we leave this extension to a future work.

The proof of Theorem \ref{thm_intro_theorem} uses similar ingredients to existing potential automorphy theorems for $G = \GL_n$ over number fields (see e.g. \cite{Bar14}). Among these, we mention:
\begin{enumerate}
\item The construction of $\hG$-valued Galois representations associated to automorphic forms on $G$.
\item Automorphy lifting theorems for $\hG$-valued Galois representations.
\item Existence of local systems with `big mod $l$ monodromy'.
\item Existence of `universally automorphic $\hG$-valued Galois representations'.
\end{enumerate}
Once these ingredients are in place, we employ the `chutes and ladders' argument diagrammed in \cite[pp. 1136--1137]{Ell05} in order to conclude Theorem \ref{thm_intro_theorem}. We now comment on each of these ingredients (i) -- (iv) in turn.

\begin{enumerate} \item As mentioned above, V. Lafforgue has constructed the Galois representations associated to cuspidal automorphic forms on the group $G$ \cite{Laf12}. In fact, he goes much farther, constructing a commutative ring of endomorphisms $\cB$ of the space $\cA_0$ of cusp forms, which contains as a subring the ring generated by Hecke operators at unramified places, and which is generated by so-called `excursion operators'. Lafforgue then defines a notion of pseudocharacter for a general reductive group (in a fashion generalizing the definition for $\GL_n$ given by Taylor \cite{Tay91}) and shows that the absolute Galois group of $K = \bbF_q(X)$ admits a pseudocharacter valued in this algebra $\cB$ of excursion operators. Furthermore, he shows that over an algebraically closed field, pseudocharacters are in bijection with completely reducible Galois representations into $\hG$, up to $\hG$-conjugation. This leads to a map from the set of prime ideals of $\cB$ to the set of completely reducible Galois representations. This work is summarized in \S \ref{sec_automorphic_forms}.

A well-known result states that deforming the pseudocharacter of an (absolutely) irreducible representation into $\GL_n$ is equivalent to deforming the representation itself (see e.g.\ \cite{Car94, Rou96}). The first main contribution of this paper is to generalize this statement to an arbitrary reductive group $\hG$. The key hypothesis we impose is that the centralizer in $\hG$ of the representation being deformed is as small as possible, i.e.\ the centre of $\hG$. If $\hG = \GL_n$, then Schur's lemma says that this condition is equivalent to 
 irreducibility (i.e.\ that the image be contained in no proper parabolic subgroup of $\hG$), but in general this condition is strictly stronger. Nevertheless, it means that after localizing an integral version of Lafforgue's algebra $\cB$ at a suitable maximal ideal, we can construct $\cB$-valued Galois representations, in a manner recalling the work of Carayol \cite{Car94}.
\item Having constructed suitable integral analogues of Lafforgue's Galois representations, we are able to prove an automorphy lifting theorem (Theorem \ref{thm_R_equals_B} and Corollary \ref{cor_automorphy_lifting}) using a generalization of the Taylor--Wiles method. The key inputs here are an understanding of deformation theory for $\hG$-valued Galois representations (which we develop from scratch here, although this idea is not original to this paper) and a workable generalization of the notion of `big' or `adequate' subgroup (see e.g.\ \cite{Tho12} for more discussion of the role these notions play in implementations of the Taylor--Wiles method). It seems likely that one could get by with a weaker notion than this (indeed, the notion of `$\hG$-abundance' that we introduce here is closer to the `enormous' condition imposed in \cite{Cal15,Kha15} than `big' or `adequate'), but it is enough for the purposes of this paper.
\item A key step in proving potential automorphy theorems is showing how to find extensions of the base field $K$ over which two given compatible systems of Galois representations are `linked' by congruences modulo primes. The reader familiar with works such as \cite{She97, Har10} may expect that in order to do this, we must construct families of $\hG$-motives (whatever this may mean). For example, in \cite{Har10} the authors consider a family of projective hypersurfaces over $\bbP^1_\bbQ$; the choice of a rational point $z \in \bbP^1_\bbQ(F)$ ($F$ a number field) determines a compatible family of Galois representations of $\Gamma_F = \Gal(F^s / F)$ acting on the primitive cohomology of the corresponding hypersurface. By contrast, here we are able to get away with compatible families of $\hG$-Galois representations (without showing they arise from motives). The reasons for this have to do with the duality between the fields whose Galois groups we are considering and the fields of rational functions on our parameter spaces of Galois representations; see for example the diagram (\ref{eqn_diagram_of_galois_groups}) below. 
\item Informally, we call a representation $\sigma : \Gamma_K \to \hG(\overline{\bbQ}_l)$ universally automorphic if for any finite separable extension $K' / K$, the restricted representation $\sigma|_{\Gamma_{K'}} : \Gamma_{K'} \to \hG(\overline{\bbQ}_l)$ is automorphic (in the sense of being associated to a prime ideal of the algebra $\cB$ of excursion operators). Of course, this is conjecturally true for any representation $\sigma$ (say of Zariski dense image), but representations for which this property can be established unconditionally play an essential role in establishing potential automorphy. We define a class of representations, called Coxeter parameters, which satisfy the condition of `$\hG$-abundance' required to apply our automorphy lifting theorems and which can be shown, in certain circumstances, to have a property close to `universal automorphy' (see Lemma \ref{lem_Coxeter_parameters_rational_over_Z_l}).

This class of Coxeter parameters generalizes, in some sense, the class of representations into $\GL_n$ which are induced from a normal subgroup with quotient cyclic of order $n$. Such representations of absolute Galois groups of number fields into $\GL_n$ are often known to be automorphic, thanks to work of Arthur--Clozel \cite{Art89} which uses the trace formula. In contrast, we deduce the automorphy of Coxeter parameters from work of Braverman--Gaitsgory on geometrization of classical Eisenstein series \cite{Bra02}. Thus the geometric Langlands program plays an essential role in the proof of Theorem \ref{thm_intro_theorem}, although it does not appear in the statement.
\end{enumerate} 
We now describe the organization of this paper. We begin in \S \ref{sec_invariant_theory} by reviewing some results from geometric invariant theory. The notion of pseudocharacter is defined using invariant theory, so this is essential for what follows. In \S \ref{sec_pseudocharacters_and_their_deformation_theory}, we review V. Lafforgue's notion of $\hG$-pseudocharacter and prove our first main result, showing how deforming pseudocharacters relates to deforming representations in good situations. In \S \ref{sec_galois_representations_and_their_deformation_theory}, we give the basic theory of deformations of $\hG$-valued representations. We also develop the Khare--Wintenberger method of constructing characteristic 0 lifts of mod $l$ Galois representations in this setting. This is a very useful application of the work of L. Lafforgue \cite{Laf02} for $\GL_n$ and the solution of de Jong's conjecture by Gaitsgory \cite{Gaitsgory}.

In \S \ref{sec_compatible_systems_of_galois_representations}, we make a basic study of compatible systems of $\hG$-valued representations. The work of L. Lafforgue again plays a key role here, when we cite the work of Chin \cite{Chi04} to show that any $\hG$-valued representation lives in a compatible system. We also make a study of compatible systems with Zariski dense image, and show that they enjoy several very pleasant properties. In \S \ref{sec_local_calculation}, we make some local calculations having to do with representation theory of $p$-adic groups at Taylor--Wiles primes. These calculations will be familiar to experts, and require only a few new ideas to deal with the possible presence of pseudocharacters (as opposed to true representations). In \S \ref{sec_automorphic_forms}, the longest section of this paper, we finally discuss automorphic forms on the group $G$ over the function field $K$, and review the work of V. Lafforgue. We then prove an automorphy lifting theorem, which is the engine driving the proof of our main Theorem \ref{thm_intro_theorem}. 

The remaining three sections are aimed at proving this potential automorphy result. In \S \ref{sec_application_of_moret-bailly} we introduce some moduli spaces of torsors; it is here that we are able to avoid the detailed study of families of motives that occurs in \cite{Har10}. In \S \ref{sec_class_of_universally_automorphic_galois_representations} we give the definition of Coxeter parameters, make a study of their basic properties, and deduce their universal automorphy from the work of Braverman--Gaitsgory. Finally in \S \ref{sec_potential_automorphy} we return to the `chutes and ladders' argument and apply everything that has gone before to deduce our main results.

At the end of the paper are two appendices, written by Gaitsgory, which establish two key properties of the automorphic functions attached to the geometric Eisenstein series constructed by Braverman--Gaitsgory. The role played by these is discussed more in the proof of Theorem \ref{thm_automorphy_of_Coxeter_homomorphisms}.
\subsection{Acknowledgments}

We all thank Gaitsgory, Labesse, Lemaire, Moeglin, Raskin, Waldspurger, and especially Genestier and Lafforgue for a number of helpful conversations. In addition, we are very grateful to Gaitsgory for providing the appendices to this paper, proving  properties of the special automorphic functions constructed in his paper with Braverman; we also thank the referee for pointing out that there were no adequate references for these facts in the literature.  G.B.  was supported by the DFG grants FG 1920 and SPP 1489. M.H.'s research received funding from the European Research Council under the European Community's Seventh Framework Programme (FP7/2007-2013) / ERC Grant agreement no. 290766 (AAMOT).  M.H. was partially supported by NSF Grant DMS-1404769. C.K.  was partially supported by NSF grants DMS-1161671  and DMS-1601692, and by a Humboldt Research Award, and thanks the Tata Institute of Fundamental Research, Mumbai for its support. This research was partially conducted during the period J.T. served as a Clay Research Fellow. J.T.'s work received funding from the European Research Council (ERC) under the European Union's Horizon 2020 research and innovation programme (grant agreement No 714405).  Both G.B. and J.T. thank M.H. and the IHES for an invitation through AAMOT.

\section{Notation and preliminaries}\label{sec_notation}

If $K$ is a field, then we will generally write $K^s$ for a fixed choice of separable closure and $\Gamma_K = \Gal(K^s / K)$ for the corresponding Galois group. If $K$ is a global field and $S$ is a finite set of places, then $K_S$ will denote the maximal subextension of $K^s$, unramified outside $S$, and $\Gamma_{K, S} = \Gal(K_S / K)$. If $v$ is a place of $K$, then $\Gamma_{K_v} = \Gal(K_v^s / K_v)$ will denote the decomposition group, and $\Gamma_{K_v} \to \Gamma_K$ the homomorphism corresponding to a fixed choice of $K$-embedding $K^s \hookrightarrow K_v^s$. If $v \not\in S$, then $\Frob_v \in \Gamma_{K, S}$ denotes a choice of geometric Frobenius element at the place $v$. If $K = \bbF_q(X)$ is the function field of a smooth projective curve $X$, then we will identify the set of places of $K$ with the set of closed points of $X$. If $v \in X$ is a place we write $q_v = \# k(v) = \# (\cO_{K_v} / \varpi_v \cO_{K_v})$ for the size of the residue field at $v$. We write $| \cdot |_v$ for the norm on $K_v$, normalized so that $|\varpi_v|_v = q_v^{-1}$; then the product formula holds. We write $\Art_{K_v} : K^\times \to \Gamma_{K_v}^\text{ab}$ for the Artin map of local class field theory, normalized to send uniformizers to geometric Frobenius elements. We write $\widehat{\cO}_{K} = \prod_{v \in X} \cO_{K_v}$. We will write $W_{K_v}$ for the Weil group of the local field $K_v$. 

In this paper, we consider group schemes both over fields and over more general bases. If $k$ is a field, then we will call a smooth affine group scheme over $k$ a linear algebraic group over $k$.  Many classical results in invariant theory are proved in the setting of linear algebraic groups. We will also wish to allow possibly non-smooth group schemes; our conventions are discussed in \S \ref{sec_invariants_over_a_field}. A variety over $k$ is, by definition, a reduced $k$-scheme of finite type.

If $G, H, \dots$ are group schemes over a base $S$, then we use Gothic letters $\frg, \frh, \dots$ to denote their Lie algebras, and $G_T, \frg_T, \dots$ to denote the base changes of these objects relative to a scheme $T \to S$. If $G$ acts on an $S$-scheme $X$ and $x \in X(T)$, then we write $Z_G(x)$ or $Z_{G_T}(x)$ for the scheme-theoretic stabilizer of $x$; it is a group scheme over $T$. We denote the centre of $G$ by $Z_G$. We say that a group scheme $G$ over $S$ is reductive if $G$ is smooth and affine with reductive (and therefore connected) geometric fibres.

When doing deformation theory, we will generally fix a prime $l$ and an algebraic closure $\overline{\bbQ}_l$ of $\bbQ_l$. A finite extension $E / \bbQ_l$ inside $\overline{\bbQ}_l$ will be called a coefficient field; when such a field $E$ has been fixed, we will write $\cO$ or $\cO_E$ for its ring of integers, $k$ or $k_E$ for its residue field, and $\varpi$ or $\varpi_E$ for a choice of uniformizer of $\cO_E$. We write $\cC_\cO$ for the category of Artinian local $\cO$-algebras with residue field $k$; if $A \in \cC_\cO$, then we write $\ffrm_A$ for its maximal ideal. Then $A$ comes with the data of an isomorphism $k \cong A / \ffrm_A$. 

\subsection{The dual group and groups over $\bbZ$}\label{sec_reductive_groups_over_Z}

In this paper, we will view the dual group of a reductive group as a split reductive group over $\bbZ$. We now recall what this means. We first recall that a root datum is a 4-tuple $(M, \Phi, M^\vee, \Phi^\vee)$ consisting of the following data:
\begin{itemize}
\item A finite free $\bbZ$-module $M$, with $\bbZ$-dual $M^\vee$. We write $\langle \cdot, \cdot \rangle : M \times M^\vee \to \bbZ$ for the tautological pairing. 
\item Finite subsets $\Phi \subset M - \{ 0 \}$ and $\Phi^\vee \subset M^\vee - \{ 0 \}$, stable under negation, and equipped with a bijection $a \leftrightarrow a^\vee$, $\Phi \leftrightarrow \Phi^\vee$.
\end{itemize}
We require that for all $a \in \Phi$, $\langle a, a^\vee \rangle = 2$, and the reflections $s_a(x) = x - \langle x, a^\vee \rangle a$ and $s_{a^\vee}(y) = y - \langle a, y \rangle a^\vee$ preserve $\Phi \subset M$ and $\Phi^\vee \subset M^\vee$, respectively. In this paper, we also require root data to be reduced, in the sense that if $a, n a \in \Phi$ for some $n \in \bbZ$, then $n \in \{ \pm 1 \}$. A based root datum is a 6-tuple $(M, \Phi, R, M^\vee, \Phi^\vee, R^\vee)$, where $(M, \Phi, M^\vee, \Phi^\vee)$ is a root datum, $R \subset \Phi$ is a root basis, and $R^\vee = \{ a^\vee \mid a \in R \}$. (See \cite[\S 1.4]{Con14}.)

If $S$ is a connected scheme, a reductive group $G$ over $S$ is said to be split if there exists a split maximal torus $T \subset G$ such that each non-zero root space $\frg_a \subset \frg$ ($a \in M = X^\ast(T)$) is a free $\cO_S$-module of rank 1. See  \cite[Definition 5.1.1]{Con14}. This condition follows from the existence of a split maximal torus if $\Pic(S) = 1$, which will always be the case in examples we consider. Associated to the triple $(G, T, M)$ is a root datum $(M, \Phi, M^\vee, \Phi^\vee)$, where $\Phi \subset M = X^\ast(T)$ is the set of roots and $\Phi^\vee \subset M^\vee = X_\ast(T)$ is the set of coroots. If  $N \subset M$ is a subgroup containing $\Phi$, then there is a subgroup $Q \subset T$ such that $X^\ast(Q) \cong M / N$. In fact, we have $Q \subset Z_G$, the quotients $G' = G / Q$ and $T' = T / Q$ exist, and $T'$ is a split maximal torus of the split reductive group $G'$. The morphism $G \to G'$ is smooth if and only if the order of the torsion subgroup of $M/N$ is invertible on $S$. In particular, if we take $N$ to be the subgroup of $M$ generated by $\Phi$, then the quotient has trivial centre, and is what we call the adjoint group $G^\text{ad}$ of $G$. The formation of the adjoint group of a split reductive group commutes with base change.  (See \cite[Corollary 3.3.5]{Con14}.)

The further choice of a Borel subgroup $T \subset B \subset G$ containing the split maximal torus $T$ determines a root basis $R \subset \Phi$, hence a based root datum $(M, \Phi, R, M^\vee, \Phi^\vee, R^\vee)$. When this data has been fixed, we will refer to a parabolic subgroup $P$ of $G$ containing $B$ as a standard parabolic subgroup of $G$. These subgroups are in bijection with the subsets of $R$. If $I \subset R$ is a subset, then the corresponding parabolic $P_I$ admits a semidirect product decomposition $P_I = M_I N_I$, where $N_I$ is the unipotent radical of $P_I$ and $M_I$ is the unique Levi subgroup of $P_I$ containing $T$ (see \cite[Proposition 5.4.5]{Con14}.) We refer to $M_I$ as the standard Levi subgroup of $P_I$. It is reductive, and its root datum with respect to $T$ is $(M, \Phi_I, M^\vee, \Phi_I^\vee)$, where $\Phi_I \subset \Phi$ is the set of roots which are sums of elements of $I$. The intersection $B \cap M_I$ is a Borel subgroup of $M_I$, and the based root datum of $T \subset B\cap M_I \subset M_I$ is $(M, \Phi_I, I, M^\vee, \Phi_I^\vee, I^\vee)$.

We now define the dual group. Let $k$ be a field, and let $G$ be a split reductive group over $k$. (We will only need to consider the split case.) To any split maximal torus and Borel subgroup $T \subset B \subset G$, we have associated the based root datum $(M, \Phi, R, M^\vee, \Phi^\vee, R^\vee)$. The dual group $\hG$ is a tuple $(\hG, \hB, \hT)$ consisting of a split reductive group $\hG$ over $\bbZ$, as well as a split maximal torus and Borel subgroup $\hT \subset \hB \subset \hG$, together with an identification of the based root datum with the dual root datum $(M^\vee, \Phi^\vee, R^\vee, M, \Phi, R)$. Then $\hG$ is determined up to non-unique isomorphism. For any split maximal torus and Borel subgroup $T' \subset B' \subset G$, there are canonical identifications
\[ X^\ast(T') \cong X^\ast(T) \cong X_\ast(\hT) \text{ and }X_\ast(T') \cong X_\ast(T) \cong X^\ast(\hT). \]
If $I \subset R$ is a subset, then the standard Levi subgroup $\hM_I \subset \hG$ of based root datum $(M^\vee, \Phi^\vee_I, I^\vee, M, \Phi_I, I)$ contains the split maximal torus and Borel subgroup $\hT \subset \hB \cap \hM_I \subset \hM_I$ and can be identified as the dual group of $M_I$. (See \cite[\S 3]{Bor79a}.)

\subsection{Chebotarev density theorem}

At several points in this paper, we will have to invoke the Chebotarev density theorem over function fields. Since this works in a slightly different way to the analogous result in the number field case, we recall the statement here. We take $X$ to be a geometrically connected, smooth, projective curve over $\bbF_q$, $K = \bbF_q(X)$, $S$ a finite set of places of $K$, and $\Gamma_{K, S}$ the Galois group of the maximal extension unramified outside $S$. This group then sits in a short exact sequence
\[ \xymatrix@1{ 1 \ar[r] & \overline{\Gamma}_{K, S} \ar[r] & \Gamma_{K, S} \ar[r] & \widehat{\bbZ} \ar[r] & 1,} \]
where the quotient $\widehat{\bbZ}$ corresponds to the everywhere unramified scalar extension $\overline{\bbF}_q \cdot K / K$, and the element $1 \in \widehat{\bbZ}$ acts as geometric Frobenius on $\overline{\bbF}_q$. The theorem is now as follows (see \cite[Theorem 4.1]{Cha97}).
\begin{theorem}\label{thm_chebotarev_density}
Suppose given a commutative diagram of groups and continuous homomorphisms
\[ \xymatrix@1{ 1 \ar[r] & \overline{\Gamma}_{K, S} \ar[d]^{\lambda_0} \ar[r] & \Gamma_{K, S} \ar[d]^{\lambda} \ar[r] & \widehat{\bbZ} \ar[d]^{1 \mapsto \gamma} \ar[r] & 1 \\
1 \ar[r] & G_0 \ar[r] & G \ar[r]^m & \Gamma \ar[r] & 1,} \]
where $G$ is finite, $\lambda_0$ is surjective, and $\Gamma$ is abelian. Let $C \subset G$ be a subset invariant under conjugation by $G$. Then we have
\[  \frac{ \# \{ v \in X \mid q_v = q^n, \lambda(\Frob_v) \in C \} }{ \# \{ v \in X \mid q_v = q^n \} } = \frac{ \# C \cap m^{-1}(\gamma^n) }{\# G_0} + O(q^{-{n/2}}), \]
where the implicit constant depends on $X$ and $G$, but not on $n$.
\end{theorem}
\begin{corollary}\label{cor_equivalence_of_representations_with_similar_trace}
\begin{enumerate} \item The set $\{ \Frob_w \}$ of Frobenius elements, indexed by places of $K_S$ not dividing $S$, is dense in $\Gamma_{K, S}$.
\item Let $l$ be a prime not dividing $q$, and let $\rho, \rho' : \Gamma_{K, S} \to \GL_n(\overline{\bbQ}_l)$ be continuous semisimple representations such that $\tr \rho(\Frob_v) = \tr \rho'(\Frob_v)$ for all $v \not\in S$. Then $\rho \cong \rho'$. 
\end{enumerate}
\end{corollary}
\begin{proof}
The second part follows from the first. The first part follows from Theorem \ref{thm_chebotarev_density}, applied to the finite quotients of $\Gamma_{K, S}$.
\end{proof}

\section{Invariant theory}\label{sec_invariant_theory}

In this section we recall some results in the invariant theory of reductive groups acting on affine varieties. We describe in these terms what it means for group representations valued in reductive groups to be completely reducible or irreducible. We first consider the theory over a field in \S \ref{sec_invariants_over_a_field}, and then consider extensions of some of these results for actions over a discrete valuation ring in \S \ref{sec_invariants_over_a_DVR}.

\subsection{Classical invariant theory}\label{sec_invariants_over_a_field}

Let $k$ be a field. 
\begin{lemma}
Let $G$ be a linear algebraic group over $k$ which acts on an integral affine variety $X$. Let $x \in X(k)$. Then:
\begin{enumerate}
\item The image of the orbit map $\mu_x : G \to X$, $g \mapsto g x$, is an open subset of its closure. We endow the image $G \cdot x$ with its induced reduced subscheme structure, and call it the orbit of $x$. It is smooth over $k$, and invariant under the action of $G$ on $X$.
\item The following are equivalent:
\begin{enumerate}
\item The map $G \to G \cdot x$ is smooth.
\item The centralizer $Z_G(x)$ is smooth over $k$.
\end{enumerate}
If these equivalent conditions hold, then we say that the orbit $G \cdot x$ is separable. 
\end{enumerate}
\end{lemma}
\begin{proof}
By Chevalley's theorem, the image $\mu_x(G) \subset X$ is a constructible subset of $X$, so contains a dense open subset $U$ of its closure $Z \subset X$. We can find a finite extension $k' / k$ and $g \in G(k')$ such that $g x = y$ lies in $U(k')$. Let $z \in \mu_x(G)$ be a closed point; then we can find a finite extension $k'' / k'$, a point $z'' \in \mu_x(G)(k'')$ lying above $z$, and $h \in G(k'')$ such that $h x = z''$. Then $h g^{-1} y = z''$, and hence $h g^{-1} U_{k''}$ is an open subset of $Z_{k''}$ which contains $z''$. Its image under the flat morphism $Z_{k''} \to Z$ is therefore an open subset containing the closed point $z$, which also lies in $\mu_x(G)$. This shows that $\mu_x(G)$ is indeed open in $Z$. The same argument using generic flatness shows that the induced map $G \to G \cdot x$ is faithfully flat, hence (since $G$ is smooth) $G \cdot x$ is geometrically reduced. The same argument once more implies that $G \cdot x$ is even smooth over $k$. The second part of the lemma now follows from the differential criterion of smoothness.
\end{proof}
Now let $G$ be a reductive group over $k$, and let $X$ be an integral affine variety on which $G$ acts. We write $X \dquot G = \Spec k[X]^G$ for the categorical quotient; since $G$ is reductive, it is again an integral affine variety, which is normal if $X$ is (see e.g. \cite[\S 2]{Bar85}). The quotient map $\pi : X \to X \dquot G$ has a number of good properties:
\begin{proposition}\label{prop_git_over_a_field}
Let notation be as above. Then:
\begin{enumerate}
\item Let $K/k$ be an algebraically closed overfield. Then $\pi$ is surjective and $G$-equivariant at the level of $K$-points.
\item If $W \subset X$ is a $G$-invariant closed set, then $\pi(W)$ is closed.
\item If $W_1, W_2 \subset X$ are disjoint $G$-invariant closed sets, then $\pi(W_1)$ and $\pi(W_2)$ are disjoint.
\item For any point $x \in (X \dquot G)(k)$, the fibre $\pi^{-1}(x)$ contains a unique closed $G$-orbit.
\item For any affine open subset $U \subset X \dquot G$, there is a natural identification $U = \pi^{-1}(U) \dquot G$.
\end{enumerate}
\end{proposition}
\begin{proof}
See \cite[Theorem 3]{Ses77}.
\end{proof}
Richardson studied the varieties $G^n$, with $G$ acting by diagonal conjugation \cite{Ric88}. His results were extended to characteristic $p$ by Bate, Martin, and R\"ohrle \cite{Bat05}. We now recall some of these results.
\begin{definition}
Let $G$ be a reductive group over $k$, and let $H \subset G$ be a closed linear algebraic subgroup. Suppose that $k$ is algebraically closed.
\begin{enumerate} \item 
We say that $H$ is $G$-completely reducible if for any parabolic subgroup $P \subset G$ containing $H$, there exists a Levi subgroup of $P$ containing $H$. We say that $H$ is $G$-irreducible if there is no proper parabolic subgroup of $G$ containing $H$.
\item We say that $H$ is strongly reductive in $G$ if $H$ is not contained in any proper parabolic subgroup of $Z_G(S)$, where $S \subset Z_G(H)$ is a maximal torus.
\end{enumerate}
\end{definition}
When the overgroup $G$ is clear from the context, we will say simply that $H$ is completely reducible (resp. strongly reductive). 
\begin{theorem}\label{thm_richardson_simultaneous_conjuation_and_stability} Let $G$ be a reductive group over $k$, and suppose that $k$ is algebraically closed. Let $x = (g_1, \dots, g_n)$ be a tuple in $G^n(k)$, and let $H \subset G$ be the smallest closed subgroup containing each of $g_1, \dots, g_n$.
\begin{enumerate}
\item The $G$-orbit of $x$ in $G^n$ is closed if and only if $H$ is strongly reductive, if and only if $H$ is $G$-completely reducible.
\item The $G$-orbit of $x$ in $G^n$ is stable (i.e.\ closed, with $Z_G(x)$ finite modulo $Z_G$) if and only if $H$ is $G$-irreducible.
\end{enumerate}
\end{theorem}
\begin{proof}
The characterization in terms of $G$-strong reductivity or irreducibility is \cite[Theorem 16.4, Proposition 16.7]{Ric88}. The equivalence between strong reductivity and $G$-complete reducibility is \cite[Theorem 3.1]{Bat05}.
\end{proof}
We apply these techniques to representation theory as follows. 
\begin{definition}\label{def_irreducibility_and_strong_irreducibility}
Let $\Gamma$ be an abstract group, and let $G$ be a reductive group over $k$.
\begin{enumerate} \item A homomorphism $\rho : \Gamma \to G(k)$ is said to be absolutely $G$-completely reducible (resp. absolutely $G$-irreducible) if the Zariski closure of $\rho(\Gamma)$ is $G$-completely reducible (resp. $G$-irreducible) after extension of scalars to an algebraic closure of $k$.
\item A homomorphism $\rho : \Gamma \to G(k)$ is said to be absolutely strongly $G$-irreducible if it is absolutely $G$-irreducible and for any other homomorphism $\rho' : \Gamma \to G(k)$ such that for all $f \in k[G]^G$ and for all $\gamma \in \Gamma$ we have $f(\rho(\gamma)) = f(\rho'(\gamma))$, $\rho'$ is also absolutely $G$-irreducible.
\end{enumerate}
\end{definition}
We observe that the Zariski closure of $\rho(\Gamma)$ is always a linear algebraic group, and that formation of this Zariski closure commutes with extension of the base field. The notions of $G$-irreducibility and $G$-complete reducibility have been studied by Serre \cite{Ser05}. We will occasionally use the word `semisimple' as a synonym for `$G$-completely reducible'. The notion of strong $G$-irreducibility is slightly unnatural, but we will require it during later arguments.

We now assume for the remainder of \S \ref{sec_invariants_over_a_field} that $k$ is algebraically closed and that $G$ is a reductive group over $k$. If $\rho : \Gamma \to G(k)$ is any representation, then we can define its semisimplification $\rho^\text{ss}$ as follows: choose a parabolic subgroup $P$ containing the image of $\rho(\Gamma)$, and minimal with respect to this property. Choose a Levi subgroup $L \subset P$. Then the composite $\rho^\text{ss} : \Gamma \to P(k) \to L(k) \to G(k)$ is $G$-completely reducible, and (up to $G(k)$-conjugation) independent of the choice of $P$ and $L$ (see \cite[Proposition 3.3]{Ser05}). This operation has an interpretation in invariant theory as well:
\begin{proposition}\label{prop_contraction_to_levi_quotient}
Let $g = (g_1, \dots, g_n) \in  G^n(k)$, let $x = \pi(g) \in (G^n \dquot G)(k)$, and let $P$ be a parabolic subgroup of $G$ minimal among those containing $g_1, \dots, g_n$. Let $L$ be a Levi subgroup of $P$. Then:
\begin{enumerate}
\item There exists a cocharacter $\lambda : \bbG_m \to G$ such that $L = Z_G(\lambda)$ and $P = \{ x \in G \mid \lim_{t \to 0} \lambda(t) x \lambda(t)^{-1} \text{ exists} \}$, with unipotent radical $\{ x \in G \mid \lim_{t \to 0} \lambda(t) x \lambda(t)^{-1} = 1 \}$.  
\item Let $g' = (g'_1, \dots, g'_n)$, where $g'_i = \lim_{t \to 0} \lambda(t) g_i \lambda(t)^{-1}$. Then $g'$ has a closed orbit in $G^n$ and $\pi(g') = \pi(g)$. Therefore $G \cdot g'$ is the unique closed orbit of $G$ in $\pi^{-1}(x)$. 
\end{enumerate}
\end{proposition}
\begin{proof}
The first part follows from \cite[Th\'eor\`eme 4.15]{Bor65}. For the second part, it follows from the definition of $g'$ that $g' \in \pi^{-1}(x)(k)$ and that $g'$ equals the image of $g \in P(k)$ in $L(k)$ (viewing $L$ as a quotient of $P$). The minimality of $P$ implies that the subgroup of $L$ generated by the components of $g'$ is $L$-irreducible, hence $G$-completely reducible (by \cite[Proposition 3.2]{Ser05}), hence $g'$ has a closed orbit in $G^n$. 
\end{proof}
\begin{proposition}\label{prop_parabolics_of_minimal_dimension}
Let $\Gamma \subset G(k)$ be a subgroup. Then:
\begin{enumerate}
\item Let $P \subset G$ be a parabolic subgroup which contains $\Gamma$, and which is minimal with respect to this property. If $\Gamma$ is $G$-completely reducible, and $L \subset P$ is a Levi subgroup containing $\Gamma$, then $\Gamma \subset L$ is $L$-irreducible.
\item The parabolic subgroups $P \subset G$ which contain $\Gamma$, and which are minimal with respect to this property, all have the same dimension. 
\end{enumerate}
\end{proposition}
\begin{proof}
See \cite[Proposition 3.3]{Ser05} and its proof, which shows that if $P, P'$ are parabolic subgroups of $G$ containing $\Gamma$, minimal with this property, that $P$ and $P'$ contain a common Levi subgroup; this implies that $P$, $P'$ have the same dimension, because of the formula $\dim P = \frac{1}{2}(\dim G + \dim L)$.
\end{proof}
We conclude this section with some remarks about separability. A closed linear algebraic subgroup $H \subset G$ is said to be separable in $G$ if its scheme-theoretic centralizer $Z_G(H)$ is smooth. If $H$ is topologically generated by elements $g_1, \dots, g_n \in G(k)$, then $H$ is separable if and only if the orbit of $(g_1, \dots, g_n)$ inside $G^n$ is separable.
\begin{theorem}\label{thm_all_subgroups_are_separable} Suppose that one of the following holds:
\begin{enumerate}
\item The characteristic of $k$ is very good for $G$.
\item $G$ admits a faithful representation $V$ such that $(\GL(V), G)$ is a reductive pair, i.e.\ $\frg \subset \frg\frl(V)$ admits a $G$-invariant complement.
\end{enumerate}
Then any linear algebraic closed subgroup of $G$ is separable in $G$.
\end{theorem}
\begin{proof}
This is \cite[Theorem 1.2]{Bat10} and \cite[Corollary 2.13]{Bat10}.
\end{proof}
We recall that if $G$ is simple, then the characteristic $l$ is said to be very good if it satisfies the following conditions, relative to the (absolute) root system of $G$:
\begin{center}
\begin{tabular}{|l|l|}
\hline
Condition     & Types              \\ \hline
$l \nmid n+1$ & $A_n$              \\ \hline
$l \neq 2$    & $B, C, D, E, F, G$ \\ \hline
$l \neq 3$    & $E, F, G$          \\ \hline
$l \neq 5$    & $E_8$              \\ \hline
\end{tabular}
\end{center}
Then $\frg$ is a simple Lie algebra, which is self-dual as a representation of $G$ (because there exists a non-degenerate $G$-invariant symmetric bilinear form on $G$; see for example \cite[Lemma I.5.3]{Spr70}). In general, we say that $l$ is a very good characteristic for $G$ if it is very good for each of the simple factors of $G$. In this case $\frg$ is a semisimple Lie algebra and the map $G \to G^\text{ad}$ is smooth; indeed, its kernel is the centre of $G$, which is smooth by Theorem \ref{thm_all_subgroups_are_separable}. 

We will often impose the condition that $l = \text{char } k$ is prime to the order of the Weyl group of $G$; this is convenient, as the following lemma shows.
\begin{lemma}
Suppose that $l = \text{char }k$ is positive and prime to the order of the Weyl group of $G$. Then $l$ is a very good characteristic for $G$.
\end{lemma}
\begin{proof}
Inspection of the tables \cite[Planches I -- IX]{Bou68}.
\end{proof}

\subsection{Invariants over a DVR}\label{sec_invariants_over_a_DVR}

Let $l$ be a prime, and let $E \subset \overline{\bbQ}_l$ be a coefficient field, with ring of integers $\cO$ and residue field $k$. If $Y$ is an $\cO$-scheme of finite type, and $y \in Y(k)$, then we write $Y^{\wedge, y}$ for the functor $\cC_\cO \to \Sets$ which sends an Artinian local $\cO$-algebra $A$ with residue field $k$ to the pre-image of $y$ under the map $Y(A) \to Y(k)$. This functor is pro-represented by the complete Noetherian local $\cO$-algebra $\widehat{\cO}_{Y, y}$. We observe that if $f : Y \to Z$ is a morphism of $\cO$-schemes of finite type over $\cO$, then $f$ is \'etale at $y$ if and only if the induced natural transformation $Y^{\wedge, y} \to Z^{\wedge, f(y)}$ is an isomorphism. 

Let $G$ be a reductive group over $\cO$. Let $X$ be an integral affine flat $\cO$-scheme of finite type on which $G$ acts, and let $x \in X(k)$. We write $G^{\wedge} = G^{\wedge, e}$ for the completion at the identity; it is a group functor, and there is a natural action 
\begin{equation}\label{eqn_completed_action} G^{\wedge} \times X^{\wedge, x} \to X^{\wedge, x}.
\end{equation}
We define the categorical quotient $X \dquot G = \Spec \cO[X]^G$ and the quotient morphism $\pi_X : X \to X \dquot G$. This space has a number of good properties that mirror what happens in the case of field coefficients:
\begin{proposition}\label{prop_git_over_a_ring}
Let notation be as above.
\begin{enumerate}
\item The space $X \dquot G$ is an integral $\cO$-scheme of finite type, and the quotient map $\pi : X \to X \dquot G$ is $G$-equivariant. If $X$ is normal, then $X \dquot G$ is normal.
\item For any homomorphism $\cO \to K$ to an algebraically closed field, the map $\pi : X(K) \to (X \dquot G)(K)$ is surjective and identifies the set  $(X \dquot G)(K)$ with the quotient of $X(K)$ by the following equivalence relation: $x_1 \sim x_2$ if and only if the closures of the $G$-orbits of $x_1$ and $x_2$ inside $X \otimes_\cO K$ have non-empty intersection. 

In particular, each orbit in $X(K)$ has a unique closed orbit in its closure, and $\pi$ induces a bijection between the closed orbits in $X(K)$ and the set $(X \dquot G)(K)$.
\item For all closed $G$-invariant subsets $W \subset X$, $\pi(W)$ is closed; and if $W_1, W_2$ are disjoint closed $G$-invariant subsets of $X$, then $\pi(W_1)$ and $\pi(W_2)$ are disjoint.
\item The formation of invariants commutes with flat base change. More precisely, if $R$ is a flat $\cO$-algebra, then the canonical map $\cO[X]^G \otimes_\cO R \to R[X_R]^{G_R}$ is an isomorphism. In particular, if $G$ acts trivially on $X$ then there is a canonical isomorphism $(G \times X) \dquot G \cong X$ (where $G$ acts on itself by left translation). 
\item Let $x \in X(k)$ have a closed $G_k$-orbit, and let $U \subset X$ be a $G$-invariant open subscheme containing $x$. Then there exists $f \in \cO[X]^G$ such that $f(x) \neq 0$. Let $D_{X \dquot G}(f) \subset X \dquot G$ denote the open subscheme where $f$ is non-vanishing. Then we can moreover choose $f$ so that $D_X(f) = \pi_X^{-1}(D_{X \dquot G}(f)) \subset U$, and there is a canonical isomorphism $D_X(f) \dquot G \cong D_{X \dquot G}(f)$.
\end{enumerate}
\end{proposition}
\begin{proof}
We first observe that $X$ admits a $G$-equivariant closed immersion into $V$, where $V$ is a $G$-module which is finite free as an $\cO$-module. Indeed, we choose $\cO$-algebra generators $f_1, \dots, f_r \in \cO[X]$. Then \cite[Proposition 3]{Ses77} shows that we can find a $G$-submodule $V \subset \cO[X]$ containing $f_1, \dots, f_r$ which is finite as an $\cO$-module. Since $X$ is flat, $V$ is free over $\cO$, and the surjection $S(V^\vee) \to \cO[X]$ then corresponds to the desired $G$-equivariant closed immersion $X \hookrightarrow V$. Most of the above now follows from \cite[Theorem 3]{Ses77} and the fact that $\cO$ is excellent, hence a universally Japanese ring (see \cite[Tag 07QS]{stacks-project}).

The fact that formation of quotient commutes with flat base change is \cite[Lemma 2]{Ses77}. For the final part, we observe that the complement $X - U$ is a $G$-invariant closed subset disjoint from the orbit of $x$, so $\pi_X(X - U) \subset X \dquot G$ is a closed subset not containing $\pi_X(x)$. We can therefore find $f \in \cO[X]^G$ such that $D_{X \dquot G}(f)$ has trivial intersection with $\pi_X(X-U)$. Then $D_X(f)$ satisfies the desired properties.
\end{proof}
Note that (ii) shows that for any homomorphism $\cO \to K$ to an algebraically closed field, the natural map $X_K \dquot G_K \to (X \dquot G)_K$ induces a bijection on $K$-points. This observation will play an important role in our study of pseudocharacters below. However, the algebras $\cO[X]^G \otimes_\cO K$ and $K[X_K]^{G_K}$ are not in general isomorphic. 

We must now establish a special case (Proposition \ref{prop_free_action_on_completion}) of Luna's \'etale slice theorem (\cite[Proposition 7.6]{Bar85}) in mixed characteristic. It seems likely that one can prove a general version of this result using the arguments of \emph{op. cit.}, but to do so here would take us too far afield. We will therefore use these arguments to prove just what we need here.
\begin{lemma}\label{lem_etale_maps_are_etale_on_quotients}
Let $X, Y$ be normal affine integral $\cO$-schemes, flat of finite type, on which $G$ acts. Let $\phi : Y \to X$ be a finite $G$-equivariant morphism, and let $y \in Y(k)$ be a point satisfying the following conditions:
\begin{enumerate}
\item $\phi$ is \'etale at $y$.
\item The orbit $G_k \cdot y$ is closed in $Y_k$.
\item The orbit $G_k \cdot \phi(y)$ is closed in $X_k$.
\item The restriction of $\phi$ to $G_k \cdot y$ is injective at the level of geometric points. 
\end{enumerate}
 Then the induced morphism $\phi \dquot G : Y \dquot G \to X \dquot G$ is \'etale at $\pi_Y(y)$.
\end{lemma}
\begin{proof}
Let $E = \Frac H^0(Y, \cO_Y)$, let $K = \Frac H^0(X, \cO_X)$, and let $L$ denote the Galois closure of $E/K$. Let $\cG = \Gal(L / K)$ and $\cH = \Gal(L/E)$. We observe that $Y$ is identified with the normalization of $X$ in $E$; we write $Z$ for the normalization of $X$ in $L$. Then $Z$ is a normal integral flat $\cO$-scheme. It is of finite type, because $\cO$ is universally Japanese. We also note that $Y \dquot G$ is identified with the normalization of $X \dquot G$ in $E$. Indeed, the normality of $X$ implies the normality of $X \dquot G$, and similarly for $Y$, by Proposition \ref{prop_git_over_a_ring} and the morphism $Y \dquot G \to X \dquot G$ is finite, because it is integral and of finite type. We give the argument for integrality: let $a \in H^0(Y, \cO_Y)^G$, and let $f(T) = T^n + a_1 T^{n-1} + \dots + a_n \in K[X]$ denote the characteristic polynomial of multiplication by $a$ on $E$. Since $a \in H^0(Y, \cO_Y)$ is integral over $H^0(X, \cO_X)$, all of the $a_i$ lie in $H^0(X, \cO_X)$ (\cite[Ch. 5. \S 1.6, Cor. 1]{Bou98}). We also see that the $a_i$ lie in $K^G$, hence in $K^G \cap H^0(X, \cO_X) = H^0(X, \cO_X)^G$, showing that $a$ is integral over $H^0(X, \cO_X)^G$.

We write $Z'$ for the normalization of $X \dquot G$ in $L$. Then $Z'$ is also a normal integral flat $\cO$-scheme of finite type. Both $Z, Z'$ receive natural actions of the group $\cG$, and there is a natural map $Z \to Z'$ respecting this action. Moreover, we can identify $\cO[Y] = \cO[Z]^\cH$, $\cO[X] = \cO[Z]^\cG$, $\cO[Y \dquot G] = \cO[Z']^\cH$ and $\cO[X \dquot G] = \cO[Z']^\cG$. (Here we need to use the fact that $\cO[X]^G$ is integrally closed in $\cO[X]$, and similarly for $Y$; compare \cite[4.2.3]{Bar85}.) In order to show that the map $Y \dquot G \to X \dquot G$ is \'etale at the point $\pi_Y(y)$, we will make appeal to the following lemma (cf. \cite[2.3.1]{Bar85}, where it is stated for varieties of finite type over a field):
\begin{lemma}\label{lem_galois_implies_etale}
Let $A$ be an excellent normal domain with field of fractions $K$, and let $L/K$ be a Galois extension of group $\cG$. Let $\cH \subset \cG$ be a subgroup, and let $E = L^\cH$, $B$ the integral closure of $A$ in $E$, $C$ the integral closure of $A$ in $L$. Let $\overline{c}$ be a geometric point above a closed point of $\Spec C$, and let $\overline{b}$, $\overline{a}$ denote its images in $\Spec B$ and $\Spec A$, respectively. Then the morphism $\Spec B \to \Spec A$ is \'etale at $\overline{b}$ if and only if $\Stab_\cG(\overline{c}) \subset \cH$. 
\end{lemma}
\begin{proof}
For such extensions, being \'etale is equivalent to being unramified \cite[Tag 0BTF]{stacks-project}. The desired characterization then follows from \cite[Ch. V, \S 2.3, Prop. 7]{Bou98}.
\end{proof}
We fix a geometric point $\overline{z}$ of $Z$ above $y$, and write $\overline{z}'$ for its image in $Z'$, $\overline{y}'$ for its image in $Y \dquot G$, and $\overline{x}'$ for its image in $X \dquot G$. We let $x = \phi(y)$. We get a commutative diagram
\[ \xymatrix{ Z \ar[r]^\nu \ar[d]_\psi & Z' \ar[d]^{\psi'} \\
Y \ar[r]^{\pi_Y} \ar[d]_{\phi} & Y \dquot G \ar[d]^{\phi \dquot G} \\
X \ar[r]^{\pi_X} & X \dquot G. } \]
Since $\phi$ is \'etale at the point $y$, Lemma \ref{lem_galois_implies_etale} implies that $\Stab_{\cG}(\overline{z}) \subset \cH$. On the other hand, we have $\psi(\Stab_{\cG}(\overline{z}') \cdot \overline{z}) \subset \pi_Y^{-1}(\overline{y}') \cap \phi^{-1}(G_k \cdot x)$. Since $\phi$ is finite, $G_k \cdot x = \phi(G_k \cdot y)$ is a closed orbit, hence $\phi^{-1}(G_k \cdot x)$ is closed, and a union of finitely many $G_k$-orbits each of which has the same dimension as $G_k \cdot x$. Therefore $\phi^{-1}(G_k \cdot x)$ is a disjoint union of closed $G_k$-orbits. The inverse image $\pi_Y^{-1}(\overline{y}')$ contains a unique closed $G_k$-orbit, namely $G_k \cdot y$, so we find that $\psi(\Stab_{\cG}(\overline{z}')\cdot \overline{z}) \subset G_k \cdot y$. 

By assumption, the restriction of $\phi$ to $G_k \cdot y$ is injective at the level of geometric points. Since $\phi\psi(\Stab_{\cG}(\overline{z}') \cdot \overline{z}) = \{ x \} = \{ \phi(y) \}$, it follows that $\psi(\Stab_{\cG}(\overline{z}') \cdot \overline{z}) = \{ y \}$. The group $\cH$ acts transitively on the fibre $\psi^{-1}(y)$ (\cite[Tag 0BRI]{stacks-project}), so we find $\Stab_{\cG}(\overline{z}') \cdot \overline{z} \subset \cH \cdot \overline{z}$, hence $\Stab_{\cG}(\overline{z}') \subset \Stab_{\cG}(\overline{z}) \cH = \cH$. The result now follows from one more application of Lemma \ref{lem_galois_implies_etale}.
\end{proof}
\begin{proposition}\label{prop_free_action_on_completion}
Suppose that $X$ is an integral affine smooth $\cO$-scheme on which $G$ acts. Let $x \in X(k)$ be a point with $G_k \cdot x$ closed, and $Z_{G_k}(x)$ scheme-theoretically trivial. Then:
\begin{enumerate}
\item The action $G^{\wedge} \times X^{\wedge, x} \to X^{\wedge, x}$ (as in (\ref{eqn_completed_action}) above) is free (i.e.\ free on $A$-points for every $A \in \cC_\cO$).
\item The natural map $\pi : X \to X \dquot G$ induces, after passage to completions, an isomorphism
\[ X^{\wedge, x} / G^{\wedge} \cong (X \dquot G)^{\wedge, \pi(x)}. \]
\end{enumerate} 
\end{proposition}
\begin{proof}
We apply \cite[Proposition 7.6]{Bar85} to obtain a locally closed subscheme $S_0 \subset X_k$ such that $x \in S_0(k)$ and the orbit map $G_k \times S_0 \to X_k$ is \'etale. In particular, $S_0$ is smooth over $k$, and we can find a sequence $\overline{f}_1, \dots, \overline{f}_r \in \cO_{X_k, x}$ of elements generating the kernel of $\cO_{X_k, x} \to \cO_{S_0, x}$ and with linearly independent image in the Zariski cotangent space of $\cO_{X_k, x}$. We lift these elements arbitrarily to $f_1, \dots, f_r \in \cO_{X, x}$. Let $U$ be a Zariski open affine neighbourhood of $x$ in $X$ such that $f_1, \dots, f_r \in \cO_X(U)$, and let $S = V(f_1, \dots, f_r)$. After possibly shrinking $U$, $S$ is an integral locally closed subscheme of $X$, smooth over $\cO$, such that $\cO_{S, x} / (\lambda) = \cO_{S_0, x}$. In particular, the action map $G \times S \to X$ is \'etale at $x$: it is unramified at $x$, by construction, and flat at $x$ by the fibral criterion of flatness \cite[Tag 039B]{stacks-project}. This shows the first part of the proposition.

After possibly shrinking $S$, we can assume that $\phi : G \times S \to X$ is \'etale everywhere; in particular, it is quasi-finite. We let $G$ act on $G \times S$ be left multiplication on $G$; then the map $\phi$ is $G$-equivariant. Let $X'$ denote the normalization of $X$ in $G \times S$. Then $i : G \times S \to X'$ is an open immersion (by Zariski's main theorem \cite[Tag 03GS]{stacks-project}) and $\eta : X' \to X$ is finite. Moreover, $X'$ is integral, affine, and flat of finite type over $\cO$. There is a unique way to extend the action of $G$ on $G \times S$ to an action on $X'$ (the key point being that normalization commutes with smooth base change \cite[Tag 03GV]{stacks-project}, so the action map $G \times X' \to X'$ exists by normalization).

We now apply Lemma \ref{lem_etale_maps_are_etale_on_quotients} to the point $i(1,x)$ of $X'(k)$; it follows that the induced morphism $\eta \dquot G : X' \dquot G \to X \dquot G$ is \'etale at $i(1,x)$. In order to be able to apply the lemma, we must check that $G_k \cdot i(1,x)$ is closed inside $X'$. Indeed, $G_k \cdot \phi(x)$ is closed inside $X$, by hypothesis, and $\eta$ is finite, so $\eta^{-1}(G_k \cdot \phi(x))$ is a closed subset of $X'_k$ consisting of finitely many $G_k$-orbits, each of the same dimension. They must therefore all be closed, implying that $G_k \cdot i(1,x)$ is itself closed. 

This also shows that $\pi_{X'}(i(1,x)) \not\in \pi_{X'}(X' - i(G \times S))$ (using the third part of Proposition \ref{prop_git_over_a_ring}). The set $\pi_{X'}(X' - i(G \times S)) \subset X' \dquot G$ is closed, so we can find a function $f \in \cO[X']^G$ such that $f(X' - i(G \times S)) = 0$ and $f(i(1, x)) \neq 0$. It then follows that $i : D_{G \times S}(f) \to D_{X'}(f)$ is an isomorphism, and $D_{X' \dquot G}(f) \cong D_{X'}(f) \dquot G \cong D_{G \times S}(f) \dquot G \cong D_{S}(f)$, hence $i \dquot G$ induces an open immersion $D_{S}(f) \to X' \dquot G$. The set $D_{S}(f)$ contains $\pi_{G \times S}(1,x)$, so this completes the proof that $\phi \dquot G = (\eta \dquot G) \circ (i \dquot G)$ is \'etale at $\pi_{G \times S}(1,x)$. 
\end{proof}
In the next section, we will apply this proposition in the following situation: $\widehat{G}$ is a split reductive group over $\cO$, $X = \widehat{G}^n$ for some $n \geq 1$, and $G = \widehat{G}^\text{ad}$ acts on $X$ by simultaneous conjugation. If $g_1, \dots, g_n \in \widehat{G}(k)$ are elements which generate a $\hG$-irreducible subgroup with scheme-theoretically trivial centralizer in $\hG^\text{ad}$, then the point $x = (g_1, \dots, g_n) \in X(k)$ satisfies the conclusion of the proposition.

\section{Pseudocharacters and their deformation theory}\label{sec_pseudocharacters_and_their_deformation_theory}

In this section, we define what it means to have a pseudocharacter of a group valued in a split reductive group $\hG$. We also prove the fundamental results that completely reducible representations  biject with pseudocharacters (when the coefficient ring is an algebraically closed field); and that representations biject with pseudocharacters (when the coefficient ring is Artinian local) provided the residual representation is sufficiently non-degenerate. 

Let $\hG$ be a split reductive group over $\bbZ$. 
\begin{definition}\label{def_pseudocharacter}
Let $A$ be a ring, and let $\Gamma$ be a group. A $\hG$-pseudocharacter $\Theta$ of $\Gamma$ over $A$ is a collection of algebra maps $\Theta_n : \bbZ[\hG^n]^{\hG} \to \Map(\Gamma^n, A) $ for each $n \geq 1$,
satisfying the following conditions:
\begin{enumerate}
\item For each $n, m \geq 1$ and for each map $\zeta : \{ 1, \dots, m \} \to \{1, \dots, n \}$, $f \in \bbZ[\hG^m]^\hG$, and $\gamma_1, \dots, \gamma_n \in \Gamma$, we have
\[ \Theta_n(f^\zeta)(\gamma_1, \dots, \gamma_n) = \Theta_m(f)(\gamma_{\zeta(1)}, \dots, \gamma_{\zeta(m)}), \]
where $f^\zeta(g_1, \dots, g_n) = f(g_{\zeta(1)}, \dots, g_{\zeta(m)})$.
\item For each $n \geq 1$, for each $\gamma_1, \dots, \gamma_{n+1} \in \Gamma$, and for each $f \in \bbZ[\hG^n]^\hG$, we have 
\[ \Theta_{n+1}(\hat{f})(\gamma_1, \dots, \gamma_{n+1}) = \Theta_n(f)(\gamma_1, \dots, \gamma_n \gamma_{n+1}), \]
where $\hat{f}(g_1, \dots, g_{n+1}) = f(g_1, \dots, g_n g_{n+1})$.
\end{enumerate}
\end{definition}
\begin{remark}
Let $\cO$ be a flat $\bbZ$-algebra, and suppose that $A$ is an $\cO$-algebra. Then it is equivalent to give a pseudocharacter $\Theta$ over $A$ or a collection of $\cO$-algebra maps $\Theta'_n : \cO[\hG^n]^\hG \to \Map(\Gamma^n, A)$ satisfying the same axioms with respect to $f \in \cO[\hG^n]^\hG$. Indeed, this follows from the fact that $\cO[\hG^n]^\hG = \bbZ[\hG^n]^\hG \otimes_\bbZ \cO$ (\cite[Lemma 2]{Ses77}). It is important to note that this may fail to be true when $\cO$ is no longer flat (for example, if $\cO$ is a field of characteristic $p > 0$), but nevertheless Theorem \ref{thm_pseudocharacters_biject_with_representations_over_fields} below is still true in this case. 
\end{remark}
The following lemma justifies our initial interest in pseudocharacters. 
\begin{lemma}
Let $A$ be a ring, and let $\Gamma$ be a group. Suppose given a homomorphism $\rho : \Gamma \to \hG(A)$. Then the collection of maps $\Theta_n(f)(\gamma_1, \dots, \gamma_n) = f(\rho(\gamma_1), \dots, \rho(\gamma_n))$ is a pseudocharacter, which depends only on $\rho$ up to $\hG(A)$-conjugation.
\end{lemma}
\begin{proof}
Immediate from the definitions.
\end{proof}
If $\rho$ is a homomorphism as in the lemma, then we will write $\tr \rho = (\Theta_n)_{n \geq 1}$ for its associated pseudocharacter.

We can change the ring and the group:
\begin{lemma}\label{lem_change_of_base_ring_or_group_for_pseudocharacters} Let $A$ be a ring, and let $\Gamma$ be a group. 
\begin{enumerate}
\item  Let $h : A \to A'$ be a ring map, and let $\Theta = (\Theta_n)_{n \geq 1}$ be a pseudocharacter over $A$. Then $h_\ast(\Theta) = (h \circ \Theta_n)_{n \geq 1}$ is a pseudocharacter over $A'$.
\item Let $\phi : \Delta \to \Gamma$ be a homomorphism, $\Theta$ a pseudocharacter of $\Gamma$ over $A$. Then the collection $\phi^\ast \Theta = (\Theta_n \circ \phi)_{n \geq 1}$ is a pseudocharacter of $\Delta$ over $A$.
\item Let $N \subset \Gamma$ be a normal subgroup, and let $\phi : \Gamma \to \Gamma / N$ be the quotient homomorphism. Then the map $\Theta \mapsto \phi^\ast \Theta$ defines a bijection between the set of pseudocharacters of $\Gamma/N$ over $A$ and the set of pseudocharacters $\Xi = (\Xi_n)_{n \geq 1}$ of $\Gamma$ over $A$ such that for all $n \geq 1$, the map $\Xi_n$ takes values in $\Map((\Gamma/N)^n, A) \subset \Map(\Gamma^n, A)$.
\end{enumerate}
\end{lemma}
\begin{proof}
Immediate from the definitions.
\end{proof}
\begin{theorem}\label{thm_pseudocharacters_biject_with_representations_over_fields}
Let $\Gamma$ be a group, and let $k$ be an algebraically closed field. Then the assignment $\rho \mapsto \Theta = \tr \rho$ induces a bijection between the following two sets:
\begin{enumerate}
\item The set of $\hG(k)$-conjugacy classes of $\hG$-completely reducible homomorphisms $\rho : \Gamma \to \hG(k)$.
\item The set of $\hG$-pseudocharacters $\Theta$ of $\Gamma$ over $k$.
\end{enumerate}
\end{theorem}
\begin{proof}
The proof of this theorem is due to Lafforgue \cite[\S 5]{Laf13}; we review it here as preparation for the infinitesimal version of the next section, and because some modifications are required in the case of positive characteristic. Before reading the proof, we invite the reader to first become reacquainted with Proposition \ref{prop_git_over_a_field}.

We first show how to construct a representation from a pseudocharacter $(\Theta_n)_{n \geq 1}$. For any $n \geq 1$, the map $\hG^n(k) \to (\hG^n \dquot \hG)(k)$ induces a bijection between the set of $\hG(k)$-conjugacy classes of tuples $(g_1, \dots, g_n)$ which generate a $\hG$-completely reducible subgroup of $\hG$, and the set $(\hG^n \dquot \hG)(k)$ (as follows from part (iv) of Proposition \ref{prop_git_over_a_field} and Theorem \ref{thm_richardson_simultaneous_conjuation_and_stability}). The datum of $\Theta_n$ determines for each tuple $\gamma = (\gamma_1, \dots, \gamma_n) \in \Gamma^n$ a point $\xi_\gamma \in (\hG^n \dquot \hG)(k)$, these points satisfying certain compatibility relations corresponding to conditions (i) and (ii) of Definition \ref{def_pseudocharacter}. We write $T(\gamma)$ for a representative of the orbit in $\hG^n(k)$ corresponding to $\xi_\gamma$.

Let $H(\gamma)$ denote the Zariski closure of the subgroup of $\hG(k)$ generated by the entries of $T(\gamma)$. For every $\gamma \in \Gamma^n$, we define $n(\gamma)$ to be the dimension of a parabolic $P \subset \hG_k$ minimal among those containing $H(\gamma)$; by Proposition \ref{prop_parabolics_of_minimal_dimension}, this is independent of the choice of $P$ satisfying this condition. 

Let $N = \sup_{n \geq 1, \gamma \in \Gamma^n} n(\gamma)$. We fix a choice of integer $n \geq 1$ and element $\delta = (\delta_1, \dots, \delta_n) \in \Gamma^n$ satisfying the following conditions:
\begin{enumerate}
\item $n(\delta) = N$.
\item For any $n' \geq 1$ and $\delta' \in \Gamma^{n'}$ also satisfying (i), we have $\dim Z_{\hG_k}(H(\delta)) \leq \dim Z_{\hG_k}(H(\delta'))$.
\item For any $n' \geq 1$ and $\delta' \in \Gamma^{n'}$ also satisfying (i) and (ii), we have $\# \pi_0( Z_{\hG_k}(H(\delta)) ) \leq \# \pi_0( Z_{\hG_k}(H(\delta')) )$.
\end{enumerate}
Write $T(\delta) = (g_1, \dots, g_n)$. We are going to show that for every $\gamma \in \Gamma$, there exists a unique element $g \in \hG(k)$ such that $(g_1, \dots, g_n, g)$ is $\hG(k)$-conjugate to $T(\delta_1, \dots, \delta_n, \gamma)$. 

We first show the existence of such an element $g$. Let $T(\delta_1, \dots, \delta_n, \gamma) = (h_1, \dots, h_n, h)$. We claim that in fact $(h_1, \dots, h_n)$ has a closed $\hG_k$-orbit in $\hG_k^n$; this is equivalent, by Theorem \ref{thm_richardson_simultaneous_conjuation_and_stability}, to the assertion that the elements $h_1, \dots, h_n \in \hG(k)$ generate a $\hG$-completely reducible subgroup. Note that $(h_1, \dots, h_n)$ lies above $\xi_\delta \in (\hG^n \dquot \hG)(k)$ (by (i) of Definition \ref{def_pseudocharacter}), so this claim will show that $(h_1, \dots, h_n)$ is $\hG(k)$-conjugate to $T(\delta)$.

To this end, let $P \subset \hG_k$ be a parabolic subgroup minimal among those containing $H(\delta_1, \dots, \delta_n, \gamma)$. We can find a Levi subgroup $M_P$ of $P$ which contains $H(\delta_1, \dots, \delta_n, \gamma)$. Let $N_P$ denote the unipotent radical of $P$, and let $Q$ be a minimal parabolic of $M_P$ containing $h_1, \dots, h_n$. Let $M_Q$ be a Levi subgroup of $Q$, and let $h_1', \dots, h_n' \in M_Q(k)$ denote the images of the elements $h_1, \dots, h_n$ in $M_Q(k)$. Then the elements $h_1', \dots, h_n'$ generate an $M_Q$-irreducible subgroup, which is therefore $\hG$-completely reducible (by \cite[Proposition 3.2]{Ser05}). It follows from Proposition \ref{prop_contraction_to_levi_quotient} that the tuple $(h_1', \dots, h_n')$ is $\hG(k)$-conjugate to $T(\delta) = (g_1, \dots, g_n)$. 

In particular, $Q$ contains a conjugate of $T(\delta)$, which implies that $Q N_P$ contains a conjugate of $T(\delta)$. Since $Q N_P$ is a parabolic subgroup of $G$, we obtain
\[ n(\delta) = N \leq \dim Q N_P \leq \dim P \leq N. \]
It follows that equality holds, $Q N_P = P$, and hence $Q = M_P$ and $h_i = h_i'$ for each $i = 1, \dots, n$. This shows that $(h_1, \dots, h_n)$ has a closed $\hG_k$-orbit in $\hG_k^n$. Consequently, we can find an element $x \in \hG(k)$ such that $x(h_1, \dots, h_n) x^{-1} = (g_1, \dots, g_n)$. We can now take $g = x h x^{-1}$.

This shows the existence of an element $g \in \hG(k)$ such that $(g_1, \dots, g_n, g)$ is $\hG(k)$-conjugate to $T(\delta_1, \dots, \delta_n, \gamma)$. To show that this element is unique, suppose $g'$ is another such element. Then we can find $y \in \hG(k)$ such that $y (g_1, \dots, g_n, g) y^{-1} = (g_1, \dots, g_n, g')$. In particular, we have $y \in Z_{\hG}(g_1, \dots, g_n)(k)$. We therefore need to show that 
\[ Z_{\hG}(g_1, \dots, g_n)(k) = Z_{\hG}(g_1, \dots, g_n, g)(k). \]
 The first group obviously contains the second. The defining properties (i), (ii) and (iii) of $\delta \in \Gamma^n$ then show that these groups must in fact be equal.

This establishes the claim, and defines a map $\gamma \in \Gamma \mapsto g = \rho(\gamma) \in \hG(k)$. We must now show that this map $\rho : \Gamma \to \hG(k)$ is a homomorphism. Let $\gamma, \gamma' \in \Gamma$. We claim that there exist $g, g' \in G(k)$ such that $(g_1, \dots, g_n, g, g')$ is $\hG(k)$-conjugate to $T(\delta_1, \dots, \delta_n, \gamma, \gamma')$, and that the pair $(g, g')$ is unique with this property. 

Let $(h_1, \dots, h_n, h, h') = T(\delta_1, \dots, \delta_n, \gamma, \gamma')$, and let $P$ be a parabolic containing the elements $h_1, \dots, h_n, h, h'$, and minimal with respect to this property. Let $M_P$ be a Levi subgroup of $P$ also containing these elements, $N_P$ the unipotent radical of $P$. Then we have $\dim P = n(\delta_1, \dots, \delta_n, \gamma, \gamma') \leq n(\delta_1, \dots, \delta_n)$. Let $Q$ be a minimal parabolic of $M_P$ containing $h_1, \dots, h_n$, and let $M_Q$ be a Levi subgroup, $h_1', \dots, h_n'$ the images of $h_1, \dots, h_n$ in $M_Q(k)$. Then the tuple $(h_1', \dots, h_n')$ is $\hG(k)$-conjugate to $(g_1, \dots, g_n)$ (again by Proposition \ref{prop_contraction_to_levi_quotient} and (i) of Definition \ref{def_pseudocharacter}); it follows that $n(\delta_1, \dots, \delta_n) = N \leq \dim Q N_P \leq \dim P \leq N$, so equality holds, $M_Q = Q = M_P$ and $Q N_P = P$, and we can find $y \in \hG(k)$ such that $y (h_1, \dots, h_n) y^{-1} = (g_1, \dots, g_n)$. We then take $(g, g') = y (h, h') y^{-1}$. The argument that this pair is unique is the same as above.

We claim that the tuples $(g_1, \dots, g_n, g)$, $(g_1, \dots, g_n, g')$, and $(g_1, \dots, g_n, g g')$ all have closed orbits in $\hG^{n+1}$. This claim will show, together with parts (i) and (ii) of Definition \ref{def_pseudocharacter}, that $\rho(\gamma) = g$, $\rho(\gamma') = g'$, and that $\rho(\gamma \gamma') = g g'$. We just show the claim in the case of $(g_1, \dots, g_n, g g')$, the others being similar. Let $P \subset \hG_k$ be a parabolic subgroup minimal among those containing $g_1, \dots, g_n, g, g'$. Then $P$ contains $g_1, \dots, g_n$, so $n(\delta) = N \leq \dim P \leq N$. Therefore equality holds, and $P$ is also minimal among those parabolic subgroups of $\hG_k$ containing $g_1, \dots, g_n$. 

Let $M_P$ be a Levi subgroup of $P$ also containing $g_1, \dots, g_n, g, g'$. Then the tuple $g_1, \dots, g_n$ is $M_P$-irreducible, hence $g_1, \dots, g_n, g g'$ is $M_P$-irreducible, hence $\hG$-completely reducible. It then follows from Theorem \ref{thm_richardson_simultaneous_conjuation_and_stability} that $g_1, \dots, g_n, g g'$ has a closed orbit in $\hG^{n+1}$. This shows that $\rho(\gamma \gamma') = \rho(\gamma) \rho (\gamma')$.

We have now shown how, given a pseudocharacter $\Theta$ over the algebraically closed field $k$, to construct a representation $\rho : \Gamma \to \hG(k)$. This representation moreover satisfies the condition $\tr \rho = \Theta$, or equivalently that for any $m \geq 1$, $\gamma = (\gamma_1, \dots, \gamma_m) \in \Gamma^m$, and $f \in \bbZ[\hG^n]^{\hG}$, we have the formula
\[ f(\rho(\gamma_1), \dots, \rho(\gamma_m)) = \Theta_m(f)(\gamma). \]
The proof of this is very similar to the verification that $\rho$ is a homomorphism, so we omit it. In order to complete the proof of the theorem, it remains to show that if $\rho, \rho'$ are two $\hG$-completely reducible homomorphisms with $\tr \rho = \tr \rho'$, then they are in fact $\hG(k)$-conjugate. 

Let us therefore fix a $\hG$-completely reducible homomorphism $\rho : \Gamma \to \hG(k)$. We will show that we can recover $\rho$ from its associated pseudocharacter $\Theta = \tr \rho$; given the constructive argument above, this is no longer surprising. We let the elements $\xi_\gamma \in (\hG^n \dquot \hG)(k)$, $T(\gamma) \in \hG^n(k)$ be as defined above. By \cite[Lemma 2.10]{Bat05}, we can find elements $\delta_1, \dots, \delta_n \in \Gamma$ such that for any parabolic subgroup $P \subset \hG_k$, and for any Levi subgroup $L \subset P$, $P$ contains $\rho(\Gamma)$ if and only if $P$ contains $\rho(\delta_1), \dots \rho(\delta_n)$, and likewise for $L$. In particular, our assumption that $\rho$ is $\hG$-completely reducible implies that $(\rho(\delta_1), \dots, \rho(\delta_n)) = (g_1, \dots, g_n)$ (say) has a closed $\hG_k$-orbit in $\hG_k^n$. After possibly augmenting the tuple $(\delta_1, \dots, \delta_n)$, we can assume moreover that $Z_{\hG}(g_1, \dots, g_n)(k) = Z_{\hG}(\rho(\Gamma))(k)$. 

Let $\gamma \in \Gamma$. We claim that $\rho(\gamma) = g$ is uniquely determined by the condition that $(g_1, \dots, g_n, g)$ is $\hG$-conjugate to $T(\delta_1, \dots, \delta_n, \gamma)$. It satisfies this property because $(g_1, \dots g_n, g)$ has a closed orbit, by construction, and because $\rho$ has associated pseudocharacter $\Theta$. On the other hand, if $g'$ is another element with this property, then we can find $x \in Z_{\hG}(\rho(\Gamma))$ such that $x g x^{-1} = g'$. Since $g \in \rho(\Gamma)$, this implies that $g = g'$, as required. 
\end{proof}
\begin{definition}
Let $R$ be a topological ring, and let $\Gamma$ be a topological group. A pseudocharacter $\Theta = (\Theta_n)_{n \geq 1}$ over $R$ is said to be continuous if for each $n \geq 1$, the map $\Theta_n$ takes values in the subset $\Map_{\text{cts}}(\Gamma^n, R) \subset \Map(\Gamma^n, R)$ of continuous maps.
\end{definition}
\begin{proposition}\label{prop_continuous_pseudocharacters_biject_with_continuous_representations}
Suppose that $\Gamma$ a profinite group, that $k$ is an algebraically closed topological field, and that $\rho : \Gamma \to \hG(k)$ is a $\hG$-completely reducible representation with $\tr \rho = \Theta$. Then:
\begin{enumerate}
\item If $\rho$ is continuous, then $\Theta$ is continuous.
\item If $k$ admits a rank 1 valuation and is of characteristic 0 \(e.g.\ $k = \overline{\bbQ}_l$\) and $\Theta$ is continuous, then $\rho$ is continuous.
\item If $k$ is endowed with the discrete topology \(e.g. $k = \overline{\bbF}_l$\) and $\Theta$ is continuous, then $\rho$ is continuous.
\end{enumerate}
\end{proposition}
\begin{proof}
The first part is clear from the definition of $\tr \rho$. The proof of the second part is contained in the proof of \cite[Proposition 5.7]{Laf13}. For the third part, we mimic the proof of uniqueness in Theorem \ref{thm_pseudocharacters_biject_with_representations_over_fields} to show that $\rho$ factors through a discrete quotient of $\Gamma$. First, we can find elements $\gamma_1, \dots, \gamma_n \in \Gamma$ such that if $g_i = \rho(\gamma_i)$, $i = 1, \dots, n$, then $\rho(\Gamma)$ is contained in the same Levi and parabolic subgroups of $\hG_k$ as $(g_1, \dots, g_n)$, and we have $Z_{\hG}(g_1, \dots, g_n) = Z_{\hG}(\rho(\Gamma))$. The proof of Theorem \ref{thm_pseudocharacters_biject_with_representations_over_fields} shows that we can then recover $\rho$ uniquely as follows: for each $\gamma \in \Gamma$, $\rho(\gamma)$ is the unique element $g \in \hG(k)$ such that $(g_1, \dots, g_n, g)$ has a closed orbit in $\hG(k)$ and for all $f \in \bbZ[\hG^{n+1}]^{\hG}$, we have $f(g_1, \dots, g_n, g) = \Theta_{n+1}(f)(\gamma_1, \dots, \gamma_n, \gamma)$. 

We now observe that for any $f \in \bbZ[\hG^{n+1}]^{\hG}$, the map $\Gamma \to k$, $\gamma \mapsto \Theta_{n+1}(f)(\gamma_1, \dots, \gamma_n, \gamma)$, is continuous. On the other hand, $\bbZ[\hG^{n+1}]^{\hG}$ is a $\bbZ$-algebra of finite type (\cite[Theorem 3]{Ses77}). It follows that we can find an open normal subgroup $N \subset \Gamma$ such that for all $f \in \bbZ[\hG^{n+1}]^{\hG}$ and for all $\gamma \in N$, we have $\Theta_{n+1}(f)(\gamma_1, \dots, \gamma_n, \gamma) = \Theta_{n+1}(f)(\gamma_1, \dots, \gamma_n, 1)$, hence $f(g_1, \dots, g_n, \rho(\gamma)) = f(g_1, \dots, g_n, 1)$. The above characterization of $\rho$ now shows that this forces $\rho(\gamma) = 1$, hence $\rho$ factors through the finite quotient $\Gamma / N$, and is \emph{a fortiori} continuous.
\end{proof}

\begin{theorem}\label{thm_reduction_modulo_l} Let $l$ be a prime, and let $\Gamma$ be a profinite group.
\begin{enumerate}
\item Let $\Theta$ be a continuous $\hG$-pseudocharacter of $\Gamma$ over $\overline{\bbQ}_l$. Then there exists a coefficient field $E \subset \overline{\bbQ}_l$ such that $\Theta$ takes values in $\cO_E$.
\item Let $\rho : \Gamma \to \hG(\overline{\bbQ}_l)$ be a continuous homomorphism. Then after replacing $\rho$ by a $\hG(\overline{\bbQ}_l)$-conjugate, we can find a coefficient field $E \subset \overline{\bbQ}_l$ such that $\rho$ takes values in $\hG(\cO_E)$. 
\item Let $\rho : \Gamma \to \hG(\overline{\bbQ}_l)$ be a continuous homomorphism. Choose a conjugate $\rho'$ with values in $\hG(\cO_E)$ for some coefficient field $E \subset \overline{\bbQ}_l$. Then the semisimplification $\overline{\rho} : \Gamma \to \hG(\overline{\bbF}_l)$ of $\rho' \text{ mod }\varpi \cO_E$ is continuous and, up to $\hG(\overline{\bbF}_l)$-conjugacy, independent of any choices.
\end{enumerate}
\end{theorem}
\begin{proof}
The first part of the theorem follows from the second: given $\Theta$, we can find a continuous representation $\rho : \Gamma \to \hG(\cO_E)$ such that $\Theta = \tr \rho$, hence $\Theta$ takes values in $\cO_E$. To prove the second part, we will use Bruhat--Tits theory (see \cite{Tit79}). Let us first note that a standard argument using the Baire category theorem shows that $\rho(\Gamma)$ is contained in $\hG(E)$ for some finite extension $E / \bbQ_l$. (Indeed, $\rho(\Gamma)$ is a Baire space which is exhausted by the closed subgroups $\rho(\Gamma) \cap \hG(E)$, as $E$ varies over all finite extensions of $\bbQ_l$ inside $\overline{\bbQ}_l$; therefore one of these contains an open subgroup of $\rho(\Gamma)$, therefore of finite index. Enlarging $E$, it will then have the desired property.)

Let $\cD \hG$ denote the derived subgroup of $\hG$. Then $\hG(E)$ acts on the building $\cB(\cD \hG, E)$. Let $\hG(E)^0$ denote the subgroup of elements $g \in G(E)$ such that for all $\chi \in X^\ast(\hG)$, $\chi(g) \in \cO_E^\times$. The maximal compact subgroups of $\hG(E)$ can all be realized as stabilizers in $\hG(E)^0$ of centroids of facets in $\cB(\cD \hG, E)$. There is a unique hyperspecial point $x_0 \in \cB(\cD \hG, E)$ such that $\Stab_{\hG(E)^0}(x_0) = \hG(\cO_E)$. 

If $E' / E$ is a finite extension, then there is an inclusion $i_{E, E'} : \cB(\cD \hG, E) \hookrightarrow \cB(\cD \hG, E')$ that is equivariant for the action of $\hG(E) \subset \hG(E')$. According to \cite[Lemma 2.4]{Lar95}, we can find a totally ramified extension $E' / E$ and a point $x \in \cB(\cD \hG, E)$ such that $i_{E, E'}(x)$ is hyperspecial and stabilized by $\rho(\Gamma)$. Since all hyperspecial vertices are conjugate under the action of $\hG^\text{ad}(E')$, this means that after replacing $\rho$ by a $\overline{\bbQ}_l$-conjugate, and possibly enlarging $E'$ further, $\rho(\Gamma) \subset \hG(\cO_{E'})$. This establishes the second part of the theorem.

After reducing modulo the maximal ideal of $\cO_{E'}$ and semisimplifying, we get the desired representation $\overline{\rho} : \Gamma \to \hG(\overline{\bbF}_l)$. It remains to check that this representation is continuous and, up to $\hG(\overline{\bbF}_l)$-conjugacy, independent of any choices. It is continuous because $\rho' \text{ mod }\ffrm_E$ is continuous. Its isomorphism class is independent of choices by Theorem \ref{thm_pseudocharacters_biject_with_representations_over_fields} and because $\tr \overline{\rho} = \overline{\Theta}$ depends only on the original representation $\rho$.
\end{proof}
\begin{definition}\label{def_reduction_modulo_l}
Let $l$ be a prime, and let $\Gamma$ be a profinite group.
\begin{enumerate}
\item If $\rho : \Gamma \to \hG(\overline{\bbQ}_l)$ is a continuous representation, then we write $\overline{\rho} : \Gamma \to \hG(\overline{\bbF}_l)$ for the semisimple residual representation associated to it by Theorem \ref{thm_reduction_modulo_l}, and call it the reduction modulo $l$ of $\rho$.
\item If $\Theta$ is a continuous pseudocharacter over $\overline{\bbQ}_l$, then we write $\overline{\Theta}$ for the continuous pseudocharacter over $\overline{\bbF}_l$ which is the reduction of $\Theta$ modulo the maximal ideal of $\overline{\bbZ}_l$, and call it the reduction modulo $l$ of $\Theta$.
\end{enumerate}
\end{definition}
If $\hG = \GL_n$, then these notions of reduction modulo $l$ are the familiar ones. 
\subsection{Artinian coefficients}

Now fix a prime $l$, and let $E \subset \overline{\bbQ}_l$ be a coefficient field. Let $\Gamma$ be a profinite group. Let $\overline{\rho} : \Gamma \to \hG(k)$ be a representation with associated pseudocharacter $\overline{\Theta} = \tr \overline{\rho}$. In this situation, we can introduce the functor $\text{PDef}_{\overline{\Theta}} : \cC_\cO \to \Sets$ which associates to any $A \in \cC_\cO$ the set of pseudocharacters $\Theta$ over $A$ with $\Theta \text{ mod } \ffrm_A = \overline{\Theta}$. We also introduce the functor $\text{Def}_{\overline{\rho}} : \cC_\cO \to \Sets$ which associates to any $A$ in $\cC_\cO$ the set of conjugacy classes of homomorphisms $\rho : \Gamma \to \hG(A)$ such that $\rho \text{ mod }\ffrm_A = \overline{\rho}$ under the group $\ker (\hG^\text{ad}(A) \to \hG^\text{ad}(k))$. (This deformation functor will be studied further in \S \ref{sec_galois_representations_and_their_deformation_theory} below.)

We then have the following infinitesimal version of Theorem \ref{thm_pseudocharacters_biject_with_representations_over_fields}, which is an analogue of Carayol's lemma for pseudocharacters valued in $\GL_n$ \cite{Car94}: 
\begin{theorem}\label{thm_pseudocharacters_biject_with_representations_over_artinian_rings}
With notation as above, suppose further that the centralizer of $\overline{\rho}$ in $\hG_k^\text{ad}$ is scheme-theoretically trivial and that $\overline{\rho}$ is absolutely $\hG$-completely reducible. Then the map $\rho \mapsto \tr \rho$ induces an isomorphism of functors $\Def_{\overline{\rho}} \to \PDef_{\overline{\Theta}}$.
\end{theorem}
\begin{proof}
If $n \geq 1$, let $X_n = \hG^n_\cO$, $Y_n = X_n \dquot \hG^\text{ad}_\cO$, $\pi : X_n \to Y_n$ the quotient map. We will use the following consequence of Proposition \ref{prop_free_action_on_completion}:
\begin{itemize}
\item Let $x = (\overline{g}_1, \dots, \overline{g}_n) \in X_n(k)$ be a point with scheme-theoretically trivial centralizer in $\hG_k^\text{ad}$ and closed $\hG$-orbit in $\hG_k^n$. Then $\hG^{\text{ad}, \wedge}_\cO$ acts freely on $X_n^{\wedge, x}$, and the map $X_n^{\wedge, x} \to Y_n^{\wedge, \pi(x)}$ factors through an isomorphism $X_n^{\wedge, x} / \hG^{\text{ad}, \wedge}_\cO \cong Y_n^{\wedge, \pi(x)}$.
\end{itemize}
Suppose that $A \in \cC_\cO$, and let $\Theta \in \PDef_{\overline{\Theta}}(A)$. We will construct a preimage $\rho : \Gamma \to \hG(A)$. Let $\gamma_1, \dots, \gamma_n \in \Gamma$ be elements such that $\overline{g}_1 = \overline{\rho}(\gamma_1), \dots, \overline{g}_n = \overline{\rho}(\gamma_n)$ cover $\overline{\rho}(\Gamma)$. Then $Z_{\hG^\text{ad}}(\overline{g}_1, \dots, \overline{g}_n) = \{ 1 \}$. Let $x = (\overline{g}_1, \dots, \overline{g}_n)$. Then $\Theta_n$ determines a point of $Y_n^{\wedge, \pi(x)}(A)$, and we choose $(g_1, \dots, g_n) \in X_n^{\wedge, x}(A)$ to be an arbitrary pre-image of this point. By the above bullet point, any other choice is conjugate to this one by a unique element of $\hG_\cO^{\text{ad}, \wedge}(A)$.

Let $\gamma \in \Gamma$ be any element, and let $y = (\overline{g}_1, \dots, \overline{g}_n, \overline{\rho}(\gamma))$. We have a commutative diagram, given by forgetting the last entry:
\[ \xymatrix{ X_{n+1}^{\wedge, y} \ar[r] \ar[d] & Y_{n+1}^{\wedge, \pi(y)} \ar[d] \\
X_n^{\wedge, x} \ar[r] & Y_n^{\wedge, \pi(x)}. } \]
The horizontal arrows are both $\hG_\cO^{\text{ad}, \wedge}$-torsors, which implies that this diagram is even Cartesian, hence there is a unique tuple $(g_1, \dots, g_n, g) \in X_{n+1}^{\wedge, y}(A)$ which lifts $(g_1, \dots, g_n)$ and which has image in $Y_{n+1}^{\wedge, \pi(y)}$ corresponding to $\Theta_{n+1}(\gamma_1, \dots, \gamma_n, \gamma)$. 

It is now an easy verification that the assignment $\gamma \leadsto g = \rho(\gamma)$ is a homomorphism $\rho : \Gamma \to \hG(A)$ with $\tr \rho = \Theta$. This shows that the natural transformation $\Def_{\overline{\rho}} \to \PDef_{\overline{\Theta}}$ is surjective. It is clear from the construction that it is also injective, so this completes the proof. 
\end{proof}

\section{Galois representations and their deformation theory}\label{sec_galois_representations_and_their_deformation_theory}

Let $\hG$ be a split semisimple group over $\bbZ$. We fix a prime $l$ which is a very good characteristic for $\hG$, as well as a coefficient field $E \subset \overline{\bbQ}_l$. Let $\Gamma$ be a profinite group satisfying Mazur's condition $\Phi_l$ \cite{Maz89}. In this section, we consider the deformation theory of representations of $\Gamma$ to $\hG$ with $l$-adic coefficients. We first describe the abstract theory, and then specialize to the case where $\Gamma = \Gamma_{K, S}$ is the Galois group of a global field of positive characteristic. 

In this case there are many effective tools available, such as the Galois cohomology of global fields, the work of L. Lafforgue on the global Langlands correspondence for $\GL_n$ \cite{Laf02}, and Gaitsgory's solution of de Jong's conjecture \cite{Gaitsgory}. We will apply these results to get a good understanding of Galois deformation rings, even before we begin to make a direct connection with automorphic forms on $G$ (see for example Theorem \ref{thm_khare_wintenberger_method}).

\subsection{Abstract deformation theory}\label{sec_abstract_deformation_theory}

We start with a fixed absolutely $\hG$-irreducible homomorphism $\overline{\rho} : \Gamma \to \hG(k)$.
\begin{lemma}\label{lem_triviality_of_scheme_theoretic_centralizer_in_very_good_characteristic}
The scheme-theoretic centralizer of $\overline{\rho}(\Gamma)$ in $\hG_k^\text{ad}$ is \'etale over $k$, and $H^0(\Gamma, \widehat{\frg}_k) = 0$.
\end{lemma}
\begin{proof}
Since $H^0(\Gamma, \widehat{\frg}_k)$ can be identified with the group of $k[\epsilon]$-points of the scheme-theoretic centralizer $Z_{\hG^\text{ad}}(\overline{\rho}(\Gamma))$, it is enough to show that this centralizer is \'etale. By Theorem \ref{thm_all_subgroups_are_separable}, this happens exactly when this centralizer is finite; and this is true, by Theorem \ref{thm_richardson_simultaneous_conjuation_and_stability}.
\end{proof}
\begin{definition}
Let $A \in \cC_\cO$. A lifting of $\overline{\rho}$ over $A$ is a homomorphism $\rho : \Gamma \to \hG(A)$ such that $\rho \text{ mod } \ffrm_A = \overline{\rho}$. Two liftings $\rho, \rho'$ are said to be strictly equivalent if there exists $g \in \ker(\hG(A) \to \hG(k))$ such that $g \rho g^{-1} = \rho'$. A strict equivalence class of liftings over $A$ is called a deformation over $A$.
\end{definition}
\begin{remark}\label{rem_smoothness} In the definition of strict equivalence, it would be equivalent to consider conjugation by $\ker(\hG^{\text{ad}}(A) \to \hG^{\text{ad}}(k))$, since the natural map $\hG_\cO^{\wedge} \to \hG_\cO^{\text{ad}, \wedge}$ is an isomorphism (because $\hG$ is semisimple and we work in very good characteristic).
\end{remark}
\begin{definition}
We write $\Def_{\overline{\rho}} : \cC_\cO \to \Sets$ for the functor which associates to $A \in \cC_\cO$ the set of deformations of $\rho$ over $A$.
\end{definition}
\begin{proposition}
The functor $\Def_{\overline{\rho}}$ is pro-represented by a complete Noetherian local $\cO$-algebra $R_{\overline{\rho}}$.
\end{proposition}
\begin{proof}
Let $\Def^\square_{\overline{\rho}}$ denote the functor of liftings of $\overline{\rho}$. Then the group functor $\hG_\cO^{\text{ad}, \wedge}$ acts freely on $\Def^\square_{\overline{\rho}}$, by Lemma \ref{lem_triviality_of_scheme_theoretic_centralizer_in_very_good_characteristic}, and there is a natural transformation $\Def^\square_{\overline{\rho}} \to \Def_{\overline{\rho}}$ that induces, for any $A \in \cC_\cO$, an isomorphism $\Def^\square_{\overline{\rho}}(A) / \hG_\cO^{\text{ad},\wedge}(A) \cong \Def_{\overline{\rho}}(A)$. It is easy to see that the functor $\Def^\square_{\overline{\rho}}$ is pro-represented by a complete Noetherian local $\cO$-algebra with residue field $k$. It now follows from \cite[Proposition 2.5]{Kha09} (quotient by a free action) that $\Def_{\overline{\rho}}$ is itself pro-representable.
\end{proof}
\begin{proposition}\label{prop_presentation_of_deformation_ring}
There exists a presentation $R_{\overline{\rho}} \cong \cO \llbracket  X_1, \dots, X_g \rrbracket / (f_1, \dots, f_r)$, where $g = \dim_k H^1(\Gamma, \widehat{\frg}_k)$ and $r = \dim_k H^2(\Gamma, \widehat{\frg}_k)$. (These dimensions are finite because we are assuming that the group $\Gamma$ satisfies Mazur's property $\Phi_l$.)
\end{proposition}
\begin{proof}
This follows from a well-known calculation with cocyles, which exactly parallels that done by Mazur \cite{Maz89}. 
\end{proof}
We now suppose that we are given a representation $i : \hG \to \GL(V)$ of finite kernel ($V$ a finite free $\bbZ$-module) such that $i \overline{\rho} : \Gamma \to \GL(V_k)$ is absolutely irreducible and $l > 2(\dim V - 1)$.
\begin{lemma}\label{lem_adjoint_module_is_semisimple}
With these assumptions, $\frg\frl(V_k)$ is a semisimple $k[\Gamma]$-module and the map $\widehat{\frg}_k \to \frg\frl(V_k)$ is split injective.
\end{lemma}
\begin{proof}
The semisimplicity of $\frg\frl(V_k)$ follows from \cite[Corollaire 5.5]{Ser05}, the irreducibility of $i \overline{\rho}$, and our hypothesis on $l$. Since $i \overline{\rho}$ is absolutely irreducible, we see that $V_{\overline{\bbQ}}$ is an irreducible highest weight module of $\hG_{\overline{\bbQ}}$. Our hypothesis on $l$ then implies that it is of low height, in the sense that for a given set of positive roots $\Phi^+$ such that $V_{\overline{\bbQ}}$ has highest weight $\lambda$, we have $\sum_{\alpha \in \Phi^+} \langle \lambda, \alpha^\vee \rangle < l$. (This condition can be checked on the principal $\SL_2$, and we have an explicit bound on the dimension of the $\SL_2$-submodules that can occur.) This, together with our hypothesis that $i$ has finite kernel, implies that the map $\widehat{\frg}_k \to \frg\frl(V_k)$ is injective (apply for example \cite[Lemma 1.2]{Lie96}). 
\end{proof}
The map $\rho \mapsto i \rho$ leads to a natural transformation $\Def_{\overline{\rho}} \to \Def_{i \overline{\rho}}$, hence (by Yoneda) a map $R_{i \overline{\rho}} \to R_{\overline{\rho}}$.
\begin{proposition}\label{prop_map_of_deformation_rings_is_finite}
The map $R_{i \overline{\rho}} \to R_{\overline{\rho}}$ is surjective.
\end{proposition}
\begin{proof}
Since we deal with complete local $\cO$-algebras with the same residue field, this can be checked on the level of tangent spaces: we must show that the map $H^1(\Gamma, \widehat{\frg}_k) \to H^1(\Gamma, \frg\frl(V_k))$ is injective. This follows from Lemma \ref{lem_adjoint_module_is_semisimple}.
\end{proof}
We record a lemma which generalizes an observation of Wiles for $\GL_2$ (see \cite[Proposition 1.2]{Wil95}).
\begin{lemma}\label{lem_completion_in_generic_fibre_a_la_Kisin}
Let $E'/E$ be a finite extension, and suppose given a homomorphism $f : R_{\overline{\rho}}[1/l] \to E'$ of $E$-algebras. Let $\rho_f : \Gamma \to \hG(E')$ denote the specialization along $f$ of a representative of the universal deformation of $\overline{\rho}$, and let $\frp = \ker f$. Then there is a canonical isomorphism 
\[ \frp / \frp^2 \otimes_{k(\frp)} E' \cong H^1(\Gamma, \widehat{\frg}_{E'})^\vee \]
of $E'$-vector spaces. In particular, if $H^1(\Gamma, \widehat{\frg}_{E'}) = 0$ then $\Spec R_{\overline{\rho}}[1/l]$ is formally unramified over $\Spec E$ at $\frp$. 
\end{lemma}
\begin{proof}
Let $R'_{\overline{\rho}}$ denote the universal deformation ring as defined on the category $\cC_{\cO_{E'}}$. Then there is a canonical isomorphism $R'_{\overline{\rho}} \cong R_{\overline{\rho}} \otimes_{\cO_E} \cO_{E'}$, and a calculation shows that after this extension of scalars, we are free to assume that the prime ideal $\frp \in \Spec R_{\overline{\rho}}[1/l]$ has residue field $k(\frp) = E = E'$. Let $\frq = \ker( R_{\overline{\rho}} \to E)$; then $\frq[1/l] = \frp$ and $\frq / \frq^2$ is a finite $\cO$-module. For any $n \geq 1$, there is an isomorphism
\[ \Hom_\cO(\frq / \frq^2, \cO / (\varpi^n)) \cong H^1(\Gamma, \widehat{\frg}_{\cO / (\varpi^n)}), \]
both sides being identified with the set of $\cO$-algebra maps $R_{\overline{\rho}} / \frq^2 \to \cO \oplus \epsilon \cO / (\varpi^n)$. The result now follows on passing to the inverse limit and inverting $l$.
\end{proof}
\subsection{The case $\Gamma = \Gamma_{K, S}$}\label{sec_galois_deformation_theory}

 We keep the assumptions of \S \ref{sec_abstract_deformation_theory} and now make a particular choice of $\Gamma$. Let $\bbF_q$ be a finite field of characteristic not $l$, and let $X$ be a smooth, projective, geometrically connected curve over $\bbF_q$, $K = \bbF_q(X)$, and $S$ a finite set of places of $K$. We now take $\Gamma = \Gamma_{K, S}$ to be the Galois group of the maximal extension of $K$ unramified outside $S$ (see \S \ref{sec_notation}). This group satisfies Mazur's condition $\Phi_l$, so we immediately obtain:
\begin{proposition}
The functor $\Def_{\overline{\rho}}$ of deformations $[ \rho : \Gamma_{K, S} \to \hG(A) ]$ is represented by a complete Noetherian local $\cO$-algebra $R_{\overline{\rho}, S}$.
\end{proposition}
We add $S$ to the notation since we will later want to vary it. 

We note that if $M$ is a discrete $k[\Gamma_{K, S}]$-module, finite-dimensional as $k$-vector space, then there are two natural cohomology groups that can be associated to it: the usual Galois cohomology $H^i(\Gamma_{K, S}, M)$, and the \'etale cohomology $H^i(X - S, M)$ of the associated sheaf on $X - S$. These groups are canonically isomorphic if either $S$ is non-empty or $X$ is not a form of $\bbP^1$. Since we are assuming that $\overline{\rho}$ exists, one of these conditions is always satisfied. In particular, we have access to the Euler characteristic formula and the Poitou--Tate exact sequence for the groups $H^i(\Gamma_{K, S}, M)$, even in the case where $S$ is empty.
\begin{proposition}\label{prop_presentation_of_deformation_ring_in_semisimple_galois_case}
There is a presentation $R_{\overline{\rho}, S} \cong \cO \llbracket  X_1, \dots X_g \rrbracket / (f_1, \dots, f_g)$, where $g = \dim_k H^1(\Gamma_{K, S}, \widehat{\frg}_k)$.
\end{proposition}
\begin{proof}
Let $h^i = \dim_k H^i$. The Euler characteristic formula (\cite[Theorem 5.1]{Mil06}) says $h^0(\Gamma_{K, S}, \widehat{\frg}_k) - h^1(\Gamma_{K, S}, \widehat{\frg}_k) + h^2(\Gamma_{K, S}, \widehat{\frg}_k) = 0$. Lemma \ref{lem_triviality_of_scheme_theoretic_centralizer_in_very_good_characteristic} implies that $h^0 = 0$. The result then follows from Proposition \ref{prop_presentation_of_deformation_ring}.
\end{proof}
\begin{proposition}\label{prop_triviality_of_tangent_space_in_generic_fibre_of_galois_deformation_ring}
Let $E'/E$ be a finite extension, and suppose given a homomorphism $f : R_{\overline{\rho}, S}[1/l] \to E'$ of $E$-algebras. Let $\rho_f : \Gamma \to \hG(E')$ denote the specialization along $f$ of a representative of the universal deformation of $\overline{\rho}$, and let $\frp = \ker f$. Then $\Spec R_{\overline{\rho}, S}[1/l]$ is formally unramified over $\Spec E$ at $\frp$. 
\end{proposition}
\begin{proof}
 We can again assume that $E' = E$. By Lemma \ref{lem_completion_in_generic_fibre_a_la_Kisin}, it is enough to show that the group $H^1(\Gamma_{K, S}, \widehat{\frg}_{E})$ vanishes. We will show this using the theory of weights. Let $\overline{K} = \overline{\bbF}_q \cdot K$, a subfield of $K^s$, and let $\overline{\Gamma}_{K, S} = \Gal(K_S / \overline{K})$. Then we have a short exact sequence of profinite groups
\[ \xymatrix@1{ 1 \ar[r] & \overline{\Gamma}_{K, S} \ar[r] & \Gamma_{K, S} \ar[r] & \widehat{\bbZ} \ar[r] & 1,} \]
where the element $1 \in \widehat{\bbZ}$ is the geometric Frobenius. Corresponding to this short exact sequence we have an inflation restriction exact sequence
\[ \xymatrix@1{ 0 \ar[r] & H^1(\widehat{\bbZ}, H^0(\overline{\Gamma}_{K, S}, \widehat{\frg}_E)) \ar[r] & H^1(\Gamma_{K, S}, \widehat{\frg}_E) \ar[r] & H^1(\overline{\Gamma}_{K, S}, \widehat{\frg}_E). } \]
Let $\overline{X} = X_{\overline{\bbF}_q}$, $\overline{S} \subset \overline{X}$ the divisor living above $S$. Then there are canonical isomorphisms for $j = 0, 1$:
\[ H^j(\overline{\Gamma}_{K, S}, \widehat{\frg}_E) \cong H^j( \overline{X} - \overline{S}, \cF), \]
where $\cF$ is the lisse $E$-sheaf on $\overline{X} - \overline{S}$ corresponding to the representation $\widehat{\frg}_E$ of $\pi_1(\overline{X} - \overline{S}) \cong \overline{\Gamma}_{K, S}$. (The implicit geometric point of $\overline{X} - \overline{S}$ is the one corresponding to the fixed separable closure $K^s$ of the function field of $X$.) We note that the representation $\rho_f$ is absolutely $\hG$-irreducible, because $\overline{\rho}$ is. Let $H$ denote the Zariski closure of $\rho_f(\Gamma_{K, S}) \subset \hG(E)$. It follows that the identity component of $H$ is a semisimple group. Indeed, $H$ is reductive, because $\rho_f$ is absolutely $\hG$-irreducible; and then semisimple, because $\hG$ is semisimple. In particular, the irreducible constituents of the $E[\Gamma_{K, S}]$-module $\widehat{\frg}_E$ have determinant of finite order. 

We find that the sheaf $\cF$ is punctually pure of weight 0 \cite[Th\'eor\`eme VII.6]{Laf02}, so Deligne's proof of the Weil conjectures \cite{Del80} shows that each group $H^j(\overline{\Gamma}_{K, S}, \widehat{\frg}_E)$, endowed with its Frobenius action, is mixed of weights $\geq j$. In particular, we get $H^0(\widehat{\bbZ}, H^1(\overline{\Gamma}_{K, S}, \widehat{\frg}_E)) = 0$. On the other hand, the space $H^1(\widehat{\bbZ}, H^0(\overline{\Gamma}_{K, S}, \widehat{\frg}_E))$ is isomorphic to the space of Frobenius coinvariants in $H^0(\overline{\Gamma}_{K, S}, \widehat{\frg}_E)$, which can be non-zero only if the space of Frobenius invariants, otherwise known as $H^0(\Gamma_{K, S}, \widehat{\frg}_E)$, is non-zero. However, this space must be zero because the space $H^0(\Gamma_{K, S}, \widehat{\frg}_k)$ is zero, by Lemma \ref{lem_triviality_of_scheme_theoretic_centralizer_in_very_good_characteristic}.
\end{proof}
\begin{theorem}\label{thm_finite_flatness_for_SL_N}
Suppose that $\hG = \SL_n$. Then $R_{\overline{\rho}, S}$ is a reduced finite flat complete intersection $\cO$-algebra.
\end{theorem}
\begin{proof}
If $n = 1$, the result is trivial, so we may assume $n \geq 2$, hence $l > 2$ (because we work in very good characteristic). We first show that $R_{\overline{\rho}, S}$ is finite flat over $\cO$. By Proposition~\ref{prop_presentation_of_deformation_ring_in_semisimple_galois_case}, it suffices to show that $R_{\overline{\rho}, S}/(\varpi)$ is a finite $k$-algebra, as then $R_{\overline{\rho}, S}/(\varpi)$ is a complete intersection and $\varpi$ is a non-zero divisor on $R_{\overline{\rho}, S}$, hence $R_{\overline{\rho}, S}$ is finite over $\cO$ and $\cO$-torsion-free, hence flat (see for example \cite[Theorem 2.1.2]{Bru93}). We follow a similar argument to \cite[\S3]{deJong}. We first observe that de Jong's conjecture, \cite[Conjecture 1.1]{deJong}, was proved for $l>2$ by Gaitsgory in \cite[Theorem~3.6]{Gaitsgory}. It asserts that the image of the group $\overline{\Gamma}_{K, S}$ under any continuous representation $\rho:\Gamma_{K, S}\to\GL_n(\overline{k((t))})$ is finite.

Suppose for contradiction that $R_{\overline{\rho}, S}/(\varpi)$ is infinite. After perhaps enlarging $k$, we can (as in \cite[3.14]{deJong}) find a $k$-algebra homomorphism $\alpha:R_{\overline{\rho}, S}/(\varpi) \to k\llbracket t \rrbracket$ with open image. Let $\rho:\Gamma\to \SL_n(k\llbracket t \rrbracket)$ be the pushforward of a representative of the universal deformation. By the above theorem of Gaitsgory, $\rho$ factors via a quotient $\Gamma_{K, S} \to \Gamma_0$ that fits into a commutative diagram of groups with exact rows
\[ \xymatrix@1{ 1 \ar[r] & \overline{\Gamma}_{K, S}  \ar[d] \ar[r] & \Gamma_{K, S}  \ar[d]\ar[r] & \widehat{\bbZ} \ar[d]^= \ar[r] & 1 \\
1 \ar[r] & \overline{\Gamma}_0 \ar[r] & \Gamma_0 \ar[r] & \widehat{\bbZ} \ar[r] & 1,} \]
where $\overline{\Gamma}_0$ is \emph{finite}. In particular, the second row of this diagram is split and the centre $Z_{\Gamma_0} \subset \Gamma_0$ is open.  By the absolute irreducibility of $\overline\rho$ the centre of $\Gamma_0$ is mapped to the centre of $\SL_n(k\llbracket t \rrbracket)$ under $\rho$. This centre is finite, so we deduce that $\rho(\Gamma_0)$ is finite, hence (applying \cite[Lemma 3.15]{deJong}) that $\rho$ is strictly equivalent to the trivial deformation of $\overline\rho$ to $k\llbracket t \rrbracket$. From the universality of $R_{\overline{\rho}, S}/(\varpi)$ for deformations to complete Noetherian local $k$-algebras with residue field $k$, one deduces that $\alpha$ factors via $k$, contradicting the openness of the image of~$\alpha$ in $k\llbracket t \rrbracket$.

We have shown that $R_{\overline{\rho}, S}$ is a finite flat complete intersection $\cO$-algebra. In particular, it is reduced if and only if it is generically reduced, e.g.\ if $R_{\overline{\rho}, S}[1/l]$ is an \'etale $E$-algebra. This follows from Proposition \ref{prop_triviality_of_tangent_space_in_generic_fibre_of_galois_deformation_ring}, and this completes the proof.
\end{proof}
We now combine this theorem with Proposition \ref{prop_map_of_deformation_rings_is_finite} to obtain the following result for a general semisimple group $\hG$.
\begin{theorem}\label{thm_khare_wintenberger_method}
Suppose that there exists a representation $i : \hG \to \GL(V)$ of finite kernel such that $i \overline{\rho} : \Gamma_{K, S} \to \GL(V_{k})$ is absolutely irreducible and $l > 2(\dim V - 1)$. Then $R_{\overline{\rho}, S}$ is a reduced finite flat complete intersection $\cO$-algebra. In particular, there exists a finite extension $E'/E$ and a continuous homomorphism $\rho : \Gamma_{K, S} \to \hG(\cO_{E'})$ such that $\rho \text{ mod }(\varpi_{E'}) = \overline{\rho}$.
\end{theorem}
\begin{proof}
Our assumptions imply that $l$ is a very good characteristic for $\SL(V)$. Theorem \ref{thm_finite_flatness_for_SL_N} implies that $R_{i \overline{\rho}, S}$ is a finite $\cO$-algebra. Proposition \ref{prop_map_of_deformation_rings_is_finite} then implies that $R_{\overline{\rho}, S}$ is a finite $\cO$-algebra. Proposition \ref{prop_presentation_of_deformation_ring_in_semisimple_galois_case} then implies that $R_{\overline{\rho}, S}$ is in fact a finite flat complete intersection $\cO$-algebra. Proposition \ref{prop_triviality_of_tangent_space_in_generic_fibre_of_galois_deformation_ring} then implies that $R_{\overline{\rho}, S}[1/l]$ is an \'etale $E$-algebra; in particular, it is reduced. Since a complete intersection ring is reduced if and only if it is generically reduced, we find that $R_{\overline{\rho}, S}$ is in fact reduced. (This is the same argument we have already applied in the case $G = \SL_n$ in the proof of Theorem \ref{thm_finite_flatness_for_SL_N}.) This completes the proof.
\end{proof}
\subsection{Taylor--Wiles places}

We continue with the notation of \S \ref{sec_galois_deformation_theory}. Thus $\hG$ is semisimple, $S$ is a finite set of places of $K = \bbF_q(X)$, and $\overline{\rho} : \Gamma_{K, S} \to \hG(k)$ is absolutely $\hG$-irreducible. We can and do assume, after possibly enlarging $E$, that for every regular semisimple element $h \in \overline{\rho}(\Gamma_{K, S})$, the torus $Z_{\hG}(h)^\circ \subset \hG_k$ is split. Let $\rho_S : \Gamma_{K, S} \to \hG(R_{\overline{\rho}, S})$ denote a representative of the universal deformation. We fix a split maximal torus $\hT \subset \hG$, and write $T$ for the split torus over $\bbZ$ with $X_\ast(T) = X^\ast(\hT)$.
\begin{lemma}\label{lem_local_structure_at_TW_primes}
Let $v \in S$ be a prime such that $\overline{\rho}|_{\Gamma_{K_v}}$ is unramified, $q_v \equiv 1 \text{ mod }l$,  and $\overline{\rho}(\Frob_v)$ is regular semisimple. Let $T_v = Z_\hG(\overline{\rho}(\Frob_v))^\circ$, and choose an inner isomorphism $\varphi : \hT_k \cong T_v$ (there are $\# W$ possible choices). Then:
\begin{enumerate}
\item There exists a unique torus $\widetilde{T}_v \subset \hG_{R_{\overline{\rho}, S}}$ lifting $T_v$ such that $\rho_S|_{\Gamma_{K_v}}$ takes values in $\widetilde{T}_v(R_{\overline{\rho}, S})$, and a unique isomorphism $\widetilde{\varphi} : \hT_{R_{\overline{\rho}, S}} \cong \widetilde{T}_v$ lifting $\varphi$.
\item The homomorphism $\widetilde{\varphi}^{-1} \circ \rho_S|_{I_{K_v}} : I_{K_v} \to \widehat{T}(R_{\overline{\rho}, S})$ has finite $l$-power order.
\end{enumerate}
\end{lemma}
\begin{proof}
Any two split maximal tori of $\hG_k$ are $\hG(k)$-conjugate, so we can choose $g \in \hG(k)$ such that $g \hT_k g^{-1} = T_v$ (this is what we mean by an inner isomorphism). We take $\varphi$ to be conjugation by this element. 

The representation $\rho_S|_{\Gamma_{K_v}}$ factors through the tame quotient $\Gamma_{K_v}^t$ of $\Gamma_{K_v}$. Let $\phi_v \in \Gamma_{K_v}^t$ be an (arithmetic) Frobenius lift and $t_v \in \Gamma_{K_v}^t$ a generator of the $l$-part of tame inertia, so that $\phi_v t_v \phi_v^{-1} = t_v^{q_v}$. Then there exists a unique maximal torus $\widetilde{T}_v \subset \hG_{R_{\overline{\rho}, S}}$ containing the element $\rho_S(\phi_v) \in \hG(R_{\overline{\rho}, S})$ (apply \cite[Exp. XIII, 3.2]{SGA3}). This torus is split, and we can even (\cite[Exp. IX, 7.3]{SGA3}) find an element $\widetilde{g} \in \hG(R_{\overline{\rho}, S})$ lifting $g$ such that $\widetilde{g} \hT_{R_{\overline{\rho}, S}} \widetilde{g}^{-1} = \widetilde{T}_v$. We take $\widetilde{\varphi}$ to be conjugation by this element. We will show that $\rho_S(\Gamma_{K_v})$ takes image in $\widetilde{T}_v(R_{\overline{\rho}, S})$. In particular, this image is abelian, and the first part of the lemma will follow. The second part will then follow by local class field theory.

To do this, we will show by induction on $i \geq 1$ that $\rho_S(t_v) \text{ mod }\ffrm^i$ lies in $\widetilde{T}_v(R_{\overline{\rho}, S} / \ffrm^i)$, where $\ffrm$ denotes the maximal ideal of $R_{\overline{\rho}, S}$. The case $i = 1$ is clear, as $\overline{\rho}$ is unramified at $v$. For the inductive step, we assume that $\rho_S(t_v) \text{ mod }\ffrm^i$ lies in $\widetilde{T}_v$, and show that the same is true mod $\ffrm^{i+1}$. Let $t'_v \in \widetilde{T}_v(R_{\overline{\rho}, S} / \ffrm^{i+1})$ be an element with $\rho_S(t_v) \equiv t'_v \text{ mod }\ffrm^i$. Thus $\rho_S(\phi_v) \text{ mod }\ffrm^{i+1}$ and $t_v'$ commute. We can write $t'_v = \rho_S(t_v) \epsilon$, for some element
\[ \epsilon \in \ker( \hG(R_{\overline{\rho}, S} / \ffrm^{i+1}) \to \hG(R_{\overline{\rho}, S} / \ffrm^{i})) \cong  \widehat{\frg}_k \otimes_k \ffrm^i / \ffrm^{i+1}. \]
(For the existence of this isomorphism, see for example \cite[Proposition 6.2]{Pin98}.) The conjugation action of $\hG(R_{\overline{\rho}, S} / \ffrm^{i+1})$ on the subgroup $ \widehat{\frg}_k \otimes_k \ffrm^i / \ffrm^{i+1}$ factors through the adjoint action of $\hG(k)$ on $\widehat{\frg}_k$. Thus the elements $t'_v$ and $\epsilon$ commute, because $t'_v \text{ mod }\ffrm = \overline{\rho}(t_v)$ is trivial.

In particular, we see that the relation $\phi_v t_v \phi_v^{-1} = t_v^{q_v}$ implies a relation
\[ \rho_S(\phi_v) \rho_S(t_v) \rho_S(\phi_v)^{-1} = \rho_S(\phi_v) t'_v \epsilon^{-1} \rho_S(\phi_v)^{-1} = t'_v \rho_S(\phi_v) \epsilon^{-1} \rho_S(\phi_v)^{-1} =\rho_S(t_v)^{q_v} = (t'_v \epsilon^{-1})^{q_v} = (t'_v)^{q_v} \epsilon^{-1}, \]
hence $(t'_v)^{q_v - 1} = \rho_S(\phi_v) \epsilon^{-1} \rho_S(\phi_v)^{-1} \epsilon$. We write $\epsilon =  X$, for some $X \in \widehat{\frg}_k \otimes_k \ffrm^i / \ffrm^{i+1}$, and decompose $X = X_0 + \sum_{\alpha \in \Phi(\hG_k, T_v)} X_\alpha$ with respect to the Cartan decomposition of $\widehat{\frg}_k$ (with respect to the torus $T_v \subset \hG_k$). We finally get
\[ \rho_S(\phi_v) \epsilon^{-1} \rho_S(\phi_v)^{-1} \epsilon =  - \Ad \overline{\rho}(\phi_v)(X) + X = (t'_v)^{q_v - 1}, \]
and the $\alpha$-component of this is $(1 - \alpha(\overline{\rho}(\phi_v)))X_\alpha = 0$. Since $\overline{\rho}(\phi_v)$ is regular semisimple, we find that $X_\alpha = 0$ for each $\alpha \in \Phi(\hG_k, T_v)$, or equivalently that $\epsilon \in \frt_v \otimes_k \ffrm^i / \ffrm^{i+1}$, where $\frt_v = \Lie T_v$. It follows that $\rho_S(t_v) \text{ mod }\ffrm^{i+1} \in \widetilde{T}_v(R_{\overline{\rho}, S} / \ffrm^{i+1})$. This is what we needed to prove.
\end{proof}
With this lemma in hand, we make the following definition.
\begin{definition}\label{def_taylor_wiles_datum}
A Taylor--Wiles datum for $\overline{\rho} : \Gamma_{K, S} \to \hG(k)$ is a pair $(Q, \{\varphi_v\}_{v \in Q})$ as follows:
\begin{enumerate}
\item $Q$ is a finite set of place $K$, disjoint from $S$, such that for each $v \in Q$, $\overline{\rho}(\Frob_v)$ is regular semisimple and $q_v \equiv 1 \text{ mod }l$.
\item For each $v \in Q$, $\varphi_v : \hT_k \cong Z_\hG(\overline{\rho}(\Frob_v))$ is a choice of inner isomorphism. In particular, the group $Z_\hG(\overline{\rho}(\Frob_v))$ is connected.
\end{enumerate}
If $Q$ is a Taylor--Wiles datum, then we define $\Delta_Q$ to be the maximal $l$-power order quotient of the group $\prod_{v \in Q} T(k(v))$.
\end{definition}
\begin{lemma}\label{lem_diamond_operators_in_deformation_ring_at_TW_primes}
If $(Q, \{\varphi_v\}_{v \in Q})$ is a Taylor--Wiles datum, then $R_{\overline{\rho}, S \cup Q}$ has a natural structure of $\cO[\Delta_Q]$-algebra, and there is a canonical isomorphism $R_{\overline{\rho}, S \cup Q} \otimes_{\cO[\Delta_Q]} \cO \cong R_{\overline{\rho}, S}$.
\end{lemma}
\begin{proof}
Let $\rho_{S \cup Q}$ denote a representative of the universal deformation. Lemma \ref{lem_local_structure_at_TW_primes} shows that for each $v \in Q$, inertia acts on $\rho_{S \cup Q}$  via a character $\chi_v = \widetilde{\varphi}_v^{-1} \circ \rho_{S \cup Q}|_{I_{K_v}} : I_{K_v} \to \hT(R_{\overline{\rho}, S \cup Q})$ which has finite $l$-power order and which is uniquely determined by $\varphi_v$.

This homomorphism $\chi_v$ factors through the quotient $I_{K_v} \to k(v)^\times$ given by local class field theory. The homomorphism $\chi_v : k(v)^\times \to \hT(R_{\overline{\rho}, S \cup Q})$ corresponds, by a simple version of Langlands duality (cf. Lemma \ref{lem_local_langlands_for_split_tori} below), to a character $\chi_v' : T(k(v)) \to R_{\overline{\rho}, S \cup Q}^\times$. After taking products we get an algebra homomorphism $\cO[\Delta_Q] \to R_{\overline{\rho}, S \cup Q}$. The quotient $R_{\overline{\rho}, S \cup Q} \otimes_{\cO[\Delta_Q]} \cO$ is identified with the maximal quotient over which $\rho_{S \cup Q}$ is unramified, hence with $R_{\overline{\rho}, S}$. This completes the proof.
\end{proof}
The following definition plays the role of `big' or `adequate' subgroups of $\GL_n(k)$ in previous works on automorphy lifting (compare \cite{Clo08, Tho12}).
\begin{definition}\label{def_abundant_subgroups}
We say that a subgroup $H \subset \hG(k)$ is $\hG$-abundant if it satisfies the following conditions:
\begin{enumerate}
\item The groups $H^0(H, \widehat{\frg}_k)$, $H^0(H, \widehat{\frg}^\vee_k)$, $H^1(H, \widehat{\frg}^\vee_k)$ and $H^1(H, k)$ all vanish. For each regular semisimple element $h \in H$, the torus $Z_{\hG}(h)^\circ \subset \hG_k$ is split.
\item For every simple $k[H]$-submodule $W \subset \widehat{\frg}^\vee_k$, there exists a regular semisimple element $h \in H$ such that $W^h \neq 0$ and $Z_{\hG}(h)$ is connected. (We recall that $Z_{\hG}(h)$ is always connected if $\hG_k$ is simply connected.)
\end{enumerate}
\end{definition}
\begin{proposition}\label{prop_existence_of_taylor_wiles_data}
Suppose that the group $\overline{\rho}(\Gamma_{K(\zeta_l)}) \subset \hG(k)$ is $\hG$-abundant and that for each $v \in S$, the group $H^0(K_v, \frg_k^\vee(1))$ is trivial. Then for each $N \geq 1$, there exists a Taylor--Wiles datum $(Q, \{ \varphi_v \}_{v \in Q})$ satisfying the following conditions:
\begin{enumerate}
\item For each $v \in Q$, $q_v \equiv 1 \text{ mod }l^N$, and $\# Q = h^1(\Gamma_{K, S}, \widehat{\frg}_k) = h^1(\Gamma_{K, S}, \widehat{\frg}_k^\vee(1))$.
\item We have $h^1_{Q\text{-triv}}(\Gamma_{K, S}, \widehat{\frg}_k^\vee(1)) = 0$. (By definition, this is the dimension of the kernel of the map $H^1(\Gamma_{K, S}, \widehat{\frg}_k^\vee(1)) \to \oplus_{v \in Q} H^1(K_v, \widehat{\frg}_k^\vee(1))$.) 
\item There exists a surjection $\cO \llbracket X_1, \dots, X_g \rrbracket \to R_{\overline{\rho}, S \cup Q}$ with $g = h^1(\Gamma_{K, S}, \widehat{\frg}_k) + (r-1)\# Q$, where $r = \rank \hG$.
\end{enumerate}
\end{proposition}
\begin{proof}
The proof is a variation on the usual themes. Fix an integer $N \geq 1$. We claim that it suffices to find a Taylor--Wiles datum $(Q, \{ \varphi_v \}_{v \in Q})$ satisfying just the following conditions:
\begin{itemize}
\item For each $v \in Q$, $q_v \equiv 1 \text{ mod }l^N$, and $\# Q = h^1(\Gamma_{K, S}, \widehat{\frg}_k^\vee(1))$.
\item We have $h^1_{Q\text{-triv}}(\Gamma_{K, S}, \widehat{\frg}_k^\vee(1)) = 0$. 
\end{itemize}
Indeed, given a Taylor--Wiles datum, the Cassels--Poitou--Tate exact sequence takes the form (see \cite[Theorem 6.2]{Ces15}):
\[ \xymatrix@1{ 0 \ar[r] & H^1_{Q\text{-triv}}(\Gamma_{K, S \cup Q}, \widehat{\frg}_k^\vee(1)) \ar[r] & H^2(\Gamma_{K, S \cup Q}, \widehat{\frg}_k) \ar[r] & \oplus_{v \in Q} H^2(K_v, \widehat{\frg}_k) \ar[r] & 0.} \]
Combining this with the Euler characteristic formula (already used in the proof of Proposition \ref{prop_presentation_of_deformation_ring_in_semisimple_galois_case}), we obtain the formula
\[ h^1(\Gamma_{K, S \cup Q}, \widehat{\frg}_k) - h^1_{Q\text{-triv}}(\Gamma_{K, S}, \widehat{\frg}_k^\vee(1)) = \sum_{v \in Q} (\rank \hG) = r \# Q. \]
(We have used the equalities  $h^2(K_v, \widehat{\frg}_k) = h^0(K_v, \widehat{\frg}^\vee_k(1)) = r$ ($v \in Q$) and $h^0(\Gamma_{K(\zeta_l)}, \widehat{\frg}_k) = 0$.) In particular, the case $Q = \emptyset$ gives the equality $h^1(\Gamma_{K, S}, \widehat{\frg}_k) = h^1(\Gamma_{K, S}, \widehat{\frg}_k^\vee(1))$. For $Q$ satisfying the above two bullet points, we obtain 
\[ h^1(\Gamma_{K, S \cup Q}, \widehat{\frg}_k) = r \# Q = h^1(\Gamma_{K, S \cup Q}, \widehat{\frg}_k) + (r-1) \# Q, \]
 hence the third point in the statement of the proposition.

By induction, it will suffice to show that for any non-zero cohomology class $[\psi] \in H^1(\Gamma_{K, S}, \widehat{\frg}_k^\vee(1))$, we can find infinitely many places $v \not\in S$ of $K$ such that $q_v \equiv 1 \text{ mod }l^N$, $\overline{\rho}(\Frob_v)$ is regular semisimple with connected centralizer in $\hG_k$, and $\res_{K_v} [\psi] \neq 0$. Indeed, we can then add one place at a time to kill off all the elements of the group $H^1(\Gamma_{K, S}, \widehat{\frg}_k^\vee(1))$. By the Chebotarev density theorem, it will even suffice to find for each non-zero cohomology class $[\psi] \in H^1(\Gamma_{K, S}, \widehat{\frg}_k^\vee(1))$ an element $\sigma \in \Gamma_{K(\zeta_{l^N})}$ such that $\overline{\rho}(\sigma)$ is regular semisimple with connected centralizer in $\hG_k$, and the $\sigma$-equivariant projection of $\psi(\sigma)$ to $\widehat{\frg}_k^\vee(1)^\sigma$ is non-zero. 

To this end, let $K_N = K(\zeta_{l^N})$, and let $L_N$ denote the extension of $K_N$ cut out by $\overline{\rho}$. Our hypothesis that $H^1(\overline{\rho}(\Gamma_{K(\zeta_l)}), k) = 0$ (part of Definition \ref{def_abundant_subgroups}) implies that we have $\overline{\rho}(\Gamma_{K_1}) = \overline{\rho}(\Gamma_{K_N})$. In particular, the group $\overline{\rho}(\Gamma_{K_N})$ is $\hG$-abundant, which implies (by inflation-restriction) that the element $\Res_{L_N} [\phi]$ determines a non-zero, $\Gamma_{K_N}$-equivariant homomorphism $f : \Gamma_{L_N} \to \widehat{\frg}^\vee_k(1)$. Let $W$ be a simple $k[\Gamma_{K_N}]$-submodule of the $k$-span of $f(\Gamma_{L_N})$, and choose $\sigma_0 \in \Gamma_{K_N}$ such that $\overline{\rho}(\sigma_0)$ is regular semisimple with connected centralizer in $\hG_k$ and $W^{\sigma_0} \neq 0$. We write $p_{\sigma_0} : \widehat{\frg}_k^\vee(1) \to \widehat{\frg}_k^\vee(1)^{\sigma_0}$ for the $\sigma_0$-equivariant projection. Then the condition $W^{\sigma_0} \neq 0$ is equivalent to the condition $p_{\sigma_0} W \neq 0$. 

If $p_{\sigma_0} \psi(\sigma_0) \neq 0$, then we're done on taking $\sigma = \sigma_0$. Otherwise, we can assume that this projection is zero, in which case we consider elements of the form $\sigma = \tau \sigma_0$ for $\tau \in \Gamma_{L_N}$. For such an element, we have
$\psi(\sigma) = f(\tau) + \psi(\sigma_0)$ and $\overline{\rho}(\sigma) = \overline{\rho}(\sigma_0)$, so the proof will be finished if we can find $\tau \in \Gamma_{L_N}$ such that $p_{\sigma_0} f(\tau) \neq 0$. Suppose for contradiction that there is no such $\tau$, or equivalently that $p_{\sigma_0} \circ f = 0$. Then the image of $f$ is contained in the unique $\sigma_0$-invariant complement of $\widehat{\frg}_k^\vee(1)^{\sigma_0} \subset \widehat{\frg}_k^\vee(1)$, implying that we must have $p_{\sigma_0} W = 0$. This is a contradiction.
\end{proof}

\section{Compatible systems of Galois representations}\label{sec_compatible_systems_of_galois_representations}

Let $\bbF_q$ be a finite field, and let $X$ be a smooth, projective, geometrically connected curve over $\bbF_q$, $K = \bbF_q(X)$. Let $\hG$ be a split reductive group over $\bbZ$, and fix an algebraic closure $\overline{\bbQ}$ of $\bbQ$.
\begin{definition}\label{def_compatible_sytems}
A compatible system of $\hG$-representations is a tuple $(S, (\rho_\lambda)_\lambda)$ consisting of the following data:
\begin{itemize}
\item A finite set $S$ of places of $K$.
\item A system of continuous and $\hG$-completely reducible representations $\rho_\lambda : \Gamma_{K, S} \to \hG(\overline{\bbQ}_\lambda)$, indexed by the prime-to-$q$ places $\lambda$ of $\overline{\bbQ}$, such that for any place $v \not\in S$ of $K$, the semisimple conjugacy class of $\rho_\lambda(\Frob_v)$ in $\hG$ is defined over $\overline{\bbQ}$ and independent of the choice of $\lambda$. 
\end{itemize}
If $\lambda_0$ is a prime-to-$q$ place of $\overline{\bbQ}$ and $\sigma : \Gamma_K \to \hG(\overline{\bbQ}_{\lambda_0})$ is a continuous, almost everywhere unramified representation, we say that the compatible system $(S, (\rho_\lambda)_\lambda)$ contains $\sigma$ if there is an isomorphism $\sigma \cong \rho_{\lambda_0}$ (i.e.\ these two representations are $\hG(\overline{\bbQ}_{\lambda_0})$-conjugate). 
\end{definition}
The semisimple conjugacy class of an element $g \in \hG(\overline{\bbQ}_\lambda)$ is by definition the conjugacy class of the semisimple part in its Jordan decomposition $g = g^s g^u$. The condition of being defined over $\overline{\bbQ}$ in this definition can be rephrased as follows: for any $f \in \bbZ[\hG]^\hG$, the number $f(\rho_\lambda(\Frob_v)) \in \overline{\bbQ}_\lambda$ in fact lies in $\overline{\bbQ}$ and is independent of $\lambda$.

If $\hG \neq \GL_n$, then compatible systems of $\hG$-representations are not generally determined by individual members. For this reason, Definition \ref{def_compatible_sytems} should be regarded as provisional. Our main observation in this section is that we recover this uniqueness property if we restrict to compatible systems of $\hG$-representations where one (equivalently, all) members have Zariski dense image.
\begin{definition}
Let $(S, (\rho_\lambda)_\lambda)$ and $(T, (\sigma_\lambda)_\lambda)$ be compatible systems of $\hG$-representations.
\begin{enumerate}
\item We say that these systems are weakly equivalent if for all $v \not\in S \cup T$, the semisimple conjugacy classes of $\rho_\lambda(\Frob_v)$ and $\sigma_\lambda(\Frob_v)$ in $\hG(\overline{\bbQ})$ are the same.
\item We say that these systems are equivalent if for all prime-to-$q$ places $\lambda$ of $\overline{\bbQ}$, the representations $\rho_\lambda$, $\sigma_\lambda$ are $\hG(\overline{\bbQ}_\lambda)$-conjugate. 
\end{enumerate}
\end{definition}
It is clear from the definition that equivalence implies weak equivalence. Note that if $\rho$ is a given representation, then any two compatible systems containing $\rho$ are, by definition, weakly equivalent.
\begin{lemma}\label{lem_weak_equivalence_of_compatible_systems}
Let $(S, (\rho_\lambda)_\lambda)$ and $(T, (\sigma_\lambda)_\lambda)$ be compatible systems of Galois representations. Then the following conditions are equivalent:
\begin{enumerate}
\item  These systems are weakly equivalent.
\item For every representation $R : \hG_{\overline{\bbQ}} \to \GL(V)$ of $\hG$ over $\overline{\bbQ}$, and for every prime-to-$q$ place $\lambda$ of $\overline{\bbQ}$  we have $R \circ \rho_\lambda \cong R \circ \sigma_\lambda$.
\item For every representation $R : \hG_{\overline{\bbQ}} \to \GL(V)$ of $\hG$ over $\overline{\bbQ}$, and for some prime-to-$q$ place $\lambda_0$ of $\overline{\bbQ}$  we have $R \circ \rho_{\lambda_0} \cong R \circ \sigma_{\lambda_0}$.
\end{enumerate}
\end{lemma} 
\begin{proof}
The condition that each $\rho_\lambda$ is $\hG$-completely reducible is equivalent to asking that the Zariski closure of the image of each $\rho_\lambda$ has reductive connected component (see \cite[Proposition 4.2]{Ser05}). This implies that for any representation $R$, the representation $R \circ \rho_\lambda$ is semisimple, and is therefore determined up to isomorphism by its character. The lemma now follows immediately from Corollary \ref{cor_equivalence_of_representations_with_similar_trace} and the fact that the ring $\overline{\bbQ}[\hG]^{\hG}$ is generated by the characters of the irreducible representations of $\hG_{\overline{\bbQ}}$.
\end{proof}
\begin{proposition}\label{prop_zariski_dense_and_locally_conjugate_implies_globally_conjugate}
Let $\lambda$ be a prime-to-$q$ place of $\overline{\bbQ}$, and let $\rho, \rho' : \Gamma_K \to \hG(\overline{\bbQ}_\lambda)$ be continuous almost everywhere unramified homomorphisms. Suppose that $\rho$ has Zariski dense image and that for all but finitely many places $v$ of $K$, the semisimple conjugacy classes of $\rho(\Frob_v)$ and $\rho'(\Frob_v)$ are the same. Then $\rho'$ also has Zariski dense image and $\rho, \rho'$ are $\hG(\overline{\bbQ}_\lambda)$-conjugate.
\end{proposition}
\begin{proof}
We can assume without loss of generality that $\rho'$ is $\hG$-completely reducible. Fix a finite set $S$ of places of $K$ such that both $\rho, \rho'$ factor through $\Gamma_{K, S}$. We first observe that for any $\gamma \in \Gamma_{K, S}$, the elements $\rho(\gamma)$ and $\rho'(\gamma)$ have $\hG(\overline{\bbQ}_\lambda)$-conjugate semisimple part. Indeed, it suffices to show that for all $f \in \bbZ[\hG]^{\hG}$, we have $f(\rho(\gamma)) = f(\rho'(\gamma))$. This follows from the corresponding statement for Frobenius elements, by Corollary \ref{cor_equivalence_of_representations_with_similar_trace}. 

Choose a faithful representation $R : \hG_{\overline{\bbQ}} \to \GL(V)$ of $\hG$, and let $\hG' = R(\hG_{\overline{\bbQ}})$. Then the image $R(\rho(\Gamma_{K, S}))$ is Zariski dense in $\hG'$, and $R \circ \rho, R \circ \rho'$ are isomorphic. For dimension reasons, we therefore find that $R(\rho'(\Gamma_{K, S}))$ is Zariski dense in $\hG'$, and hence that $\rho'$ has Zariski dense image in $\hG$.

We now show that the representations $R \circ \rho$ and $R \circ \rho'$ are $\hG'(\overline{\bbQ}_\lambda)$-conjugate, as homomorphisms into $\hG'(\overline{\bbQ}_\lambda)$. Let $g \in \GL(V)(\overline{\bbQ}_\lambda)$ be such that $g( R \circ \rho )g^{-1} = R \circ \rho'$. Then $g \in N_{\GL(V)}(\hG')(\overline{\bbQ}_\lambda)$, and we need to show that $g$ induces an inner automorphism of $\hG'$. Let $\theta : \hG' \to \hG'$ denote the automorphisms induced by conjugation by $g$. We know that if $\Frob_v \in \Gamma_{K, S}$ is a Frobenius element, then $\theta$ leaves invariant the semisimple conjugacy class of $\rho(\Frob_v)$. This implies that $\theta$ leaves invariant all semisimple conjugacy classes of $\hG'$. Indeed, these semisimple conjugacy classes are in bijection with points of the quotient $\hG' \dquot \hG'$, and the $\overline{\bbQ}_\lambda$-points corresponding to elements $\rho(\Frob_v)$ are Zariski dense (as follows from an $l$-adic variant of the Chebotarev density theorem).

We therefore need to show that if $\theta$ is an automorphism of a reductive group $H$ over an algebraically closed field of characteristic 0, and $\theta$ acts trivially on $H \dquot H$, then $\theta$ is an inner automorphism. After composing $\theta$ with an inner automorphism, we can assume that $\theta$ preserves a pinning $(T, B, \{ X_\alpha \}_{\alpha \in R})$, where $T \subset B \subset H$ are a maximal torus and Borel subgroup, $R \subset \Phi(H, T)$ is the corresponding set of simple roots, and $X_\alpha$ ($\alpha \in R$) is a basis of the $\alpha$-root space $\frh_\alpha \subset \frh$. Then $\theta$ corresponds to a symmetry of the Dynkin diagram of $H$; in particular, it is the trivial automorphism if and only if its restriction to $T$ is the identity (see \cite[Ch. VIII, \S 5, No. 2]{Bou05}). Let $W$ denote the Weyl group of $G$, let $\Lambda = \bbZ \Phi(H, T) \subset X^\ast(T)$ denote the root lattice, and let $\Aut(\Lambda)$ denote the group of automorphisms of $\Lambda$ which leave $\Phi$, $\Phi^\vee$ invariant. Then there is a split short exact sequence
\[ \xymatrix@1{ 1 \ar[r] & W \ar[r] & \Aut(\Lambda) \ar[r] & \Out(\Lambda) \ar[r] & 1, } \]
where the splitting is given by lifting a class of ``outer'' automorphisms to the unique one which leaves $R$ invariant. In particular, the image of $\theta$ in $\Aut(\Lambda)$ lies in the image of this splitting, by construction. It follows that the restriction of $\theta$ to $T$ is non-trivial if and only if its restriction to the quotient $T / W$ is non-trivial. Our assumption that $\theta$ acts trivially on the quotient $H \dquot H \cong T / W$ then implies that $\theta$ is indeed the identity automorphism. This completes the proof.
\end{proof}
\begin{theorem}\label{thm_existence_of_compatible_systems_containing_a_given_representation}
Suppose that $\hG$ is semisimple. Let $\lambda_0$ be a prime-to-$q$ place of $\overline{\bbQ}$, and let $\rho: \Gamma_{K, S}  \to \hG(\overline{\bbQ}_{\lambda_0})$ be a continuous homomorphism with Zariski dense image. Then:
\begin{enumerate}
\item There exists a compatible system $(S, (\rho_\lambda)_\lambda)$ containing $\rho$. Moreover, each constituent representation $\rho_\lambda$ has Zariski dense image.
\item Any other compatible system containing $\rho$ is equivalent to $(S, (\rho_\lambda)_\lambda)$.
\end{enumerate}
\end{theorem}
\begin{proof}
We first show the existence of the compatible system. We apply \cite[Theorem 1.4]{Chi04}, which along with the other results of that paper (in particular, \cite[Theorem 6.12]{Chi04}) says the following:
\begin{itemize}
\item Let $(S, (\alpha_\lambda)_\lambda)$ be a compatible system of $\GL_n$-representations, each pure of weight 0. For each prime-to-$q$ place $\lambda$ of $\overline{\bbQ}$, let $G_\lambda$ denote the Zariski closure of the image of $\alpha_\lambda$. Suppose that $G_{\lambda_0}$ is connected.

Then there exists a reductive group $G_0$ over $\overline{\bbQ}$ and for each prime-to-$q$ place $\lambda$ an isomorphism $\phi_\lambda : G_\lambda \cong G_{0, \overline{\bbQ}_\lambda}$, with the following property: for any irreducible representation $\theta : G_0 \to \GL(V)$, the system $(S, (\theta_{\overline{\bbQ}_\lambda} \circ \phi_\lambda \circ \alpha_\lambda)_\lambda)$ is a compatible system of $\GL(V)$-representations.
\end{itemize}
To apply this, let $R : \hG_{\overline{\bbQ}} \to \GL(V)$ be a faithful representation, and let $\alpha_{\lambda_0} = R \circ \rho$. Then $\alpha_{\lambda_0}$ is semisimple, and we have a tautological isomorphism $j_0 : \hG_{\overline{\bbQ}_\lambda} \cong G_{\lambda_0}$. Each irreducible constituent of $\alpha_{\lambda_0}$ has trivial determinant (because $\rho$ has Zariski dense image and $\hG$ is semisimple, hence has no non-trivial characters). By \cite[Th\'eor\`eme VII.6]{Laf02}, $\alpha_{\lambda_0}$ lives in a compatible system $(S, (\alpha_\lambda)_\lambda)$ of $\GL_n$-representations, each pure of weight 0. Let us apply Chin's results, and define $\sigma_\lambda = \phi_\lambda \circ \alpha_\lambda$. We claim that $(S, (\sigma_\lambda)_\lambda)$ is a compatible system of $G_0$-representations. 

Since the ring of invariant functions on $G_0$ is generated by characters, it is enough to show that for any irreducible representation $\theta$ of $G_0$, and for any place $v \not\in S$ of $K$, the number $\tr \theta_{\overline{\bbQ}_\lambda} ( \sigma_\lambda ( \Frob_v) )$ lies in $\overline{\bbQ}$ and is independent of $\lambda$. However, this is exactly the statement that the representations $\theta_{\overline{\bbQ}_\lambda} \circ \sigma_\lambda = \theta_{\overline{\bbQ}_\lambda} \circ \phi_\lambda \circ \alpha_\lambda$ lie in a compatible system.

To obtain a compatible system of $\hG$-representations, we choose an isomorphism $k_0 : \hG_{\overline{\bbQ}} \cong G_0$ which is in the inner class of $\phi_{\lambda_0} \circ j_0$, and define $\rho_\lambda = k_{0, \overline{\bbQ}_\lambda}^{-1} \circ \sigma_\lambda$. Then $(S, (\rho_\lambda)_{\lambda})$ is a compatible system of $\hG$-representations, and $\rho_{\lambda_0}$ is $\hG(\overline{\bbQ}_{\lambda_0})$-conjugate to $\rho$. 

This completes the construction of the compatible system containing $\rho$, and shows that each constituent has Zariski dense image. We now show that any other compatible system containing $\rho$ is equivalent to $(S, (\rho_\lambda)_\lambda)$. If $(T, (\rho'_\lambda)_\lambda)$ is such a compatible system, then for every prime-to-$q$ place $\lambda$ of $\overline{\bbQ}$, the representations $\rho_\lambda$ and $\rho'_\lambda$ are weakly equivalent, by Lemma \ref{lem_weak_equivalence_of_compatible_systems}. Moreover, $\rho_\lambda$ has Zariski dense image. By Proposition \ref{prop_zariski_dense_and_locally_conjugate_implies_globally_conjugate}, $\rho_\lambda$ and $\rho'_\lambda$ are $\hG(\overline{\bbQ}_\lambda)$-conjugate. This implies that the two compatible systems are in fact equivalent.
\end{proof}
The following proposition is an application of results of Larsen \cite{Lar95}. 
\begin{proposition}\label{prop_big_image_in_compatible_systems}
Suppose that $\hG$ is semisimple and simply connected. Let $(S, (\rho_\lambda)_\lambda)$ be a compatible system of $\hG$-representations such that some (equivalently, every) representation $\rho_\lambda$ has Zariski dense image in $\hG(\overline{\bbQ}_\lambda)$. Then, after passing to an equivalent compatible system, we can find a number field $E \subset \overline{\bbQ}$ with the following properties:
\begin{enumerate}
\item For every prime-to-$q$ place $\lambda$ of $\overline{\bbQ}$, the image of $\rho_\lambda$ is contained inside $\hG(E_\lambda)$.
\item There exists a set $\cL'$ of rational primes of Dirichlet density 0 with the following property: if $l$ splits in $E / \bbQ$ and $l \not\in \cL'$, and $\lambda$ is a place of $\overline{\bbQ}$ above $l$, then $\rho_\lambda$ has image equal to $\hG(\bbZ_l)$.
\end{enumerate}
\end{proposition}
\begin{proof}
Let $R : \hG_{\overline{\bbQ}} \to \GL(V)$ be a faithful representation. Then (\cite[Theorem 4.6]{Chi04}) we can find a number field $E \subset \overline{\bbQ}$ such that for every prime-to-$q$ place $\lambda$ of $\overline{\bbQ}$, and for every place $v \not\in S$ of $K$, the characteristic polynomial of $(R \circ \rho_\lambda)(\Frob_v)$ has coefficients in $E \subset E_\lambda$ and is independent of $\lambda$. By \cite[Lemma 6.4]{Chi04}, we can find a place $v_0 \not\in S$ of $K$ such that for every prime-to-$q$ place $\lambda$ of $\overline{\bbQ}$, the Zariski closure of the group generated by $(R \circ \rho_\lambda)(\Frob_{v_0})^\text{s}$ is a maximal torus of $R(\hG_{\overline{\bbQ}_\lambda})$; in particular, it is connected, and $\rho_\lambda(\Frob_{v_0})$ is a regular semisimple element of $\hG(\overline{\bbQ}_\lambda)$.

After possibly enlarging $E$, we can assume that for every place $v \not\in S$ of $K$, the conjugacy class of $\rho_\lambda(\Frob_v)$ in $\hG(\overline{\bbQ}_\lambda)$ is defined over $E$ and independent of $\lambda$. We can moreover assume that the characteristic polynomial of each $(R \circ \rho_\lambda)(\Frob_{v_0})$ has all of its roots in $E$, and that the conjugacy class of $\rho_\lambda(\Frob_{v_0})$ has a representative in $\hG(E)$. Choose a place $w_0$ of $K_S$ lying above $v_0$. After passing to an equivalent compatible system, we can suppose that $\rho_\lambda(\Frob_{w_0}) \in \hG(E)$ is independent of $\lambda$; the Zariski closed subgroup it generates is a split maximal torus of $\hG_E$.

 Let $\sigma \in \Gamma_{E_\lambda} = \Gal(\overline{\bbQ}_\lambda / E_\lambda)$. The representation $\rho_\lambda^\sigma$ is equivalent to $\rho_\lambda$, by Proposition \ref{prop_zariski_dense_and_locally_conjugate_implies_globally_conjugate}, so there exists a (necessarily unique) $g \in \hG^\text{ad}(\overline{\bbQ}_\lambda)$ such that $g \rho_\lambda^\sigma g^{-1} = \rho_\lambda$. In particular, we have $g \rho_\lambda(\Frob_{w_0}) g^{-1} = \rho_\lambda(\Frob_{w_0})$, hence $g$ lies in the centralizer $Z_{\hG^\text{ad}}(\rho_\lambda(\Frob_{w_0})) = T$, say. We have thus defined a 1-cocycle $\sigma \mapsto g$ with values in $T(\overline{\bbQ}_\lambda)$. (We note that this 1-cocycle is continuous when $T(\overline{\bbQ}_\lambda)$ is endowed with the discrete topology, since $\rho_\lambda$ can be defined over a finite extension of $\bbQ_l$, by the second part of Theorem \ref{thm_reduction_modulo_l}.) Since $T$ is a split torus over $E$, the group $H^1(E_\lambda, T(\overline{\bbQ}_\lambda))$ is trivial, showing that we can conjugate $\rho_\lambda$ by an element of $T(\overline{\bbQ}_\lambda)$ to force it to take values in $\hG(E_\lambda)$, as desired. This completes the proof of the first part of the proposition.

For the second part, we can assume that $E$ is Galois over $\bbQ$, and that the faithful representation $R : \hG_E \to \GL(V)$ is defined over $E$. We will apply \cite[Proposition 7.1]{Sno09} to the compatible system $(S, (R \circ \rho_\lambda)_\lambda)$. This result is deduced from the main theorem of \cite{Lar95}, and implies the following: there exists an open normal subgroup $\Delta \subset \Gamma_K$ and a set $\cL'$ of rational primes of Dirichlet density $0$, such that if $\lambda$ is a place of $\overline{\bbQ}$ above a prime $l$ split in $E$, and $l \not\in \cL'$, then $\rho_\lambda(\Delta)$ is a hyperspecial subgroup of $\hG(E_\lambda) = \hG(\bbQ_l)$. All hyperspecial subgroups of $\hG(\bbQ_l)$ are conjugate under the action of $\hG^\text{ad}(\bbQ_l)$, so we can further assume (after replacing $(S, (\rho_\lambda)_\lambda)$ by an equivalent compatible system) that for any such place $\lambda$, $\rho_\lambda(\Delta)$ actually equals $\hG(\bbZ_l)$. 

Then $\rho_\lambda(\Gamma_K)$ is a compact subgroup of $\hG(\bbQ_l)$ which contains the hyperspecial subgroup $\hG(\bbZ_l)$ as a subgroup of finite index. Since hyperspecial subgroups can be characterized as those compact subgroups of $\hG(\bbQ_l)$ of maximal volume, it follows that we in fact have $\rho_\lambda(\Gamma_K) = \hG(\bbZ_l)$. This completes the proof of the proposition.
\end{proof}
The following proposition will later be combined with the results of \S \ref{sec_galois_representations_and_their_deformation_theory} and \S \ref{sec_compatible_systems_of_galois_representations} to produce compatible systems of representations with Zariski dense image (see Proposition \ref{prop_construction_of_unramified_and_dense_compatible_system} below).
\begin{proposition}\label{prop_large_residual_image_implies_zariski_dense}
Suppose that $\hG$ is simple and simply connected, and let $l$ be a very good characteristic for $\hG$. Let $E \subset \overline{\bbQ}_l$ be a coefficient field, and let $H \subset \hG(\cO)$ be a closed subgroup such that its image in $\hG(k)$ contains $\hG(\bbF_l)$. Suppose that $l > 2 \dim_{\bbF_l} \widehat{\frg}_{\bbF_l}$. Then $H$ contains a conjugate of $\hG(\bbZ_l)$. In particular, $H$ is Zariski dense in $\hG_{\overline{\bbQ}_l}$.
\end{proposition}
\begin{proof}
 Results of this type will be studied exhaustively in \cite{Boc16}. After shrinking $H$, we can assume that the residue image of $H$ equals $\hG(\bbF_l)$. By the argument of \cite[Lemma 4.2]{Kha16}, it suffices to establish the following claims:
\begin{enumerate}
\item The group $H^1(\hG(\bbF_l), \widehat{\frg}_{\bbF_l})$ is zero.
\item The group homomorphism $\hG(\bbZ / l^2 \bbZ) \to \hG(\bbF_l)$ is non-split.
\item The module $\widehat{\frg}_{\bbF_l}$ is an absolutely irreducible $\bbF_l[\hG(\bbF_l)]$-module.
\end{enumerate}
We observe that our assumption $l > 2 \dim_{\bbF_l} \widehat{\frg}_{\bbF_l}$ implies in particular that $l > 5$. (We can assume that $\hG$ is not the trivial group.) Claim (i) therefore follows from the main theorem of \cite{Vol89}. Claim (iii) follows because the characteristic $l$ is very good for $\hG$ (see the table in \S \ref{sec_invariants_over_a_field}). Claim (ii) is implied by the following:
\begin{enumerate}
\setcounter{enumi}{3}
\item Let $\gamma \in \hG(\bbF_l)$ be an element of exact order $l$. Then no pre-image of $\gamma$ in $\hG(\bbZ / l^2 \bbZ)$ has exact order $l$.
\end{enumerate}
To show (iv), we note that the adjoint representation $\hG \to \GL(\widehat{\frg})$ has kernel $Z_{\hG}$, which has order prime to $l$ (because we work in very good characteristic). It is therefore enough to show that if $l > 2n$ and $\gamma \in \GL_n(\bbF_l)$ has exact order $l$, then no pre-image of $\gamma$ in $\GL_n(\bbZ / l^2 \bbZ)$ has order $l$. (We are applying this with $n = \dim_{\bbF_l} \widehat{\frg}_{\bbF_l}$.)

We show this by direct computation. After conjugation, we can assume that $\gamma = 1 + \overline{u}$, where $\overline{u}$ is an upper-triangular matrix in $M_n(\bbF_l)$ with 0's on the diagonal. Any pre-image of $\gamma$ in $\GL_n(\bbZ / l^2 \bbZ)$ has the form $\widetilde{\gamma} = 1 + u+ l v$, where $u$ is an upper-triangular matrix with 0's on the diagonal in $M_n(\bbZ / l^2 \bbZ)$, and $v \in M_n(\bbZ / l^2 \bbZ)$ is arbitrary. Moreover, $u \not\equiv 0 \text{ mod }l \bbZ$ (since $\gamma \neq 1$, by assumption). We calculate
\[ \widetilde{\gamma}^l = 1 + l (u + lv) + \frac{l(l-1)}{2}(u + lv)^2 + \dots + l (u + lv)^{l-1} + (u + lv)^l \equiv (1 + u)^l - u^l + (u + lv)^l \text{ mod }l^2 \bbZ. \]
We also have
\[ (u + lv)^l \equiv u^l + l( u^{l-1} v + u^{l-2} v u + \dots + v u^{l-1}) \text{ mod }l^2 \bbZ. \]
Since $l > 2n$ and $u^n = 0$, we find $(u+lv)^l \equiv 0 \text { mod }l^2 \bbZ$ and hence $\widetilde{\gamma}^l = (1 + u)^l$. Since $u$ is not zero mod $l$, $(1 + u)$ has exact order $l^2$ in $\GL_n(\bbZ / l^2 \bbZ)$, showing that $\widetilde{\gamma}^l \neq 1$, as claimed. This shows claim (iv) and completes the proof.
\end{proof}
\section{A local calculation}\label{sec_local_calculation}
In this section, we will analyse Hecke modules of the type that arise when considering Taylor--Wiles places. Some of the calculations are quite similar to those of \cite[\S 5]{Kha15}.

Let $K$ be a non-archimedean local field with residue field $\bbF_q$ of characteristic $p$, and fix a choice of uniformizer $\varpi_K$. Fix a prime $l \neq p$ and a square root $p^{1/2}$ of $p$ in $\overline{\bbQ}_l$, as well as a coefficient field $E \subset \overline{\bbQ}_l$ with ring of integers $\cO$ and residue field $k$.

Let $G$ be a split reductive group over $\cO_K$, and fix a choice of split maximal torus and Borel subgroup $T \subset B \subset G$ (defined over $\cO_K$); this determines a set $\Phi^+ \subset \Phi = \Phi(G, T)$ of positive roots, a root basis $\Delta \subset \Phi^+$, and sets $X_\ast(T)^+ \subset X_\ast(T)$ and $X^\ast(T)^+ \subset X^\ast(T)$ of dominant cocharacters and characters, respectively. We write $N$ for the unipotent radical of $B$. We set $W = W(G, T)$. We write $\hT \subset \hB \subset \hG$ for the dual group of $G$, viewed (as usual) as a split reductive group over $\bbZ$. We establish the following running assumptions, which will hold throughout \S \ref{sec_local_calculation}:
\begin{itemize}
\item $q \equiv 1 \text{ mod }l$.
\item $l \nmid \# W$.
\end{itemize}
We introduce open compact subgroups $U = G(\cO_K)$, $U_0 = $ pre-image of $B(\bbF_q)$ under $U \to G(\bbF_q)$, and $U_1 = $ maximal pro-prime-to-$l$ subgroup of $U_0$. Thus $U_0$ is an Iwahori subgroup of $G(K)$, and there is a canonical isomorphism of $U_0 / U_1$ with the maximal $l$-power quotient of $T(\bbF_q)$.  We have the following simple case of Langlands duality:
\begin{lemma}\label{lem_local_langlands_for_split_tori}
Let $A$ be a ring. Then there is a canonical bijection $\chi \leftrightarrow \chi^\vee$ between the following two sets:
\begin{enumerate}
\item The set of characters $\chi : T(K) \to A^\times$.
\item The set of homomorphisms $\chi^\vee : W_K \to \hT(A)$.
\end{enumerate}
It is uniquely characterized as follows: if $\lambda \in X_\ast(T) = X^\ast(\hT)$, then $\lambda \circ \chi^\vee \circ \Art_K = \chi \circ \lambda$ as characters $K^\times \to A^\times$.
\end{lemma}
\begin{proof}
It suffices to note that if $S$ is a split torus over a field $k$, and $A$ is a ring, then there are canonical isomorphisms
\[ \Hom(S(k), A^\times) \cong \Hom(X_\ast(S) \otimes_\bbZ k^\times, A^\times) \cong \Hom(k^\times, \Hom(X^\ast(\widehat{S}), A^\times)) \cong \Hom(k^\times, \widehat{S}(A)). \]
\end{proof}
If $V \subset G(K)$ is an open compact subgroup, then we write $\cH_V$ for the convolution algebra of compactly supported $V$-biinvariant functions $f : G(K) \to \cO$, with respect to the Haar measure that gives $V$ volume equal to 1. If $R$ is an $\cO$-algebra, then we write $\cH_{V, R} = \cH_{V} \otimes_\cO R$. This can be interpreted as a double coset algebra; see \cite[\S 2.2]{New15} or \S \ref{sec_automorphic_forms} below. In particular, it has a basis consisting of the characteristic functions $[V g V]$ of double cosets. If $\Pi$ is any smooth $R[G(K)]$-module, then $\Pi^V$ has a canonical structure of $\cH_{V, R}$-module. If $V \subset V'$ is another open compact subgroup, then there is an inclusion $\cH_{V'} \subset \cH_V$. However, this is not in general an algebra homomorphism (it does not preserve the unit unless $V = V'$).

We will be concerned with the actions of the algebras $\cH_{U_0}$ and $\cH_{U_1}$. We now make some comments on these in turn. The Iwahori--Hecke algebra $\cH_{U_0}$ is extremely well-studied. It has a presentation, the Bernstein presentation, which is an isomorphism
\begin{equation}\label{eqn_Bernstein_presentation} \cH_{U_0} \cong \cO[X_\ast(T)] \widetilde{\otimes} \cO[U_0 \backslash U / U_0].
\end{equation}
 The subalgebra $\cO[U_0 \backslash U / U_0] \subset \cH_{U_0}$ of functions supported in $U$ is finite free as an $\cO$-module, having a basis consisting of the elements $T_w = U_0 \dot{w} U_0$ ($w \in W$), because of the existence of the Bruhat decomposition of $G(\bbF_q)$. (Here $\dot{w}$ denotes a representative in $U$ of the Weyl element $w \in W$.) The other terms appearing are defined as follows. (We refer the reader to \cite{Hai10} for more details.)
\begin{itemize}
\item $\cO[X_\ast(T)]$ is the group algebra of $X_\ast(T)$, a free $\bbZ$-module. The embedding $\cO[X_\ast(T)] \to \cH_{U_0}$ is defined as follows: if $e_\lambda$ is the basis element in $\cO[X_\ast(T)]$ corresponding to a dominant cocharacter $\lambda \in X_\ast(T)^+$, then we send $e_\lambda$ to $q^{-\langle \rho, \lambda \rangle} [ U_0 \lambda(\varpi_K) U_0]$, where $\rho$ is the usual half-sum of the positive roots. One can show that this defines an algebra homomorphism $\cO[X_\ast(T)^+] \to \cH_{U_0}$, which then extends uniquely to a homomorphism $\cO[X_\ast(T)] \to \cH_{U_0}$.
\item The tensor product $\widetilde{\otimes}$ is the usual tensor product as $\cO$-modules, but with a twisted multiplication, which is characterized on basis elements by the formula (for $\alpha \in \Delta$, $\lambda \in X_\ast(T)$):
\begin{equation}\label{eqn_Bernstein_relation} T_{s_\alpha} e_\lambda = e_{s_\alpha(\lambda)} T_{s_\alpha} + ( q - 1)\frac{e_{s_\alpha(\lambda)} - e_\lambda}{1 - e_{- \alpha^\vee}}. 
\end{equation}
We note that the fraction, a priori an element of the fraction field of $\cO[X_\ast(T)]$, in fact lies in $\cO[X_\ast(T)]$.
\end{itemize}
The subalgebra $\cO[X_\ast(T)]^W \subset \cH_{U_0}$ is central. We define $T(K)_0 = T(K) / T(\cO_K)$. In what follows, we will use the identification $T(K)_0 \cong X_\ast(T)$, so that if, for example, $\Pi$ is a smooth $\cO[G(K)]$-module, then $\Pi^{U_0}$ gets the structure of $\cO[T(K)_0]$-module (via the inclusion $\cO[T(K)_0] \cong \cO[X_\ast(T)] \subset \cH_{U_0}$).

If $\chi : T(K) \to \overline{\bbQ}_l^\times$ is a smooth character, then we write $i_B^G \chi$ for the normalized induction. Explicitly, we have
\[ i_B^G \chi = \{ f : G(K) \to \overline{\bbQ}_l \text{ locally constant} \mid \forall b \in B(K), g \in G(K), f (bg) = \delta(b)^{1/2} \chi(b) f(g) \}, \]
with $G(K)$ acting on $f \in i_B^G \chi$ by right translation, and $\delta(t n) = | 2\rho(t) |_K$ the usual modulus character. The Iwahori subgroup $U_0 \subset G(K)$ has the following well-known property:
\begin{lemma}\label{lem_action_of_abelian_algebra_on_iwahori_invariants}
Let $\pi$ be an irreducible admissible $\overline{\bbQ}_l[G(K)]$-module. Then:
\begin{enumerate}
\item $\pi^{U_0} \neq 0$ if and only if $\pi$ is isomorphic to a submodule of a representation $i_B^G \chi$, where $\chi : T(K)_0 \to \overline{\bbQ}_l^\times$ is an unramified character. In this case, $\chi$ is determined up to the action of the Weyl group, and the characters of $\cO[T(K)_0]$ which appear in $\pi^{U_0}$ are among the $w(\chi)$, $w \in W$. 
\item If there exists an $\cO$-lattice $M \subset \pi^{U_0}$ which is stable under the action of $\cO[T(K)_0]$, then $\chi$ in fact takes values in $\overline{\bbZ}_l^\times$. Let $\overline{\chi} = \chi \text{ mod }\ffrm_{\overline{\bbZ}_l}$ denote the reduction modulo $l$ of $\chi$. If $\overline{\chi} : T(K)_0 \to \overline{\bbF}_l^\times$ has trivial stabilizer in the Weyl group, and $\ffrm_{w(\overline{\chi})}$ is the kernel of the homomorphism $\cO[T(K)_0] \to \overline{\bbF}_l$ associated to the character $w(\overline{\chi})$, for some $w \in W$, then $(\pi^{U_0})_{\ffrm_{w(\overline{\chi})}}$ has dimension either 0 or 1 as a $\overline{\bbQ}_l$-vector space.
\end{enumerate} 
\end{lemma}
\begin{proof}
If $\pi$ is any admissible $\overline{\bbQ}_l[G(K)]$-module, then there is a canonical isomorphism $\pi^{U_0} \cong r_N (\pi)^{T(\cO_K)}$ of $\overline{\bbQ}_l[T(K)_0]$-modules, where $r_N (\pi) = \pi_N \otimes \delta^{-1/2}$ is the normalized Jacquet module. Indeed, \cite[Proposition 2.4]{Cas80} says that the projection map $p : \pi^{U_0} \to r_N (\pi)^{T(\cO_K)}$ is a vector space isomorphism, and \cite[Proposition 2.5]{Cas80} says that for any $\lambda \in X_\ast(T)^+$, $v \in \pi^{U_0}$, we have $p([U_0 \lambda(\varpi_K) U_0] v) = \delta^{-1/2}(\lambda(\varpi_K)) (\lambda(\varpi_K) \cdot p(v))$. 

If $\pi$ is any irreducible admissible  $\overline{\bbQ}_l[G(K)]$-module, and $\chi : T(K) \to \overline{\bbQ}_l^\times$ is a smooth character, then Frobenius reciprocity gives a canonical isomorphism
\[ \Hom_{G(K)}(\pi, i_B^G \chi) \cong \Hom_{T(K)}(r_N (\pi), \chi). \]
We see that if $\pi^{U_0} \neq 0$, then $\pi$ embeds as a submodule of $i_B^G \chi$ for an unramified character $\chi$. Conversely, if $\pi$ is a submodule of a representation $i_B^G \chi$, then $\pi^{U_0}$ has $\chi$ as a quotient, hence is in particular non-zero. 

It follows from \cite[2.9, Theorem]{Ber77} that two representations $i_B^G \chi$ and $i_B^G \chi'$ have a common Jordan--H\"older factor if and only if $\chi, \chi'$ are conjugate under the action of the Weyl group. On the other hand, the Jordan--H\"older factor $\overline{\bbQ}_l[T(K)_0]$-modules of $r_N (i_B^G \chi)$ are, with multiplicity, the $w(\chi)$ ($w \in W$), so the factors of $\pi^{U_0}$ must be among the $w(\chi)$. This shows the first part of the lemma. Everything in the second part now follows easily.
\end{proof}
If $\overline{\chi} : T(K)_0 \to \overline{\bbF}_l^\times$ is a character, we will henceforth write $\ffrm_{\overline{\chi}} \subset \cO[T(K)_0]$ for the maximal ideal which is the kernel of the homomorphism $\cO[T(K)_0] \to \overline{\bbF}_l$ associated to the character $\overline{\chi}$, as in the statement of the lemma.

We now consider the algebra $\cH_{U_1}$. We define $T(\cO_K)^l \subset T(\cO_K)$ to be the maximal pro-prime-to-$l$ subgroup of $T(\cO_K)$, $T(\cO_K)_l = T(\cO_K) / T(\cO_K)^l$, and $T(K)_l = T(K) / T(\cO_K)^l$. There is a canonical isomorphism $U_0 / U_1 \cong T(\cO_K)_l$. We define a submonoid $T(K)_l^+ \subset T(K)_l$ as the set of elements of the form $t T(\cO_K)^l$ with $t T(\cO_K) = \lambda(\varpi_K) T(\cO_K)$ for some dominant cocharacter $\lambda$. Observe that the choice of $\varpi_K$ determines an isomorphism $T(K)_l \cong X_\ast(T) \times T(\cO_K)_l$, and that $\lambda$ is uniquely determined by the coset $t T(\cO_K)$. 
\begin{lemma}\label{lem_abelian_subalgebra_of_pro-p_Hecke_algebra}
\begin{enumerate} \item The assignment $e_t \in \cO[T(K)_l^+] \mapsto q^{-\langle \rho, \lambda \rangle} [U_1 t U_1] \in \cH_{U_1}$ determines an algebra homomorphism $\cO[T(K)_l^+] \to \cH_{U_1}$, which extends uniquely to an algebra homomorphism $\cO[T(K)_l] \to \cH_{U_1}$.
\item Let $\Pi$ be a smooth $\cO[G(K)]$-module, $t \in T(K)_l$, and $v \in \Pi^{U_0}$. Then $[U_0 t U_0] v = [U_1 t U_1] v$. In other words, the inclusion $\Pi^{U_0} \subset \Pi^{U_1}$ is compatible with the algebra map $\cO[T(K)_l] \to \cO[T(K)_0]$.
\end{enumerate}
\end{lemma}
\begin{proof}
Let $U_p \subset U$ denote the maximal pro-$p$-subgroup; then $U_p \subset U_1 \subset U_0$ and $U / U_p \cong T(\bbF_q)$. The Hecke algebra $\cH_{U_p}$ enjoys many of the same properties as the Iwahori--Hecke algebra $\cH_{U_0}$; in particular, it admits Iwahori--Matsumoto- and Bernstein-style presentations, see \cite{Vig05, Fli11}. The proofs of many of these properties can be transposed word-for-word to the algebra $\cH_{U_1}$. This is in particular the case for the first part of the current lemma; see e.g.\ \cite[Lemma 2.3]{Fli11} and the remark immediately following \cite[Proposition 4.4]{Fli11}. 

For the second part, we can assume without loss of generality that $t \in T(K)_l^+$. Recall that the action of these Hecke operators can be given as follows: if $v \in \Pi^{V}$, then we decompose $V g V = \coprod_i g_i V$, and set $[V g V] v = \sum_i g_i \cdot v$. It is therefore enough to show that given $t \in T(K)_l^+$, we can find elements $g_1, \dots, g_n \in G(K)$ such that $U_0 t U_0 = \coprod_i g_i U_0$ and $U_1 t U_1 = \coprod_i g_i U_1$. It even suffices to show that the natural map $U_1 / U_1 \cap t U_1 t^{-1} \to U_0 / U_0 \cap t U_0 t^{-1}$ is bijective. It is surjective, because $U_0 = U_1 T(\cO_K)$ and $T(\cO_K) \subset U_0 \cap t U_0 t^{-1}$.

To show injectivity, suppose that $u, v \in U_1$ have the same image in $U_0 / U_0 \cap t U_0 t^{-1}$. Then we can write $u = v w$ with $w \in U_0 \cap t U_0 t^{-1}$, hence $w \in U_1 \cap t U_0 t^{-1}$. To finish the proof, it is therefore enough to show that $U_1 \cap t U_0 t^{-1} = U_1 \cap t U_1 t^{-1}$. Let $N$ denote the unipotent radical of $B$, $\overline{N}$ the unipotent radical of the opposite Borel. We set $U_0^+ = N(K) \cap U_0$, $U_0^- = \overline{N}(K) \cap U_0$, and define $U_1^+, U_1^-$ similarly. Then we have the Iwahori decomposition: the product map $U_0^- \times T(\cO_K) \times U_0^+ \to U_0$ is a bijection. Similarly the map $U_1^- \times T(\cO_K)^l \times U_1^+ \to U_1$ is a bijection and we have $U_0^+ = U_1^+$, $U_0^- = U_1^-$. Since $t$ normalizes $N(K)$ and $\overline{N}(K)$, the result now follows from the existence of the Iwahori decomposition and the fact that the multiplication map $\overline{N} \times T \times N \to G$ is an open immersion: we have
\[ t U_0 t^{-1} = t U_0^- t^{-1} T(\cO_K) t U_0^+ t^{-1} = t U_1^- t^{-1} T(\cO_K) t U_1^+ t^{-1}, \]
hence $U_1 \cap t U_0 t^{-1} = U_1 \cap t U_1 t^{-1}$.
\end{proof}
\begin{lemma}\label{lem_irreducible_admissibles_with_non-zero_pro-p_invariants} Let $\pi$ be an irreducible admissible $\overline{\bbQ}_l[G(K)]$-module. Then:
\begin{enumerate}
\item $\pi^{U_1} \neq 0$ if and only if $\pi$ is isomorphic to a submodule of a representation $i_B^G \chi$, where $\chi : T(K) \to \overline{\bbQ}_l^\times$ is a smooth character which factors through $T(K) \to T(K)_l$. In this case, $\chi$ is determined up to the action of the Weyl group, and the characters of $\cO[T(K)_l]$ which appear in $\pi^{U_1}$ are among the $w(\chi)$, $w \in W$.
\item If $\pi^{U_1} \neq 0$ and there exists an $\cO$-lattice $M \subset \pi^{U_1}$ which is stable under the action of $\cO[T(K)_l]$, then $\chi$ in fact takes values in $\overline{\bbZ}_l^\times$. In this case the reduction modulo $l$ $\overline{\chi} : T(K)_l \to \overline{\bbF}_l^\times$ is unramified; if it has trivial stabilizer in the Weyl group, and $\ffrm_{w(\overline{\chi})}$ is the kernel of the homomorphism $\cO[T(K)_l] \to \overline{\bbF}_l$ associated to $w(\overline{\chi})$, for some $w \in W$, then $(\pi^{U_1})_{\ffrm_{w(\overline{\chi})}}$ has dimension either $0$ or $1$ as a $\overline{\bbQ}_l$-vector space. 
\item If $\pi^{U_1} \neq 0$, there exists an $\cO$-lattice $M \subset \pi^{U_1}$ which is stable under the action of $\cO[T(K)_l]$, $\overline{\chi}$ has trivial stabilizer in the Weyl group, and $(\pi^{U_1})_{\ffrm_{w(\overline{\chi})}} \neq 0$, then the action of $T(K)_l$ on this 1-dimensional vector space is by the character $w(\chi)$. 
\end{enumerate}
\end{lemma}
\begin{proof}
The last two points follow easily from the first. For the first, we observe (cf. the proof of \cite[Theorem 2.1]{Fli11}) that for any admissible $\overline{\bbQ}_l[G(K)]$-module $\pi$, the projection $\pi^{U_1} \to r_N(\pi)^{T(\cO_K)^l}$ is an isomorphism of $\overline{\bbQ}_l[T(K)_l]$-modules. The remainder of the lemma then follows from \cite{Ber77} in the same way as in the proof of Lemma \ref{lem_action_of_abelian_algebra_on_iwahori_invariants}.
\end{proof}
If $\overline{\chi} : T(K)_0 \to \overline{\bbF}_l^\times$ is an unramified character, then we will write $\ffrm_{\overline{\chi}} \subset \cO[T(K)_l]$ for the maximal ideal which is the kernel of the homomorphism $\cO[T(K)_l] \to \overline{\bbF}_l$ which is associated to the character $\overline{\chi}$. This is an abuse of notation, since we have used the same notation to denote a maximal ideal of $\cO[T(K)_0]$. However, we hope that it will not cause confusion, because there is a canonical surjective homomorphism $\cO[T(K)_l] \to \cO[T(K)_0]$ which induces a bijection on maximal ideals. 

With these preliminaries out of the way, we can now start our work proper. Our assumption that $q \equiv 1 \text{ mod }l$ has the following important consequence:
\begin{lemma}\label{lem_iwahori_hecke_algebra_semidirect_product}
There is an isomorphism $\cH_{U_0} \otimes_\cO k \cong k[X_\ast(T) \rtimes W]$.
\end{lemma}
\begin{proof}
This is just the reduction modulo $l$ of the Bernstein presentation (\ref{eqn_Bernstein_presentation}), on noting that $k[U_0 \backslash U / U_0] \cong k[W]$, and that the twisted tensor product (\ref{eqn_Bernstein_relation}) becomes the defining relation of the semidirect product $X_\ast(T) \rtimes W$, because $q = 1$ in $k$.
\end{proof}
\begin{lemma}\label{lem_hecke_projection_to_U_invariants}
Let $\Pi$ be a smooth $\cO[G(K)]$-module, flat over $\cO$, such that $\Pi^{U_0}$ is a finite free $\cO$-module. Suppose that there is exactly one maximal ideal of $\cO[X_\ast(T)]^W$ in the support of $\Pi^{U_0} \otimes_\cO k$, corresponding to the $W$-orbit of an unramified character $ \overline{\chi} : T(K) \to \overline{\bbF}_l^\times$ with trivial stabilizer in $W$. Then:
\begin{enumerate}
\item There is a decomposition $\Pi^{U_0} = \oplus_{w \in W} (\Pi^{U_0})_{\ffrm_{w(\overline{\chi})}}$.
\item The map $[U] : \Pi^{U_0} \to \Pi^U$ induced by the element $[U] = \sum_{w \in W} [U_0 \dot{w} U_0] \in \cH_{U_0}$ restricts to an isomorphism $(\Pi^{U_0})_{\ffrm_{\overline{\chi}}} \to \Pi^U$. 
\end{enumerate}
\end{lemma}
\begin{proof}
Since $\Pi^{U_0}$ is finite free as an $\cO$-module, we have a direct sum decomposition $\Pi^{U_0} = \oplus_\ffrm (\Pi^{U_0})_\ffrm$, where the direct sum runs over the set of maximal ideals of $\cO[T(K)_0]$ which are in the support of $\Pi^{U_0}$; equivalently, the set of maximal ideals of $k[T(K)_0]$ which are in the support of $\Pi^{U_0} \otimes_\cO k$. By assumption, if $\ffrm$ is a maximal ideal in the support corresponding to a character $\overline{\psi} : T(K) \to \overline{\bbF}_l^\times$, then its pullback to $\cO[X_\ast(T)]^W$ equals the pullback of $\ffrm_{\overline{\chi}}$. This in turn implies that $\overline{\psi}$ is a $W$-conjugate of $\overline{\chi}$.

For the second part, we note that both $(\Pi^{U_0})_{\ffrm_{\overline{\chi}}}$ and $\Pi^U$ are finite free $\cO$-modules. To show that the given map is an isomorphism, it therefore suffices to show that the map
\begin{equation}\label{eqn_U_projection_over_O} [U] : (\Pi^{U_0})_{\ffrm_{\overline{\chi}}} \otimes_\cO k \to \Pi^U \otimes_\cO k 
\end{equation}
is an isomorphism. Since $[U : U_0] \equiv \# W \text{ mod }l$, and $l \nmid \# W$ by assumption, we have $\Pi^U = [U] \Pi^{U_0}$, and $[U : U_0]^{-1} [U]$ is a projector onto $\Pi^U$. Writing $M = \Pi^{U_0} \otimes_\cO k$, a $\cH_{U_0} \otimes_\cO k$-module, it is therefore enough to show that the map
\begin{equation}\label{eqn_U_projection_over_k} [U] : M_{\ffrm_{\overline{\chi}}} \to [U] M 
\end{equation}
is an isomorphism. To show this, we note that if $w \in W$, then Lemma \ref{lem_iwahori_hecke_algebra_semidirect_product} implies that the action of $w \in \cH_{U_0} \otimes_\cO k$ sends $M_{\ffrm_{\overline{\chi}}}$ isomorphically to $M_{\ffrm_{w(\overline{\chi})}}$. Since $[U]$ corresponds to the element $\sum_{w \in W} w \in k[W] \subset k[X_\ast(T) \rtimes W]$ under the isomorphism of Lemma \ref{lem_iwahori_hecke_algebra_semidirect_product}, this finally shows that the map given by (\ref{eqn_U_projection_over_k}) is indeed an isomorphism.
\end{proof}
\begin{lemma}\label{lem_torus_characters_with_same_pseudocharacter}
Let $\Omega$ be an algebraically closed field, and let $\psi, \psi' : W_K \to \hT(\Omega)$ be smooth characters, and suppose $Z_{\hG}(\psi) = \hT_\Omega$. Let $\iota : \hT \to \hG$ denote the natural inclusion. Then the following are equivalent:
\begin{enumerate}
\item The characters $\psi$ and $\psi'$ are $W$-conjugate.
\item The $\hG$-pseudocharacters $\tr \iota \psi$ and $\tr \iota \psi'$ of $W_K$ are equal.
\end{enumerate}
\end{lemma}
\begin{proof}
The first condition clearly implies the second. For the converse direction, we note that both $\psi$ and $\psi'$ are $\hG$-completely reducible, so it follows from Theorem \ref{thm_pseudocharacters_biject_with_representations_over_fields} that we can find $g \in \hG(\Omega)$ with $g \psi' g^{-1} = \psi$. In particular, the centralizer of the image of $\psi'$ is a maximal torus of $\hG_\Omega$; since the image is contained inside $\hT(\Omega)$, by assumption, we find that $Z_{\hG}(\psi') = \hT_\Omega$. It then follows  that $g$ normalizes $\hT_\Omega$, hence that $g$ represents an element of $W$.
\end{proof}
The following proposition, which is of a technical nature, will be used in our implementation of the Taylor--Wiles method below; compare \cite[Proposition 5.9]{Tho12}.
\begin{proposition}\label{prop_picking_our_iwahori_spherical_forms}
Let $R$ be a complete Noetherian local $\cO$-algebra with residue field $k$, and let $\Pi$ be a smooth $R[G(K)]$-module such that for each open compact subgroup $V \subset G(K)$, $\Pi^V$ is a finite free $\cO$-module, and $\Pi \otimes_\cO \overline{\bbQ}_l$ is a semisimple admissible $\overline{\bbQ}_l[G(K)]$-module. If $V \subset G(K)$ is an open compact subgroup, define $R_V$ to be the quotient of $R$ that acts faithfully on $\Pi^V$. It is a finite flat local $\cO$-algebra.

Suppose that there exists a homomorphism $\rho_{U_0} : \Gamma_K \to \hG(R_{U_0})$ satisfying the following conditions:
\begin{enumerate}
\item $\rho_{U_0}|_{W_K} \text{ mod }\ffrm_{R_{U_0}} : W_K \to \hG(k)$ is the composition of an unramified character $\overline{\chi}^\vee : W_K \to \hT(k)$ with trivial stabilizer in $W$ with the inclusion $\iota : \hT \to \hG$.
\item For each irreducible admissible $\overline{\bbQ}_l[G(K)]$-module $\pi$ which is a submodule of a representation $i_B^G \chi_\pi$, where $\chi_\pi : T(K) \to \overline{\bbQ}_l^\times$ is an unramified character, let $V_\pi$ denote the $\pi$-isotypic component of $\Pi \otimes_\cO \overline{\bbQ}_l$, and let $R_{U_0, \pi}$ denote the quotient of $R_{U_0}$ which acts faithfully on $\Pi^{U_0} \cap V_\pi$. Let $\rho(\pi)$ denote the representation $\Gamma_K \to \hG(R_{U_0}) \to \hG(R_{U_0, \pi})$. Then the  $\hG$-pseudocharacter of $\rho(\pi)|_{W_K}$ takes values in the scalars 
\[ \overline{\bbQ}_l \subset R_{U_0, \pi} \otimes_\cO \overline{\bbQ}_l \subset \End_{\overline{\bbQ}_l}(V_\pi^{U_0}), \]
 where it is equal to the $\hG$-pseudocharacter of the representation $\iota \chi_\pi^\vee$.
\end{enumerate}
Then there is a canonical isomorphism $(\Pi^{U_0})_{\ffrm_{\overline{\chi}}} \cong \Pi^U$ of $R_{U_0}$-modules.
\end{proposition}
\begin{proof}
The result will follow from Lemma \ref{lem_hecke_projection_to_U_invariants} if we can show that the only maximal ideal of $\cO[X_\ast(T)]^W$ in the support of $\Pi^{U_0} \otimes_\cO k$ is the one corresponding to the $W$-orbit of $\overline{\chi}$. This follows from the existence of $\rho_{U_0}$ and the compatibility condition on its pseudocharacter, as we now explain.

By \cite[Exp. IX, 7.3]{SGA3} and the argument of Lemma \ref{lem_local_structure_at_TW_primes}, we can assume, after perhaps conjugating $\rho_{U_0}$ by an element of $\hG(R_U)$ with trivial image in $\hG(k)$, that $\rho_{U_0}|_{W_K} = \iota \chi^\vee$ for some character $\chi^\vee : W_K \to \hT(R_U)$. If $\pi$ is an irreducible admissible $\overline{\bbQ}_l[G(K)]$-module such that $\pi^{U_0} \neq 0$, then it can be written as a submodule of $i_B^G \chi_\pi$ for some character $\chi_\pi : T(K)_0 \to \overline{\bbQ}_l^\times$. If further $(\Pi^{U_0} \otimes_\cO \overline{\bbQ}_l) \cap V_\pi \neq 0$, then the image of $\Pi^{U_0}$ under the $\cH_{U_0}$-equivariant projection $\Pi^{U_0} \otimes_\cO \overline{\bbQ}_l \to V_\pi^{U_0}$ is an $\cO$-lattice. This implies that $\chi_\pi$ must in fact take values in $\overline{\bbZ}_l^\times$. We need to show that $\chi_\pi \text{ mod }\ffrm_{\overline{\bbZ}_l}$ and $\overline{\chi}$ are $W$-conjugate as characters $T(K) \to \overline{\bbF}_l^\times$. 

Let $\chi(\pi) : T(K) \to R_{U_0, \pi}^\times$ denote the image of $\chi$ under the quotient $R_{U_0} \to R_{U_0, \pi}$. The point $(ii)$ above means that the compositions
\[ \iota \chi_\pi^\vee : W_K \to \hT(\overline{\bbZ}_l) \to \hG(\overline{\bbZ}_l) \]
and
\[ \iota \chi(\pi)^\vee : W_K \to \hT(R_{U_0, \pi}) \to \hG(R_{U_0, \pi}) \]
have the same associated pseudocharacter, which therefore takes values in $\overline{\bbZ}_l$ (viewed as a subring of the scalar endomorphisms inside $\End_{\overline{\bbQ}_l}(V_\pi^{U_0})$). This implies in particular that $\iota \chi_\pi^\vee \text{ mod }\ffrm_{\overline{\bbZ}_l}$ and $\iota \chi^\vee \text{ mod }\ffrm_{R_{U_0}} = \iota \chi(\pi)^\vee \text{ mod }\ffrm_{R_{U_0, \pi}}$ have the same associated $\hG$-pseudocharacter over $\overline{\bbF}_l$; and this implies by Lemma \ref{lem_torus_characters_with_same_pseudocharacter} that they are in fact $W$-conjugate. Applying again the bijection of Lemma \ref{lem_local_langlands_for_split_tori} now concludes the proof.
\end{proof}
We now state another proposition, analogous to \cite[Proposition 5.12]{Tho12}, that will be used in our implementation of the Taylor--Wiles method. A similar argument has been used by Guerberoff \cite{Gue11}.
\begin{proposition}\label{prop_picking_out_pro-p_spherical_forms}
Let $R$ be a complete Noetherian local $\cO$-algebra with residue field $k$, and let $\Pi$ be a smooth $R[G(K)]$-module such that for each open compact subgroup $V \subset G(K)$, $\Pi^V$ is a finite free $\cO$-module, and $\Pi \otimes_\cO \overline{\bbQ}_l$ is a semisimple admissible $\overline{\bbQ}_l[G(K)]$-module. If $V \subset G(K)$ is an open compact subgroup, define $R_V$ to be the quotient of $R$ that acts faithfully on $\Pi^V$, a finite flat local $\cO$-algebra. 

Suppose that there exists a homomorphism $\rho_{U_1} : \Gamma_K \to \hG(R_{U_1})$ satisfying the following conditions:
\begin{enumerate}
\item $\rho_{U_1}|_{W_K} \text{ mod }\ffrm_{R_{U_1}} : W_K \to \hG(k)$ is the composition of an unramified character $\overline{\chi}^\vee : W_K \to \hT(k)$ with trivial stabilizer in $W$ with the inclusion $\iota : \hT \to \hG$. (After conjugating $\rho_{U_1}$ by an element of $\hG(R_{U_1})$ with trivial image in $\hG(k)$, we can then assume that $\rho_{U_1}|_{W_K} = \iota \chi^\vee$ for a uniquely determined character $\chi : T(K) \to R_{U_1}^\times$.)
\item For each irreducible admissible $\overline{\bbQ}_l[G(K)]$-module $\pi$ which is a submodule of a representation $i_B^G \chi_\pi$, where $\chi_\pi : T(K) \to \overline{\bbQ}_l^\times$ is a character which factors through the quotient $T(K) \to T(K)_l$, let $V_\pi$ denote the $\pi$-isotypic component of $\Pi \otimes_\cO \overline{\bbQ}_l$, and let $R_{U_1, \pi}$ denote the quotient of $R_{U_1}$ which acts faithfully on $\Pi^{U_1} \cap V_\pi$. Let $\rho(\pi)$ denote the representation $\Gamma_K \to \hG(R_{U_1}) \to \hG(R_{U_1, \pi})$. Then the pseudocharacter of $\rho(\pi)|_{W_K}$ takes values in the scalars 
\[ \overline{\bbQ}_l \subset R_{U_1, \pi} \otimes_\cO \overline{\bbQ}_l \subset \End_{\overline{\bbQ}_l}(V_{\pi}^{U_1}), \]
 where it is equal to the pseudocharacter of the representation $\iota \chi_\pi^\vee$.
\end{enumerate}
Consider the two actions of the group $T(\cO_K)_l$ on  $(\Pi^{U_1})_{\ffrm_{\overline{\chi}}}$ defined as follows. The first is via the canonical isomorphism $T(\cO_K)_l \cong U_0 / U_1$. The second is defined as in Lemma \ref{lem_diamond_operators_in_deformation_ring_at_TW_primes} using the restriction of the character $\chi : T(K)_l \to R_{U_1}^\times$ to $T(\cO_K)_l$. Under the above conditions, these two actions are the same.
\end{proposition}
\begin{proof}
The proof is similar to, but not exactly the same as, the proof of Proposition \ref{prop_picking_our_iwahori_spherical_forms}. Let $\pi$ be an irreducible admissible $\overline{\bbQ}_l[G(K)]$-module, submodule of $i_B^G \chi_\pi$, where $\chi_\pi : T(K)_l \to \overline{\bbQ}_l^\times$ is a character. Let $\chi(\pi)$ denote the image of $\chi$ under the quotient $R_{U_1} \to R_{U_1, \pi}$. To prove the proposition, it is enough to show that the two actions of $T(\cO_K)_l$ on $(\Pi^{U_1})_{\ffrm_{\overline{\chi}}} \cap V_\pi$, one via the inclusion $T(\cO_K)_l \to T(K)_l$ and the other via the map $\chi(\pi)|_{T(\cO_K)_l} : T(\cO_K)_l \to  R_{U_1, \pi}^\times$, are the same. We can assume that $(\Pi^{U_1})_{\ffrm_{\overline{\chi}}} \cap V_\pi \neq 0$, which implies that $\chi_\pi$ takes values in $\overline{\bbZ}_l^\times$.

The point (ii) above says that the compositions
\[ \iota \chi_\pi^\vee : W_K \to \hT(\overline{\bbZ}_l) \to \hG(\overline{\bbZ}_l) \]
and
\[ \iota \chi(\pi)^\vee : W_K \to \hT(R_{U_1, \pi}) \to \hG(R_{U_1, \pi}) \]
have the same associated $\hG$-pseudocharacter. In particular, Lemma \ref{lem_torus_characters_with_same_pseudocharacter} shows that there exists a (necessarily unique) element $w \in W$ such that $\overline{\chi} = w(\overline{\chi}_\pi)$, and hence $\ffrm_{\overline{\chi}} = \ffrm_{w(\overline{\chi}_\pi)}$. By Lemma \ref{lem_irreducible_admissibles_with_non-zero_pro-p_invariants}, the action of $T(K)_l$ on $(\Pi^{U_1})_{\ffrm_{\overline{\chi}}} \cap V_\pi$ is via the character $w(\chi_\pi)$. Another application of Lemma \ref{lem_torus_characters_with_same_pseudocharacter} shows that for any projection $p : R_{U_1, \pi} \otimes_\cO \overline{\bbQ}_l \to \overline{\bbQ}_l$, the two characters $p \chi(\pi)$ and $p w(\chi_\pi)$ are equal. We must show that $\chi(\pi)|_{T(\cO_K)_l}$ and $w (\chi_\pi)|_{T(\cO_K)_l}$ are equal as characters $T(\cO_K)_l \to (R_{U_1, \pi} \otimes_\cO \overline{\bbQ}_l)^\times$. 

To see this, let $A_\pi$ denote the maximal \'etale $\overline{\bbQ}_l$-subalgebra of $R_{U_1, \pi} \otimes_\cO \overline{\bbQ}_l$; it maps isomorphically to the maximal \'etale quotient  $\overline{\bbQ}_l$-algebra of the Artinian $\overline{\bbQ}_l$-algebra $R_{U_1, \pi} \otimes_\cO \overline{\bbQ}_l$, over which we have shown that $\chi(\pi)|_{T(\cO_K)_l}$ and $w(\chi_\pi)|_{T(\cO_K)_l}$ are indeed equal. The proof is now complete on observing that these characters, being of finite order, in fact take values in $A_\pi^\times$. 
\end{proof}
\section{Automorphic forms}\label{sec_automorphic_forms}

In this section, the longest of this paper, we discuss spaces of automorphic forms with integral structures and prove an automorphy lifting theorem. We fix notation as follows: let $X$ be a smooth, projective, geometrically connected curve over the finite field $\bbF_q$ of residue characteristic $p$, and let $K = \bbF_q(X)$. Let $G$ be a split semisimple group over $\bbF_q$, and fix a choice of split maximal torus and Borel subgroup $T \subset B \subset G$. We write $\hG$ for the dual group of $G$ (considered as a split reductive group over $\bbZ$); it is equipped with a split maximal torus and Borel subgroup $\hT \subset \hB \subset \hG$, and there is a canonical identification $X_\ast(T) = X^\ast(\hT)$ (see \S \ref{sec_reductive_groups_over_Z}).

 The section is divided up as follows. In \S \ref{sec_cusp_forms_and_hecke_algebras} we establish basic notation and describe the Satake transform (which relates unramified Hecke operators on $G$ and on its Levi subgroups). In \S \ref{sec_summary_of_lafforgue} we summarize the work of V. Lafforgue, constructing Galois representations attached to cuspidal automorphic forms, in a way which is suitable for our intended applications. In \S \ref{sec_automorphic_forms_free_over_diamond_operators}, we prove an auxiliary result stating that under suitable hypotheses, certain spaces of cuspidal automorphic forms with integral structures are free over rings of diamond operators. Finally, in \S \ref{sec_R_equals_B} we combine Lafforgue's work with the Taylor--Wiles method, using the technical results established in \S \ref{sec_automorphic_forms_free_over_diamond_operators}, to prove our automorphy lifting result.

\subsection{Cusp forms and Hecke algebras}\label{sec_cusp_forms_and_hecke_algebras}

\begin{proposition}
\begin{enumerate}
\item $G(K)$ is a discrete subgroup of $G(\bbA_K)$.
\item For any compact open subgroup $U  \subset \prod_v G(\cO_{K_v})$, and for any $g \in G(\bbA_K)$, the intersection $gG(K)g^{-1} \cap U$ (taken inside $G(\bbA_K)$) is finite. 
\end{enumerate}
\end{proposition}
\begin{proof}
The discreteness of $G(K)$ in $G(\bbA_K)$ follows from the discreteness of $K$ in $\bbA_K$ (note that $G$ is an affine group scheme). This implies that all intersections $g G(K) g^{-1} \cap U$ are finite. 
\end{proof}
We will define spaces of automorphic forms with integral coefficients. If $R$ is any $\bbZ[1/p]$-algebra and $U \subset G(\bbA_K)$ is any open compact subgroup, then we define $X_U = G(K) \backslash G(\bbA_K) / U$ and:
\begin{itemize}
\item $C(U, R)$ to be the $R$-module of functions $f : X_U \to R$;
\item $C_c(U, R) \subset C(U, R)$ to be the $R$-submodule of functions $f$ which have finite support;
\item and $C_{\text{cusp}}(U, R) \subset C(U, R)$ to be the $R$-submodule of functions $f$ which are cuspidal, in the sense that for all proper parabolic subgroups $P \subset G$ and for all $g \in G(\bbA_K)$, the integral
\[ \int_{n \in N(K) \backslash N(\bbA_K)} f(ng)\, dn \]
vanishes, where $N$ is the unipotent radical of $P$.
\end{itemize}
This last integral is normalized by endowing $N(K) \backslash N(\bbA_K)$ with its probability Haar measure. It makes sense because we are assuming that $p$ is a unit in $R$.
\begin{proposition}\label{prop_cusp_forms_on_G_have_compact_support}
Suppose that $R$ is a Noetherian $\bbZ[1/p]$-algebra which embeds in $\bbC$. For any open compact subgroup $U \subset G(\bbA_K)$, we have $C_{\text{cusp}}(U, R) \subset C_c(U, R)$, and $C_{\text{cusp}}(U, R)$ is a finite $R$-module. In particular, cuspidal automorphic forms are compactly supported in $X_U$. 
\end{proposition}
\begin{proof}
If $R = \bbC$, then the stronger statement that there exists a finite subset $Z \subset X_U$ such that all functions $f \in C_{\text{cusp}}(U, \bbC)$ are supported in $Z$ is proved in \cite[Corollary 1.2.3]{Har74}. In general this shows that $C_{\text{cusp}}(U, R)$ is contained in the finite free $R$-module consisting of functions supported on $Z$, and is therefore itself a finite $R$-module.
\end{proof}
Let $\Omega$ be an algebraically closed field of characteristic 0. The $\Omega$-vector space
\[ C_{\text{cusp}}(\Omega) = \ilim_U C_{\text{cusp}}(U, \Omega) \]
has a natural structure of semisimple admissible $\Omega[G(\bbA_K)]$-module. A cuspidal automorphic representation of $G(\bbA_K)$ over $\Omega$ is, by definition, an irreducible admissible $\Omega[G(\bbA_K)]$-module which is isomorphic to a subrepresentation of $C_{\text{cusp}}(\Omega)$. We observe that any cuspidal automorphic representation over $\bbC$ or $\overline{\bbQ}_l$ can in fact be defined over $\overline{\bbQ}$.

If $H$ is any locally profinite group and $U \subset H$ is an open compact subgroup, then we write $\cH(H, U)$ for the algebra of compactly supported $U$-biinvariant functions $f : H \to \bbZ$, with unit $[U]$ (the characteristic function of $U$). The basic properties of this algebra are very well-known, and can be found (for example) in \cite[\S 2.2]{New15}. In particular, if $M$ is a smooth $\bbZ[H]$-module, then the set $M^U$ of $U$-fixed vectors has a canonical structure of $\cH(H, U)$-module.
We will use this most often when $H = G(\bbA_K)$ and $U \subset G(\bbA_K)$ is an open compact subgroup: thus $\cH(G(\bbA_K), U)$ acts on all the spaces $C(U, R)$, $C_c(U, R)$, $C_{\text{cusp}}(U, R)$ via $R$-module homomorphisms, in a way compatible with the natural inclusions. 

If $v$ is a place of $K$, then the Satake isomorphism gives a complete description of the algebra $\cH(G(K_v), G(\cO_{K_v}))$ (see \cite{Gro98}): it is an isomorphism \begin{equation}\label{eqn_satake_isomorphism} \cH(G(K_v), G(\cO_{K_v})) \otimes_\bbZ \bbZ[q_v^{\pm 1/2}] \cong \bbZ[X^\ast(\hT)]^{W(\hG, \hT)}  \otimes_\bbZ \bbZ[q_v^{\pm 1/2}].
\end{equation}
Let $\Omega$ be an algebraically closed field of characteristic 0. If $V$ is an irreducible representation of $\hG_{\Omega}$, then the restriction of its character $\chi_V$ to $\hT$ is an element of $\bbZ[X^\ast(\hT)]^W$ (and these elements form a $\bbZ$-basis for $\bbZ[X^\ast(\hT)]^W$ as $V$ varies). If we fix a choice of square-root $p^{1/2}$ of $p$ (and hence of $q_v$) in $\Omega$, then the Satake isomorphism (\ref{eqn_satake_isomorphism}) determines from this data an operator $T_{V, v} \in \cH(G(K_v), G(\cO_{K_v})) \otimes_\bbZ \Omega$. We can characterize it uniquely using the following property: let $\chi : T(K_v) \to \Omega^\times$ be an unramified character. Then the space $(i_B^G \chi)^{G(\cO_{K_v})}$ is a $\cH(G(K_v), G(\cO_{K_v})) \otimes_\bbZ \Omega$-module, 1-dimensional as $\Omega$-vector space, and we have the equality ($\varpi_v \in \cO_{K_v}$ a uniformizer, $\chi^\vee$ as in Lemma \ref{lem_local_langlands_for_split_tori}):
\begin{equation}\label{eqn_relation_characterizing_satake_isomorphism} \chi_V(\chi^\vee(\Frob_v)) = \text{ eigenvalue of }T_{V, v}\text{ on }(i_B^G \chi)^{G(\cO_{K_v})}. 
\end{equation}
The Satake isomorphism has a relative version that we will also use. Let $P = MN$ be a standard parabolic subgroup of $G$ with its Levi decomposition. We can define a map (the Satake transform)
\[ \cS^G_P : \cH(G(K_v), G(\cO_{K_v}))  \otimes_\bbZ \bbZ[q_v^{\pm 1/2}]\to \cH(M(K_v), M(\cO_{K_v}))  \otimes_\bbZ \bbZ[q_v^{\pm 1/2}] \]
by the formula $(\cS^G_P f)(m) = \delta_P^{1/2}(m) \int_{n \in N(K_v)} f(mn) \, dn$, $\delta_P(m) = | \det \Ad(m)|_{\frn_{K_v}} |_v$. (The Haar measure on $N(K_v)$ is normalized by giving the subgroup $N(\cO_{K_v})$ volume 1.) This is transitive, in the sense that if $Q \subset P$ is another standard parabolic subgroup, then $\cS^G_Q = \cS^M_{Q \cap M} \circ \cS^G_P$. If $P = B$, then $\cS^G_B$ agrees with the usual Satake isomorphism under the identifications $\cH(T(K_v), T(\cO_{K_v})) = \bbZ[X_\ast(T)] = \bbZ[X^\ast(\hT)]$. For any choice of $P$ the map $\cS^G_P$ fits into a commutative diagram
\begin{equation}\label{eqn_normalized_satake_transform} \begin{aligned} \xymatrix{ \cH(G(K_v), G(\cO_{K_v}))\ar[d]_{\cS^G_P}  \otimes_\bbZ \bbZ[q_v^{\pm 1/2}] \ar[r]^-{\cS^G_B} & \bbZ[X^\ast(\hT)]^{W(\hG, \hT)}  \otimes_\bbZ \bbZ[q_v^{\pm 1/2}]\ar[d] \\
\cH(M(K_v), M(\cO_{K_v}))  \otimes_\bbZ \bbZ[q_v^{\pm 1/2}] \ar[r]^-{\cS^M_{B \cap M}} & \bbZ[X^\ast(\hT)]^{W(\hM, \hT)}  \otimes_\bbZ \bbZ[q_v^{\pm 1/2}], }\end{aligned} 
\end{equation}
where the right-hand arrow is the obvious inclusion. The Satake transform has the following compatibility with normalized induction: let $\pi$ be an irreducible admissible $\Omega[M(K_v)]$-module such that $\pi^{M(\cO_{K_v})} \neq 0$. Then this space is 1-dimensional and $\cH(M(K_v), M(\cO_{K_v}))$ acts on it via a character $\phi : \cH(M(K_v), M(\cO_{K_v})) \to \Omega$. Moreover, the normalized induction $\Pi = i_P^G \pi$ satisfies $\dim_{\Omega} \Pi^{G(\cO_{K_v})} = 1$, and the algebra $\cH(G(K_v), G(\cO_{K_v}))$ acts on this space via the character $\phi \circ \cS^G_P$. Specializing once more to the case $P = B$, we recover the relation (\ref{eqn_relation_characterizing_satake_isomorphism}). 

The existence of the Satake isomorphism leads to the unramified local Langlands correspondence, as a consequence of the following proposition.
\begin{proposition}\label{prop_chevalley_restriction_over_Z}
Let $M$ be a standard Levi subgroup of $G$. Then the natural restriction map $\bbZ[\hM]^\hM \to \bbZ[\hT]^{W(\hM, \hT)}$ is an isomorphism.
\end{proposition}
\begin{proof}
This is Chevalley's restriction theorem over $\bbZ$. The injectivity can be checked after extending scalars to $\overline{\bbQ}$. The surjectivity follows from the fact that the ring $\bbZ[\hT]^{W(\hM, \hT)}$ is spanned as a $\bbZ$-module by the restrictions to $\hT$ of the characters of the irreducible highest weight representations of $\hM_{\bbQ}$; and these representation admit $\hM$-stable $\bbZ$-lattices (see \cite[I.10.4, Lemma]{Jan03}), so their characters lie in $\bbZ[\hM]^{\hM}$. This completes the proof.
\end{proof}
The isomorphism classes of irreducible admissible $\bbC[G(K_v)]$-modules $\pi$ with $\pi^{G(\cO_{K_v})} \neq 0$ are in bijection with the homomorphisms $\cH(G(K_v), G(\cO_{K_v})) \to \bbC$ (see \cite[Ch. 1, 4.3]{Bus06}) . The above proposition, combined with the Satake isomorphism, shows that these are in bijection with the set $(\hG \dquot \hG)(\bbC)$, itself in bijection (see \S \ref{sec_invariants_over_a_field}) with the set of $\hG(\bbC)$-conjugacy classes of semisimple elements of $\hG(\bbC)$, or equivalently the equivalence classes of Frobenius-semisimple and unramified homomorphisms $W_{K_v} \to \hG(\bbC)$: this is the unramified local Langlands correspondence. The same discussion applies over any algebraically closed field $\Omega$ of characteristic 0 which is equipped with a $\bbZ[q^{\pm 1/2}]$-algebra structure. If $\pi$ is an irreducible admissible $\Omega[G(K_v)]$-module with non-zero $G(\cO_{K_v})$-invariants, then we will write $\cS(\pi) \in (\hG \dquot \hG)(\Omega)$ for the image in this quotient of the corresponding semisimple conjugacy class. We will also use the same notation with $G$ replaced by a standard Levi subgroup $M$.

We now discuss twisting of unramified representations. Let $\Omega$ be an algebraically closed field of characteristic 0. For any standard parabolic subgroup $P = MN$ of $G$, the dual torus of the split torus $C_M$ (the cocentre of $M$) is canonically identified with $Z_{\hM}^\circ$ (the connected component of the centre of $\hM$). We write $M(K_v)^1 \subset M(K_v)$ for the subgroup $\{ m \in M(K_v) \mid \forall \chi \in X^\ast(C_M), | \chi(m)|_v = 1\} $. Since $M$ is split, $M(K_v) \to C_M(K_v)$ is surjective and the quotient $M(K_v) / M(K_v)^1$ is isomorphic to the quotient of $C_M(K_v)$ by its maximal compact subgroup. There is a canonical isomorphism
 \begin{equation} M(K_v) / M(K_v)^1 \cong X_\ast(C_M) \cong X^\ast(Z_{\hM}^\circ), 
 \end{equation}
hence
\begin{equation}\label{eqn_cocenter_and_center}
\Hom(M(K_v) / M(K_v)^1, \Omega^\times) \cong Z_{\hM}^\circ(\Omega). 
\end{equation}
This isomorphism has the following reinterpretation in terms of the unramified local Langlands correspondence. The centre $Z_{\hM}$ acts on $\hM$ by left multiplication in a way commuting with the adjoint action of $\hM$; this action therefore passes to the quotient $\hM \dquot \hM$. Suppose that $\pi$ is an irreducible admissible $\Omega[M(K_v)]$-module with $\pi^{M(\cO_{K_v})} \neq 0$, and let $\psi : M(K_v) / M(K_v)^1 \to \Omega^\times$ correspond to the element $z_\psi \in Z_{\hM}^\circ(\Omega)$ under the isomorphism (\ref{eqn_cocenter_and_center}).  The representation $\pi \otimes \psi$ also has non-zero $M(\cO_{K_v})$-fixed vectors, and we have the equality
\begin{equation}\label{eqn_twisting_satake_parameters} \cS(\pi \otimes \psi) =  z_\psi \cdot \cS(\pi)
\end{equation}
inside $(\hM \dquot \hM)(\Omega)$.

There is also a global version of this twisting construction. Let $\| \cdot \| : K^\times \backslash \bbA_K^\times \to \bbR_{>0}$ denote the norm character, and define for any standard Levi subgroup $M$ of $G$
\[ M(\bbA_K)^1 = \{ m \in M(\bbA_K) \mid \forall \chi \in X^\ast(C_M), \text{ } \| \chi(m) \| = 1 \}. \]
The pairing $\langle m, \chi \rangle = - \log_q \| \chi(g) \|$ gives rise to an isomorphism $M(\bbA_K) / M(\bbA_K)^1 \cong \Hom(X^\ast(C_M), \bbZ) \cong X_\ast(C_M) \cong X^\ast(Z_{\hM}^\circ)$, hence an isomorphism
\begin{equation}\label{eqn_global_cocenter_and_center}
\Hom(M(\bbA_K) / M(\bbA_K)^1, \Omega^\times) \cong Z_{\hM}^\circ(\Omega). 
\end{equation}
The compatibility between the local and global isomorphisms (\ref{eqn_cocenter_and_center}) and (\ref{eqn_global_cocenter_and_center}) is expressed by the commutativity of the following diagram:
\[ \xymatrix{ \Hom(M(\bbA_K) / M(\bbA_K)^1, \Omega^\times) \ar[r] \ar[d] & Z_{\hM}^\circ(\Omega) \ar[d] \\ \Hom(M(K_v) / M(K_v)^1, \Omega^\times) \ar[r] & Z_{\hM}^\circ(\Omega), } \]
where the left vertical arrow is given by restriction and the right vertical arrow is given by multiplication by the degree $[k(v) : \bbF_q]$. In particular, if $\pi$ is an irreducible admissible $\Omega[M(\bbA_K)]$-module and $\psi : M(\bbA_K) / M(\bbA_K)^1 \to \Omega^\times$ is a character corresponding to the element $z_\psi \in Z_{\hM}^\circ(\Omega)$ under the isomorphism (\ref{eqn_global_cocenter_and_center}), and $v$ is a place of $K$ such that $\pi_v^{M(\cO_{K_v})} \neq 0$, then we have the equality
\begin{equation}\label{eqn_global_twisting_satake_parameters} \cS( (\pi \otimes \psi)_v ) =  z_\psi^{[k(v) : \bbF_q]} \cdot \cS(\pi_v) 
\end{equation}
inside $(\hM \dquot \hM)(\Omega)$, which is the global version of the equation (\ref{eqn_twisting_satake_parameters}).
\subsection{Summary of V. Lafforgue's work}\label{sec_summary_of_lafforgue}

We continue with the notation of the previous section, and now summarize some aspects of the construction of Galois representations attached to automorphic forms by V. Lafforgue \cite{Laf12}. Let $l \nmid q$ be a prime. Let $N = \sum_v n_v \cdot v \subset X$ be an effective divisor, and let $U(N) = \ker( \prod_v G(\cO_{K_v}) \to G(\cO_N) )$. Let $T$ denote the set of places $v$ of $K$ for which $n_v$ is non-zero. Lafforgue constructs for each finite set $I$, each tuple $(\gamma_i)_{i \in I} \in \Gamma_K^I$, and each function $f \in \bbZ[\hG^I]^{\hG}$, an operator $S_{I, (\gamma_i)_{i \in I}, f} \in \End_{\overline{\bbQ}_l}(C_{\text{cusp}}(U(N), \overline{\bbQ}_l))$, called an excursion operator. 
\begin{proposition}\label{prop_properties_of_excursion_operators}
The excursion operators enjoy the following properties:
\begin{enumerate}
\item The $\overline{\bbQ}_l$-subalgebra $\cB(U(N), \overline{\bbQ}_l)$ of $\End_{\overline{\bbQ}_l}(C_{\text{cusp}}(U(N), \overline{\bbQ}_l))$ generated by these operators (for all $I$, $(\gamma_i)_{i \in I}$, and $f$) is commutative. (Note that it is necessarily a finite $\overline{\bbQ}_l$-algebra, because it acts faithfully on the finite-dimensional $\overline{\bbQ}_l$-space $C_{\text{cusp}}(U(N), \overline{\bbQ}_l)$.)
\item The action of excursion operators on $C_{\text{cusp}}(U(N), \overline{\bbQ}_l)$ commutes with the action of the abstract Hecke algebra $\cH(G(\bbA_K), U(N))$.
\item If $v \not\in T$ and $\chi_V \in \bbZ[\hG]^{\hG}$ is the character of an irreducible representation $V$ of $\hG_{\overline{\bbQ}_l}$, then $S_{\{ 0 \}, \Frob_v, \chi_V} = T_{V, v}$ is the unramified Hecke operator corresponding to $V$ and $v$ under the Satake isomorphism. In particular, the algebra of excursion operators contains the algebra generated by all the unramified Hecke operators.
\item If $N \subset N'$ is another effective divisor, then the action of excursion operators commutes with the inclusion $C_{\text{cusp}}(U(N), \overline{\bbQ}_l) \subset C_{\text{cusp}}(U(N'), \overline{\bbQ}_l)$.
\end{enumerate}
\end{proposition}
\begin{proof}
These properties follow from the constructions given in \cite{Laf12}. Note that our notation is slightly different (since we take $f \in \bbZ[\hG^n]^{\hG}$ rather than $f \in \bbZ[\hG^{n+1}]^{\hG \times \hG}$).
\end{proof}
The most important feature of the excursion operators is the following:
\begin{theorem}\label{thm_excursion_operators_form_a_pseudocharacter}
Define for each $n \geq 1$ a map $\Theta_{U(N), n} : \bbZ[\hG^n]^{\hG} \to \Map(\Gamma_K^n, \cB(U(N), \overline{\bbQ}_l))$ by 
\[ \Theta_{U(N), n}(f)(\gamma_1, \dots, \gamma_n) = S_{\{1, \dots, n\}, (\gamma_i)_{i=1}^n, f}. \]
Then $\Theta_{U(N)} = (\Theta_{U(N), n})_{n \geq 1}$ is a $\cB(U(N), \overline{\bbQ}_l)$-valued pseudocharacter, which factors through the quotient $\Gamma_K \to \Gamma_{K, T}$. It is continuous when $\cB(U(N), \overline{\bbQ}_l)$ is endowed with its $l$-adic $(\overline{\bbQ}_l$-vector space) topology.
\end{theorem}
\begin{proof}
See \cite[Proposition-Definition 11.3]{Laf12}. 
\end{proof}
\begin{corollary}\label{cor_existence_of_Lafforgues_Galois_representations}
Let $\frp \subset \cB(U(N), \overline{\bbQ}_l)$ be a prime ideal. Then:
\begin{enumerate} \item There exists a continuous, $\hG$-completely reducible representation $\sigma_\frp : \Gamma_K \to \hG(\overline{\bbQ}_l)$ satisfying the following condition: for all excursion operators $S_{I, (\gamma_i)_{i \in I}, f}$, we have
\begin{equation}\label{eqn_compatibility_with_excursion_operators} f( (\sigma_\frp(\gamma_i))_{i \in I}) = S_{I, (\gamma_i)_{i \in I}, f} \text{ mod }\frp. 
\end{equation}
\item The representation $\sigma_\frp$ is uniquely determined up to $\hG(\overline{\bbQ}_l)$-conjugacy by (\ref{eqn_compatibility_with_excursion_operators}).
\item The representation $\sigma_\frp$ is unramified outside the support $T$ of $N$. If $v \not\in T$, then it satisfies the expected local-global compatibility relation at $v$: for all irreducible representations $V$ of $\hG_{\overline{\bbQ}_l}$, we have $T_{V, v} \in \cB(U(N), \overline{\bbQ}_l)$ and
\[ \chi_V( \sigma_\frp(\Frob_v) ) = T_{V, v} \text{ mod }\frp. \]
\end{enumerate}
\end{corollary}
\begin{proof}
The first two parts follow from Theorem \ref{thm_excursion_operators_form_a_pseudocharacter} and Theorem \ref{thm_pseudocharacters_biject_with_representations_over_fields}. The third part follows from the third part of Proposition \ref{prop_properties_of_excursion_operators}.
\end{proof}
\begin{definition}\label{def_automorphic_galois_representation}
We say that a continuous, $\hG$-completely reducible representation $\sigma : \Gamma_K \to \hG(\overline{\bbQ}_l)$ is automorphic if there exist $N$ and $\frp \subset \cB(U(N), \overline{\bbQ}_l)$ as above such that $\sigma$ and $\sigma_\frp$ are $\hG(\overline{\bbQ}_l)$-conjugate.
\end{definition}
We warn the reader that the defining property here implies, but is not a priori implied by, the existence of an automorphic representation with the correct Hecke eigenvalues at unramified places (because of the possible existence of homomorphisms $\sigma, \sigma'$  such that $\sigma|_{\Gamma_{K_v}}$ and $\sigma'|_{\Gamma_{K_v}}$ are $\hG(\overline{\bbQ}_l)$-conjugate for every place $v$ of $K$, but nevertheless $\sigma$ and $\sigma'$ are not $\hG(\overline{\bbQ}_l)$-conjugate). However, things are well-behaved in this way at least when $\sigma$ has Zariski dense image:
\begin{lemma}\label{lem_zariski_dense_and_classically_automorphic_implies_lafforgue_automorphic}
Let $\sigma : \Gamma_K \to \hG(\overline{\bbQ}_l)$ be a continuous representation with Zariski dense image. Then $\sigma$ is automorphic if and only if there exists a cuspidal automorphic representation $\pi$ of $G(\bbA_K)$ with the following property: for almost every place $v$ of $K$ such that $\pi^{G(\cO_{K_v})} \neq 0$, $\sigma|_{\Gamma_{K_v}}$ is unramified, and we have the relation ($V$ an irreducible representation of $\hG_{\overline{\bbQ}_l}$)
\begin{equation}\label{eqn_hecke_frobenius_satake_relation} \chi_V(\sigma(\Frob_v)) = \text{ eigenvalue of }T_{V, v}\text{ on }\pi_v^{G(\cO_{K_v})}. 
\end{equation}
If these equivalent conditions hold, then $\sigma|_{\Gamma_{K_v}}$ is unramified at every place $v$ of $K$ such that  $\pi^{G(\cO_{K_v})} \neq 0$, and satisfies the relation (\ref{eqn_hecke_frobenius_satake_relation}) there.
\end{lemma}
\begin{proof}
As we have already noted, if $\sigma$ is automorphic then $\pi$ exists: if $\sigma = \sigma_\frp$, we can choose $\pi$ to be the representation generated by a non-zero vector in $C_{\text{cusp}}(U(N), \overline{\bbQ}_l)[\frp]$. Suppose conversely that there exists a cuspidal automorphic representation $\pi$ as in the statement of the lemma, and let $N$ be minimal with $\pi^{U(N)} \neq 0$. Then we can find a maximal ideal $\frp \subset \cB(U(N), \overline{\bbQ}_l)$ such that $(\pi^{U(N)})_\frp \neq 0$. It then follows from Proposition \ref{prop_zariski_dense_and_locally_conjugate_implies_globally_conjugate} that $\sigma$ and $\sigma_\frp$ are $\hG(\overline{\bbQ}_l)$-conjugate representations.
\end{proof}
Next, we state a theorem of Genestier and V. Lafforgue \cite{Gen17}, which gives a partial description of the restriction of the pseudocharacter $\Theta_{U(N)}$ to decomposition groups at ramified places. They actually prove a much more general result, but in the interest of simplicity we state here only what we need. This result will be used to understand the ramification of the pseudocharacter $\Theta_{U(N)}$ at Taylor--Wiles places.
\begin{theorem}\label{thm_local_global_compatibility}
Let $v$ be a place of $K$, and let $\pi_v$ be an irreducible admissible $\overline{\bbQ}_l[G(K_v)]$-module which is a submodule of $i_B^G \chi_v$ for some smooth character $\chi_v : T(K_v) \to \overline{\bbQ}_l^\times$. Let $V_{\pi_v} \subset C_{\text{cusp}}(\overline{\bbQ}_l)$ be the $\pi_v$-isotypic part, and let $\Theta_{U(N), \pi_v}$ denote the image of $\Theta_{U(N)}$ along the projection $\cB(U(N), \overline{\bbQ}_l) \to \End_{\overline{\bbQ}_l}(V_{\pi_v}^{U(N)})$. Then $\Theta_{U(N), \pi_v}|_{W_{K_v}}$ takes values in the scalars $\overline{\bbQ}_l \subset \End_{\overline{\bbQ}_l}(V_{\pi_v}^{U(N)})$, where it coincides with the pseudocharacter associated to the representation $\iota \chi_v^\vee$.
\end{theorem}
The notation is as in \S \ref{sec_local_calculation}: $\iota : \hT \to \hG$ is the natural inclusion, and $\chi_v^\vee : W_{K_v} \to \hT(\overline{\bbQ}_l)$ is the character dual to $\chi_v$ under the local Langlands correspondence for split tori (Lemma \ref{lem_local_langlands_for_split_tori}). 
Finally, we need to extend slightly the definition of the algebra of excursion operators. Our first observation is that if $U = \prod_v U_v \subset \prod_v G(\cO_{K_v})$ is any open compact subgroup, then we can define an operator $S_{I, (\gamma_i)_{i \in I}, f} \in \End_{\overline{\bbQ}_l}(C_{\text{cusp}}(U, \overline{\bbQ}_l))$ as follows: choose $N$ such that $U(N) \subset U$, and restrict the action of the operator in $C_{\text{cusp}}(U(N), \overline{\bbQ}_l)$. It follows immediately from Proposition \ref{prop_properties_of_excursion_operators} that this is well-defined and independent of the choice of $N$, and that the $\overline{\bbQ}_l$-subalgebra $\cB(U, \overline{\bbQ}_l) \subset \End_{\overline{\bbQ}_l}(C_{\text{cusp}}(U, \overline{\bbQ}_l))$ generated by all excursion operators is commutative (it is a quotient of $\cB(U(N), \overline{\bbQ}_l)$). We write $\Theta_U$ for the associated pseudocharacter valued in $\cB(U, \overline{\bbQ}_l)$ (image of $\Theta_{U(N)}$).

Our next observation is that the excursion operators preserve natural rational and integral structures. We fix a choice of coefficient field $E \subset \overline{\bbQ}_l$.
\begin{proposition}\label{prop_excursion_operators_preserve_integral_structures}
Let $U = \prod_v U_v \subset \prod_v G(\cO_{K_v})$ be an open compact subgroup. Let $I$ be a finite set, $(\gamma_i)_{i \in I} \in \Gamma_K^I$, and $f \in \bbZ[\hG^I]^{\hG}$. Then the operator $S_{I, (\gamma_i)_{i \in I}, f}$ on $C_{\text{cusp}}(U, \overline{\bbQ}_l)$ preserves the $\cO$-submodule $C_{\text{cusp}}(U, \cO)$.
\end{proposition}
\begin{proof}
This will follow from \cite[Proposition 13.1]{Laf12} if we can show that any function $f \in \bbZ[\hG^n]^{\hG \times \hG}$ (where $\hG \times \hG$ acts on $\hG^n$ by diagonal left and right translation) is of the form $f(g_1, \dots, g_n) = \xi( (g_1, \dots, g_n) \cdot x)$, where $W$ is a representation of $\hG^n$ on a free $\bbZ$-module and $x : \bbZ \to W$ and $\xi : W \to \bbZ$ are $\hG$-equivariant morphisms (where $\hG$ acts via the diagonal $\hG \to \hG^n$). 

We can show this using the recipe of \cite[Lemme 10.5]{Laf12}. Let $f \in \bbZ[\hG^n]^{\hG \times \hG}$, and let $W \subset \bbZ[\hG^n]$ denote the submodule generated by the left translates of $f$. This can be constructed as the intersection $W = (\hG^n(\bbC) \cdot f )\cap \bbZ[\hG^n]$. In particular, every element of $W$ is invariant under diagonal right translation. Let $x = f$, and let $\xi : W \to \bbZ$ be the restriction to $W$ of the functional `evaluation at the identity'. Then $x$ and $\xi$ are both $\hG$-invariant, and together they give the desired representation of the function $f$.
\end{proof}
We write $\cB(U, \cO)$ for the $\cO$-subalgebra of $\End_{\cO}(C_{\text{cusp}}(U, \cO))$ generated by all excursion operators. It follows from Proposition \ref{prop_excursion_operators_preserve_integral_structures} that it is a finite flat $\cO$-algebra, and it follows from the definitions that the pseudocharacter $\Theta_U$ in fact takes values in $\cB(U, \cO)$. 
\begin{corollary}\label{cor_existence_of_galois_representations_attached_to_maximal_ideals_in_excursion_algebra}
Let $\ffrm \subset \cB(U, \cO)$ be a maximal ideal, and choose an embedding $\cB(U, \cO) / \ffrm \hookrightarrow \overline{\bbF}_l$. Then:
\begin{enumerate}
\item There is a $\hG$-completely reducible representation $\overline{\sigma}_\ffrm : \Gamma_K \to \hG(\overline{\bbF}_l)$ satisfying the following condition: for all excursion operators $S_{I, (\gamma_i)_{i \in I}, f}$, we have
\begin{equation}\label{eqn_compatibility_with_mod_l_excursion_operators} f( (\overline{\sigma}_\ffrm(\gamma_i))_{i \in I}) = S_{I, (\gamma_i)_{i \in I}, f} \text{ mod }\ffrm. 
\end{equation}
\item The representation $\overline{\sigma}_\ffrm$ is uniquely determined up to $\hG(\overline{\bbF}_l)$-conjugacy by (\ref{eqn_compatibility_with_mod_l_excursion_operators}).
\item If $v$ is a place of $K$ such that $U_v = G(\cO_{K_v})$, then $\overline{\sigma}_\ffrm$ is unramified at $v$ and satisfies the expected local-global compatibility relation there: for all irreducible representations $V$ of $\hG_{\overline{\bbQ}_l}$, we have $T_{V, v} \in \cB(U, \cO)$ and
\[ \chi_V( \overline{\sigma}_\ffrm(\Frob_v) ) = T_{V, v} \text{ mod }\ffrm. \]
\end{enumerate} 
\end{corollary}
We note that the expression in $(iii)$ makes sense, because the character of the reduction of $V$ mod $l$ is defined independent of any choices.
\begin{proof}
Proposition \ref{prop_excursion_operators_preserve_integral_structures} implies that $\Theta_U$ is in fact a $\cB(U, \cO)$-valued pseudocharacter. The result follows on applying Theorem \ref{thm_pseudocharacters_biject_with_representations_over_fields} to the projection of $\Theta_U$ to the quotient $\cB(U, \cO) / \ffrm$.
\end{proof}
\subsection{Automorphic forms are free over $\cO[\Delta]$}\label{sec_automorphic_forms_free_over_diamond_operators}

Our goal in the remainder of \S \ref{sec_automorphic_forms} is to prove an automorphy lifting theorem. We will accomplish this in \S \ref{sec_R_equals_B} as an application of the Taylor--Wiles method. This requires us to first establish as a key technical input that certain integral spaces of automorphic forms are free over suitable integral group rings of diamond operators; this will be accomplished in the current section (see Theorem \ref{thm_freeness_over_rings_of_diamond_operators} below), as an application of the principle outlined in \cite[Appendix]{Boc06} (which appears here as Lemma \ref{lem_freeness_of_set_theoretic_action_of_diamond_operators}).  

\subsubsection{Construction of an unramified Hecke operator}
We first consider an abstract situation. The setup is as at the beginning of \S \ref{sec_automorphic_forms}. Suppose given the following data:
\begin{itemize}
\item A subring $R \subset \bbC$ which is a discrete valuation ring, which contains $q^{1/2}$, and with finite residue field $k$ of characteristic $l \neq p$.
\item A finite set $S$ of places of $K$ and an open compact subgroup $U = \prod_v U_v \subset G(\bbA_K)$ such that for all $v \not\in S$, $U_v = G(\cO_{K_v})$.
\item For each standard parabolic subgroup $P = MN$ of $G$ (including $G$ itself) a finite set $I_P$ of irreducible admissible $\bbC[M(\bbA_K)]$-modules $(\pi_i)_{i \in I_P}$ such that $(i_P^G \pi_i)^U \neq 0$. This implies that for each $i \in I_P$ and for each $v \not\in S$, the space $\pi_{i, v}^{M(\cO_{K_v})}$ is non-zero. We ask further that the associated homomorphism $\phi_{i, v} : \cH(M(K_v), M(\cO_{K_v})) \to \bbC$ giving the action of the Hecke algebra on this space takes values in $R$.
\item For each $i \in I_P$, a continuous homomorphism $\overline{\rho}_i : \Gamma_K \to \hM(k)$, unramified outside $S$, and such that for each $v \not\in S$, the homomorphism $\cH(M(K_v), M(\cO_{K_v})) \to k$ associated to the conjugacy class of $\overline{\rho}_i(\Frob_v)$ (which exists by Proposition \ref{prop_chevalley_restriction_over_Z}) is equal to the reduction of $\phi_{i, v}$ modulo $\ffrm_R$.
\item An element $i_0 \in I_G$ such that $\overline{\rho}_{i_0} : \Gamma_K \to \hG(k)$ is absolutely strongly $\hG$-irreducible, in the sense of Definition \ref{def_irreducibility_and_strong_irreducibility}.
\end{itemize}
We can associate to each $i \in I_P$ a homomorphism $f_{i, v} : \cH(G(K_v), G(\cO_{K_v})) \otimes_\bbZ R \to R[Z_{\hM}^\circ]$, which represents the formula $z_\psi \mapsto \cS( i_P^G( \pi_{i, v} \otimes \psi) )$ (notation as in \S \ref{sec_cusp_forms_and_hecke_algebras}). If $\Sigma$ is a finite set of places of $K$, disjoint from $S$, then we write $\bbT_\Sigma = \otimes_{v \in \Sigma} \cH(G(K_v), G(\cO_{K_v})) \otimes_\bbZ R$ and $f_{i, \Sigma} = \otimes_{v \in \Sigma} f_{i, v} : \bbT_\Sigma \to R[Z_{\hM}^\circ]$.
\begin{proposition}\label{prop_construction_of_special_hecke_operators}
With notation and assumptions as above, we can find $\Sigma$ and $t \in \bbT_\Sigma$ satisfying the following conditions:
\begin{enumerate}
\item For every standard parabolic subgroup $P \neq G$ and for every $i \in I_P$, we have $f_{i, \Sigma}(t) = 0$. In particular, for every character $\psi : M(\bbA_K) / M(\bbA_K)^1 \to \bbC^\times$, we have $t (i_P^G (\pi \otimes \psi) )^U = 0$. 
\item We have $f_{i_0, \Sigma}(t) \not\equiv 0 \text{ mod }\ffrm_R$.
\end{enumerate}
\end{proposition}
\begin{proof}
Given a finite set $\Sigma$ of places of $K$, disjoint from $S$, let $x_\Sigma \in \Spec \bbT_{\Sigma}(k)$ be the (closed) point corresponding to $\overline{\rho}_{i_0}$. We want to show that we can choose $\Sigma$ so that for any standard parabolic subgroup $P = MN \neq G$ and for any $i \in I_P$, the image of $Z^\circ_{\hM, R}$ under the corresponding finite map $\Spec f_{i, \Sigma} : Z^\circ_{\hM, R} \to \bbT_{\Sigma}$ does not contain $x_\Sigma$. Indeed, the union of these finitely many images is closed, so if this union misses $x_\Sigma$ then we can find $t \in \bbT_{\Sigma}$ such that the distinguished affine open $D(t) = \{ \frp \in \Spec \bbT_{\Sigma} \mid t \not\in \frp \}$ contains $x_\Sigma$ but has empty intersection with this union. This element $t$ then has the desired properties (i) and (ii). 

Let $L'/K$ be the compositum of the extensions cut out by all of the $\overline{\rho}_i$, $i \in I_P$. Let $L = L' \cdot \bbF_q^d$, where $d$ is the exponent of the group $l^\times$, and $l/k$ is the compositum of all extensions of degree at most $\# W$. We choose $\Sigma$ so that every conjugacy class of the group $\Gal(L/K)$ contains an element of the form $\Frob_v$, $v \in \Sigma$. We claim that this choice of $\Sigma$ works.

Suppose for contradiction that $x_\Sigma$ is the image under some $f_{i, \Sigma}$ of a point $y \in Z^\circ_{\hM, R}$, where $P = MN$ is a proper standard parabolic subgroup of $G$. The residue field of $y$ is a finite extension of $k$ of degree at most $\# W$. It follows that there is a continuous character $\lambda : \Gamma_K \to \Gal(K \cdot \overline{\bbF}_q / K) \to Z^\circ_{\hM}(l)$ such that for each $v \in \Sigma$, the elements $\overline{\rho}_{i_0}(\Frob_v)$ and $(\overline{\rho}_i \otimes \lambda)(\Frob_v) \in \hG(l)$ determine the same homomorphism $\cH(G(K_v), G(\cO_{K_v})) \to l$. Moreover $\overline{\rho}_i \otimes \lambda$ factors through $\Gal(L/K)$.

Since the elements $\Frob_v$ ($v \in \Sigma$) cover $\Gal(L/K)$, we find that for all $f \in \bbZ[\hG]^{\hG}$ and for all $\gamma \in \Gamma_K$, we have $f(\overline{\rho}_{i_0}(\gamma)) = f (\overline{\rho}_i \otimes \lambda(\gamma))$. Applying the strong irreducibility property of $\overline{\rho}_{i_0}$, we conclude that $\overline{\rho}_i \otimes \lambda$ has image contained in no proper parabolic subgroup of $\hG_k$. However, it has (by construction) image contained inside $\hM_k$. This contradiction concludes the proof. 
\end{proof}
\subsubsection{Interlude on constant terms and Eisenstein series}\label{sec_constant_terms_and_eisenstein_series}
We now want to apply Proposition \ref{prop_construction_of_special_hecke_operators} to construct unramified Hecke operators which kill compactly supported automorphic forms which are not cuspidal. This will be accomplished in Lemma \ref{lem_construction_of_cuspidal_projector} below, but we must first recall some basic notions from the complex theory of automorphic forms. Our reference is the book of Moeglin--Waldspurger \cite{Moe95}. The additional `automorphic' notation introduced here (especially any object in script font $\mathscr{A}, \mathscr{B}, \mathscr{C}, \dots$) will be used only in this \S \ref{sec_constant_terms_and_eisenstein_series}.

Fix a place $v_0$ of $K$, and let $\frz$ denote the Bernstein center of the group $G(K_{v_0})$. (This auxiliary choice of place is used in the definition of spaces of automorphic forms in the following paragraphs. However, the objects constructed are independent of this choice, and it plays no role in our arguments; see \cite[\S I.3.6]{Moe95}.) Let $U_0 = G(\widehat{\cO}_K)$. For each standard parabolic subgroup $P = MN$ of $G$, there is a map $m_P : G(\bbA_K) \to M(\bbA_K) / M(\bbA_K)^1$ which sends $g  = nmu$ with $n \in N(\bbA_K)$, $m \in M(\bbA_K)$, $u \in U_0$, to the element $m_P(g) = m M(\bbA_K)^1$. This is well-defined because $M(\bbA_K) \cap U_0 \subset M(\bbA_K)^1$. The composite
\[ Z_M(\bbA_K) \to M(\bbA_K) \to M(\bbA_K) / M(\bbA_K)^1 \]
has image of finite index, and compact kernel. We define 
\[ \Re \fra_M = \Hom_\bbZ(X^\ast(C_M), \bbR) \cong M(\bbA_K) / M(\bbA_K)^1 \otimes_\bbZ \bbR. \]
If $P = M N$ is a standard parabolic subgroup of $G$, then we define a space $\mathscr{A}(N(\bbA_K) M(K) \backslash G(\bbA_K))$ of automorphic forms as the set of locally constant functions $\phi : N(\bbA_K) M(K) \backslash G(\bbA_K) \to \bbC$ satisfying the following conditions:
\begin{itemize}
\item $\phi$ has moderate growth (see \cite[I.2.3]{Moe95}).
\item $\phi$ is $U_0$-finite and $\frz$-finite. 
\end{itemize} 
We write $\mathscr{C}(N(\bbA_K) M(K) \backslash G(\bbA_K))$ for the space of all continuous functions $N(\bbA_K) M(K) \backslash G(\bbA_K) \to \bbC$, and $\mathscr{C}_c(N(\bbA_K) M(K) \backslash G(\bbA_K))$ for its subspace of continuous functions of compact support.

If $P = MN$ is a standard parabolic subgroup of $G$ and $\phi : N(K) \backslash G(\bbA_K) \to \bbC$ is a continuous function, then we define the constant term of $\phi$ along $P$ by the integral
\[ \phi_P(g) = \int_{n \in N(K) \backslash N(\bbA_K)} \phi(ng) \, dn. \]
Note that $N(K) \backslash N(\bbA_K)$ is compact and so has its canonical probability measure, induced from a left-invariant Haar measure on $N(\bbA_K)$. This means that the integral is well-defined. If $P' \subset P$ are standard parabolic subgroups of $G$ and $\phi \in \mathscr{A}(N(\bbA_K) M(K) \backslash G(\bbA_K))$, then in fact $\phi_{P'} \in \mathscr{A}(N'(\bbA_K) M'(K) \backslash G(\bbA_K))$. 

A function $\phi \in \mathscr{A}(N(\bbA_K) M(K) \backslash G(\bbA_K))$ is said to be cuspidal if for all proper standard parabolic subgroups $P' \subset P$, the constant term $\phi_{P'}$ vanishes. We write $\mathscr{A}_0(N(\bbA_K) M(K) \backslash G(\bbA_K))$ for the subspace of cusp forms. If $\chi : Z_{M}(K) \backslash Z_{M}(\bbA_K) \to \bbC^\times$ is a character, then we define
\[ \mathscr{A}(N(\bbA_K) M(K) \backslash G(\bbA_K))_\chi \subset \mathscr{A}(N(\bbA_K) M(K) \backslash G(\bbA_K)) \]
and 
\[\mathscr{A}_0(N(\bbA_K) M(K) \backslash G(\bbA_K))_\chi \subset \mathscr{A}_0(N(\bbA_K) M(K) \backslash G(\bbA_K)) \]
to be the subspaces of functions satisfying the relation $\phi(zg) = \chi(z) \phi(g)$ for all $z \in Z_{M}(\bbA_K)$. The group $G(\bbA_K)$ acts on all of the above-defined spaces of functions by right translation, and we have isomorphisms of $\bbC[G(\bbA_K)]$-modules
\[ \mathscr{A}(N(\bbA_K) M(K) \backslash G(\bbA_K)) \cong \Ind_{P(\bbA_K)}^{G(\bbA_K)} \mathscr{A}(M(K) \backslash M(\bbA_K)) \]
and
\[ \mathscr{A}_0(N(\bbA_K) M(K) \backslash G(\bbA_K)) \cong \Ind_{P(\bbA_K)}^{G(\bbA_K)} \mathscr{A}_0(M(K) \backslash M(\bbA_K)), \]
where the spaces of automorphic forms on $M$ are defined in the same way as for $G$ (see \cite[I.2.17]{Moe95}). 
\begin{proposition}\label{prop_cusp_forms_have_compact_support}
Let $P = MN$ be a standard parabolic subgroup of $G$, and let $U \subset U_0$ be an open compact subgroup. Then there exists a compact subset $C \subset G(\bbA_K)$ such that any function 
\[ \phi \in \mathscr{A}_0(N(\bbA_K) M(K) \backslash G(\bbA_K))^U \]
has support contained in $Z_{M}(\bbA_K) N(\bbA_K) M(K) C$. In particular, for any character $\chi : Z_{M}(K) \backslash Z_{M}(\bbA_K) \to \bbC^\times$, the space
\[ \mathscr{A}_0(N(\bbA_K) M(K) \backslash G(\bbA_K))^U_\chi \]
is finite-dimensional, and for any $\phi \in \mathscr{A}_0(N(\bbA_K) M(K) \backslash G(\bbA_K))^U$, the support of $\phi$ is compact modulo $Z_{M}(\bbA_K) P(K)$.
\end{proposition}
\begin{proof}
If $P = G$, then $\mathscr{A}_0 = C_{\text{cusp}}$ and this follows from Proposition \ref{prop_cusp_forms_on_G_have_compact_support}. The general case can be reduced to Proposition \ref{prop_cusp_forms_on_G_have_compact_support}, or rather its generalization \cite[Theorem 1.2.1]{Har74} to reductive (and not just semisimple) groups, as we now explain. After possibly shrinking $U$, we can assume that $U$ is normal in $U_0$. Then for all $u \in U_0$, the function $\phi_u : m \mapsto \phi(mu)$ defines an element of $\mathscr{A}_0(M(K) \backslash M(\bbA_K))^{U \cap M(\bbA_K)}$. Applying the analogue of Proposition \ref{prop_cusp_forms_on_G_have_compact_support} to the reductive group $M$, we get a compact subset $C' \subset M(\bbA_K)$ such that any such function $\phi_u$ has support in $Z_{M}(\bbA_K) M(K) C'$. It follows that $\phi$ has support in $N(\bbA_K) Z_{M}(\bbA_K) M(K) C' U_0$. We can therefore take $C = C' U_0$.
\end{proof}
\begin{proposition}\label{prop_automorphic_forms_with_central_character}
Let $P = M N$ be a standard parabolic subgroup of $G$.
\begin{enumerate}
\item Given characters $\chi : Z_{M}(K) \backslash Z_{M}(\bbA_K) \to \bbC^\times$ and $\psi : M(\bbA_K) / M(\bbA_K)^1 \to \bbC^\times$, there are canonical isomorphisms
\[ \mathscr{A}(N(\bbA_K) M(K) \backslash G(\bbA_K))_\chi \cong \mathscr{A}(N(\bbA_K) M(K) \backslash G(\bbA_K))_{\chi \psi} \]
and
\[ \mathscr{A}_0(N(\bbA_K) M(K) \backslash G(\bbA_K))_\chi \cong \mathscr{A}_0(N(\bbA_K) M(K) \backslash G(\bbA_K))_{\chi \psi} \]
given by the formula $\phi \mapsto (g \mapsto \psi(m_P(g)) \phi(g))$.
\item Let $\bbC[\Re \fra_M]$ denote the ring of polynomials on the real vector space $\Re \fra_M$ with complex coefficients. Then there are canonical isomorphisms
\[ \bbC[\Re \fra_M] \otimes_\bbC \oplus_\chi \mathscr{A}(N(\bbA_K) M(K) \backslash G(\bbA_K))_\chi \cong \mathscr{A}(N(\bbA_K) M(K) \backslash G(\bbA_K)) \]
and
\[ \bbC[\Re \fra_M] \otimes_\bbC \oplus_\chi \mathscr{A}_0(N(\bbA_K) M(K) \backslash G(\bbA_K))_\chi \cong \mathscr{A}_0(N(\bbA_K) M(K) \backslash G(\bbA_K)) \]
given by the formula $Q \otimes \phi \mapsto (g \mapsto Q(m_P(g)) \phi(g))$.
\end{enumerate}
\end{proposition}
\begin{proof}
See \cite[I.3]{Moe95}.
\end{proof}
Let $P = MN$ be a proper standard parabolic subgroup of $G$. We write $\mathscr{C}_0(N(\bbA_K) M(K) \backslash G(\bbA_K))$ for the space of functions $N(\bbA_K) M(K) \backslash G(\bbA_K) \to \bbC$ spanned by functions of the form $g \mapsto b(m_P(g)) \phi(g)$, where $\phi \in \mathscr{A}_0(N(\bbA_K) M(K) \backslash G(\bbA_K))$ and $b : M(\bbA_K) / M(\bbA_K)^1 \to \bbC$ has compact support (see \cite[I.3.4]{Moe95}). It follows from Proposition \ref{prop_cusp_forms_have_compact_support} that any function $\phi \in \mathscr{C}_0(N(\bbA_K) M(K) \backslash G(\bbA_K))$ has compact support modulo $P(K)$. We can therefore define a linear map
\[ \begin{split} \Eis_P : \mathscr{C}_0(N(\bbA_K) M(K) \backslash G(\bbA_K)) & \to \mathscr{C}_c(G(K) \backslash G(\bbA_K)) \\ \phi & \mapsto \left( \Eis_P \phi : g \mapsto \sum_{\gamma \in P(K) \backslash G(K)} \phi(\gamma g) \right).
\end{split} \]
This is $G(\bbA_K)$-equivariant. If $\psi : G(K) \backslash G(\bbA_K) \to \bbC$ is any locally constant function, then we have the formula
\begin{equation}\label{eqn_constant_terms_of_pseudo_eisenstein_series} \int_{x \in N(\bbA_K) M(K) \backslash G(\bbA_K)} \phi(x) \psi_P(x) \, dx = \int_{g \in G(K) \backslash G(\bbA_K)} \Eis_P(\phi)(g) \psi(g) \, dg. 
\end{equation}
In particular, if $\psi \in \mathscr{A}_0(G(K) \backslash G(\bbA_K))$, then these integrals both vanish. 
\begin{lemma}\label{lem_density_of_cusp_forms_with_central_character}
Let $S$ be a finite set of places of $K$, and let $U = \prod_v U_v \subset G(\bbA_K)$ be an open compact subgroup such that $U_v = G(\cO_{K_v})$ if $v \not\in S$. Let $t \in \cH(G(\bbA_K^S), U^S) \otimes_\bbZ \bbC$ be an operator such that for any character $\chi : Z_{M}(K) \backslash Z_{M}(\bbA_K) \to \bbC^\times$, $t \mathscr{A}_0(N(\bbA_K) M(K) \backslash G(\bbA_K))_\chi^U = 0$. Then $t \mathscr{C}_c(G(K) \backslash G(\bbA_K))^U \subset \mathscr{A}_0(G(K)\backslash G(\bbA_K))^U$.
\end{lemma}
\begin{proof}
We introduce the Hilbert space $\mathscr{H} = L^2(G(K) \backslash G(\bbA_K) / U)$, which is the completion of the space $\mathscr{C}_c(G(K) \backslash G(\bbA_K))^U$ with respect to its natural pre-Hilbert structure given by the inner product
\[ (\phi, \psi) = \int_{g \in G(K) \backslash G(\bbA_K)} \overline{\phi}(g) \psi(g) \, dg. \]
 The operator $t$ induces a continuous linear endomorphism of $\mathscr{H}$ which leaves invariant the finite-dimensional closed subspace $\mathscr{A}_0(G(K) \backslash G(\bbA_K))^U$. To prove the lemma, it suffices to show that $t$ acts as 0 on a dense subspace of the orthogonal complement of $\mathscr{A}_0(G(K) \backslash G(\bbA_K))^U$. By \cite[Proposition I.3.4]{Moe95} and the relation (\ref{eqn_constant_terms_of_pseudo_eisenstein_series}), $\sum_{P \subsetneq G} \Eis_P \mathscr{C}_0(N(\bbA_K) M(K) \backslash G(\bbA_K))^U$ is such a subspace.
 
We therefore need show that for each proper standard parabolic $P = MN$ of $G$, we have \[ t \mathscr{C}_0(N(\bbA_K) M(K) \backslash G(\bbA_K))^U = 0. \]
There is a Hermitian pairing
\[ \langle \cdot, \cdot \rangle_P : \mathscr{C}_0(N(\bbA_K) M(K) \backslash G(\bbA_K))^U \times \mathscr{C}(N(\bbA_K) M(K) \backslash G(\bbA_K))^U\to \bbC, \]
given by the formula $\langle \phi, \psi \rangle_P = \int_{x \in N(\bbA_K) M(K) \backslash G(\bbA_K)} \overline{\phi}(x) \psi(x) \, dx$. Let $(\cdot)^\ast$ denote the anti-involution of $\cH(G(\bbA_K^S), U^S) \otimes_\bbZ \bbC$ given by the formula $s^\ast(g) = \overline{s(g^{-1})}$. Then for any $s \in \cH(G(\bbA_K^S), U^S) \otimes_\bbZ \bbC$ we have the formula $\langle s \phi, \psi \rangle = \langle \phi, s^\ast \psi \rangle$. 

For any $\phi \in \mathscr{C}_0(N(\bbA_K) M(K) \backslash G(\bbA_K))$ and any finite subset $X \subset M(\bbA_K) / M(\bbA_K)^1$, we can find $\psi \in \oplus_\chi \mathscr{A}_0(N(\bbA_K) M(K) \backslash G(\bbA_K))_\chi$ such that $\phi(g) = \psi(g)$ whenever $m_P(g) \in X$. Indeed, this follows from Proposition \ref{prop_automorphic_forms_with_central_character} and the definition of the space $\mathscr{C}_0(N(\bbA_K) M(K) \backslash G(\bbA_K))$. To finish the proof of the lemma, we choose $\phi \in \mathscr{C}_0(N(\bbA_K) M(K) \backslash G(\bbA_K))$. We can find  $\psi \in \oplus_\chi \mathscr{A}_0(N(\bbA_K) M(K) \backslash G(\bbA_K))_\chi$ such that the restriction of $\psi$ to the support of $t^\ast t \phi$ equals $\phi$. We then obtain
\[ \langle t \phi, t \phi \rangle = \langle t^\ast  t \phi, \psi \rangle = \langle t \phi, t \psi \rangle = 0, \]
since $t$ annihilates $\psi$, by hypothesis. This forces $t \phi = 0$. Since $\phi$ was arbitrary this shows the desired vanishing.
\end{proof}
In the statement (but not the proof) of the following lemma, we use the notation of \S \ref{sec_summary_of_lafforgue}.
\begin{lemma}\label{lem_construction_of_cuspidal_projector}
Let $S$ be a finite set of places of $K$, and let $U \subset G(\widehat{\cO}_K)$ be an open compact subgroup such that $U_v = G(\cO_{K_v})$ if $v \not\in S$. Let $l$ be a prime and let $E \subset \overline{\bbQ}_l$ be a coefficient field. Let $\ffrm \subset \cB(U, \cO)$ be a maximal ideal such that $\overline{\sigma}_\ffrm : \Gamma_{K, S} \to \hG(\overline{\bbF}_l)$ is strongly irreducible. Then, after possibly enlarging $E$, we can find $t \in \cH(G(\bbA_K^S), U^S) \otimes_\bbZ \cO$ such that $t \text{ mod } \ffrm \neq 0$ and $t C_c(U, \cO) \subset C_{\text{cusp}}(U, \cO)$.
\end{lemma}
\begin{proof}
If $P = M N$ is a standard parabolic subgroup of $G$, say that two characters $\chi, \chi' : Z_{M}(K) \backslash Z_{M}(\bbA_K) \to \bbC^\times$ are twist equivalent if they differ by multiplication by an element of the image of the map
\[ \Hom(M(\bbA_K) / M(\bbA_K)^1, \bbC^\times) \to \Hom(Z_{M}(K) \backslash Z_{M}(\bbA_K), \bbC^\times). \]
 There are finitely many twist equivalence classes of characters $\chi : Z_M(K) \backslash Z_{M}(\bbA_K) / (U \cap Z_{M}(\bbA_K)) \to \bbC^\times$, and each equivalence class contains a character of finite order. This being the case, we can find the following data:
\begin{itemize}
\item A number field $L \subset \bbC$.
\item For each standard parabolic subgroup $P = M N \subset G$, a finite set $X_P$ of representatives 
\[ \chi : Z_{M}(K) \backslash Z_{M}(\bbA_K) / (U \cap Z_{M}(\bbA_K)) \to \bbC^\times \]
 of the twist equivalence classes, each $\chi \in X_P$ taking values in $\cO_L^\times$ and being of finite order. 
\end{itemize}
After possibly enlarging $L$, we can assume that for each $\chi \in X_P$ and for each of the finitely many irreducible constituents $\pi \subset \mathscr{A}_0(N(\bbA_K) M(K) \backslash G(\bbA_K))_\chi$ with non-zero $U$-invariants, the unramified Hecke operators $T_{V, v}$ ($v \not \in S$) on $\pi^U$ have all eigenvalues in $\cO_L[1/p]$. Indeed, there is an isomorphism of admissible $\bbC[G(\bbA_K)]$-modules
\[ \mathscr{A}_0(N(\bbA_K) M(K) \backslash G(\bbA_K))_\chi \cong \Ind_{P(\bbA_K)}^{G(\bbA_K)} \mathscr{A}_0(M(K) \backslash M(\bbA_K))_\chi, \]
and $\mathscr{A}_0(M(K) \backslash M(\bbA_K))_\chi$ has a natural $\cO_L[1/p]$-structure which is preserved by the $T_{V, v}$ (consisting of those functions $\phi : M(K) \backslash M(\bbA_K) \to \bbC$ which take values in $\cO_L[1/p]$). 

Let $\lambda$ be a place of $L$ of residue characteristic $l$, and let $R = \cO_{L, (\lambda)} \subset \bbC$. Fix an isomorphism $\overline{\bbQ}_l \cong \bbC$ which induces the place $\lambda$ of $L$. After possibly enlarging $E$, we can assume that $E$ contains $L$ under this identification. If $P = MN$ is a standard parabolic subgroup of $G$, let $I_P$ denote the set of all irreducible constituents of $\mathscr{A}_0(N(\bbA_K) M(K) \backslash G(\bbA_K))_\chi$ with non-zero $U$-invariants, as $\chi$ ranges over $X_P$. We are now in the situation of Proposition \ref{prop_construction_of_special_hecke_operators}, so we obtain a Hecke operator $t \in \cH(G(\bbA_K^S), U^S) \otimes_\bbZ R$ such that $t$ has non-zero image in $\cB(U, \cO) / \ffrm$ and $t$ acts as 0 on any representation $i_{P(\bbA_K)}^{G(\bbA_K)} (\pi \otimes \psi)$ ($\pi \in I_P$, $\psi \in \Hom(M(\bbA_K) / M(\bbA_K)^1, \bbC^\times)$) when $P$ is a proper standard parabolic subgroup of $G$, hence on 
\[ \oplus_\chi  \mathscr{A}_0(N(\bbA_K) M(K) \backslash G(\bbA_K))_\chi^U. \]
It now follows from Lemma \ref{lem_density_of_cusp_forms_with_central_character} that $t$ has the properties claimed in the statement of the current lemma. 
\end{proof}

\subsubsection{Application to freeness of integral automorphic forms}

We now come to the main result of \S \ref{sec_automorphic_forms_free_over_diamond_operators}. We fix a prime $l$ and a  coefficient field $E \subset \overline{\bbQ}_l$.
\begin{theorem}\label{thm_freeness_over_rings_of_diamond_operators}
Let $U = \prod_v U_v$ be an open compact subgroup of $G(\widehat{\cO}_K)$, and let $V = \prod_v V_v \subset U$ be an open normal subgroup such that $U/V$ is abelian of $l$-power order. Let $v_0$ be a place of $K$, and let $l^M$ denote the order of an $l$-Sylow subgroup of $G(\bbF_{q_{v_0}})$. Let $V \subset W \subset U$ be a subgroup such that $(U/V)[l^M] \subset W/V$. Finally, let $\ffrm \subset \cB(W, \cO)$ be a maximal ideal such that $\overline{\sigma}_\ffrm$ is strongly $\hG$-irreducible. Then $C_{\text{cusp}}(W, \cO)_\ffrm$ is a finite free $\cO[U/W]$-module.
\end{theorem}
The starting point for the proof is the following lemma.
\begin{lemma}\label{lem_freeness_of_set_theoretic_action_of_diamond_operators}
Let $U = \prod_v U_v$ be an open compact subgroup of $G(\widehat{\cO}_K)$, and let $V = \prod_v V_v \subset U$ be an open normal subgroup such that $U/V$ is abelian of $l$-power order. Let $v_0$ be a place of $K$, and let $l^M$ denote the order of an $l$-Sylow subgroup of $G(\bbF_{q_{v_0}})$. Let $V \subset W \subset U$ be a subgroup such that $(U/V)[l^M] \subset W/V$. Then the quotient $U/W$ acts freely on $X_W = G(K) \backslash G(\bbA_K) / W$.
\end{lemma}
\begin{proof}
For any place $v$ of $K$ and for any $g \in G(\bbA_K)$, the finite group $g G(K) g^{-1} \cap U $ injects into $U_v$ under projection to the $v$-component. Consequently, the $l$-part of its order divides $l^M$. 

The group $U$ acts on the discrete sets $X_V$ and $X_W$ by right translation. We want to show that if $g \in G(\bbA_K)$, then $\Stab_U (G(K) g W) = W$. We have $\Stab_U (G(K) g W) = \Stab_U (G(K) g V) \cdot W$, so it suffices to show that $\Stab_U (G(K) g V) \subset W$. On the other hand, we have $\Stab_U (G(K) g V) = (U \cap g^{-1} G(K) g) \cdot V$, so it even suffices to show that $U \cap g^{-1} G(K) g \subset W$. 

Let $u \in U \cap g^{-1} G(K) g$. We must show that $u$ lies in the kernel of the composite group homomorphism
\[ \xymatrix@1{ U \cap g^{-1} G(K) g \ar[r]^-\alpha & U/V \ar[r]^-\beta & U/W. } \]
The element $\alpha(u)$ has $l$-power order, hence satisfies $\alpha(u)^{l^M} = e$, hence (by hypothesis) $(\beta \circ \alpha)(u) = e$. This completes the proof.
\end{proof}
The lemma implies in particular that $C_c(W, \cO)$ is a free $\cO[U/W]$-module (although of infinite rank). In order to prove Theorem \ref{thm_freeness_over_rings_of_diamond_operators}, we will show that $C_{\text{cusp}}(W, \cO)_\ffrm$ can be realized as a direct summand $\cO[U/W]$-module of $C_c(W, \cO)$. Let $S$ be a finite set of places of $K$ such that $U_v = V_v = G(\cO_{K_v})$ if $v \not\in S$. By Lemma \ref{lem_construction_of_cuspidal_projector}, we can find an operator $t \in \cH(G(\bbA_K^S), U^S) \otimes_\bbZ \cO$ such that $t C_c(W, \cO)  \subset C_{\text{cusp}}(W, \cO)$ and $t$ has non-zero image in $\cB(W, \cO)/\ffrm$. Since $\cB(W, \cO)$ is a finite $\cO$-algebra, we can find $s \in \cB(W, \cO)$ such that $s C_{\text{cusp}}(W, \cO) = C_{\text{cusp}}(W, \cO)_\ffrm$ and $s$ has non-zero image in $\cB(W, \cO)/\ffrm$. Let $z = s \circ t$, viewed as an endomorphism of $C_c(W, \cO)$. Then we have 
\[ z C_c(W, \cO) = z C_{\text{cusp}}(W, \cO)_\ffrm = C_{\text{cusp}}(W, \cO)_\ffrm, \]
and the restriction of $z$ to $C_{\text{cusp}}(W, \cO)_\ffrm$ is an automorphism. Since the action of $z$ commutes with the action of $\cO[U/W]$, we conclude that $C_{\text{cusp}}(W, \cO)_\ffrm$ is a direct summand $\cO[U/W]$-module of $C_c(W, \cO)$, and hence that $C_{\text{cusp}}(W, \cO)_\ffrm$ is a finite free $\cO[U/W]$-module, as required. This completes the proof of Theorem \ref{thm_freeness_over_rings_of_diamond_operators}.

\subsection{Automorphic forms are free over $R_{\overline{\rho}}$}\label{sec_R_equals_B}

We now show how to use our work so far to prove an `$R = \cB$' theorem, which identifies in certain cases part of the integral algebra $\cB(U, \cO)$ of excursion operators with a Galois deformation ring of the type introduced in \S \ref{sec_galois_deformation_theory}. Let $l$ be a prime not dividing $q$, let $E \subset \overline{\bbQ}_l$ be a coefficient field, and let $U = G(\widehat{\cO}_K)$. Let $\ffrm \subset \cB(U, \cO)$ be a maximal ideal, and let $\overline{\sigma}_\ffrm : \Gamma_K \to \hG(\overline{\bbF}_l)$ denote the representation associated to the maximal ideal $\ffrm$ by Corollary \ref{cor_existence_of_galois_representations_attached_to_maximal_ideals_in_excursion_algebra}. We can assume, after possibly enlarging $E$, that $\overline{\sigma}_\ffrm$ takes values in $\hG(k)$. We make the following assumptions:
\begin{enumerate}
\item $l \nmid \# W$. This implies in particular that $l$ is a very good characteristic for $\hG$.
\item The subgroup $Z_{\hG^\text{ad}}(\overline{\sigma}_\ffrm(\Gamma_K))$ of $\hG_k$ is scheme-theoretically trivial.
\item  The representation $\overline{\sigma}_\ffrm$ is absolutely strongly $\hG$-irreducible (Definition \ref{def_irreducibility_and_strong_irreducibility}).
\item The subgroup $\overline{\sigma}_\ffrm(\Gamma_{K(\zeta_l)})$ of $\hG(k)$ is $\hG$-abundant (Definition \ref{def_abundant_subgroups}). 
\end{enumerate}
We note that assumptions (i) and (iii) together imply that $Z_{\hG^\text{ad}}(\overline{\sigma}_\ffrm(\Gamma_K))$ is a finite \'etale group (by Theorem \ref{thm_richardson_simultaneous_conjuation_and_stability} and Theorem \ref{thm_all_subgroups_are_separable}). Point (ii) can therefore be checked at the level of geometric points.

The universal deformation ring $R_{\overline{\sigma}_\ffrm, \emptyset}$ is then defined (see \S \ref{sec_galois_deformation_theory}). We write $\Theta_{U, \ffrm}$ for the projection of the pseudocharacter $\Theta_U$ to $\cB(U, \cO)_\ffrm$, and $\sigma^\text{univ} : \Gamma_K \to \hG(R_{\overline{\sigma}_\ffrm, \emptyset})$ for a representative of the universal deformation.
\begin{lemma}\label{lem_existence_of_map_R_to_B}
There is a unique morphism $f_\ffrm : R_{\overline{\sigma}_\ffrm, \emptyset} \to \cB(U, \cO)_\ffrm$ of $\cO$-algebras such that $f_{\ffrm, \ast} \tr \sigma^\text{univ} = \Theta_{U, \ffrm}$. It is surjective.
\end{lemma}
\begin{proof}
The existence and uniqueness of the map follows from Theorem \ref{thm_pseudocharacters_biject_with_representations_over_artinian_rings} and our assumption on the centralizer of the image of $\overline{\sigma}_\ffrm$ in the group $\hG^\text{ad}$. The map is surjective because, by definition, the ring $\cB(U, \cO)_\ffrm$ is generated by the excursion operators $S_{I, (\gamma_i)_{i \in I}, f}$, and each such operator can be explicitly realized as the image of the element  $f(\sigma^\text{univ}(\gamma_i)_{i \in I}) \in R_{\overline{\sigma}_\ffrm, \emptyset}$.
\end{proof}
The rest of this section is now devoted to proving the following result.
\begin{theorem}\label{thm_R_equals_B}
With assumptions as above,  $f_\ffrm$ is an isomorphism, $C_{\text{cusp}}(U, \cO)_\ffrm$ is a free $R_{\overline{\sigma}_\ffrm, \emptyset}$-module, and $R_{\overline{\sigma}_\ffrm, \emptyset}$ is a complete intersection $\cO$-algebra. 
\end{theorem}
Before giving the proof of Theorem \ref{thm_R_equals_B}, we explain the role played by our hypotheses (i) -- (iv) above. The first condition (i) is convenient; it removes the need to deal with some technical issues (such as possible non-smoothness of the map $\hG \to \hG^\text{ad}$). The condition (ii) is absolutely essential, since it is only in this case that the pseudocharacter in $\cB(U, \cO)$ constructed by Lafforgue can be upgraded to a true representation, as in Lemma \ref{lem_existence_of_map_R_to_B}. Since we know how to control deformations of representations using Galois cohomology, but not how to control deformations of pseudocharacters, we do not see a way to avoid this at the present time. 

The condition (iii) is used to establish an essential technical lemma (Lemma \ref{lem_freeness_over_rings_of_diamond_operators} below, which is a reformulation of Theorem \ref{thm_freeness_over_rings_of_diamond_operators}). We note that the $\hG$-irreducibility of $\overline{\sigma}_\ffrm$ is already implied by (ii), so it is the strong irreducibility that is important here. Finally, the condition (iv) is used to construct sets of auxiliary Taylor--Wiles places of $K$. Some condition of this type is essential. It is possible that this could be weakened in the future (as the notion of `bigness' has been replaced by `adequacy' in analogous theorems for $\GL_n$), but the condition of $\hG$-abundance is sufficient for our purposes.
\begin{corollary}\label{cor_automorphy_lifting}
Let $\rho : \Gamma_K \to \hG(E)$ be a continuous, everywhere unramified representation such that $\overline{\rho} \cong \overline{\sigma}_\ffrm$. Then there exists a cuspidal automorphic representation $\Pi$ of $G(\bbA_K)$ over $\overline{\bbQ}_l$ such that $\Pi^U \neq 0$ and for every place $v$ of $K$, and every irreducible representation $V$ of $\hG_{\overline{\bbQ}_l}$, we have
\[ \chi_V(\rho(\Frob_v)) = \text{eigenvalue of }T_{V, v}\text{ on }\Pi_v^{U_v}. \]
Fix a choice of character $\psi : N(K) \backslash N(\bbA_K) \to \overline{\bbZ}_l^\times$, and suppose further that there exists a minimal prime ideal $\frp \subset \cB(U, \cO)_\ffrm$ and $f \in C_\text{cusp}(U, \cO)_\ffrm[\frp]$ satisfying
\[ \int_{n \in N(K) \backslash N(\bbA_K)} f(n) \psi(n) \, dn \not\equiv 0 \text{ mod }\ffrm_{\overline{\bbZ}_l}. \]
Then we can moreover assume that $\Pi$ is generated by a vector $F \in C_\text{cusp}(U, \overline{\bbQ}_l)$ satisfying 
\[ \int_{n \in N(K) \backslash N(\bbA_K)} F(n) \psi(n) \, dn \neq 0. \]
 In particular, if $\psi$ is a generic character, then $\Pi$ is globally generic.
\end{corollary}
\begin{proof}
The representation $\rho$ determines a homomorphism $R_{\overline{\sigma}_\ffrm, \emptyset} \to \cO$; let $\frq$ denote its kernel. Then $C_{\text{cusp}}(U, \cO)_\ffrm[\frq]$ is a non-zero finite free $\cO$-module. If $x \in C_{\text{cusp}}(U, \cO)_\ffrm[\frq]$ is a non-zero element, then for every place $v$ of $K$ and for every irreducible representation $V$ of $\hG_{\overline{\bbQ}_l}$, we have $T_{V, v} x = \chi_V(\rho(\Frob_v)) x$. We can take $\Pi$ to be the irreducible  $\overline{\bbQ}_l[G(\bbA_K)]$-module generated by $x$. This establishes the first claim.

For the second, we note that we are free to enlarge $\cO$. We can therefore assume that $\cB(U, \cO)_\ffrm / \frp = \cO$ (i.e.\ that the eigenvalues of the excursion operators on $f$ all lie in $\cO$) and that the functional 
\[ \Psi :C_{\text{cusp}}(U, \cO)_\ffrm \to \overline{\bbZ}_l, \, F \mapsto \Psi(F) = \int_{n \in N(K) \backslash N(\bbA_K)} F(n) \psi(n) \, dn, \]
in facts takes values in $\cO$, and therefore defines an element of $\Hom_\cO(C_{\text{cusp}}(U, \cO)_\ffrm, \cO) = H_0^\ast$, say. 

By Theorem \ref{thm_R_equals_B}, both $H_0$ and $H_0^\ast$ are free $\cB(U, \cO)_\ffrm$-modules. The duality $H_0 \times H_0^\ast \to \cO$ induces a perfect duality 
\[ (H_0 / \varpi H_0)[\ffrm] \times H_0^\ast / \ffrm H_0^\ast \to k. \]
Since $f \in H_0[\frp]$ and $\frp + (\varpi) = \ffrm$, the element $\overline{f} = f \text{ mod } \varpi H_0$ lies in $(H_0 / \varpi H_0)[\ffrm]$. Moreover, we have $\overline{\Psi}(\overline{f}) \neq 0$, showing that $\Psi$ defines a non-trivial element of $H_0^\ast / \ffrm H_0^\ast$. Let $\Psi = \Psi_1, \dots, \Psi_m$ be elements of $H_0^\ast$ which project to a basis of $H_0^\ast / \ffrm H_0^\ast$; by Nakayama's lemma, they are actually free $\cB(U, \cO)_\ffrm$-module generators for $H_0^\ast$. This shows that for any minimal prime ideal $\frq \subset \cB(U, \cO)_\ffrm$, the restriction of $\Psi$ to $H_0[\frq]$ is non-zero. For the second part of the corollary, we can therefore take $\Pi$ to be the irreducible $\overline{\bbQ}_l[G(\bbA_K)]$-module generated by an element of $H_0[\frq]$ which does not lie in the kernel of $\Psi$.
\end{proof}
\begin{proof}[Proof of Theorem \ref{thm_R_equals_B}]
We will use the Taylor--Wiles method. We first introduce some notation. We recall (Definition \ref{def_taylor_wiles_datum}) that a Taylor--Wiles datum for $\overline{\sigma}_\ffrm$ is a pair $(Q, \{ \varphi_v \}_{v \in Q})$, where:
\begin{itemize}
\item $Q$ is a finite set of places $v$ of $K$ such that $\overline{\sigma}_\ffrm(\Frob_v)$ is regular semisimple and $q_v \equiv 1 \text{ mod }l$.
\item For each $v \in Q$, $\varphi_v : \hT_k \cong Z_\hG(\overline{\sigma}_\ffrm(\Frob_v))$ is a choice of inner isomorphism.
\end{itemize}
We recall that we have fixed a choice $T \subset B \subset G$ of split maximal torus and Borel subgroup of $G$. If $Q$ is a Taylor--Wiles datum, then we define $\Delta_Q$ to be the maximal $l$-power order quotient of the group $\prod_{v \in Q} T(k(v))$. According to Lemma \ref{lem_local_structure_at_TW_primes}, the ring $R_{\overline{\sigma}_\ffrm, Q}$ then has a canonical structure of $\cO[\Delta_Q]$-algebra. We will write $\fra_Q \subset \cO[\Delta_Q]$ for the augmentation ideal; the same lemma shows that there is a canonical isomorphism $R_{\overline{\sigma}_\ffrm, Q} / (\fra_Q) \cong R_{\overline{\sigma}_\ffrm, \emptyset}$.

If $Q$ is a Taylor--Wiles datum, then we define open compact subgroups $U_1(Q) \subset U_0(Q) \subset U$ as follows:
\begin{itemize}
\item $U_0(Q) = \prod_v U_0(Q)_v$, where $U_0(Q)_v = U_v = G(\cO_{K_v})$ if $v \not\in Q$, and $U_0(Q)_v$ is the Iwahori group $U_0$ defined in \S \ref{sec_local_calculation} if $v \in Q$.
\item $U_1(Q) = \prod_v U_1(Q)_v$, where $U_1(Q)_v = U_v = G(\cO_{K_v})$ if $v \not\in Q$, and $U_1(Q)_v$ is the group $U_1$ defined in \S \ref{sec_local_calculation} if $v \in Q$.
\end{itemize}
Thus $U_1 \subset U_0$ is a normal subgroup, and there is a canonical isomorphism $U_0 / U_1 \cong \Delta_Q$. 

We now need to define auxiliary spaces of modular forms. We define $H'_0 = C_{\text{cusp}}(U, \cO)_\ffrm$. If $Q$ is a Taylor--Wiles datum, then there are surjective maps $\cB(U_1(Q), \cO) \to \cB(U_0(Q), \cO) \to \cB(U, \cO)$, and we write $\ffrm$ as well for the pullback of $\ffrm \subset \cB(U, \cO)$ to these two algebras. Just as in Lemma \ref{lem_existence_of_map_R_to_B}, we have surjective morphisms \[ R_{\overline{\sigma}_\ffrm, Q} \to \cB(U_1, \cO)_\ffrm \to \cB(U_0, \cO)_\ffrm. \]
We define $H'_{Q, 1} = C_{\text{cusp}}(U_1(Q), \cO)_\ffrm$ and $H'_{Q, 0} = C_{\text{cusp}}(U_0(Q), \cO)_\ffrm$. 

There is a structure of $R_{\overline{\sigma}_\ffrm, Q}[ \prod_{v \in Q} X_\ast(T) ]$-module on $H'_{Q, 0}$, where the copy of $X_\ast(T)$ corresponding to $v \in Q$ acts via the embedding $\cO[X_\ast(T)] \to \cH_{U_0(Q)_v}$ described in Lemma \ref{lem_action_of_abelian_algebra_on_iwahori_invariants} and the preceding paragraphs. We write $\frn_{Q, 0} \subset \cO[ \prod_{v \in Q} X_\ast(T) ]$ for the maximal ideal which is associated to the tuple of characters ($v \in Q$):
\begin{equation}\label{eqn_tuple_of_characters_associated_to_TW_datum} \varphi_v^{-1} \circ \overline{\sigma}_\ffrm|_{W_{K_v}} : W_{K_v} \to \hT(k), 
\end{equation}
as in the paragraph following Lemma \ref{lem_action_of_abelian_algebra_on_iwahori_invariants}. Then $H'_{Q, 0, \frn_{Q, 0}}$ is a direct factor $R_{\overline{\sigma}_\ffrm, Q}$-module of $H'_{Q, 0}$, and Proposition \ref{prop_picking_our_iwahori_spherical_forms} shows that there is a canonical isomorphism $H'_{Q, 0, \frn_{Q, 0}} \cong H'_0$ of $R_{\overline{\sigma}_\ffrm, Q}$-modules.  (We note that the key hypothesis in Proposition \ref{prop_picking_our_iwahori_spherical_forms} of compatibility of two different pseudocharacters is satisfied in our situation by Theorem \ref{thm_local_global_compatibility}. We note as well that we assume in \S \ref{sec_local_calculation} that if $v \in Q$, then the stabilizer in the Weyl group of the regular semisimple element $\overline{\rho}_\ffrm(\Frob_v) \in \hG(k)$ is trivial; this is equivalent to the condition that the centralizer in $\hG_k$ of $\overline{\rho}_\ffrm(\Frob_v)$ is connected, which is part of the definition of a Taylor--Wiles datum.)

Similarly, if $v \in Q$ then we write $T(K_v)_l$ for the quotient of $T(K_v)$ by its maximal pro-prime-to-$l$ subgroup. Then there is a structure of $R_{\overline{\sigma}_\ffrm, Q}[ \prod_{v \in Q} T(K_v)_l ]$-module on $H'_{Q, 1}$, where the copy of $T(K_v)_l$ corresponding to $v \in Q$ acts via the embedding $\cO[T(K_v)_l] \to \cH_{U_1(Q)_v}$ described in Lemma \ref{lem_abelian_subalgebra_of_pro-p_Hecke_algebra}. We write $\frn_{Q, 1} \subset \cO[ \prod_{v \in Q} T(K_v)_l ]$ for the maximal ideal which is associated to the tuple of characters (\ref{eqn_tuple_of_characters_associated_to_TW_datum})
as in the paragraph following Lemma \ref{lem_irreducible_admissibles_with_non-zero_pro-p_invariants}. Then $H'_{Q, 1, \frn_{Q, 1}}$ is a direct factor $R_{\overline{\sigma}_\ffrm, Q}$-module of $H'_{Q, 1}$, and Proposition \ref{prop_picking_out_pro-p_spherical_forms} shows that the two structures of $\cO[\Delta_Q]$-module on $H'_{Q, 1, \frn_{Q, 1}}$, one arising from the homomorphism $\cO[\Delta_Q] \to R_{\overline{\sigma}_\ffrm, Q}$ and the other from the homomorphism $\cO[\Delta_Q] \to \cO[\prod_{v \in Q} T(K_v)_l]$, are the same. (We are again invoking Theorem \ref{thm_local_global_compatibility} to justify the application of Proposition \ref{prop_picking_out_pro-p_spherical_forms}.)

It follows from the second part of Lemma \ref{lem_abelian_subalgebra_of_pro-p_Hecke_algebra} that the natural inclusion 
\[ C_{\text{cusp}}(U_0(Q), \cO) \subset C_{\text{cusp}}(U_1(Q), \cO) \]
induces an identification $H'_{Q, 0, \frn_{Q, 0}} = (H'_{Q, 1, \frn_{Q, 1}})^{\Delta_Q}$. In order to complete the proof, we will require one more key property of the modules $H'_{Q, 1}$.
\begin{lemma}\label{lem_freeness_over_rings_of_diamond_operators}
Fix a place $v_0$ of $K$, and let $l^M$ denote the order of the $l$-Sylow subgroup of $G(\bbF_{q_{v_0}})$. Let $Q$ be a Taylor--Wiles datum, and suppose given an integer $N \geq M$ such that for each $v \in Q$, $q_v \equiv 1 \text{ mod }l^N$. Then $(H'_{Q, 1})^{l^{N-M} \Delta_Q}$ is a free $\cO[\Delta_Q / l^{N-M} \Delta_Q]$-module.
\end{lemma}
\begin{proof}
This follows from Theorem \ref{thm_freeness_over_rings_of_diamond_operators}. (This is the point in the proof where we use our assumption that $\overline{\sigma}_\ffrm$ is strongly $\hG$-irreducible.)
\end{proof}
 This property implies in turn that $(H'_{Q, 1, \frn_{Q, 1}})^{l^{N-M} \Delta_Q}$ is a free $\cO[\Delta_Q / l^{N-M} \Delta_Q]$-module. Observe that the abelian group $\Delta_Q / l^{N-M} \Delta_Q$  is a free $\bbZ / l^{N-M} \bbZ$-module of rank $ r \# Q$, where $r = \rank \hG$.

Henceforth we fix a place $v_0$ of $K$ and let $l^M$ be as in Lemma \ref{lem_freeness_over_rings_of_diamond_operators}. If $Q$ is a Taylor--Wiles datum as in Lemma \ref{lem_freeness_over_rings_of_diamond_operators}, we will then define $H_Q = \Hom_\cO((H'_{Q, 1, \frn_{Q, 1}})^{l^{N-M} \Delta_Q}, \cO)$ and $H_0 = \Hom_\cO(H'_0, \cO)$, and endow these finite free $\cO$-modules with their natural structures of $R_{\overline{\sigma}_\ffrm, Q} \otimes_{\cO[\Delta_Q]} \cO[\Delta_Q / l^{N-M} \Delta_Q]$- and $R_{\overline{\sigma}_\ffrm, \emptyset}$-module, respectively. We can summarize the preceding discussion as follows:
\begin{itemize}
\item The module $H_Q$ is a finite free $\cO[\Delta_Q / l^{N-M} \Delta_Q]$-module, where $\cO[\Delta_Q / l^{N-M} \Delta_Q]$ acts via the algebra homomorphism 
\[ \cO[\Delta_Q / l^{N-M} \Delta_Q] \to R_{\overline{\sigma}_\ffrm, Q} \otimes_{\cO[\Delta_Q]}  \cO[\Delta_Q / l^{N-M} \Delta_Q]. \]
\item There is a natural surjective map $H_Q \to H_0$, which factors through an isomorphism $(H_Q)_{\Delta_Q} \to H_0$, and is compatible with the isomorphism $R_{\overline{\sigma}_\ffrm, Q} / (\fra_Q) \cong R_{\overline{\sigma}_\ffrm, \emptyset}$.
\end{itemize}
 Indeed, the freeness of $(H'_{Q, 1, \frn_{Q, 1}})^{l^{N-M} \Delta_Q}$ implies that $H_Q$ is itself a free $\cO[\Delta_Q / l^{N-M} \Delta_Q]$-module, and that there is a natural isomorphism
\[ (H_Q)_{\Delta_Q} \cong \Hom_\cO( (H'_{Q, 1, \frn_{Q, 1}})^{l^{N-M} \Delta_Q})^{\Delta_Q}, \cO) = \Hom_\cO( H'_{Q, 0, \frn_{Q, 0}}, \cO ) \cong \Hom_\cO( H_0', \cO) = H_0. \]
Let $q = h^1(\Gamma_{K, \emptyset}, \widehat{\frg}_k)$. By Proposition \ref{prop_existence_of_taylor_wiles_data}, we can find for each $N \geq 1$ a Taylor--Wiles datum $(Q_N, \{ \varphi_v \}_{v \in Q_N})$ which satisfies the following conditions:
\begin{itemize}
\item For each $v \in Q_N$, we have $q_v \equiv 1 \text{ mod }l^{N+M}$ and $\# Q_N = q$.
\item There exists a surjection $\cO \llbracket X_1, \dots, X_g \rrbracket \to R_{\overline{\sigma}_\ffrm, Q_N}$, where $g = qr$.
\end{itemize}
We now patch these objects together. Define $R_\infty = \cO \llbracket X_1, \dots, X_g \rrbracket$, $\Delta_\infty = \bbZ_l^g$, $\Delta_N = \Delta_\infty / l^{N}\Delta_\infty$, $S_\infty = \cO\llbracket \Delta_\infty \rrbracket$, $S_N = \cO[\Delta_N]$. We define $\frb_N = \ker (S_\infty \to S_N)$, $\frb_0 = \ker( S_\infty \to \cO)$ (i.e.\ $\frb_0$ is the augmentation ideal of this completed group algebra). We choose for each Taylor--Wiles datum $Q_N$ a surjection $\Delta_\infty \to \Delta_{Q_N}$; this endows each ring $R_{\overline{\sigma}_\ffrm, Q_N}$ with an $S_\infty$-algebra structure, and hence each ring $R_{\overline{\sigma}_\ffrm, Q_N} / (\frb_N)$ with an $S_N$-algebra structure. The discussion above shows that $H_{Q_N}$ has a natural structure of $R_{\overline{\sigma}_\ffrm, Q_N}/(\frb_N)$-module, and that $H_{Q_N}$ is free as an $S_N$-module, when $S_N$ acts via the map $S_N \to R_{\overline{\sigma}_\ffrm, Q_N}/(\frb_N)$. We also fix a choice of surjection $R_\infty \to R_{\overline{\sigma}_\ffrm, Q_N}$.

Let $a = \dim_k (H_0 \otimes_\cO k)$, and if $m \geq 1$ set $r_m = g(g+1)aml^m$. This integer is chosen so that for any $N \geq 1$, we have
\[ \ffrm_{R_{\overline{\sigma}_\ffrm, Q_N}}^{r_N} H_{Q_N} \subset \varpi^N H_{Q_N}. \]
If $m \geq 1$ is an integer, then we define a patching datum of level $m$ to be a tuple $(R_m, H_m, \alpha_m, \beta_m)$ consisting of the following data:
\begin{itemize}
\item $R_m$ is a complete Noetherian local $\cO$-algebra with residue field $k$. It is equipped with a homomorphism $S_m \to R_m$ and a surjective homomorphism $R_\infty \to R_m$ (both morphisms of $\cO$-algebras).
\item $H_m$ is a finite $R_m$-module.
\item $\alpha_m$ is an isomorphism $R_m / (\frb_0) \cong R_{\overline{\sigma}_\ffrm, \emptyset} / \ffrm_{R_{\overline{\sigma}_\ffrm, \emptyset}}^{r_m}$ of $\cO$-algebras.
\item $\beta_m$ is an isomorphism $H_m / (\frb_0) \cong H_0 / (\varpi^m)$ of $\cO$-modules.
\end{itemize}
We require these data to satisfy the following conditions:
\begin{itemize}
\item The maximal ideal of $R_m$ satisfies $\ffrm_{R_m}^{r_m} = 0$.
\item $H_m$ is a free $S_m / (\varpi^m)$-module, where $S_m$ acts via $S_m \to R_m$.
\item The isomorphisms $\alpha_m$, $\beta_m$ are compatible with the structure of $R_{\overline{\sigma}_\ffrm, \emptyset} / \ffrm_{R_{\overline{\sigma}_\ffrm, \emptyset}}^{r_m}$-module on $H_0 / (\varpi^m)$.
\end{itemize}
We define a morphism between two patching data $(R_m, H_m, \alpha_m, \beta_m)$, $(R'_m, H'_m, \alpha'_m, \beta'_m)$ of level $m$ to be a pair of isomorphisms $i : R_m \to R'_m$, $j : H_m \to H'_m$ satisfying the following conditions:
\begin{itemize}
\item $i$ is compatible with the structures of $R_\infty$- and $S_m$-algebra, and satisfies $\alpha_m = \alpha'_m i$.
\item $j$ is compatible with the structures of $R_m$- and $R'_m$-module via $i$, and satisfies $\beta_m = \beta'_m j$.
\end{itemize}
Then the collection $\cD_m$ of patching data of level $m$ forms a category (in fact a groupoid). This category has finitely many isomorphism classes of objects: indeed, our conditions imply that $R_m$ and $H_m$ have cardinality bounded solely in terms of $m$, and the finiteness follows from this. For any $m' \geq m$, there is a functor $F_{m', m} : \cD_{m'} \to \cD_m$ which sends $(R_{m'}, H_{m'}, \alpha_{m'}, \beta_{m'})$ to the datum $(R_m, H_m, \alpha_m, \beta_m)$ given as follows:
\begin{itemize}
\item We set $R_m = R_{m'} / (\frb_m, \ffrm_{R_{m'}}^{r_m})$.
\item We set $H_m = H_{m'} / (\frb_m, \varpi^m)$.
\item We let $\alpha_m = \alpha_{m'} \text{ mod }\ffrm_{R_{m'}}^{r_m}$ and $\beta_m = \beta_{m'} \text{ mod } \varpi^m$, noting that there are canonical isomorphisms $R_m / (\frb_0) \cong R_{m'} / (\frb_0, \ffrm_{R_{m'}}^{r_m})$ and $H_m / (\frb_0) \cong H_{m'} / (\frb_0, \varpi^m)$.
\end{itemize}
For any $1 \leq m \leq N$, we can write down a patching datum $\cP_{m, N} = (R_{m, N}, H_{m, N}, \alpha_{m, N}, \beta_{m, N})$ of level $m$ as follows:
\begin{itemize}
\item We set $R_{m, N} = R_{\overline{\sigma}_\ffrm, Q_N} / (\ffrm_{R_{\overline{\sigma}_\ffrm, Q_N}}^{r_m}, \frb_m)$. Our choices determine maps $R_\infty \to R_{m, N}$ and $S_m \to R_{m, N}$.
\item We set $H_{m, N} = H_{Q_N} / (\frb_m, \varpi^m)$.
\item We let $\alpha_{m, N}$ denote the reduction modulo $\ffrm_{R_{\overline{\sigma}_\ffrm, Q_N}}^{r_m}$ of the canonical isomorphism $R_{\overline{\sigma}_\ffrm, Q_N} / (\frb_0) \cong R_{\overline{\sigma}_\ffrm, \emptyset}$.
\item We let $\beta_{m, N}$ denote the reduction modulo $\varpi^m$ of the canonical isomorphism $H_{Q_N} / (\frb_0) \cong H_0$.
\end{itemize}
Using the finiteness of the skeleton category of $\cD_m$ for each $m \geq 1$, we find that we can choose integers $N_1 \leq N_2 \leq N_3 \leq \dots$ and for each $m \geq 1$ an isomorphism $F_{m+1, m}(\cP_{m+1, N_{m+1}}) \cong \cP_{m, N_m}$ of patching data. This means that we can pass to the inverse limit, setting
\[ R^\infty = \plim_m R_{m, N_m}, H_\infty = \plim_m H_{m, N_m}, \]
to obtain the following objects:
\begin{itemize}
\item $R^\infty$, a complete Noetherian local $\cO$-algebra with residue field $k$, which is equipped with structures of $S_\infty$-algebra and a surjective map $R_\infty \to R^\infty$.
\item $H_\infty$, a finite $R^\infty$-module.
\item $\alpha_\infty$, an isomorphism $R^\infty / (\frb_0) \cong R_{\overline{\sigma}_\ffrm, \emptyset}$.
\item $\beta_\infty$, an isomorphism $H_\infty / (\frb_0) \cong H_0$. 
\end{itemize}
These objects have the following additional properties:
\begin{itemize}
\item $H_\infty$ is free as an $S_\infty$-module.
\item The isomorphisms $\alpha_\infty$, $\beta_\infty$ are compatible with the structure of $R_{\overline{\sigma}_\ffrm, \emptyset}$-module on $H_0$.
\end{itemize}
We find that
\[ \dim R^\infty \geq \depth_{R^\infty} H_\infty \geq \depth_{S_\infty} H_\infty = \dim S_\infty = \dim R_\infty \geq \dim R^\infty, \]
and hence that these inequalities are equalities, $R_\infty \to R^\infty$ is an isomorphism, and (by the Auslander--Buchsbaum formula) $H_\infty$ is also a free $R^\infty$-module. It follows that $H_\infty / (\frb_0) \cong H_0$ is a free $R^\infty / (\frb_0) \cong R_{\overline{\sigma}_\ffrm, \emptyset}$-module, and that $R_{\overline{\sigma}_\ffrm, \emptyset}$ is an $\cO$-flat complete intersection. This in turn implies that $C_{\text{cusp}}(U, \cO)_\ffrm \cong \Hom_\cO(H_0, \cO)$ is a free $R_{\overline{\sigma}_\ffrm, \emptyset}$-module (complete intersections are Gorenstein). This completes the proof of the theorem.
\end{proof} 

\section{Application of theorems of L. Moret-Bailly}\label{sec_application_of_moret-bailly}

In this section we make some simple geometric constructions regarding torsors under finite groups. We first recall some basic facts. Let $H$ be a finite group. We recall that an $H$-torsor over a scheme $S$ is an $S$-scheme $X$, faithfully flat and of locally finite type, on which $H$ acts on the left by $S$-morphisms, and such that the natural morphism $\underline{H}_S \times_S X \to X \times_S X,$ $(h, x) \mapsto (hx, x)$, is an isomorphism. (Here $\underline{H}_S$ denotes the constant group over $S$ attached to $H$.) 

Two torsors $X, X'$ over $S$ are said to be isomorphic if there exists an $H$-equivariant $S$-isomorphism $f : X \to X'$. The \'etale sheaf $\Isom_{S,H}(X, X')$ of isomorphisms of $H$-torsors is representable by a finite \'etale $S$-scheme. If $X = \underline{H}_S$ is the trivial $H$-torsor over $S$, then we have a canonical identification $\Isom_{S,H}(X, X') = X'$. 

Now suppose that $S$ is connected, and let $\overline{s}$ be a geometric point of $S$. If $X$ is an $H$-torsor over $S$, then the choice of a geometric point $\overline{x}$ of $X$ above $\overline{s}$ determines a homomorphism $\varphi_{X, \overline{x}} : \pi_1(S, \overline{s}) \to H$, given by the formula $\gamma \cdot \overline{x} = \varphi_{X, \overline{x}}(\gamma) \cdot \overline{x}$ ($\gamma \in \pi_1(S, \overline{s})$). A different choice of $\overline{x}$ replaces $\varphi_{X, \overline{x}}$ by an $H$-conjugate. This assignment $X \mapsto \varphi_{X, \overline{x}}$ determines a bijection between the following two sets (see \cite[Exp. V, No. 5]{SGA1})  :
\begin{itemize}
\item The set of $H$-torsors $X$ over $S$, up to isomorphism.
\item The set of homomorphisms $\pi_1(S, \overline{s}) \to H$, up to $H$-conjugation.
\end{itemize}
We will apply these considerations in the following context. Let $X, Y$ be smooth, geometrically connected curves over $\bbF_q$, and set $K = \bbF_q(X)$, $F = \bbF_q(Y)$. Let $\overline{\eta}_X, \overline{\eta}_Y$ be the geometric generic points of $X$ and $Y$, respectively, corresponding to fixed choices of separable closures $K^s / K$ and $F^s / F$. We write $\overline{\bbF}_q$ for the algebraic closure of $\bbF_q$ inside $K^s$.

Let $\varphi : \pi_1(X, \overline{\eta}_X) \to H$, $\psi : \pi_1(Y, \overline{\eta}_Y) \to H$ be homomorphisms. We now consider the pullbacks of these homomorphisms to various different (although closely related) fundamental groups. The curve $Y_K$ is a smooth, geometrically connected curve over $K$, and the given data determines two torsors $X_\psi$ and $X_\varphi$ over $Y_K$. We define $Z_{\psi, \varphi} = \Isom_{Y_K, H}(X_\varphi, X_\psi)$. Let $\Omega$ denote a separable closure of the function field of $Y_{K^s}$, and let $\overline{\eta}$ denote the corresponding geometric generic point of $Y_{K^s}$. Fix a choice of $F$-embedding $F^s \hookrightarrow \Omega$. We then have a commutative diagram of fundamental groups:
\begin{equation}\begin{aligned}\label{eqn_diagram_of_fundamental_groups} \xymatrix{ & & & \pi_1(Y_{K^s}, \overline{\eta}) \ar[dl] \ar@{->>}[rd] \\
& & \pi_1(Y_K, \overline{\eta}) \ar@{->>}[dl] \ar@{->>}[rd] & & \pi_1(Y_{\overline{\bbF}_q}, \overline{\eta}_Y) \ar[dl] \\
& \pi_1(\Spec K, \overline{\eta}_X) \ar[dl]_{\varphi} & & \pi_1(Y, \overline{\eta}_Y) \ar[dr]^{\psi} \\
H & & & & H} 
\end{aligned}
\end{equation}
We recall (\cite[Exp. V, Proposition 6.9]{SGA1}) that if $S' \to S$ is a morphism of connected schemes, then the surjectivity of the map $\pi_1(S', \overline{\eta}_{S'}) \to \pi_1(S, \overline{\eta}_S)$ is equivalent to the following statement: for each connected finite \'etale cover $Z \to S$, $Z_{S'} \to S'$ is still connected. This condition is easily checked for the morphisms leading to the surjective arrows in the above diagram. For example, let $Y' \to Y$ be a connected finite \'etale cover, and let $\bbF_{q'}$ denote the algebraic closure of $\bbF_q$ in $\bbF_q(Y')$. Then $Y'$ is geometrically connected over $\bbF_{q'}$, which shows that $Y'_K$ is connected.
\begin{lemma}\label{lem_formalism_of_torsors} Let notation be as above.
\begin{enumerate}
\item Suppose that $\psi$ is surjective, even after restriction to $\pi_1(Y_{\overline{\bbF}_q}, \overline{\eta}_Y)$. Then $Z_{\psi, \varphi}$ is a geometrically connected finite \'etale $K$-scheme.
\item Let $K'/K$ be a finite separable extension, which is contained inside $K^s$, and let $z \in Z_{\psi, \varphi}(K')$. Let $y \in Y_K(K') = Y(K')$ denote the image of $z$, and suppose that $y \not\in \im(Y(\overline{\bbF}_q \cap K') \to Y(K'))$. Then $z$ determines an $\bbF_q$-embedding $\beta : F \hookrightarrow K'$ such that $\beta^\ast \psi$ and $\varphi|_{\Gamma_{K'}}$ are $H$-conjugate as homomorphisms $\Gamma_{K'} \to H$.
\end{enumerate}
\end{lemma}
\begin{proof}
For the first part, it suffices to note that $Z_{\psi, \varphi, K^s} \cong X_{\psi, K^s} \cong X_{\psi|_{\pi_1(Y_{K^s}, \overline{\eta})}}$ is connected. Indeed, the diagram (\ref{eqn_diagram_of_fundamental_groups}) shows, together with our assumption, that $\psi|_{\pi_1(Y_{K^s}, \overline{\eta})} : \pi_1(Y_{K^s}, \overline{\eta}) \to H$ is surjective.

For the second part, we first note that $\beta^\ast \psi$ depends on a choice of a compatible embedding $F^s \to K^s$; however, its $H$-conjugacy class is independent of any such choice, so the conclusion of the proposition makes sense. The point $z$ determines a morphism $y : \Spec K' \to Z_{\psi, \varphi} \to Y_K \to Y$. Our assumption that $y$ does not come from $Y(\overline{\bbF}_q)$ says that the point of $\Spec K'$ is mapped to the generic point of $Y$, hence determines an $\bbF_q$-embedding $\beta : F \hookrightarrow K'$, and a commutative diagram (where we now omit base points for simplicity):
\begin{equation}\label{eqn_diagram_of_galois_groups}\begin{aligned} \xymatrix{ & \pi_1(\Spec K') \ar[rd] \ar[dd]\ar[rr]^{\beta_\ast} && \pi_1(\Spec F) \ar[dd] \\
& & \pi_1(Y_K) \ar[dl] \ar[dr] \\
& \pi_1(\Spec K) \ar[dl]_\varphi & & \pi_1(Y) \ar[dr]^\psi \\
H & & & & H.} 
\end{aligned}
\end{equation}
 The existence of the point $z$ then says that the homomorphisms $\beta^\ast \psi$ and $\varphi|_{\Gamma_{K'}}$ are $H$-conjugate, as desired. 
\end{proof}
We can now apply a well-known theorem of Moret-Bailly (see \cite{MoretBailly}) to deduce the following consequence.
\begin{proposition}\label{prop_potential_agreement_of_residual_representations}
Let $X, Y$ be smooth, geometrically connected curves over $\bbF_q$, and set $K = \bbF_q(X)$, $F = \bbF_q(Y)$. Let $\varphi : \pi_1(X, \overline{\eta}_X) \to H$, $\psi : \pi_1(Y, \overline{\eta}_Y) \to H$ be homomorphisms such that $\psi$ is surjective, even after restriction to $\pi_1(Y_{\overline{\bbF}_q}, \overline{\eta}_Y)$. Let $L/K$ be a finite Galois extension. Then we can find a finite Galois extension $K' / K$ and an $\bbF_q$-embedding $\beta : F \hookrightarrow K'$ satisfying the following conditions:
\begin{enumerate}
\item The extension $K' / K$ is linearly disjoint from $L/K$, and $K' \cap \overline{\bbF}_q = \bbF_q$.
\item The homomorphisms $\varphi|_{\Gamma_{K'}}$ and $\beta^\ast \psi$ are $H$-conjugate. 
\end{enumerate}
\end{proposition}
\begin{proof}
By the first part of Lemma \ref{lem_formalism_of_torsors}, $Z_{\psi, \varphi}$ is a smooth, geometrically connected curve over $K$. By spreading out and the Weil bounds, the set $Z_{\psi, \varphi}(K_v)$ is non-empty for all but finitely many places $v$ of $K$. Let $S$ be a finite set of places of $K$ such that if $L/M/K$ is an intermediate field Galois over $K$, with $\Gal(M/K)$ simple and non-trivial, then there is $v \in S$ which does not split in $M$; and if $v \in S$, then $Z_{\psi, \varphi}(K_v)$ is non-empty (it is easy to construct such a set using the Chebotarev density theorem, i.e.\ Theorem \ref{thm_chebotarev_density}). We see that $S$ has the following property: any Galois extension $K'/K$ which is $S$-split is linearly disjoint from $L/K$. After adjoining two further places to $S$ of coprime residue degrees, we see that $K' /K$ will also have the property that $K' \cap \overline{\bbF}_q = \bbF_q$.

By Lemma \ref{lem_formalism_of_torsors}, the proposition will now follow from the following statement: there exists a finite Galois extension $K' / K$ satisfying the following conditions:
\begin{enumerate}
\item The extension $K' / K$ is $S$-split.
\item The set $Z_{\psi, \varphi}(K')$ contains a point which does not lie above $Y(\bbF_q)$.
\end{enumerate}
It follows from \cite[Th\'eor\`eme 1.3]{MoretBailly} that such an extension exists. This completes the proof.
\end{proof}
We conclude this section by recalling another useful result of Moret-Bailly and applying it to the existence of compatible systems (see \cite{Mor90}).
\begin{theorem}\label{thm_potential_inverse_galois}
Let $X$ be a smooth, geometrically connected curve over $\bbF_q$, and let $K = \bbF_q(X)$. Let $S$ be a finite set of places of $K$. Suppose given a finite group $H$ and for each $v \in S$, a Galois extension $M_v / K_v$ and an injection $\varphi_v : \Gal(M_v / K_v) \to H$. Then we can find inside $K^s$ a finite extension $K'/K$ and a Galois extension $M / K'$, satisfying the following conditions:
\begin{enumerate}
\item $K' / K$ is $S$-split and there is given an isomorphism $\varphi : \Gal(M/K') \to H$.
\item For each place $w$ of $M$ above a place $v$ of $S$, let $v' = w|_{K'}$, so that there is a canonical isomorphism $K_v \cong K'_{v'}$. Then there is an isomorphism $M_w \cong M_v$ of $K_v$-algebras such that the composite map
\[ \xymatrix@1{ \Gal(M_v / K_v) \ar[r]^\cong & \Gal(M_w / K'_{v'}) \ar[r] & \Gal(M/K') \ar[r]^-\varphi & H} \]
equals $\varphi_v$.
\end{enumerate}
\end{theorem}
\begin{remark}  The main theorem of the article \cite{GL} of Gan and Lomel{\'{\i}} contains a strong automorphic analogue of this theorem.
\end{remark}

\begin{proposition}\label{prop_construction_of_unramified_and_dense_compatible_system}
Let $G$ be a split simple adjoint group over $\bbF_q$, and let $l$ be a prime such that $l > 2 \dim G$. Let $K = \bbF_q(X)$, where $X$ is a smooth, geometrically connected curve over $\bbF_q$, and let $S$ be a finite set of places of $K$. Then we can find the following data:
\begin{enumerate}
\item A finite extension $K'/K$ inside $K^s$ which is $S$-split.
\item A coefficient field $E \subset \overline{\bbQ}_l$ and a continuous, everywhere unramified homomorphism $\rho : \Gamma_{K'} \to \hG(\cO_E)$ such that $\overline{\rho} = \rho \text{ mod }\ffrm_E$ has image $\hG(\bbF_l)$.
\end{enumerate}
The representation $\rho = \rho_\lambda$ fits into a compatible system of continuous, everywhere unramified representations $(\emptyset, (\rho_\lambda)_\lambda)$, each of which has Zariski dense image in $\hG(\overline{\bbQ}_\lambda)$.
\end{proposition}
\begin{proof}
The last sentence will follow from the results of \S \ref{sec_compatible_systems_of_galois_representations} once we establish the rest, the Zariski density of the image of $\rho$ being a consequence of (ii) and Proposition \ref{prop_large_residual_image_implies_zariski_dense}. We first apply Theorem \ref{thm_potential_inverse_galois} to obtain the following data:
\begin{itemize}
\item A finite extension $K'/K$ inside $K^s$, which is $S$-split.
\item A continuous, $S$-unramified and surjective homomorphism $\overline{\rho} : \Gamma_{K'} \to \hG(\bbF_l)$.
\end{itemize}
After replacing $K'$ by an extension $K' \cdot K''$ and restricting $\overline{\rho}$ to the Galois group of this extension, we can further assume:
\begin{itemize}
\item $\overline{\rho}$ is everywhere unramified.
\end{itemize}
The result will now follow from Theorem \ref{thm_khare_wintenberger_method} (applied to the ring $R_{\overline{\rho}, \emptyset}$) if we can show that the following conditions are satisfied:
\begin{itemize}
\item $l$ is a very good characteristic for $\hG$.
\item $\overline{\rho}$ is absolutely $\hG$-irreducible.
\item There exists a representation $i : \hG \to \GL(V)$ of finite kernel such that $i \overline{\rho} : \Gamma_{K'} \to \GL(V_{\bbF_l})$ is absolutely irreducible and $l > 2 ( \dim V - 1)$.
\end{itemize}
Our hypothesis on $l$ implies that it is a very good characteristic for $\hG$, and that $l$ is larger than the Coxeter number of $\hG$. Consequently, $\hG(\bbF_l)$ is a $\hG$-absolutely irreducible subgroup of $\hG$ (because its saturation equals $\hG(\overline{\bbF}_l)$, see \cite[\S 5.1]{Ser05}). We can then satisfy the third point above by taking $V$ to be the adjoint representation of $\hG$.
\end{proof}
\section{A class of universally automorphic Galois representations}\label{sec_class_of_universally_automorphic_galois_representations}

In this section we introduce a useful class of Galois representations, which we call Coxeter parameters, and which can, in certain circumstances, be shown to be ``universally automorphic''; see the introduction for a discussion of the role played by this property, or Lemma \ref{lem_Coxeter_parameters_rational_over_Z_l} below for a precise formulation of what we actually use. We first describe these Coxeter parameters abstractly and establish their basic properties in \S \ref{sec_abstract_coxeter_parameters}, and then relate them to Galois representations and study their universal automorphy in \S \ref{sec_galois_coxeter_parameters}. We learned the idea of using Coxeter elements of Weyl groups to build Langlands parameters from the work of Gross and Reeder (see for example \cite{Gro10}).

\subsection{Abstract Coxeter parameters}\label{sec_abstract_coxeter_parameters}

Let $\hG$ be a split, simply connected and \textbf{simple} group over $\bbZ$ of rank $r$. Let $\hT \subset \hB \subset \hG$ be a choice of split maximal torus and Borel subgroup, and let $\Phi = \Phi(\hG, \hT)$ denote the corresponding set of roots, $R \subset \Phi$ the corresponding set of simple roots. We write $W = W(\hG, \hT) = N_{\hG}(\hT) / \hT$ for the Weyl group. Let $k$ be an algebraically closed field. We will assume throughout \S \ref{sec_class_of_universally_automorphic_galois_representations} that the characteristic of $k$ is either 0 or $l > 0$, where $l$ satisfies the following conditions:
\begin{enumerate}
\item $l > 2 h - 2$, where $h$ is the Coxeter number of $\hG$ (defined in the statement of Proposition \ref{prop_basic_properties_of_Coxeter_elements} below).
\item $l$ is prime to $\# W$. (In fact, this condition is implied by (i).)
\end{enumerate}
These conditions are satisfied by any sufficiently large prime $l$. Under these assumptions, it follows that if $H$ is a group and $\phi : H \to \hG(k)$ is a homomorphism, then $\phi$ is $\hG$-completely reducible if and only if $\widehat{\frg}_k$ is a semisimple $k[H]$-module (see \cite[Corollaire 5.5]{Ser05}).
\begin{definition}
We call an element $w \in W$ a Coxeter element if it is conjugate in $W$ to an element constructed as follows: choose an ordering $\alpha_1, \dots, \alpha_r$ of $R$, and take the product $s_{\alpha_1} \dots s_{\alpha_r}$ of the corresponding simple reflections in $W$.
\end{definition}
The properties of Coxeter elements are very well-known. We recall some of them here:
\begin{proposition}\label{prop_basic_properties_of_Coxeter_elements}
\begin{enumerate}
\item There is a unique conjugacy class of Coxeter elements in $W$. Their common order $h$ is called the Coxeter number of $\hG$.
\item If $w \in W$ is a Coxeter element, then it acts freely on $\Phi$, with $r = \#R$ orbits. In particular, we have $\# \Phi = r h$. 
\item Let $w \in W$ be a Coxeter element, and let $\dot{w} \in N_\hG(\hT)(k)$ denote a lift of $w$. Then the $\hG(k)$-conjugacy class of $\dot{w}$ is independent of any choices and $\hT_k^w = Z_{\hG_k}$. The element $\dot{w}$ is regular semisimple and $\dot{w}^h \in Z_{\hG}(k)$. We write $\dot{h}$ for the order of $\dot{w}$, which depends only on $\hG$.
\item Suppose that $t$ is a prime number not dividing $\text{char } k$ or $\# W$ and such that $t \equiv 1 \text{ mod }h$, and let $w \in W$ be a Coxeter element. Let $q \in \bbF_t^\times$ be a primitive $h^\text{th}$ root of unity. Then $\hT(k)[t]^{w = q}$ is a 1-dimensional $\bbF_t$-vector space, and every non-zero element $v$ of this space is regular semisimple in $\hG(k)$. Conversely, if $w' \in W$ and $v \in \hT(k)[t]^{w = q} - \{ 0 \}$, and $w' v = a v$ for some $a \in \bbF_t^\times$, then $w'$ is a power of $w$. In particular, if $w' v = q v$, then $w' = w$.
\end{enumerate}
\end{proposition}
\begin{proof}
The first two points can be checked after extending scalars to $\bbC$, in which case see \cite{Cox34} and \cite{Kos59}, respectively. For the first assertion of the third part, it is enough to show that if $\dot{w}'$ is another lift of $w$ to $N_{\hG}(\hT)(k)$, then $\dot{w}$ and $\dot{w}'$ are $\hG(k)$-conjugate. We will show that they are in fact $\hT(k)$-conjugate. Indeed, if $x \in \hT(k)$, then we have $x \dot{w} x^{-1} = x^{1-w} \dot{w}$; on the other hand, we have $\dot{w}' = y \dot{w}$ for some $y \in \hT(k)$. We therefore need to show that the map $\hT(k) \to \hT(k), x \mapsto x^{1-w}$, is surjective. This will follow if we can show that the scheme-theoretic centralizer of $w$ in ${\hT_k}$ is equal to $Z_{\hG_k}$. To show this, it suffices to observe that the map $1 - w : X^\ast(\hT) \to X^\ast(\hT)$ is injective, with cokernel of order equal to $Z_{\hG}(k)$. This can again be checked in the case $k = \bbC$, in which case it is a known fact; see e.g.\ \cite[Lemma 6.2]{Gro10}. For the second assertion of the third part, we observe that in the Cartan decomposition
\[ \widehat{\frg}_k = \widehat{\frt}_k \oplus \bigoplus_{\alpha \in \Phi(\hG, \hT)} \widehat{\frg}_\alpha, \]
$\dot{w}^h$ acts trivially on $\widehat{\frt}_k$ and leaves invariant each root space $\widehat{\frg}_\alpha$, by definition of the Coxeter number $h$. On the other hand, $\dot{w}$ permutes these root spaces freely (by the second part of the proposition). It now follows from the facts that $\widehat{\frt}_k^{\dot{w}} = 0$ and that $\dim_k \widehat{\frg}_k \geq r$ that $\dim_k \widehat{\frg}_k^{\dot{w}} = \# R = r$, hence that $\dot{w}$ is regular and $\dot{w}^h$ acts trivially on $\widehat{\frg}_k$. We deduce that $\dot{w}$ is regular semisimple and that $\dot{w}^h \in Z_{\hG}(k)$.

We now come to the fourth point. We first prove the analogous claims for the Lie algebra $\widehat{\frt}_{\bbQ_t}$: if $\zeta \in \bbQ_t^\times$ is any primitive $h^\text{th}$ root of unity, then the eigenspace $\widehat{\frt}_{\bbQ_t}^{w = \zeta}$ is 1-dimensional, and all of its non-zero elements are regular semisimple. Indeed, it follows from \cite[Theorem 4.2]{Spr74} (and the fact that Coxeter elements are regular and have $a(h) = 1$, in the notation of \emph{op. cit.}) that the eigenspace $\widehat{\frt}_{\bbQ_t}^{w = \zeta}$ is 1-dimensional and is spanned by regular semisimple elements, and that the centralizer of $w$ in $W$ is the cyclic group generated by $w$. It then follows from \cite[Proposition 3.5]{Spr74} that the only elements of $W$ which preserve the eigenspace $\widehat{\frt}_{\bbQ_t}^{w = \zeta}$ are the powers of $w$. 

We must now deduce the corresponding statement for the group $\hT(k)$. Since $t \nmid \# W$, $X_\ast(\hT) \otimes_\bbZ \bbZ_t$ is a projective $\bbZ_t[W]$-module, and we have isomorphisms of $\bbZ_t[W]$-modules $X_\ast(\hT) \otimes_\bbZ \mu_t(k) \cong \hT(k)[t]$ (by evaluation), $X_\ast(\hT) \otimes_\bbZ \bbQ_t \cong \widehat{\frt}_{\bbQ_t}$ (via the canonical isomorphism $X_\ast(\bbG_m) \cong \bbZ \cong \Lie \bbG_m$). The characteristic polynomial in $\bbF_t[X]$ of $w$ acting on $\hT(k)[t]$ has distinct roots. There is a unique $h^\text{th}$ root of unity $\zeta \in \bbQ_t$ lifting $q$. Let $v$ be a non-zero element of $\hT(k)[t]^{w = q}$, a 1-dimensional $\bbF_t$-vector space, and suppose that $v$ is not regular semisimple, hence that the stabilizer of $v$ in $W$ is non-trivial (because $\hG$ is assumed simply connected). Let $H_v = \Stab_W(v)$; then $H_v$ is normalized by $w$, and we write $\widetilde{H}_v$ for the subgroup of $W$ generated by $H_v$ and $w$. Thus $\widetilde{H}_v$ acts on $v$ by a character. By lifting this character to characteristic 0, we can find $\widetilde{v} \in (X_\ast(\hT) \otimes_\bbZ \bbZ_t)^{w = \zeta}$ such that $H_v \subset \Stab_W(\widetilde{v})$, which contradicts the above paragraph after passage to $\bbQ_t$. This shows that $\Stab_W(v)$ is trivial and that $v$ is in fact regular semisimple.

Finally, let $w' \in W$, and let $v$ be a non-zero element of $\hT(k)[t]^{w = q}$. Suppose  that $v$ is a $w'$-eigenvector. Let $H'_v$ denote the stabilizer in $W$ of the line $\bbF_t \cdot v = \hT(k)[t]^{w = q}$ (so that both $w, w'$ lie in $H'_v$). Then $\bbF_t \cdot v$ is a 1-dimensional subrepresentation of $\hT(k)[t]$ which occurs with multiplicity 1 (because $q$ appears with multiplicity 1 as an eigenvalue of $w$). By lifting this character of $H'_v$ to characteristic 0, we can find a lift $\widetilde{v} \in (X_\ast(\hT) \otimes_\bbZ \bbZ_t)^{w = \zeta}$ of $v$, such that the line $\bbZ_t \cdot \widetilde{v}$ is invariant by $H'_v$. Appealing once more to the case of Lie algebras now shows that $w'$ is a power of $w$.
\end{proof}
\begin{definition}\label{def_abstract_Coxeter_homomorphism}
Let $\Gamma$ be a group. We call a homomorphism $\phi : \Gamma \to \hG(k)$ an abstract Coxeter homomorphism if it satisfies the following conditions:
\begin{enumerate}
\item There exists a maximal torus $T \subset \hG_k$ such that $\phi(\Gamma) \subset N_{\hG_k}(T)$, and the image of $\phi(\Gamma)$ in $W(\hG_k, T)$ is the cyclic group generated by a Coxeter element $w$. Let $\phi^\text{ad}$ denote the composite $\Gamma \to \hG(k) \to \hG^\text{ad}(k)$, $T^\text{ad}$ the image of $T$ in $\hG^\text{ad}_k$.
\item There exists a prime $t \equiv 1 \text{ mod }h$ not dividing $\text{char }k$ or $\# W$ and a primitive $h^\text{th}$ root of unity $q \in \bbF_t^\times$ such that $\phi^\text{ad}(\Gamma) \cap T^\text{ad}(k)$ is cyclic of order $t$ and $w v w^{-1} = v^q$ for any $v \in \phi^\text{ad}(\Gamma) \cap T^\text{ad}(k)$.
\end{enumerate}
\end{definition}
The definition includes some useful technical conditions that we can ensure are satisfied in applications. We will soon  see (Proposition \ref{prop_basic_properties_of_abstract_Coxeter_homomorphisms}) that the maximal torus $T$ appearing in this definition is the unique one with the listed properties.
\begin{lemma}\label{lem_reduction_modulo_l_of_abstract_Coxeter_homomorphism}
Let $\phi : \Gamma \to \hG(\overline{\bbQ})$ be an abstract Coxeter homomorphism. Let $\lambda$ be a place of $\overline{\bbQ}$, $\phi_\lambda : \Gamma \to \hG(\overline{\bbQ}) \to \hG(\overline{\bbQ}_\lambda)$ the composite, and $\overline{\phi}_\lambda : \Gamma \to \hG(\overline{\bbF}_l)$ the reduction modulo $l$. Suppose that $l > 2 h - 2$ and that $l$ does not divide the order of $\phi(\Gamma)$. Then $\overline{\phi}_\lambda$ is an abstract Coxeter homomorphism.
\end{lemma}
\begin{proof}
If $\phi : \Gamma \to \hG(\overline{\bbQ})$ is an abstract Coxeter homomorphism, we can assume (after conjugation) that the torus appearing in the definition is $\hT$. Since every element of the Weyl group $W(\hG, \hT)$ admits a representative in $\hG(\bbZ)$, this means we can even assume that $\phi$ takes values in $\hG(\overline{\bbZ})$ (i.e.\ the points of $\hG$ with values in the algebraic integers in $\overline{\bbQ}$), and that $\phi^\text{ad}$ takes values in $\hG(\bbZ[\zeta_t])$. Then the ``physical'' reduction of mod $\lambda$ (i.e.\ composition with $\hG(\bbZ[\zeta_t]) \to \hG(\overline{\bbF}_\lambda)$) of $\phi^\text{ad}$ has image of order prime to $l$, so is $\hG$-completely reducible, so is $\hG(\overline{\bbF}_l)$-conjugate to $\overline{\phi}_\lambda$ (i.e.\ the reduction modulo $\lambda$ of $\phi_\lambda$ defined using the pseudocharacter as in Definition \ref{def_reduction_modulo_l}). This implies that $\overline{\phi}_\lambda$ is itself an abstract Coxeter homomorphism. 
\end{proof}
\begin{lemma}\label{lem_adjoint_representation_of_abstract_Coxeter_homomorphism}
Let $\phi : \Gamma \to \hG(k)$ be an abstract Coxeter homomorphism, and let the torus $T$ be as in Definition \ref{def_abstract_Coxeter_homomorphism}. Then:
\begin{enumerate}
\item Let $\Delta = (\phi^\text{ad})^{-1}(T^\text{ad}(k))$. Then $\Gamma / \Delta$ is cyclic of order $h$. We write $w \in W(\hG_k, T)$ for the image of a generator. Then $w$ is a Coxeter element, and there are isomorphisms $\phi^\text{ad}(\Gamma) \cong \phi^\text{ad}(\Delta) \rtimes \langle w \rangle \cong \bbZ / t \bbZ \rtimes \bbZ / h \bbZ$, where $1 \in \bbZ / h \bbZ$ acts on the cyclic normal subgroup as multiplication by $q$.
\item Let $\frt = \Lie T$, $\widehat{\frg}_\alpha$ the $\alpha$-root space inside $\widehat{\frg}_k$ corresponding to a root $\alpha \in \Phi(\hG_k, T)$. Then there is an isomorphism of $k[\Gamma]$-modules:
\[ \widehat{\frg}_k \cong \frt \oplus \bigoplus_{\alpha \in \Phi(\hG_k, T) / w} \Ind_{\Delta}^\Gamma \widehat{\frg}_\alpha, \]
where each summand $\Ind_{\Delta}^\Gamma \widehat{\frg}_\alpha$ is an irreducible $k[\Gamma]$-module.
\end{enumerate}
\end{lemma}
\begin{proof}
We note that $w$ is a Coxeter element, because it generates the same cyclic group as a Coxeter element (\cite[Proposition 4.7]{Spr74}). Let $\sigma \in \Gamma$ be an element which projects to $w$, and let $\dot{w} = \phi^\text{ad}(\sigma) \in \hG^\text{ad}(k)$. Proposition \ref{prop_basic_properties_of_Coxeter_elements} shows that $\dot{w}$ has order $h$, and this implies the existence of the semidirect decomposition in the first part of the proposition.

To finish the proof of the lemma, we must show that for any $\alpha \in \Phi(\hG_k, T)$,  there is an isomorphism
\[ \widehat{\frg}_\alpha \oplus \widehat{\frg}_{\alpha^w} \oplus \dots \oplus \widehat{\frg}_{\alpha^{w^{h-1}}} \cong \Ind_{\Delta}^\Gamma \widehat{\frg}_\alpha \]
of irreducible $k[\Gamma]$-modules. By Frobenius reciprocity, it is enough to show that the induced representation is irreducible; and this is a consequence of the fact that the representations $\widehat{\frg}_\alpha, \widehat{\frg}_{\alpha^w}, \dots, \widehat{\frg}_{\alpha^{w^{h-1}}}$ are pairwise non-isomorphic. Indeed, if $\delta \in \Delta$ has $\phi^\text{ad}(\delta) \neq 1$, then $\phi^\text{ad}(\delta)$ generates $T(k)[t]^{w = q}$, hence is regular semisimple, by Proposition \ref{prop_basic_properties_of_Coxeter_elements}, hence satisfies $\alpha(\phi(\delta)) \neq 1$, showing that $\alpha(\phi(\delta))$ is a primitive $t^\text{th}$ root of unity. Then $\alpha^{w^i}(\phi(\delta)) = \alpha(\phi(\delta))^{q^i}$, and these elements are distinct as $i = 0, 1, \dots, h-1$.
\end{proof}
\begin{lemma}\label{lem_Coxeter_homomorphisms_have_abundant_image}
Let $\phi : \Gamma \to \hG(k)$ be an abstract Coxeter homomorphism. Then $H^0(\Gamma, \widehat{\frg}_k) = H^0(\Gamma, \widehat{\frg}^\vee_k) = 0$ and for every simple  non-trivial $k[\Gamma]$-submodule $V \subset \widehat{\frg}_k^\vee$, there exists $\gamma \in \Gamma$ such that $\phi(\gamma)$ is regular semisimple with connected centralizer in $\hG_k$ and $V^\gamma \neq 0$. 
\end{lemma}
\begin{proof}
The vanishing follows easily from Lemma \ref{lem_adjoint_representation_of_abstract_Coxeter_homomorphism}. To show the second part, we again apply Lemma \ref{lem_adjoint_representation_of_abstract_Coxeter_homomorphism}, which tells us what the possible choices for $V$ are. Under our assumptions the $k[\Gamma]$-module $\widehat{\frg}_k$ is self-dual, so it is enough to show that for any simple $k[\Gamma]$-module $V \subset \widehat{\frg}_k$, there exists $\gamma \in \Gamma$ such that  $\phi(\gamma)$ is regular semisimple and $V^\gamma \neq 0$. If $V \subset \frt$, then we can choose $\gamma \in \Delta$. If $V \subset  \oplus_{\alpha \in \Phi(\hG_k, T) / w } \Ind_{\Delta}^\Gamma \widehat{\frg}_\alpha$, then we can choose $\gamma$ so that $\phi(\gamma)$ projects to a Coxeter element. It follows from the second and third parts of Proposition \ref{prop_basic_properties_of_Coxeter_elements} that $\phi(\gamma)$ has non-trivial invariants in $V$.
\end{proof}
\begin{proposition}\label{prop_basic_properties_of_abstract_Coxeter_homomorphisms}
Let $\phi : \Gamma \to \hG(k)$ be an abstract Coxeter homomorphism.
\begin{enumerate}
\item $\phi$ is $\hG$-irreducible and $Z_{\hG^\text{ad}_k}(\phi(\Gamma))$ is scheme-theoretically trivial.
\item  There is exactly one maximal torus $T \subset \hG_k$ such that $\phi(\Gamma) \subset N_{\hG_k}(T)$.
\item Suppose that $\psi : \Gamma \to \hG(k)$ is another homomorphism, and for all $\gamma \in \Gamma$, the elements $\phi(\gamma)$ and $\psi(\gamma)$ of $\hG(k)$ have the same image in $(\hG \dquot \hG)(k)$. Then $\phi$ and $\psi$ are $\hG(k)$-conjugate. 
\end{enumerate}
\end{proposition}
\begin{proof}
Since $\phi(\Gamma)$ has order prime to the characteristic of $k$, it is $\hG$-completely reducible. In particular, if it is not $\hG$-irreducible then it is contained in a Levi subgroup of a proper parabolic of $\hG_k$, so centralizes a non-trivial torus. To show the first part, it is therefore enough to show that $Z_{\hG^\text{ad}_k}(\phi(\Gamma))$ is scheme-theoretically trivial, or even that $Z_{\hG_k}(\phi(\Gamma))$ is equal to $Z_{\hG_k}$. (We note that this really depends on the fact that $\hG$ is simply connected, and would be false in general otherwise, as one sees already by considering the example $\hG = \PGL_2$.)

Let $T$ be a maximal torus of $\hG_k$ and $w \in W(\hG_k, T)$ a Coxeter element as in the definition of abstract Coxeter homomorphism. Since $\hG_k$ is simply connected, the centralizer of a regular semisimple element of $T$ is $T$ itself \cite[Ch. 2]{Hum95}. The definition of Coxeter homomorphism shows that the centralizer of $\phi(\Gamma)$ is therefore contained inside $T^w$. By Proposition \ref{prop_basic_properties_of_Coxeter_elements}, we have $T^w = Z_{\hG_k}$, and this group is \'etale over $k$ (because we work in very good characteristic). This shows the first part of the proposition.

For the second part, suppose that $T'$ is another maximal torus such that $\phi(\Gamma) \subset N_{\hG_k}(T')$, and let $\Delta = (\phi^\text{ad})^{-1}(T(k))$. Then $\phi^\text{ad}(\Delta)$ has trivial projection to $W(\hG_k, T')$ (because it has $t$-power order, and $t$ does not divide the order of the Weyl group). It follows that $\phi^\text{ad}(\Delta) \subset T'$, hence $T = Z_{\hG_k}(\phi^\text{ad}(\Delta)) = T'$.

For the third part, we observe that the given condition means that for all $\gamma \in \Gamma$, $\phi(\gamma)$ and $\psi(\gamma)$ have the same semisimple part, up to $\hG(k)$-conjuation. Let $\widehat{\frg}_k \circ \phi$ and $\widehat{\frg}_k \circ \psi$ denote the two $k[\Gamma]$-modules coming from the adjoint representation of $\hG$. They have the same character, so they are isomorphic (up to semisimplification). This implies that $\psi$ is also $\hG$-irreducible. Indeed, if not then we can replace $\psi$ by its semisimplification to obtain a $\hG$-completely reducible representation which centralizes a non-trivial torus; this would imply that $\widehat{\frg}_k \circ \psi \cong \widehat{\frg}_k \circ \phi$ contains a non-trivial subspace on which $\Gamma$ acts trivially, which is not the case. It follows (\cite[Corollaire 5.5]{Ser05}, and the assumptions at the beginning of \S \ref{sec_abstract_coxeter_parameters}) that $\widehat{\frg}_k \circ \psi$ is a semisimple $k[\Gamma]$-module, hence there is an isomorphism $\widehat{\frg}_k \circ \phi \cong \widehat{\frg}_k \circ \psi$. This shows that $\psi(\Gamma)$ has finite order prime to the characteristic of $k$. In particular, every element of the image of $\psi$ is semisimple, so we find that for all $\gamma \in \Gamma$, $\phi(\gamma)$ and $\psi(\gamma)$ are $\hG(k)$-conjugate, hence $\ker \phi = \ker \psi$ and $\ker \phi^\text{ad} = \ker \psi^\text{ad}$.

Let $H = \phi(\Gamma)$, $H_0 = \phi^\text{ad}(\Gamma)$. We identify both $\phi$ and $\psi$ with homomorphisms $H \to \hG(k)$. Let $g \in H$ be an element that maps to a generator of $H_0 \cap T^\text{ad}(k)$, a 1-dimensional $\bbF_t$-vector space. After replacing $\psi$ by a $\hG(k)$-conjugate, and $g$ by a power, we can assume that $\phi(g) = \psi(g) \in T(k)[t]$. In particular, $g \in H$ has order $t$. Let $v = \phi(g)$. Let $g' \in H$ be a pre-image of $w \in W(\hG_k, T)$. Then there exists a primitive $h^\text{th}$ root of unity $q \in \bbF_t^\times$ such that $w v w^{-1} = v^q$, hence $g' g (g')^{-1} = g^q$. 

We have $\psi(g') \in N_{\hG_k}(T)(k)$. Let $w'$ denote the image of $\psi(g')$ in $W(\hG_k, T)$. Then we have $w' v (w')^{-1} = \psi(g^q) = \psi(g)^q = v^q$, so Proposition \ref{prop_basic_properties_of_Coxeter_elements} implies that $w' = w$. After replacing $\psi$ by a $T(k)$-conjugate, we can therefore assume that $\psi(g') = \phi(g')$, without disturbing our assumption that $\psi(g) = \phi(g)$. Since $g, g'$ generate $H_0$, this shows that $\phi^\text{ad} = \psi^\text{ad}$ (equality, not just isomorphism). It follows that there exists a character $\omega : H \to Z_{\hG}(k)$ such that for all $x \in H$, we have $\phi(x) = \omega(x) \psi(x)$. To finish the proof, we must show that $\omega = 1$. However, the elements $\phi(x)$, $\psi(x)$ are $\hG(k)$-conjugate, so for any $x \in H$ we can find $w_x \in W(\hG_k, T)$ such that $\psi(x) = \omega(x)^{-1}  \phi(x)  = \phi(x)^{w_x}$, hence $\omega(x) = \phi(x)^{1 - w_x}$. If $\phi(x) \in T(k)$, then $\phi(x)^{1 - w_x}$ has order $t$, implying that $\omega(x) = 1$ (since $t$ is prime to the order of $Z_{\hG}(k)$). To complete the proof, we now just need to observe that $\phi(g') = \psi(g')$, and $H$ is generated by $g'$, together with the subgroup $\phi(\Gamma) \cap T(k)$.
\end{proof}
\begin{corollary}\label{cor_strong_irreducibility_and_abundance_of_abstract_Coxeter_homomorphism}
Let $\phi : \Gamma \to \hG(k)$ be an abstract Coxeter homomorphism. Then $\phi$ is strongly $\hG$-irreducible and $\phi(\Gamma) \subset \hG(k)$ is a $\hG$-abundant subgroup.
\end{corollary}
\begin{proof}
The $\hG$-abundance follows from Lemma \ref{lem_Coxeter_homomorphisms_have_abundant_image}, and the strong irreducibility follows from Proposition \ref{prop_basic_properties_of_abstract_Coxeter_homomorphisms}.
\end{proof}
\subsection{Galois Coxeter parameters}\label{sec_galois_coxeter_parameters}

We now continue with the assumptions of \S \ref{sec_abstract_coxeter_parameters}, and specify a particular choice of $\Gamma$, as follows: we take $X$ to be a smooth, projective, geometrically connected curve over $\bbF_q$ of characteristic $p$, $K = \bbF_q(X)$, and $\Gamma = \Gamma_{K, S}$, where $S$ is a finite set of places of $K$. 
\begin{definition}\label{def_Galois_Coxeter_homomorphism}
Let $k$ be an algebraically closed field of characteristic 0 or $l > 2h - 2$. A Coxeter homomorphism over $k$ is a homomorphism $\phi : \Gamma_{K, S} \to \hG(k)$ which is an abstract Coxeter homomorphism, in the sense of Definition \ref{def_abstract_Coxeter_homomorphism}, and which is continuous, when $\hG(k)$ is endowed with the discrete topology.
\end{definition}
\begin{proposition}\label{prop_basic_properties_of_Galois_Coxeter_homomorphisms}
Let $\phi : \Gamma_{K, S} \to \hG(\overline{\bbQ})$ be a Coxeter homomorphism, and let $\lambda$ be a place of $\overline{\bbQ}$ of residue characteristic $l \neq p$. Let $\phi_\lambda$ be the composite of $\phi$ with the inclusion $\hG(\overline{\bbQ}) \subset \hG(\overline{\bbQ}_\lambda)$. Then:
\begin{enumerate}
\item  If $\psi : \Gamma_{K, S} \to \hG(\overline{\bbQ}_\lambda)$ is a continuous homomorphism such that for every place $v \not\in S$ of $K$, $\phi(\Frob_v)$ and $\psi(\Frob_v)$ have the same image in $(\hG \dquot \hG)(\overline{\bbQ}_\lambda)$, then $\phi$, $\psi$ are $\hG(\overline{\bbQ}_\lambda)$-conjugate.
\item If $l$ is prime to $ \# \phi(\Gamma_{K, S})$ and $l > 2 h - 2$, then the residual representation $\overline{\phi}_\lambda : \Gamma_K \to \hG(\overline{\bbF}_l)$ is a Coxeter homomorphism. In particular, it is strongly $\hG$-irreducible and $\overline{\phi}_\lambda(\Gamma_K)$ is a $\hG$-abundant subgroup of $\hG(\overline{\bbF}_l)$.
\item Let $L / K$ denote the extension cut out by $\phi^\text{ad}$, and let $K' / K$ be a finite separable extension linearly disjoint from $L$. Then $\phi|_{\Gamma_{K'}}$ is a Coxeter homomorphism. 
\end{enumerate}
\end{proposition}
\begin{proof}
The representations $\phi$, $\psi$ determine continuous maps $\Gamma_{K, S} \to (\hG \dquot \hG)(\overline{\bbQ}_\lambda)$ which agree on the set of Frobenius elements. This set is dense in $\Gamma_{K, S}$, by Corollary \ref{cor_equivalence_of_representations_with_similar_trace}, so the first part of the proposition follows from Proposition \ref{prop_basic_properties_of_abstract_Coxeter_homomorphisms}. The second part follows from Lemma \ref{lem_reduction_modulo_l_of_abstract_Coxeter_homomorphism} and Corollary \ref{cor_strong_irreducibility_and_abundance_of_abstract_Coxeter_homomorphism}. For the third part, it is enough to show that $\phi^\text{ad}(\Gamma_{K'}) = \phi^\text{ad}(\Gamma_K)$. This follows immediately from our hypothesis.
\end{proof}
We now come to the most important result of this section. Let $G$ denote a split reductive group over $K$, and let us suppose that $\hG$ is in fact the dual group of $G$. We recall that we are assuming that $\hG$ is simply connected, so this implies that $G$ is in fact an adjoint group.
\begin{theorem}\label{thm_automorphy_of_Coxeter_homomorphisms}
Let $\phi : \Gamma_{K, \emptyset} \to \hG(\overline{\bbQ})$ be a Coxeter homomorphism, and let $\lambda$ be a place of $\overline{\bbQ}$ of residue characteristic $l \neq p$. Suppose that $\phi(\Gamma_{K \cdot \overline{\bbF}_q})$ is contained in a conjugate of $\hT(\overline{\bbQ})$. Then there exists a cuspidal automorphic representation $\pi$ over $\overline{\bbQ}_\lambda$ such that $\pi^{G(\widehat{\cO}_K)} \neq 0$ and which corresponds everywhere locally to $\phi|_{W_{K_v}}$ (under the unramified local Langlands correspondence). In other words, for every irreducible representation $V$ of $\hG_{\overline{\bbQ}_\lambda}$, and for every place $v$ of $K$, we have the relation
\[ \chi_V( \phi_\lambda(\Frob_v) ) = \text{eigenvalue of }T_{V, v} \text{ on }\pi^{G(\widehat{\cO}_K)}. \]
\end{theorem}
\begin{proof}
\cite[Theorem 2.2.8]{Bra02} implies the existence of a spherical automorphic function 
\[ f : G(K) \backslash G(\bbA_K) / G(\widehat{\cO}_K) \to \overline{\bbQ}_\lambda\]
 with unramified Hecke eigenvalues which correspond to $\phi|_{W_{K_v}}$ under the unramified local Langlands correspondence. (It is assumed in this reference that $G$ is a reductive group with simply connected derived group; the general case can be reduced to this one using $z$-extensions \cite{Kot82}.) If we knew that this function was cuspidal and non-zero, then we could take $\pi$ to be the representation generated by $f$. These properties of the function $f$ are established in the two appendices to this paper by Gaitsgory. We note that the proof that $f$ is non-zero shows in fact that the first Whittaker coefficient is non-zero, implying that $\pi$ is in fact globally generic. 
\end{proof}
\begin{lemma}\label{lem_Coxeter_parameters_rational_over_Z_l}
Suppose given a prime $t \nmid p \# W$ such that $q \text{ mod }t$ has exact order $h$. Then we can find a Coxeter parameter $\phi : \Gamma_{K, S} \to \hG(\overline{\bbQ}
)$ with the following properties: 
\begin{enumerate}
\item $\phi(\Gamma_K) \subset \hG(\bbZ[\zeta_t])$.
\item For any prime-to-$p$ place $\lambda$ of $\overline{\bbQ}$, and for any finite separable extension $K' / K$, linearly disjoint from the extension of $K$ cut out by $\phi^\text{ad}$, and such that $\phi|_{\Gamma_{K'}}$ is everywhere unramified, the representation $\phi_\lambda|_{\Gamma_{K'}} : \Gamma_{K'} \to \hG(\overline{\bbQ}_\lambda)$ is automorphic in the sense of Definition \ref{def_automorphic_galois_representation}, being associated to a prime ideal $\frp$ of the excursion algebra $\cB(G(\widehat{\cO}_{K'}), \overline{\bbQ}_\lambda)$. 
\end{enumerate}
\end{lemma}
\begin{proof}
Let $\alpha \in K$ be an element that has valuation 1 at some place of $K$, let $f(Y) = Y^t - \alpha$, and let $E_0 / K$ denote the splitting field of $f(Y)$. Then $f(Y) \in K[Y]$ is irreducible, even over $K \cdot \overline{\bbF}_q$, and there is an isomorphism $\Gal(E_0/K) \cong \mu_t \rtimes \bbZ / h \bbZ$, where $1 \in \bbZ / h \bbZ$ is a lift of Frobenius, which acts on $\mu_t$ by $\zeta \mapsto \zeta^q$. Let $w \in W(\hG, \hT)$ be a Coxeter element, and let $\dot{w} \in \hG(\bbZ)$ be a lift to $\hG$. Let $\dot{h}$ denote the order of $\dot{w}$, and $E = E_0 \cdot \bbF_{q^{\dot{h}}}$. Then there is an isomorphism $\Gal(E / K) \cong \mu_t \rtimes \bbZ / \dot{h} \bbZ$, where again $1 \in \bbZ / \dot{h} \bbZ$ acts on $\mu_t$ by $\zeta \mapsto \zeta^q$. 

By Proposition \ref{prop_basic_properties_of_Coxeter_elements}, the $\bbF_t$-vector space $\hT(\bbZ[\zeta_t])[t]^{w = q} = \hT(\overline{\bbQ})[t]^{w = q}$ is 1-dimensional; let $v$ be a non-zero element. Let $\phi : \Gal(E/K) \to \hG(\bbZ[\zeta_t])$ be the homomorphism which sends a generator of $\mu_t$ to $v$ and $1 \in \bbZ / \dot{h} \bbZ$ to $\dot{w}$. Then $\phi$ is a Coxeter homomorphism into $\hG(\overline{\bbQ})$ which takes values in $\hG(\bbZ[\zeta_t])$.

For any place $\lambda$ and any extension $K' / K$ as in the statement of the lemma, $\phi_\lambda|_{\Gamma_{K'}}$ is an everywhere unramified Coxeter parameter, which takes values in $\hT(\overline{\bbQ}_\lambda)$ after restriction to the geometric fundamental group. Theorem \ref{thm_automorphy_of_Coxeter_homomorphisms} implies the existence of an everywhere unramified cuspidal automorphic representation $\pi$ which corresponds to $\phi$ everywhere locally, hence an everywhere unramified automorphic Galois representation $\sigma_\frp$ (associated to a prime ideal $\frp$ of the algebra of excursion operators at level $G(\widehat{\cO}_{K'})$) such that for every place $v$ of $K'$, $\phi_\lambda(\Frob_v)$ and $\sigma_\frp(\Frob_v)$ have the same image in $(\hG \dquot \hG)(\overline{\bbQ}_\lambda)$. Proposition \ref{prop_basic_properties_of_Galois_Coxeter_homomorphisms} then implies that $\phi_\lambda|_{\Gamma_{K'}}$ and $\sigma_\frp$ are in fact $\hG(\overline{\bbQ}_\lambda)$-conjugate, showing that $\phi_\lambda|_{\Gamma_{K'}}$ is automorphic, as required.
\end{proof}

\section{Potential automorphy}\label{sec_potential_automorphy}
Let $X$ be a smooth, geometrically connected curve over $\bbF_q$, and let $K = \bbF_q(X)$. Let $G$ be a split semisimple group over $\bbF_q$. Our goal in this section the following result, which is the main theorem of this paper.
\begin{theorem}\label{thm_main_theorem}
Let $l \nmid q$ be a prime, and let $\rho : \Gamma_K \to \hG(\overline{\bbQ}_l)$ be a continuous, everywhere unramified representation with Zariski dense image. Then there exists a finite Galois extension $K'/ K$ such that $\rho|_{\Gamma_{K'}}$ is automorphic: there exists a cuspidal automorphic representation $\Pi$ of $G(\bbA_{K'})$ over $\overline{\bbQ}_l$ such that $\Pi^{G(\widehat{\cO}_{K'})} \neq 0$ and for every place $v$ of $K'$, $\rho|_{W_{K'_v}}$ and $\Pi_v$ are matched under the unramified local Langlands correspondence at $v$. In other words, for every irreducible representation $V$ of $\hG_{\overline{\bbQ}_l}$, and for every place $v$ of $K'$, we have the relation
\[ \chi_V( \rho(\Frob_v) ) = \text{eigenvalue of }T_{V, v} \text{ on }\Pi^{G(\widehat{\cO}_{K'})}. \]
\end{theorem}
\begin{proof}
Let $H$ denote the adjoint group of $G$, $\eta : G \to H$ the canonical isogeny. Then $H$ is a product of its simple factors, and the theorem is therefore true for $H$ by Theorem \ref{thm_main_theorem_simply_connected_case} below. We have a dual isogeny $\widehat{\eta} : \hH \to \hG$, which presents $\hH$ as the simply connected cover of $\hG$. By \cite[Theorem 1.4]{Con15}, the representation $\rho$ lifts to a representation $\rho_H : \Gamma_K \to \hH(\overline{\bbQ}_l)$ such that $\widehat{\eta} \circ \rho_H = \rho$. Then $\rho_H$ has Zariski dense image as well, and becomes unramified after a finite base change. Moreover, it lives in a compatible system, by Theorem \ref{thm_existence_of_compatible_systems_containing_a_given_representation}. We can therefore apply Theorem \ref{thm_main_theorem_simply_connected_case} to find a finite Galois extension $K' / K$ and a cuspidal automorphic representation $\Pi_H$ of $H(\bbA_{K'})$ over $\overline{\bbQ}_l$ such that for every place $v$ of $K'$, $\Pi_{H, v}^{H(\cO_{K'_v})} \neq 0$, and $\Pi_{H, v}$ and $\rho_H|_{W_{K'_v}}$ are matched under the unramified local Langlands correspondence.

 We let $f_H : H({K'}) \backslash H(\bbA_{K'}) / H(\widehat{\cO}_{K'}) \to \overline{\bbQ}_l$ be a cuspidal function which generates $\Pi_H$, and set $f_G = f_H \circ \eta$. Then $f_G : G({K'}) \backslash G(\bbA_{K'}) / G(\widehat{\cO}_{K'}) \to \overline{\bbQ}_l$ is a cuspidal function, and if $f_G \neq 0$ then its Hecke eigenvalues are matched everywhere locally with $\rho$ (because the Satake isomorphism is compatible with isogenies). The proof of the theorem will be finished if we can show that $f_G \neq 0$ (as then we can take $\Pi = \Pi_G$ to be the cuspidal automorphic representation generated by $f_G$). This non-vanishing follows immediately from Proposition \ref{prop_restriction_to_covering_group_non_zero}.
\end{proof}
\begin{proposition}\label{prop_restriction_to_covering_group_non_zero}
Let $\eta : G \to H$ denote the adjoint group of $G$, and let $l \nmid q$ be a prime. Let $\rho : \Gamma_K \to \hH(\overline{\bbQ}_l)$ be a continuous, everywhere unramified representation with Zariski dense image, and suppose that there exists a non-zero cuspidal function $f : H(K) \backslash H(\bbA_K) / H(\widehat{\cO}_K) \to \overline{\bbQ}_l$ such that for every place $v$ of $K$ and every irreducible representation $V$ of $\hH_{\overline{\bbQ}_l}$, we have
\[ T_{V, v}(f) = \chi_V(\rho(\Frob_v)) f. \]
Then $f|_{\eta(G(\bbA_K))}$ is not zero.
\end{proposition}
\begin{proof}
The idea of the proof is as follows: if $f(\eta(G(\bbA_K))) = 0$, then the automorphic representation generated by $f$ should be a lift from a twisted endoscopic group of $H$. This would contradict the Zariski density of the image of $\rho$. To implement this idea, we must get our hands dirty.

Let $T_G$ be a split maximal torus of $G$, and let $T_H$ denote its image in $H$. Then we have a commutative diagram of $\bbF_q$-groups with exact rows
\begin{equation}\label{eqn_covering_group_of_adjoint_group} \begin{aligned} \xymatrix{ 1 \ar[r] & Z_G \ar[d] \ar[r] & T_G \ar[r] \ar[d] & T_H \ar[r] \ar[d] & 1 \\
1 \ar[r] & Z_G \ar[r] & G \ar[r] & H \ar[r] & 1. } \end{aligned}
\end{equation}
We observe that $\eta(G(\bbA_K))$ is a normal subgroup of $H(\bbA_K)$ which contains the derived group of $H(\bbA_K)$, which implies that the subgroup of $H(\bbA_K)$ generated by $\eta(G(\bbA_K))$, $H(K)$, and $H(\widehat{\cO}_K)$ actually equals $H(K) \cdot \eta(G(\bbA_K)) \cdot H(\widehat{\cO}_K)$, and is normal.  Let $\cY = H(\bbA_K) /( H(K) \cdot \eta(G(\bbA_K)) \cdot H(\widehat{\cO}_K))$, and let $\cY^\ast$ denote the group of characters $\omega : \cY \to \overline{\bbQ}_l^\times$. If $\omega \in \cY^\ast$, we define a new function $f \otimes \omega : H(K) \backslash H(\bbA_K) / H(\widehat{\cO}_K) \to \overline{\bbQ}_l$ by the formula $(f \otimes \omega)(h) = f(h) \omega(h)$. Then $f \otimes \omega$ is also cuspidal. The proposition will follow from the following two claims:
\begin{enumerate}
\item The group $\cY$ is finite.
\item The set $\{ f \otimes \omega \}_{\omega \in \cY^\ast}$ is linearly independent over $\overline{\bbQ}_l$.
\end{enumerate}
Indeed, (i) implies that we can find (unique) constants $a_\omega \in \overline{\bbQ}_l$ such that for $h \in H(\bbA_K)$, we have
\[ \sum_{\omega \in \cY^\ast} a_\omega \omega(h) = \left\{ \begin{array}{ll} 1 & \text{ if }h\text{ maps to the trivial element in }\cY; \\ 0 & \text{ otherwise.} \end{array}\right. \]
If $f(\eta(G(\bbA_K))) = 0$, then we have $\sum_{\omega \in \cY^\ast} a_\omega (f \otimes \omega) = 0$, which contradicts (ii).

We therefore have to establish  the claims (i) and (ii) above. Taking $K$-points in the diagram (\ref{eqn_covering_group_of_adjoint_group}) and applying Theorem 90 leads to a commutative diagram
\[ \begin{aligned} \xymatrix{ 1 \ar[r] & Z_G(K) \ar[d] \ar[r] & T_G(K) \ar[r] \ar[d] & T_H(K) \ar[r] \ar[d] & H^1(K, Z_G) \ar[d]^= \ar[r] & 1 \ar[d]\\
1 \ar[r] & Z_G(K) \ar[r] & G(K) \ar[r] & H(K) \ar[r] & H^1(K, Z_G) \ar[r] & H^1(K, G), } \end{aligned}
\]
where the cohomology is flat cohomology. We find that the connecting homomorphism induces an isomorphism $H(K) / \eta( G(K) ) \cong H^1(K, Z_G)$. The same is true if $K$ is replaced by $K_v$ or $\cO_{K_v}$ (for any place $v$ of $K$). We obtain an isomorphism
\[ \cY \cong \ilim_S H(\cO_{K, S}) \backslash \left[ \prod_{v \in S} \frac{H(K_v)}{\eta(G(K_v)) \cdot H(\cO_{K_v})} \right]\cong \ilim_S H(\cO_{K, S}) \backslash \left[ \prod_{v \in S} \frac{H^1(K_v, Z_G)}{H^1(\cO_{K_v}, Z_G)} \right] \]
\[ \cong (\prod_v \res_v)(H^1(K, Z_G)) \backslash \left[ \bigoplus_v  \frac{H^1(K_v, Z_G)}{H^1(\cO_{K_v}, Z_G)} \right]. \]
This group is finite. Indeed, $Z_G$ is a finite $\bbF_q$-group with constant dual, so it suffices to show that for any $n \geq 1$, the group
\[
(\prod_v \res_v)(H^1(K, \mu_n)) \backslash \left[ \bigoplus_v  \frac{H^1(K_v, \mu_n)}{H^1(\cO_{K_v}, \mu_n)} \right] \cong K^\times \backslash \bbA_K^\times / \widehat{\cO}_K^\times (\bbA_K^\times)^n \cong \Pic(X) \otimes_\bbZ \bbZ / n \bbZ
\]
is finite, and this is true. This shows claim (i) above. In order to show claim (ii), let $Z_G^D$ denote the Cartier dual of $Z_G$, and let $\cX$ denote the group of everywhere unramified characters $\chi : \Gamma_K \to Z_G^D$ (a subgroup of $H^1(K, Z_G^D)$); note that $Z_G^D$ is a constant \'etale group scheme. Let $\inv_v : H^2(K_v, \bbG_m) \to \bbQ / \bbZ$ be the isomorphism of class field theory. For each place $v$ of $K$, local duality gives a perfect pairing $H^1(K_v, Z_G^D) \times H^1(K_v, Z_G) \to \bbQ / \bbZ$, defined explicitly by the formula $(\chi, h) \mapsto \inv_v(\chi \cup h)$. The Cassels--Poitou--Tate exact sequence (\cite[Theorem 6.2]{Ces15}) gives an exact sequence
\[
\xymatrix@1{ 0 \ar[r] & \cY \ar[r]^-{\cQ^\vee} & \Hom(\cX, \bbQ / \bbZ) \ar[r] & H^2(K, Z_G) \ar[r]^-{\prod_v \res_v} & \oplus_v H^2(K_v, Z_G),}
\]
where $\cQ^\vee$ is defined by $\cQ^\vee( h )(\chi) = \sum_v \inv_v(\chi \cup h)$. The last map is injective (we can again reduce to the analogous statement for $\mu_n$, where it follows from class field theory), so after dualizing we obtain an isomorphism $\cQ : \cX \to \Hom(\cY, \bbQ / \bbZ)$. 

We now fix an isomorphism $\bbQ / \bbZ \cong \overline{\bbQ}_l^\times[\text{tors}]$. Then there is a canonical isomorphism
\begin{equation}\label{eqn_dual_of_centre_is_centre_of_dual} Z_G^D \cong \ker( Z_{\hH}(\overline{\bbQ}_l) \to Z_{\hG}(\overline{\bbQ}_l) ),
\end{equation}
which allows us to identify the inverse of $\cQ$ with an isomorphism
\[ \cP : \cY^\ast \to  \Hom(\Gamma_{K, \emptyset}, \ker( Z_{\hH}(\overline{\bbQ}_l) \to Z_{\hG}(\overline{\bbQ}_l) ). \]
Let $V$ be an irreducible representation of $\hH_{\overline{\bbQ}_l}$ of highest weight $\lambda \in X^\ast(\hT_H)$ (with respect to the fixed Borel subgroup $\hB_H \subset \hH$). Let $v$ be a place of $K$. We note that $T_{V, v}$, considered as a compactly supported function on $H(K_v)$, is supported on double cosets of the form $H(\widehat{\cO}_K) \mu(\varpi_v) H(\widehat{\cO}_K)$, where $\mu \in X_\ast(T_H)$ has the property that $\lambda - \mu$ is a sum of positive coroots, hence lies in the image of $X_\ast(T_G) \to X_\ast(T_H)$. In particular, for any character $\omega \in \cY^\ast$, we have $\omega(\lambda(\varpi_v)) = \omega(\mu(\varpi_v))$. This allows us to calculate for $g \in H(\bbA_K)$:
\begin{equation}\label{eqn_hecke_eigenvalues_of_twist} \begin{split} (T_{V, v} (f \otimes \omega))(g) & = \int_{h \in H(K_v)} T_{V, v}(h) f(gh) \omega(gh) \, dh  \\
& = \omega(g) \int_{h \in H(K_v)} T_{V, v}(h) f(gh) \omega(\lambda(\varpi_v)) \, dh = \omega(\lambda(\varpi_v))  \chi_V(\rho(\Frob_v)) (f \otimes \omega)(g). \end{split}
\end{equation}
We now observe that claim (ii) above follows from the following:
\begin{enumerate}
\setcounter{enumi}{2}
\item For any place $v$ of $K$, for any irreducible representation $V$ of $\hH_{\overline{\bbQ}_l}$ of highest weight $\lambda \in X_\ast(T_H) = X^\ast(\hT_H)$, and for any $\omega \in \cY^\ast$, we have the equality
\begin{equation}\label{eqn_twisting_identity} \omega(\lambda(\varpi_v)) = \lambda( \cP(\omega)(\Frob_v) ). 
\end{equation}
\end{enumerate}
Indeed, the equations (\ref{eqn_hecke_eigenvalues_of_twist}) and (\ref{eqn_twisting_identity}) together imply the identity
\begin{equation}
T_{V, v} (f \otimes \omega) = \lambda( \cP(\omega)(\Frob_v) )  \chi_V(\rho(\Frob_v)) = \chi_V( (\rho \otimes \cP(\omega))(\Frob_v)), 
\end{equation}
showing that the function $f \otimes \omega$ has unramified Hecke eigenvalues matching $\rho \otimes \cP(\omega)$ under the unramified local Langlands correspondence. 

To show that the $f \otimes \omega$ ($\omega \in \cY^\ast$) are linearly independent, it suffices to show that they have pairwise distinct sets of Hecke eigenvalues. However, if two forms $f \otimes \omega$, $f \otimes \omega'$ have the same Hecke eigenvalues, it follows from Proposition \ref{prop_zariski_dense_and_locally_conjugate_implies_globally_conjugate} and our assumption that $\rho$ has Zariski dense image that $\rho \otimes \cP(\omega)$ and $\rho \otimes \cP(\omega')$ are actually $\hH(\overline{\bbQ}_l)$-conjugate, hence there exists $g \in \hH(\overline{\bbQ}_l)$ such that $g (\rho \otimes \cP(\omega) )g^{-1} = \rho \otimes \cP(\omega' )$. In particular, $g \not\in Z_{\hH}(\overline{\bbQ}_l)$. If $L / K$ denotes a finite extension such that $\omega$, $\omega'$ are trivial on $\Gamma_L$, then $\rho(\Gamma_L)$ is still Zariski dense in $\hH(\overline{\bbQ}_l)$ and we have $g \rho|_{\Gamma_L} g^{-1} = \rho|_{\Gamma_L}$, implying that $g \in Z_{\hH}(\overline{\bbQ}_l)$. This contradiction shows (conditional on the claim (iii) above) that no two forms $f \otimes \omega$, $f \otimes \omega'$ have the same Hecke eigenvalues, hence that claim (ii) is true.

To finish the proof of the proposition, we therefore just need to establish claim (iii). After unwinding the definitions, this is equivalent to the conclusion of Lemma \ref{lem_reciprocity_and_cup_products} below. This concludes the proof of the proposition.
\end{proof}
\begin{lemma}\label{lem_reciprocity_and_cup_products}
Let $v$ be a place of $K$, and consider an exact sequence
\begin{equation}\label{eqn_exact_sequence_of_split_tori} \xymatrix@1{ 0 \ar[r] & Z \ar[r] & T \ar[r]^-f & T' \ar[r] & 0,} 
\end{equation}
where $f$ is an isogeny of split tori over $K_v$. Let $\iota$ be the map defined by the dual exact sequence
\begin{equation} \xymatrix@1{ 0 \ar[r] & Z^D \ar[r]^-\iota & X^\ast(T') \otimes \bbQ / \bbZ \ar[r] & X^\ast(T) \otimes \bbQ / \bbZ \ar[r] & 0. } 
\end{equation}
Let $\chi \in H^1(K_v, Z^D)$ be an unramified element. Then for any $\lambda \in X^\ast(\hT') = X_\ast(T')$, we have the formula
\begin{equation} \inv_v( \chi \cup (\delta\lambda(\varpi_v))) = \lambda( \iota \chi(\Frob_v) ), 
\end{equation}
where the connecting homomorphism $\delta$ is defined by the exact sequence (\ref{eqn_exact_sequence_of_split_tori}).
\end{lemma}
\begin{proof}
The desired formula is linear in $\lambda$. We can choose isomorphisms $T \cong \bbG_m^n$, $T' \cong \bbG_m^n$ such that $X^\ast(f)$ is given by a diagonal matrix (because of the existence of Smith normal form). We can therefore reduce to the case $T = T' = \bbG_m$, $f(x) = x^m$ for some non-zero integer $m$, $\lambda : \bbG_m \to \bbG_m$ the identity map. In this case we must show the identity
\[ \inv_v(   \chi \cup \delta(\varpi_v)) = \chi(\Frob_v). \]
In fact we have for any $b \in K_v^\times$, $\chi : \Gamma_{K_v} \to \bbZ / m \bbZ$:
\[  \inv_v(  \chi \cup \delta(b)) = - \inv_v( b \cup \delta(\chi) ) = \chi(\Art_{K_v}(b)), \]
using \cite[Ch. XIV, \S 1, Proposition 3]{Ser62} (and noting that $\Art_{K_v}$ is the reciprocal of the reciprocity map defined there). This is the desired result.
\end{proof}
We have now reduced Theorem \ref{thm_main_theorem} to the case where $G$ is a simple adjoint group over $\bbF_q$. This is the case treated by the following theorem.
\begin{theorem}\label{thm_main_theorem_simply_connected_case}
Suppose that $\hG$ is simple and simply connected, and let $(\emptyset, (\rho_\lambda)_\lambda)$ be a compatible system of representations $\rho_\lambda : \Gamma_{K, \emptyset} \to \hG(\overline{\bbQ}_\lambda)$, each of which has Zariski dense image. Then there exists a finite Galois extension $K' / K$ and a cuspidal automorphic representation $\Pi$ of $G(\bbA_{K'})$ over $\overline{\bbQ}$ such that $\Pi^{G(\widehat{\cO}_{K'})} \neq 0$ and for every place $v$ of $K'$, and for each prime-to-$q$ place $\lambda$ of $\overline{\bbQ}$, $\rho_\lambda|_{W_{K'_v}}$ and $\Pi_v$ are matched under the unramified local Langlands correspondence.
\end{theorem}
\begin{proof}
Let $Y$ be another geometrically connected curve over $\bbF_q$ and let $F = \bbF_q(Y)$. After possibly replacing $Y$ by a finite cover, we can find a  compatible system $(\emptyset, (R_\lambda)_\lambda)$ of everywhere unramified and Zariski dense $\hG$-representations of $\Gamma_F$ (apply Proposition \ref{prop_construction_of_unramified_and_dense_compatible_system}). By Proposition \ref{prop_big_image_in_compatible_systems}, we can replace $(\rho_\lambda)_\lambda$ and $(R_\lambda)_\lambda$ by equivalent compatible systems and find a number field $E$ and a set $\cL'$ of rational primes with the following properties:
\begin{itemize}
\item For each place $\lambda$ of $\overline{\bbQ}$, both $\rho_\lambda$ and $R_\lambda$ take values in $\hG(E_\lambda)$.
\item The set $\cL'$ has Dirichlet density 0. If $l \neq p$ is a rational prime split in $E$, and $l \not\in \cL'$, and $\lambda$ lies above $l$, then both $R_\lambda$ and $\rho_\lambda$ have image equal to $\hG(\bbZ_l)$. 
\end{itemize}
We recall that $h$ denotes the Coxeter number of $\hG$. The group $\Gal(E(\zeta_{h^\infty}) / E)$ embeds naturally (via the cyclotomic character) as a finite index subgroup of $\prod_{r | h} \bbZ_r^\times$; we fix an integer $b \geq 1$ such that it contains the subgroup of elements congruent to $1 \text{ mod } h^b$. Let $t > \# W$ be a prime such that $h^{b+1}$ divides the multiplicative order of $q \text{ mod } t$; there exist infinitely many such primes, by the main theorem of \cite{Mor05}.

By the Chebotarev density theorem, we can find primes $l_0, l_1$ not dividing $q$ and satisfying the following conditions:
\begin{itemize}
\item $l_0$ splits in $E$, $l_0 > \# W$, and $l_0 \not\in \cL'$.
\item $l_1$ splits in $E(\zeta_t)$, $l_1 > t$, and $l_1 \not\in \cL'$.
\item If $r | h$ is a prime, then $l_1 \equiv 1 \text{ mod }h^b$ but $l_1 \not\equiv 1 \text{ mod }r h^b$.
\item The groups $\hG(\bbF_{l_0})$ and $\hG(\bbF_{l_1})$ are perfect and have no isomorphic non-trivial quotients.
\end{itemize}
In particular, $l_0$ and $l_1$ are both very good characteristics for $\hG$. If $r$ is a prime and $g$ is an integer prime to $r$, let $o_r(g)$ denote the order of the image of $g$ in $\bbF_r^\times$. Thus $o_t(q)$ is divisible by $h^{b+1}$. Let $\alpha = o_t(q) / h$, so that $o_t(q^\alpha) = h$ and $q^\alpha \text{ mod }t$ is a primitive $h^\text{th}$ root of unity, and $\alpha$ is divisible by $h^b$. We now observe that:
\begin{itemize}
\item The degree $[ \bbF_{q^\alpha}(\zeta_{l_1}) : \bbF_{q^\alpha} ]$ is prime to $h$.
\end{itemize}
Indeed, the degree of this field extension equals $o_{l_1}(q^\alpha)$, which in turn divides $o_{l_1}(q^{h^b})$, which itself in turn divides $(l_1 - 1) / h^b$. This quantity is prime to $h$, by construction. 

Let $K_0 = K \cdot \bbF_{q^\alpha}$. We can now apply Lemma \ref{lem_Coxeter_parameters_rational_over_Z_l} to obtain a Coxeter parameter $\phi : \Gamma_{K_0} \to \hG(\overline{\bbQ})$ satisfying the following conditions:
\begin{itemize}
\item $\phi$ takes values in $\hG(\bbZ[\zeta_t])$.
\item For any prime-to-$p$ place $\lambda$ of $\overline{\bbQ}$ and for any separable extension $K' / K_0$ linearly disjoint from the extension of $K_0$ cut out by $\phi^\text{ad}$ such that $\phi|_{\Gamma_{K'}}$ is everywhere unramified, the composite $\phi_\lambda|_{\Gamma_{K'}} : \Gamma_{K'} \to \hG(\bbZ[\zeta_t]) \to \hG(\overline{\bbQ}_\lambda)$ is automorphic (that is, associated to a prime ideal $\frp$ of the algebra $\cB(G(\widehat{\cO}_{K'}), \overline{\bbQ}_\lambda)$ of excursion operators). 
\end{itemize}
Fix places $\lambda_0, \lambda_1$ of $\overline{\bbQ}$ of characteristics $l_0, l_1$, respectively. We see that the following additional condition is satisfied:
\begin{itemize}
\item Let $K' / K_0$ be a separable extension, linearly disjoint from the extension of $K_0$ cut out by $\phi^\text{ad}$. Then $\overline{\phi}_{\lambda_1}(\Gamma_{K'(\zeta_{l_1})})$ is a $\hG$-abundant subgroup of $\hG(\overline{\bbF}_{l_1})$, and $\overline{\phi}_{\lambda_1}|_{\Gamma_{K'}}$ is strongly $\hG$-irreducible.
\end{itemize}
Indeed, it suffices to show that the extension $K'(\zeta_{l_1}) / K_0$ is linearly disjoint from the extension $M/K_0$ cut out by $\phi^\text{ad}$, as then $\phi|_{\Gamma_{K'(\zeta_{l_1})}}$ is a Coxeter parameter (Proposition \ref{prop_basic_properties_of_Galois_Coxeter_homomorphisms}) and we can appeal to Corollary \ref{cor_strong_irreducibility_and_abundance_of_abstract_Coxeter_homomorphism}. To show this disjointness, note that the map $\Gamma_{K'} \to \Gal(M/K_0)$ is surjective, by construction; the image of $\Gamma_{K'(\zeta_{l_1})}$ in $\Gal(M/K_0)$ is a normal subgroup with abelian quotient of order dividing $o_{l_1}(q^\alpha)$. The abelianization of $\Gal(M/K_0)$ is cyclic of order $h$. Using that the $o_{l_1}(q^\alpha)$ is prime to $h$, we find that $\Gamma_{K'(\zeta_{l_1})} \to \Gal(M/K_0)$ is surjective and hence that the extensions $K'(\zeta_{l_1}) / K_0$ and $M / K_0$ are indeed linearly disjoint, as required.

Let $F_0 = F \cdot \bbF_{q^\alpha}$. We now apply Proposition \ref{prop_potential_agreement_of_residual_representations} with the following data:
\begin{itemize}
\item $H = \hG(\bbF_{l_0}) \times \hG(\bbF_{l_1})$.
\item $\varphi = \overline{\rho}_{\lambda_0}|_{\Gamma_{K_0}} \times \overline{\phi}_{\lambda_1}$.
\item $\psi = \overline{R}_{\lambda_0}|_{\Gamma_{F_0}} \times \overline{R}_{\lambda_1}|_{\Gamma_{F_0}}$. Note that $\psi$ is surjective, by Goursat's lemma. It is even surjective after restriction to the geometric fundamental group, because $H$ is perfect.
\item $L / K_0$ is the extension cut out by $\varphi$.
\end{itemize}
We obtain a finite Galois extension $K' / K_0$ and an $\bbF_{q^\alpha}$-embedding $\beta : F_0 \hookrightarrow K'$ satisfying the following conditions:
\begin{itemize}
\item The extension $K' / K_0$ is linearly disjoint from $L / K_0$, and $K' \cap \overline{\bbF}_q = \bbF_{q^\alpha}$.
\item The homomorphisms $\overline{\rho}_{\lambda_0}|_{\Gamma_{K'}}$ and $\beta^\ast \overline{R}_{\lambda_0}$ are $\hG(\bbF_{l_0})$-conjugate.
\item The homomorphisms $\overline{\phi}_{\lambda_1}|_{\Gamma_{K'}}$ and $\beta^\ast \overline{R}_{\lambda_1}$ are $\hG(\bbF_{l_1})$-conjugate.
\item The homomorphism $\phi_{\lambda_1}|_{\Gamma_{K'}}$ is everywhere unramified.
\end{itemize}
The first three points follow from Proposition \ref{prop_potential_agreement_of_residual_representations}; after possibly enlarging the field $K'$, we can ensure the fourth point also holds. Consequently $\phi_{\lambda_1}|_{\Gamma_{K'}}$ is automorphic. 

We are now going to apply our automorphy lifting theorem to $\phi_{\lambda_1}|_{\Gamma_{K'}}$. We first need to make sure that the hypotheses (i) -- (iv) at the beginning of \S \ref{sec_R_equals_B} are satisfied. The first condition ($l_1 \nmid \# W$) is satisfied by construction. The remaining conditions are satisfied because $\phi|_{\Gamma_{K'}}$ is a Coxeter parameter, and remains so after restriction to $\Gamma_{K'(\zeta_{l_1})}$, by construction. Corollary \ref{cor_automorphy_lifting} now applies, and we deduce that there exists an everywhere unramified cuspidal automorphic representation $\Pi$ of $G(\bbA_{K'})$ over $\overline{\bbQ}$ which corresponds everywhere locally to the representation $\beta^\ast R_{\lambda_1}$, hence to the representation $\beta^\ast R_{\lambda_0}$. The representation $\beta^\ast R_{\lambda_0} : \Gamma_{K'} \to \hG(\overline{\bbQ}_{\lambda_0})$ has Zariski dense image, so Lemma \ref{lem_zariski_dense_and_classically_automorphic_implies_lafforgue_automorphic} implies that $\beta^\ast R_{\lambda_0}$ is automorphic in the sense of Definition \ref{def_automorphic_galois_representation}.

On the other hand, the residual representation of $\beta^\ast R_{\lambda_0}$ is $\hG(\overline{\bbF}_{l_0})$-conjugate to that of $\overline{\rho}_{\lambda_0}|_{\Gamma_{K'}}$, which has image $\hG(\bbF_{l_0})$. We can again apply Corollary \ref{cor_automorphy_lifting} to deduce the existence of an everywhere unramified cuspidal automorphic representation $\pi$ of $G(\bbA_{K'})$ over $\overline{\bbQ}$ which corresponds everywhere locally to the representation $\rho_{\lambda_0}|_{\Gamma_{K'}}$. Since every representation in the compatible family $(\emptyset, \rho_\lambda|_{\Gamma_{K'}})$ has Zariski dense image, the theorem now follows from the existence of $\pi$ and another application of Lemma \ref{lem_zariski_dense_and_classically_automorphic_implies_lafforgue_automorphic}.
\end{proof}
\subsection{Descent and a conjectural application}

In this section, we discuss informally a conjectural application of Theorem \ref{thm_main_theorem} that links descent and the local Langlands conjecture, and which was our initial motivation for studying these problems. Let $G$ be a split semisimple group over $\bbF_q$, and let $l$ be a prime not dividing $q$.

Let $K = \bbF_q(X)$ be the function field of a smooth, geometrically connected curve, and let $v$ be a place of $K$.  If $\pi$ is a cuspidal automorphic representation of $G(\bbA_K)$, then the work of V. Lafforgue (see \S \ref{sec_summary_of_lafforgue}) associates to $\pi$ at least one continuous, $\hG$-completely reducible representation $\rho_\pi : \Gamma_K \to \hG(\overline{\bbQ}_l)$, which is unramified at those places where $\pi$ is, and which is compatible with $\pi_v$ under the unramified local Langlands correspondence at such places.

If $v$ is a place at which $\pi_v$ is ramified, then we still obtain a $\hG$-completely reducible representation $\rho_\pi|_{W_{K_v}}^\text{ss} : W_{K_v} \to \hG(\overline{\bbQ}_l)$. It is natural to expect that this representation depends only on $\pi_v$, and not on the choice of cuspidal automorphic representation $\pi$ which realizes $\pi_v$ as a local factor; and also that this representation depends only on $K_v$ as a local field, and not on its realization as a completion of the global field $K$.

These expectations have been announced by Genestier--Lafforgue \cite{Gen17}. We have already cited part of this work in Theorem \ref{thm_local_global_compatibility}, which played an essential role in the proof of our main automorphy lifting theorem (Theorem \ref{thm_R_equals_B}). If $F = \bbF_q((t))$, then this leads to a  map
\[ \text{LLC}^\text{ss}_F : \left\{ \begin{array}{cc} \text{Irreducible admissible representations} \\ \text{of } G(F), \text{ up to isomorphism} \end{array} \right\} \rightarrow  \left\{ \begin{array}{c} \hG\text{-completely reducible homomorphisms }\\W_{F} \to \widehat{G}(\overline{\bbQ}_l),\text{ up to }\widehat{G}\text{-conjugation} \end{array} \right\}, \]
which deserves to be called the semisimple local Langlands correspondence over $F$. 
The question we would like to answer is whether or not this map $\text{LLC}^\text{ss}_F$ is surjective. The most important case is whether all of the $\hG$-irreducible homomorphisms $W_F \to \hG(\overline{\bbQ}_l)$ appear in the image; one expects to be able to reduce to this case. In this case, we have the following proposition.
\begin{proposition}\label{prop_application_of_descent}
Assume Conjecture \ref{conj_existence_of_descent} below. Then the image of the map $\text{LLC}^\text{ss}_F$ contains all $\hG$-irreducible homomorphisms. 
\end{proposition}
Conjecture \ref{conj_existence_of_descent} is as follows:
\begin{conjecture}\label{conj_existence_of_descent}
Let $E/K$ be a cyclic extension of global fields, as above, and let $\rho : \Gamma_K \to \hG(\overline{\bbQ}_l)$ be a continuous, almost everywhere unramified homomorphism of Zariski dense image. Suppose that there exists a cuspidal automorphic representation $\pi_E$ of $G(\bbA_E)$ over $\overline{\bbQ}_l$ such that at each place of $E$ at which $\pi_E$ and $\rho|_{\Gamma_E}$ are unramified, they correspond under the unramified local Langlands correspondence.

Then there exists a cuspidal automorphic representation $\pi_K$ of $G(\bbA_K)$ such that at each place of $K$ at which $\pi_K$ and $\rho$ are unramified, they correspond under the unramified local Langlands correspondenec. 
\end{conjecture}
We now sketch the proof of Proposition \ref{prop_application_of_descent}. Let $\sigma_0 : W_F \to \hG(\overline{\bbQ}_l)$ be a $\hG$-irreducible homomorphism. Using similar techniques to those appearing in the proof of Proposition \ref{prop_construction_of_unramified_and_dense_compatible_system}, one can find the following:
\begin{itemize}
\item A global field $K = \bbF_q(X)$, together with a place $v_0$ and an $\bbF_q$-isomorphism $F \cong K_{v_0}$.
\item A continuous representation $\sigma : \Gamma_K \to \hG(\overline{\bbQ}_l)$ of Zariski dense image which is unramified outside $v_0$ and which satisfies $\sigma|_{W_{K_{v_0}}} \cong \sigma_0$. 
\end{itemize}
A slight generalization of Theorem \ref{thm_main_theorem} then allows us to find as well:
\begin{itemize}
\item A Galois extension $K' / K$ such that $\sigma|_{\Gamma_{K'}}$ is everywhere unramified.
\item An everywhere unramified cuspidal automorphic representation $\pi$ of $G(\bbA_{K'})$ such that at each place $v$ of $K'$, $\pi$ and $\sigma|_{\Gamma_{K'}}$ are related under the unramified local Langlands correspondence. 
\end{itemize}
Let $w_0$ be a place of $K'$ above $v_0$, and let $K_0$ denote the fixed field inside $K'$ of the decomposition group $\Gal(K'_{w_0} / K_{v_0})$. Let $u_0$ denote the place of $K_0$ below $w_0$. Then the extension $K' / K_0$ is Galois; $w_0$ is the unique place of $K'$ above $u_0$; the inclusion $K_{v_0} \to K_{0,u_0}$ is an isomorphism; and the inclusion $\Gal(K'_{w_0} / K_{0,u_0}) \to \Gal(K' /K_0)$ is an isomorphism. In particular, the extension $K' / K_0$ is soluble.

We can now repeatedly apply Conjecture \ref{conj_existence_of_descent} to the abelian layers of this soluble extension to obtain the following:
\begin{itemize}
\item A cuspidal automorphic representation $\pi_0$ of $G(\bbA_{K_0})$ such that at every place where $\pi_0$ and $\sigma|_{\Gamma_{K_0}}$ are unramified, they are related by the unramified local Langlands correspondence. 
\end{itemize}
In particular, if $\rho_0 : \Gamma_{K_0} \to \hG(\overline{\bbQ}_l)$ denotes one of the representations associated to $\pi_0$ by V. Lafforgue, then $\rho_0$ has Zariski dense image, is uniquely determined by $\pi_0$, and is $\hG(\overline{\bbQ}_l)$-conjugate to $\sigma|_{\Gamma_{K_0}}$ (see Lemma \ref{lem_zariski_dense_and_classically_automorphic_implies_lafforgue_automorphic}).

We can now conclude. The above discussion shows that the representation $\sigma|_{W_{K_{0, u_0}}}^\text{ss}$ is the image of $\pi_{0, u_0}$ under the map $ \text{LLC}^\text{ss}_{K_{0, u_0}}$. Now pulling back along the isomorphisms $F \cong K_{v_0} \cong K_{0, u_0}$ and using the identification $\sigma_0 \cong \sigma|_{W_{K_{0, u_0}}}^\text{ss}$ shows that the representation $\sigma_0$ is in the image of $\text{LLC}^\text{ss}_{F}$, as desired. 

\subsection{Existence of Whittaker models}

In this final section we sketch a variant of our main theorem. As usual, we let $X$ be a smooth, geometrically connected curve over $\bbF_q$, and let $K = \bbF_q(X)$. Let $G$ be a split semisimple group over $\bbF_q$, and fix a split maximal torus and Borel subgroup $T \subset B \subset G$. We recall that a character $\psi : N(\bbA_K) \to \overline{\bbQ}^\times$ is said to be generic if it is non-trivial on restriction to each simple root subgroup of $N(\bbA_K) / [N(\bbA_K), N(\bbA_K)] $ (cf. \cite{Jia07}).
\begin{definition}
Let $\pi$ be a cuspidal automorphic representation of $G(\bbA_K)$ over $\overline{\bbQ}$. We say that $\pi$ is globally generic if there exists an embedding $\phi : \pi \to C_{\text{cusp}}(\overline{\bbQ})$, a function $f \in \phi(\pi)$, and a generic character $\psi : N(K) \backslash N(\bbA_K) \to \overline{\bbQ}^\times$ such that the integral
\[ \int_{n \in N(K) \backslash N(\bbA_K)} f(n) \psi(n) \, dn \]
is non-zero.
\end{definition}
Globally generic automorphic representations play an important role in questions such as multiplicity one or the construction of L-functions. We are going to sketch a proof of the following result.
\begin{theorem}\label{thm_globally_generic_potential_automorphy}
Suppose given a $\hG$-compatible system $(\emptyset, (\rho_\lambda : \Gamma_K \to \hG(\overline{\bbQ}_\lambda))_\lambda)$ of continuous, everywhere unramified representations with Zariski dense image. Then there exists a finite extension $K'/ K$ and a globally generic, everywhere unramified, cuspidal automorphic representation $\Pi$ of $G(\bbA_{K'})$ over $\overline{\bbQ}$ such that $\rho|_{\Gamma_{K'}}$ and $\Pi$ are matched everywhere locally under the unramified local Langlands correspondence.
\end{theorem}
The two main points in the proof are Corollary \ref{cor_automorphy_lifting} and the following strengthening of Theorem \ref{thm_automorphy_of_Coxeter_homomorphisms}:
\begin{theorem}\label{thm_Whittaker_models_for_Coxeter_parameters}
Suppose that $G$ is simple and adjoint. Let $\phi : \Gamma_{K, \emptyset} \to \hG(\overline{\bbQ})$ be a Coxeter homomorphism, and let $\lambda$ be a place of $\overline{\bbQ}$ of residue characteristic $l \nmid q$. Suppose that $\phi(\Gamma_{K \cdot \overline{\bbF}_q})$ is contained in a conjugate of $\hT(\overline{\bbQ})$. Then we can find the following:
\begin{enumerate}
\item A coefficient field $E \subset \overline{\bbQ}_\lambda$.
\item A cuspidal automorphic representation $\pi$ of $G(\bbA_K)$ over $\overline{\bbQ}_l$ such that $\pi^{G(\widehat{\cO}_K)} \neq 0$.
\item An embedding $\phi : \pi \to C_{\text{cusp}}(\overline{\bbQ}_l)$ and a function $f \in C_\text{cusp}(G(\widehat{\cO}_K), \cO)$ such that $f$ spans $\phi(\pi)^{G(\widehat{\cO}_K)}$.
\item A generic character $\psi : N(K) \backslash N(\bbA_K) \to \overline{\bbZ}_l^\times$ such that 
\begin{equation}\label{eqn_whittaker_generic_coxeter_automorphic_form} \int_{n \in N(K) \backslash N(\bbA_K)} f(n) \psi(n) \, dn = 1. 
\end{equation}
\end{enumerate}
In particular, $\pi$ is globally generic. 
\end{theorem}
\begin{proof}
This is the same as Theorem \ref{thm_automorphy_of_Coxeter_homomorphisms}, except we now require the existence of an $\cO$-valued function $f$ satisfying the identity (\ref{eqn_whittaker_generic_coxeter_automorphic_form}). The existence of a function $f$ satisfying this condition is proved in Appendix B of this paper; the same computation has already been used in the proof of Theorem \ref{thm_automorphy_of_Coxeter_homomorphisms} to show that the function $f$ considered there is in fact non-zero.
\end{proof}
We now sketch the proof of Theorem \ref{thm_globally_generic_potential_automorphy}.
\begin{proof}[Proof of Theorem \ref{thm_globally_generic_potential_automorphy}]
We can reduce, as in at the beginning of \S \ref{sec_potential_automorphy}, to the case where $G$ is a simple adjoint group. By Proposition \ref{prop_big_image_in_compatible_systems}, we can find, after passing to an equivalent compatible system, a number field $M \subset \overline{\bbQ}$ and a set $\cL'$ of rational primes of Dirichlet density 0, all satisfying the following conditions:
\begin{itemize}
\item For each prime-to-$q$ place $\lambda$ of $\overline{\bbQ}$, $\rho_\lambda$ takes values in $\hG(M_\lambda)$.
\item If $l \nmid q$ is a rational prime split in $M$, and $l \not\in \cL'$, and $\lambda$ lies above $l$, then $\rho_{\lambda}$ has image equal to $\hG(\bbZ_{l})$. 
\end{itemize}
Let $h$ denote the Coxeter number of $\hG$. The group $\Gal(M(\zeta_{h^\infty}) / M)$ embeds naturally as a finite index subgroup of $\prod_{r | h} \bbZ_r^\times$; we fix an integer $b \geq 1$ such that it contains the subgroup of elements congruent to $1 \text{ mod } h^b$. Let $t > \# W$ be a prime such that $h^{b+1}$ divides the multiplicative order of $q \text{ mod } t$. By the Chebotarev density theorem, we can find a prime $l \nmid q$ satisfying the following conditions:
\begin{itemize}
\item $l$ splits in $M(\zeta_t)$, $l > t$, and $l \not\in \cL'$.
\item If $r | h$ is a prime, then $l \equiv 1 \text{ mod }h^b$ but $l \not\equiv 1 \text{ mod }r h^b$.
\item The group $\hG(\bbF_{l})$ is perfect.
\end{itemize}
If $r$ is a prime and $g$ is an integer prime to $r$, then we write  $o_r(g)$ for the order of the image of $g$ in $\bbF_r^\times$, as in the proof of Theorem \ref{thm_main_theorem_simply_connected_case}. Let $\alpha = o_t(q) / h$, so that $o_t(q^\alpha) = h$ and $q^\alpha \text{ mod }t$ is a primitive $h^\text{th}$ root of unity, and $\alpha$ is divisible by $h^b$. Then the degree $[ \bbF_{q^\alpha}(\zeta_{l}) : \bbF_{q^\alpha} ]$ is prime to $h$.

We can now apply Theorem \ref{thm_Whittaker_models_for_Coxeter_parameters} and similar arguments as in the proof of Lemma \ref{lem_Coxeter_parameters_rational_over_Z_l} to obtain the following:
\begin{itemize}
\item A smooth, projective, geometrically connected curve $Y$ over $\bbF_{q^\alpha}$ with function field $F = \bbF_{q^\alpha}(Y)$.
\item An everywhere unramified Coxeter parameter $\phi : \Gamma_F \to \hG(\overline{\bbQ})$ which in fact takes values in $\hG(\bbZ[\zeta_t])$.
\item For any place $\lambda$ of $\overline{\bbQ}$ above $l$ and for any finite separable extension $F' / F$ linearly disjoint from the extension of $F$ cut out by $\phi^\text{ad}$, a coefficient field $E \subset \overline{\bbQ}_{\lambda}$, a function $f : G(F') \backslash G(\bbA_{F'}) / G(\widehat{\cO}_{F'}) \to \cO$ which is matched everywhere locally with $\phi_{\lambda}$ under the unramified local Langlands correspondence, and a generic character $\psi : N(F') \backslash N(\bbA_{F'}) \to  \overline{\bbZ}_l^\times$ such that 
\[ \int_{n \in N(F') \backslash N(\bbA_{F'})} f(n) \psi(n) \, dn = 1. \]
\end{itemize}
Let $K_0 = K \cdot \bbF_{q^\alpha}$, and fix a choice of place $\lambda$ of $\overline{\bbQ}$ above $l$. We now apply Proposition \ref{prop_potential_agreement_of_residual_representations} with the following data:
\begin{itemize}
\item $H = \hG(\bbF_{l})$.
\item $\varphi = \overline{\phi}_{\lambda}$.
\item $\psi = \overline{ \rho}_{\lambda}|_{\Gamma_{K_0}}$.
\item $L / F$ is the extension cut out by $\overline{\phi}_{\lambda}$.
\end{itemize}
In order to avoid confusion, we note that the roles of $F$ and $K$ here are reversed relative to their roles in the statement of Proposition \ref{prop_potential_agreement_of_residual_representations}.
We obtain a finite Galois extension $F' / F$ and an $\bbF_{q^\alpha}$-embedding $\beta : K_0 \hookrightarrow F'$ satisfying the following conditions:
\begin{itemize}
\item The extension $F' / F$ is linearly disjoint from $L / F$, and $F' \cap \overline{\bbF}_q = \bbF_{q^\alpha}$.
\item The homomorphisms $\overline{\phi}_{\lambda}|_{\Gamma_{F'}}$ and $\beta^\ast \overline{\rho}_{\lambda}$ are $\hG(\bbF_{l})$-conjugate.
\end{itemize}
Let $E \subset \overline{\bbQ}_{\lambda}$, $f : G(F') \backslash G(\bbA_{F'}) / G(\widehat{\cO}_{F'}) \to \cO$, and $\psi : N(F') \backslash N(\bbA_{F'}) \to  \overline{\bbZ}_l^\times$ be the objects associated to $\phi_{\lambda}|_{\Gamma_{F'}}$ above. The representation $\beta^\ast \rho_{\lambda}$ takes values in $\hG(\bbZ_{l})$ and has residual representation conjugate to the Galois Coxeter homomorphism $\overline{\phi}_{\lambda}|_{\Gamma_{F'}}$. 

Let $U = G(\widehat{\cO}_{F'})$. There is a unique maximal ideal $\ffrm \subset \cB(U, \cO)$ such that $f \in C_{\text{cusp}}(U, \cO)_\ffrm$; indeed, this follows from Proposition \ref{prop_basic_properties_of_abstract_Coxeter_homomorphisms}. We can then apply Theorem \ref{thm_R_equals_B} to deduce that the natural map $R_{\overline{\phi}_{\lambda}|_{\Gamma_{F'}}, \emptyset} \to \cB(U, \cO)_\ffrm$ is an isomorphism, and that both of these rings are finite flat complete intersection $\cO$-algebras. It then follows from Proposition \ref{prop_triviality_of_tangent_space_in_generic_fibre_of_galois_deformation_ring} that they are even reduced, hence $\cB(U, \cO)_\ffrm[1/l]$ is an \'etale $E$-algebra.

It follows from Proposition \ref{prop_basic_properties_of_abstract_Coxeter_homomorphisms} and these observations that there is a unique minimal prime ideal $\frp \subset \cB(U, \cO)_\ffrm$ such that $f \in C_{\text{cusp}}(U, \cO)[\frp]$. We can now apply Corollary \ref{cor_automorphy_lifting} to deduce the existence of a function $f' \in C_{\text{cusp}}(U, \overline{\bbQ}_{\lambda})$ such that
\[ \int_{n \in N(F') \backslash N(\bbA_{F'})} f'(n) \psi(n) \neq 0 \]
and such that $f'$ is matched everywhere locally with $\beta^\ast \rho_{\lambda}$ under the unramified local Langlands correspondence. Since the space $C_\text{cusp}(U, \overline{\bbQ}_\lambda)$ with its Hecke action is defined over $\bbQ$, we can even assume that $f'$ lies in $C_{\text{cusp}}(U, \overline{\bbQ})$. The proof of the theorem is now complete on taking $\pi$ to be the cuspidal automorphic representation generated by the function $f'$, and $K'$ to be $F'$, viewed as an extension of $K$ via $\beta|_{K} : K \hookrightarrow F'$.
\end{proof}

\section*{Appendix A. Cuspidality of Eisenstein series, by D. Gaitsgory}
\addcontentsline{toc}{section}{Appendix A. Cuspidality of Eisenstein series, by D. Gaitsgory}
\bigskip

\noindent{\bf 1.}
We assume being in the setting of \cite[Theorem 2.2.8]{Bra02}. We are given a $\check{T}$-local system 
$\ol{E}_{\check{T}}$ on $\ol{X}$, such that the induced $\check{G}$-local system
$$\ol{E}_{\check{G}}:=\on{Ind}^{\check{G}}_{\check{T}}(\ol{E}_{\check{T}})$$
is equipped with a Weil structure, and as such is \emph{irreducible}. 

\medskip

According to \cite[Proposition 2.2.9]{Bra02}, $\ol{E}_{\check{T}}$ is \emph{regular} (i.e., 
the $\bbG_m$-local system $\alpha(\ol{E}_{\check{T}})$ is \emph{non-constant} for all roots $\alpha$ of $\check{G}$), and 
there exists an element $w\in W$ and its lift $\wt{w}\in N(\check{T})$ such that 
\begin{equation} \label{e:w}
Fr^*(\ol{E}_{\check{T}})\simeq  \ol{E}^w_{\check{T}},
\end{equation} 
and the identification 
$$\on{Ind}^{\check{G}}_{\check{T}}(\ol{E}_{\check{T}})\simeq \ol{E}_{\check{G}}\simeq Fr^*(\ol{E}_{\check{G}})\simeq 
\on{Ind}^{\check{G}}_{\check{T}}(Fr^*(\ol{E}_{\check{T}}))\overset{\text{\eqref{e:w}}}\simeq \on{Ind}^{\check{G}}_{\check{T}}(\ol{E}^w_{\check{T}})$$
is induced by $\wt{w}$. 

\bigskip

\noindent{\bf 2.}  Consider the object
$$\on{Aut}_{\ol{E}_{\check{G}}}:=\on{Eis}^G_T(\on{Aut}_{\ol{E}_{\check{T}}}).$$

We define the structure of Weil sheaf on $\on{Aut}_{\ol{E}_{\check{G}}}$ by
\begin{equation} \label{e:Weil}
Fr^*(\on{Eis}^G_T(\on{Aut}_{\ol{E}_{\check{T}}}))\simeq \on{Eis}^G_T(\on{Aut}_{Fr^*(\ol{E}_{\check{T}})})\overset{\text{\eqref{e:w}}}\simeq
\on{Eis}^G_T(\on{Aut}_{\ol{E}^w_{\check{T}}})\simeq \on{Eis}^G_T(\on{Aut}_{\ol{E}_{\check{T}}}),
\end{equation} 
where the last isomorphism is the Functional Equation for Eisenstein series (\cite[Theorem 2.2.4]{Bra02}). 

\medskip

Our goal is to show that the spherical automorphic function corresponding to $\on{Aut}_{\ol{E}_{\check{G}}}$ with 
the above Weil structure is cuspidal. 

\bigskip

\noindent{\bf 3.}  First, we claim that we can assume that $\ol{E}_{\check{T}}$ is \emph{strongly regular}, i.e., $\ol{E}^{w'}_{\check{T}}$
is non-isomorphic to $\ol{E}_{\check{T}}$ for any $w'\in W$. 

\medskip

Indeed, embed $G$ into $G_1:=G\overset{Z(G)}\times T'$, where $T'$ is a torus. There exists a $\check{G}_1$-local system $\ol{E}_{\check{G}_1}$,
equipped with a Weil structure so that $\ol{E}_{\check{G}}$ (equipped with its own Weil structure) is induced from $E_{\check{G}_1}$ 
by means of the homomorphism $\check{G}_1\to \check{G}$. 

\medskip

Let
$\ol{E}_{\check{T}_1}$ be the corresponding $\check{T}_1$-local system over $\ol{X}$. We claim that $\ol{E}_{\check{T}_1}$ is automatically
strongly regular. This follows from the fact that the derived group of $\cG_1$ is simply connected. Now, if we know the
cuspidality assertion for $G_1$, the cuspidality assertion for $G$ follows. 

\bigskip

\noindent{\bf 4.}  For a parabolic $P$ with Levi quotient $M$ let
$$\on{CT}^G_M:\on{Sh}(\ol{\Bun}_G)\to \on{Sh}(\ol{\Bun}_M)$$
be the corresponding constant term functor, i.e., 
$$\on{CT}^G_M=(\frq_P)_!\circ (\frp_P)^*.$$

\medskip

Consider the object
\begin{equation} \label{e:CT of Eis}
\on{CT}^G_M(\on{Eis}^G_T(\on{Aut}_{\ol{E}_{\check{T}}}))\in  \on{Sh}(\ol{\Bun}_M),
\end{equation} 
with the Weil structure induced by the Weil structure on $\on{Eis}^G_T(\on{Aut}_{\ol{E}_{\check{T}}})$, given by 
\eqref{e:Weil}.

\medskip

We need to show that the function on $\Bun_M({\mathbb F}_q)$ corresponding to \eqref{e:CT of Eis} is zero.

\bigskip

\noindent{\bf 5.} Let 
$$'\!\Eis^M_T: \on{Sh}(\ol{\Bun}_T)\to \on{Sh}(\ol{\Bun}_M)$$
be the \emph{non-compactified} Eisenstein series functor (see \cite[Sect. 2.2.10]{Bra02}). 

\medskip

The standard calculation of the constant term applied to Eisenstein series (see \cite[Proposition 10.8]{BG2} for the case $P=B$) says that
the functor 
\begin{equation} \label{e:CT of Eis funct}
\on{CT}^G_M\circ \on{Eis}^G_T
\end{equation}  
can be canonically written as an extension of functors 
$$\cS^{w'}: \on{Sh}(\ol{\Bun}_T)\to \on{Sh}(\ol{\Bun}_M),$$
where each $\cS^{w'}$ is an extension of functors 
$$\on{Sh}(\ol{\Bun}_T) \overset{\text{some Hecke functor}}\longrightarrow 
\on{Sh}(\ol{\Bun}_T) \overset{w'}\to \on{Sh}(\ol{\Bun}_T)  \overset{'\!\Eis^M_T}\longrightarrow 
\on{Sh}(\ol{\Bun}_M),$$
and $w'$ runs over a set of representatives of $W_M\backslash W$. 

\bigskip

\noindent{\bf 6.} Now, the assumption that $\ol{E}_{\check{T}}$ is strongly regular implies that 
\begin{equation} \label{e:orthogonality}
R\Hom(\cS^{w'_1}(\on{Aut}_{\ol{E}_{\check{T}}}),\cS^{w'_2}(\on{Aut}_{\ol{E}^{w'}_{\check{T}}}))=0
\end{equation} 
unless $w'_1=w'_2\cdot w\, \on{mod}\, W_M$, see \cite[Proposition 10.6]{BG2}.

\medskip 

In particular, taking in \eqref{e:orthogonality} $w'_1=1$, we obtain that the object \eqref{e:CT of Eis} is canonically a direct sum:
\begin{equation} \label{e:decom}
\on{CT}^G_M(\on{Eis}^G_T(\on{Aut}_{\ol{E}_{\check{T}}}))\simeq \underset{w'}\oplus\, \cS^{w'}(\on{Aut}_{\ol{E}_{\check{T}}}).
\end{equation} 

\bigskip

\noindent{\bf 6.} In terms of the isomorphism \eqref{e:decom} the Weil structure on \eqref{e:CT of Eis} 
is an isomorphism
$$Fr^*(\underset{w'}\oplus\, \cS^{w'}(\on{Aut}_{\ol{E}_{\check{T}}}))\overset{\text{\eqref{e:w}}}\simeq 
\underset{w'}\oplus\, \cS^{w'}(\on{Aut}_{\ol{E}^w_{\check{T}}})\simeq \underset{w'}\oplus\, \cS^{w'}(\on{Aut}_{\ol{E}_{\check{T}}}).$$

\medskip

Applying \eqref{e:orthogonality} again, we obtain that this isomorphism is given by a collection of isomorphisms
$$Fr^*(\cS^{w'}(\on{Aut}_{\ol{E}_{\check{T}}}))\to \cS^{w'_1}(\on{Aut}_{\ol{E}_{\check{T}}}),$$
where $w'_1$ is such that $w'_1=w'\cdot w\, \on{mod}\, W_M$.

\medskip

Note, however, that for no $w'$ do we have $w'=w'_1$. Indeed, this would mean that $w$ belongs to a subgroup
conjugate to $W_M$, contradicting the irreducibility of $\ol{E}_{\check{G}}$ as a Weil local system.  

\bigskip

\noindent{\bf 7.}
Now, the required vanishing of the function follows from the next general claim: let $\cY$ be a stack, and
let $\cF$ be an object of $\on{Sh}(\cY)$, equipped with a Weil structure. Assume that $\cF$ is written
as a direct sum
\begin{equation} \label{e:decom abs}
\cF=\underset{i\in I}\oplus\, \cF_i,
\end{equation} 
where $I$ is some finite set. 

\medskip

Assume that in terms of \eqref{e:decom abs}, the Weil structure on $\cF$ corresponds to a system of isomorphisms
$$Fr^*(\cF_i)\to \cF_{\phi(i)},$$
where $\phi:I\to I$ is an automorphism of $I$.  

\medskip

Assume that $\phi(i)\neq i$ for all $i\in I$. Then the function on $\cY({\mathbb F}_q)$ corresponding to $\cF$ vanishes.

\section*{Appendix B. Non-vanishing of Whittaker coefficients, by D. Gaitsgory}
\addcontentsline{toc}{section}{Appendix B. Non-vanishing of Whittaker coefficients, by D. Gaitsgory}
\noindent{\bf Temporary notation:} For a stack/scheme $\cY$ over $\bbF_q$, we will denote by $\overline{\cY}$ its base change to $\overline{\bbF}_q$.
We will denote by $\Sh(\overline{\cY})$ the derived category of sheaves on $\overline{\cY}$, and by $\Sh(\cY)$ the category of
objects in $\Sh(\overline{\cY})$, equipped with a Weil structure, i.e., pairs $(\cF\in \Sh(\overline{\cY}),\alpha:\on{Fr}_\cY^*(\cF)\simeq \cF)$). 

\medskip

Note that for $\cY=\on{pt}:=\Spec(\bbF_q)$, the category $\Sh(\on{pt})$ is that of objects of $\Vect$ (=graded $\ol\bbQ_\ell$-vector spaces) 
equipped with an automorphism. 

\bigskip

\noindent{\bf 1.} 
We assume being in the setting of \cite{Bra02}. We are given a $\check{T}$-local system 
$\ol{E}_{\check{T}}$ on $\ol{X}$, and let 
$$\on{Aut}_{\ol{E}_{\check{T}}}\in \on{Sh}(\Bun_T)$$
be the 1-dimensional local system that corresponds to it by geometric Class Field Theory.
Consider the object
$$\on{Aut}_{\ol{E}_{\check{G}}}:=\on{Eis}^G_T(\on{Aut}_{\ol{E}_{\check{T}}})\in \Sh(\ol\Bun_G).$$

Assume that $\on{Aut}_{\ol{E}_{\check{G}}}$ is equipped with a Weil structure. Our goal is to prove:

\medskip

\noindent{\bf Theorem.}  {\it The function on $\Bun_G(\bbF_q)$ that corresponds to $\on{Aut}_{\ol{E}_{\check{G}}}$
is non-zero.}

\medskip

We will prove Theorem 1 by showing that its first Whittaker coefficient is non-zero. 

\bigskip

\noindent{\bf 2.} Pick a square root $\omega^{\frac{1}{2}}_X$ of the canonical line bundle $\omega_X$ on $X$. 
Let $\rho(\omega_X)$ denote the $T$-bundle on $X$ induced from $\omega^{\frac{1}{2}}_X$ using the cocharacter
$2\rho:\bbG_m\to T$. 

\medskip

Let $\Bun_{N^-}^{\rho(\omega_X)}$ be the stack 
$$\Bun_{B^-}\underset{\Bun_T}\times \{\rho(\omega_X)\};$$
this is a twisted version of $\Bun_{N^-}$. We have a canonical map
$$B^-\to \underset{i\in I}\Pi\, B^-_i,$$
where $I$ is the set of vertices of the Dynkin diagram, and $B^-_i$ the negative Borel subgroup of the adjoint quotient of
the corresponding subminimal Levi. The above homomorphism defines a map of stacks
$$\Bun_{N^-}^{\rho(\omega_X)}\to \underset{i\in I}\Pi\, \Bun_{N^-_i}^{\rho_i(\omega_X)}.$$

\medskip

Note that each $\Bun_{N^-_i}^{\rho_i(\omega_X)}$ is the stack classifying extensions 
$$0\to \omega_X\to \cE_i\to \cO_X\to 0,$$
and hence admits a canonical map to $\bbA^1$. Let $\on{A-Sch}$ denote the Artin-Schreier sheaf on $\bbA^1$. 
Let $\chi$ denote the *-pullback to $\Bun_{N^-}^{\rho(\omega_X)}$ of $\on{A-Sch}$ along the map
$$\Bun_{N^-}^{\rho(\omega_X)}\to \underset{i\in I}\Pi\, \Bun_{N^-_i}^{\rho_i(\omega_X)} \to \underset{i\in I}\Pi\, \bbA^1
\overset{\on{sum}}\longrightarrow \bbA^1.$$

\medskip

We define the functor
$$\Whit:\Sh(\Bun_G)\to \Sh(\on{pt})$$
to be the composition
$$\Sh(\Bun_G) \to \Sh(\Bun_{N^-}^{\rho(\omega_X)}) \overset{-\otimes \chi}\longrightarrow \Sh(\Bun_{N^-}^{\rho(\omega_X)}) \to \Sh(\on{pt}),$$
where the first arrow is *-pullback with respect to the natural projection $\Bun_{N^-}^{\rho(\omega_X)}\to \Bun_G$, the the last arrow
is the functor of cohomology with compact supports. We will use the same symbol $\Whit$ to denote the corresponding functor
$$\Sh(\ol\Bun_G)\to \Vect$$
(i.e., when we ignore the Weil structure). 

\medskip

It is clear that for $\cF\in \Sh(\Bun_G)$ and the corresponding function $f$ on $\Bun_G(\bbF_q)$, the first Whittaker coefficient of $f$
equals the trace of the Frobenius on $\Whit(\cF)$. 

\medskip

Hence, Theorem 1 follows from the next result:

\medskip

\noindent{\bf Theorem.}  {\it There is an isomorphism
$$\Whit(\on{Aut}_{\ol{E}_{\check{G}}})\simeq \ol\bbQ_\ell,$$
up to a cohomological shift.} 

\bigskip

\noindent{\bf Notation change:} From now on we will work over $\ol\bbF_q$, and we will omit putting the bar over the objects
involved. So, for example, from now on $X$ is a curve over $\ol\bbF_q$, and $E_{\check{T}}$ is a $\check{T}$-local system
on $X$, etc. 

\bigskip

\noindent{\bf 2'.} We will in fact prove the following generalization of Theorem 2:

\medskip

\noindent{\bf Theorem.}  {\it The functor $$\Whit\circ \on{Eis}^G_T:\Sh(\Bun_T)\to \Vect$$ identifies canonically with the functor of *-fiber 
at the point $\rho(\omega_X)\in \Bun_T$.} 

\bigskip

\noindent{\bf 3.} Consider the diagram
\begin{equation} \label{e:sq}
\CD
\Bun_{N^-}^{\rho(\omega_X)}\underset{\Bun_G}\times \BunBb   @>{'\frp^-}>>  \BunBb  @>{\ol\frq}>>   \Bun_T \\
@V{'\ol\frp}VV  @VV{\ol\frp}V  \\
\Bun_{N^-}^{\rho(\omega_X)}  @>{\frp^-}>>  \Bun_G.
\endCD
\end{equation} 

By base change, the functor $\Whit\circ \on{Eis}^G_T$ is calculated as follows
$$\cF\mapsto 
H_c\left(\Bun_{N^-}^{\rho(\omega_X)}\underset{\Bun_G}\times \BunBb,(\ol\frq\circ {}'\frp^-)^*(\cF)\otimes 
({}'\frp^-)^*(\IC_{\BunBb})\otimes ({}'\ol\frp)^*(\chi)\right)[-\dim(\Bun_T)].$$

\medskip

The fiber product $\Bun_{N^-}^{\rho(\omega_X)}\underset{\Bun_G}\times \BunBb$ admits a decomposition into locally closed substacks
$$(\Bun_{N^-}^{\rho(\omega_X)}\underset{\Bun_G}\times \BunBb)^w,\quad w\in W$$
indexed by the relative position of the $N^-$-reduction and the $B$-reduction of a given $G$-bundle over the generic point of $X$.

\medskip

We have the following basic assertion:

\medskip

\noindent{\bf Proposition.}  {\it For any $w\neq 1$ and $\cF\in \Sh(\Bun_T)$, the cohomology
$$H_c\left((\Bun_{N^-}^{\rho(\omega_X)}\underset{\Bun_G}\times \BunBb)^w,(\ol\frq\circ {}'\frp^-)^*(\cF)\otimes 
({}'\frp^-)^*(\IC_{\BunBb}\otimes ({}'\ol\frp)^*(\chi)\right)$$
vanishes.} 

\medskip

The proof is obtained by repeating the argument of \cite[Sect. 10.9]{BG2}.

\bigskip

\noindent{\bf 4.} Denote 
$$\cZ:=(\Bun_{N^-}^{\rho(\omega_X)}\underset{\Bun_G}\times \BunBb)^1;$$
this is the open stratum, where the two reductions are mutually transversal. The stack $\cZ$ is known to be a scheme and is called
the \emph{Zastava} space, see \cite[Sect. 2.2]{BFGM}.

\medskip

From Proposition 3, we obtain that $\Whit\circ \on{Eis}^G_T$ can calculated be calculated by
$$\cF\mapsto H_c\left(\cZ,(\ol\frq\circ {}'\frp^-)^*(\cF)\otimes 
({}'\frp^-)^*(\IC_{\BunBb})\otimes ({}'\ol\frp)^*(\chi)\right)[-\dim(\Bun_T)],$$
where by a slight abuse of notation we continue to denote by ${}'\frp^-$ and $'\ol\frp$ the maps from \eqref{e:sq}, restricted to $\cZ$. 

\medskip

The scheme $\cZ$ splits into connected components, indexed by the elements of $\Lambda^{\on{pos}}$, the semi-group
of coweights of $G$ equal to non-negative linear combinations of positive simple coroots; for $\lambda\in \Lambda^{\on{pos}}$,
let $\cZ^\lambda$ denote the corresponding connected component. 

\medskip

Let $X^\lambda$ denote the corresponding partially
symmetrized power of the curve. I.e., if $\lambda=\underset{i\in I}\Sigma\, n_i\cdot \alpha_i$, then
$$X^\lambda=\underset{i\in I}\Pi\, X^{(n_i)}.$$

\medskip

According to \cite[Sect. 2.2]{BFGM}, there is a canonical map
$$\pi^\lambda:\cZ^\lambda\to X^\lambda,$$
such that the composition
$$\ol\frq\circ {}'\frp^-:\cZ^\lambda\to \Bun_T$$
equals 
$$\cZ^\lambda\overset{\pi^\lambda}\longrightarrow X^\lambda \overset{\on{AJ}}\longrightarrow \Bun_T,$$
where $\on{AJ}$ is a version of the Abel-Jacobi map that sends 
$$D\in X^\lambda \,\mapsto\, \rho(\omega_X)(-D).$$

\medskip

In addition, according to \cite{BFGM}, we have:
$$({}'\frp^-)^*(\IC_{\BunBb})\simeq \IC_{\cZ}[\dim(\Bun_G)-\dim(\Bun_{N^-}^{\rho(\omega_X)})].$$

Applying the projection formula, we obtain that $\Whit\circ \on{Eis}^G_T(\cF)$ is the direct sum over $\lambda\in \Lambda^{\on{pos}}$
of the expressions
\begin{equation} \label{e:direct summand}
H_c(X^\lambda,\pi^\lambda_!\circ ({}'\ol\frp)^*(\chi)\otimes \on{AJ}^*(\cF))[\dim(\Bun_G)-\dim(\Bun_{N^-}^{\rho(\omega_X)})-\dim(\Bun_T)].
\end{equation} 

\bigskip

\noindent{\bf 5.} We now apply the following result of \cite[Theorem 3.4.1]{Ras}:

\medskip

\noindent{\bf Theorem.}  {\it For $\lambda\neq 0$, the object 
$$\pi^\lambda_!\circ ({}'\ol\frp)^*(\chi)\in \Sh(X^\lambda)$$
is zero.} 

\medskip

Thus, we obtain that among the summands in \eqref{e:direct summand}, only the one with $\lambda=0$ is non-zero. In this case $\cZ^0=\on{pt}$,
and the assertion of Theorem 2' follows.  

\newcommand{\etalchar}[1]{$^{#1}$}
\def\polhk#1{\setbox0=\hbox{#1}{\ooalign{\hidewidth
  \lower1.5ex\hbox{`}\hidewidth\crcr\unhbox0}}}

\end{document}